\documentclass[11pt, reqno]{amsart}
\usepackage[english]{babel}
\usepackage[T1]{fontenc}
\usepackage[utf8]{inputenc}

\usepackage{mathrsfs}
\usepackage{amsmath}
\usepackage{amssymb}
\usepackage{amsfonts}
\usepackage{amsthm}
\usepackage{amscd}%
\usepackage{graphicx}
\usepackage{setspace}
\usepackage{caption}
\usepackage{subcaption}
\usepackage[T1]{fontenc}
\usepackage{geometry}
\usepackage{mathtools}
\usepackage{faktor}
\usepackage[dvipsnames]{xcolor}
\usepackage[colorlinks,citecolor=NavyBlue,urlcolor=NavyBlue]{hyperref}
\usepackage{cleveref}
\usepackage[all]{xy}
\usepackage{enumitem}
\usepackage{comment}
\usepackage{ dsfont }
\usepackage{ stmaryrd }
\usepackage[backend=biber,style=alphabetic]{biblatex}
\newcounter{ustep}

\newcommand{\UStep}[1][]{%
	\refstepcounter{ustep}
	\par\noindent\textbf{Step U.I.~\theustep.}%
	\ifx\relax#1\relax\else\ \textbf{#1}\fi
	\
	}

\newcommand{\stepref}[1]{%
	\hyperref[#1]{\textcolor{red}{(U.I.\ref*{#1})}}%
}

\newcounter{estep}

\newcommand{\EStep}[1][]{%
	\refstepcounter{estep}
	\par\noindent\textbf{Step H.U.I.~\theestep.}%
	\ifx\relax#1\relax\else\ \textbf{#1}\fi
	\
}

\newcommand{\estepref}[1]{%
	\hyperref[#1]{\textcolor{red}{(H.U.I.\ref*{#1})}}%
}

\newcounter{sstep}

\newcommand{\SStep}[1][]{%
	\refstepcounter{sstep}
	\par\noindent\textbf{Step S.I.~\thesstep.}%
	\ifx\relax#1\relax\else\ \textbf{#1}\fi
	\
}

\newcommand{\sstepref}[1]{%
	\hyperref[#1]{\textcolor{red}{(S.I.\ref*{#1})}}%
}
\makeatletter
\def\saveenum{\xdef\@savedenum{\the\c@enumi\relax}}
\def\resetenum{\global\c@enumi\@savedenum}
\makeatother

\numberwithin{equation}{section}

\newcommand{\bbR}{\mathbb{R}}

\newcommand{\bbZ}{\mathbb{Z}}

\newcommand{\bbN}{\mathbb{N}}

\newcommand{\cC}{\mathcal{C}}

\newcommand{\hx}{\hat{x}}
\newcommand{\hy}{\hat{y}}
\newcommand{\hf}{\hat{f}}
\newcommand{\Es}{E^s}

\newcommand{\Eu}{E^u}

\newcommand{\Ws}{W^s}

\newcommand{\Wu}{W^u}

\newcommand{\inv}{^{-1}}

\newcommand{\cO}{\mathcal{O}}
\newcommand{\intert}{\cap\kern-0.7em|\kern0.7em}

\newcommand{\bbP}{\mathbb{P}}

\newcommand{\diam}{\operatorname{diam}}
\newcommand{\id}{\operatorname{id}}
\newcommand{\supp}{\operatorname{supp}}

\newcommand{\hmu}{\hat{\mu}}
\newcommand{\hnu}{\hat{\nu}}
\newcommand{\hz}{\hat{z}}

\newcommand{\hO}{\hat{\mathcal{O}}}

\newcommand{\norm}[1]{
	\lVert #1 \rVert
}

\newtheorem{thm}{Theorem}[section]
\newtheorem{cor}[thm]{Corollary}
\newtheorem{lemma}[thm]{Lemma}
\newtheorem{prop}[thm]{Proposition}
\newtheorem{question}{\textbf{Question}}

\newtheorem*{conjecture1}{\textbf{Conjecture 1}}
\newtheorem*{conjecture2}{\textbf{Conjecture 2}}

\newtheorem{claim}[thm]{\textbf{Claim}}
\newtheorem{defn}[thm]{Definition}
\newtheorem{hyp}{Hypothesis}

\newtheorem{mainthm}{Theorem}

\newtheorem*{maincor}{Corollary}

\theoremstyle{remark}
\newtheorem*{remark} {\textbf{Remark}}

\newcommand{\NUH}{\Lambda_{\chi, \epsilon, l}}
\newcommand{\hNUH}{\hat{\Lambda}_{\chi, \epsilon, l}}

\newcommand{\simh}{\overset{h}{\sim}}

\newcommand{\parbreak}{%
	\par\vspace{1em}\noindent
}

\newcounter{mysubequations}

\title{Finiteness of measures of maximal entropy\\ for smooth saddle surface endomorphisms}
\author{Matéo Ghezal}
\date{}

\begin{document}

\begingroup
\renewcommand\thefootnote{}
\addtocounter{footnote}{-1}
\endgroup

\begingroup
\renewcommand\thefootnote{}
\addtocounter{footnote}{-1}
\endgroup

\begin{abstract}
	We show that $\cC^{\infty}$ local diffeomorphisms of closed surfaces whose topological entropy is larger than the logarithm of their degree admit a finite number of ergodic measures of maximal entropy. To do this, we construct families of rectangles, with a nice geometry, displaying a Markov property. We then analyze the behavior of the iterates of unstable curves intersecting these rectangles, using Yomdin theory.
\end{abstract}

\maketitle

\tableofcontents

\section{Introduction}
\subsection{Main result on measures of maximal entropy}
Throughout this paper, $M$ denotes a closed surface: a compact two-dimensional $\cC^{\infty}$ Riemannian manifold without boundary. Let $f:M\rightarrow M$ be a local diffeomorphism: a $\cC^1$ map such that $d_xf$, the tangent map at $x$, is an isomorphism for any $x \in M$. Denote by $\deg(f)$ the topological degree of $f$. Recall that the topological entropy $h_{top}(f)$ is related to the measured entropy $h(f,\mu)$ of $f$-invariant Borel probability measures $\mu$ by the variational principle (see Section~\ref{sec:EntropyYomdin} for more precise definitions):
\begin{equation*}
	h_{top}(f) = \sup_{\mu \in \bbP(f)}h(f,\mu) = \sup_{\mu \in \bbP_e(f)} h(f,\mu)
\end{equation*}
where $\bbP(f)$ is the set of $f$-invariant Borel probability measures and $\bbP_e(f)$ the set of ergodic $\mu \in \bbP(f)$. We say that $\mu \in \bbP_e(f)$ is a measure of maximal entropy, also called m.m.e., if $h_{top}(f)=h(f,\mu)$.

In~\cite{MisuirewiczPrzytycki}, Misiurewicz and Przytycki have shown the following link between the entropy of $f$ and its topological degree:
\begin{equation*}
	h_{top}(f)\geq \log \deg(f).
\end{equation*}
In a recent major work~\cite{buzzi2022measures}, Buzzi, Crovisier, and Sarig have shown that for $\cC^{\infty}$ systems, if $\deg(f)=1$ and $h_{top}(f)>0$, then $f$ admits a finite number of measures of maximal entropy. In the present work, we show the following generalization to any degree.

\begin{mainthm}
	Let $M$ be a closed surface and let $f:M \rightarrow M$ be a $\cC^{\infty}$ local diffeomorphism. If $h_{top}(f) > \log \deg(f)$, then $f$ admits a finite number of ergodic measures of maximal entropy.
	\label{mainthm:A}
\end{mainthm}

Note that this theorem is not true without the assumption $h_{top}(f)>\log\deg(f)$. Indeed, consider the product of the identity and a uniformly hyperbolic map on the circle (for example $x\mapsto 2x$). This gives trivial counter-examples since any fiber will carry a measure of maximal entropy.

The hypothesis $h_{top}(f)>\log\deg(f)$ forces the m.m.e.'s to have one positive and one negative Lyapunov exponent. When $h_{top}(f)=\log\deg(f)$, this is no longer the case, but we believe that our methods could apply to systems displaying enough hyperbolic properties. This will be discussed later in sub-Section~\ref{subsec:Intro:Comments}.

Note that Buzzi, Crovisier, and Sarig obtained uniqueness of the measure of maximal entropy for $\cC^{\infty}$ transitive surface diffeomorphisms with positive entropy. Here, we do not know if transitivity gives uniqueness. Even under stronger hypotheses like strong transitivity, (i.e. assuming that for any $x \in M$, the set $\cup_{n\geq0}f^{-n}(x)$ is dense), we do not know if such a local diffeomorphism could admit more than one m.m.e.

\subsection{Homoclinic classes}
Given a dynamical system, a natural way to try to understand the dynamics is to decompose the system into invariant elementary pieces. A famous example is Smale's spectral decomposition for Axiom A diffeomorphisms~\cite{smale1967differentiable}. In a more general setting, a good candidate for such elementary pieces are homoclinic classes. Buzzi, Crovisier, and Sarig gave a generalization of Smale's theorem to the non-uniformly hyperbolic setting when $f$ is a surface diffeomorphism in~\cite{buzzi2022measures}. We discuss here partial generalizations of their result on \textit{measured homoclinic classes} for any topological degree.

Let us first introduce the natural extension of $f$. Define the set:
\begin{equation*}
	M_f = \big\{(x_n)_{n\in \bbZ} \in M^{\bbZ}, \ f(x_n) = x_{n+1} \ \forall n \in \bbZ\big\}.
\end{equation*}
Let $\hf : M_f \rightarrow M_f$ be the homeomorphism defined by:
\begin{equation*}
	\hf\big((x_n)_{n\in \bbZ}\big) = (f(x_n))_{n\in \bbZ} = (x_{n+1})_{n \in \bbZ}.
\end{equation*}
Let $\pi:M_f\rightarrow M$ be the natural projection on the $0$-th coordinate. The system $(M_f,\hf)$ is an extension of $(M,f)$ in the sense that $\pi\circ\hf = f\circ\pi$.

Let us also recall the notion of measured homoclinic classes. It was first introduced in the case of periodic orbits in~\cite{NewhouseHypLimSet}. The measured version was defined in~\cite{buzzi2022measures}. In the case of endomorphisms, these notions were revisited in~\cite{lima2024measures}. See Section~\ref{sec:HomClasses} for a more precise introduction. Assume here that $f:M\rightarrow M$ is a $\cC^r$, $r>1$ local diffeomorphism. Let $\mu$ be an ergodic $f$-invariant measure. There exists a unique $\hf$-invariant lift to $M_f$, which we denote by $\hmu$. We say that $\mu$ is \textit{hyperbolic of saddle type} if $\mu$ has one positive and one negative Lyapunov exponent. In that case, $\hmu$-almost every $\hx$ has a well-defined stable and unstable manifold $W^s(\hx),W^u(\hx)\subset M_f$. Note that a point $y$ belongs to $\Wu(\hx)$ if it has a lift $\hy\in M_f$ such that the past of $\hy$ converges to the past of $\hx$. We define Smale's pre-order on the set of ergodic hyperbolic measures of saddle type as follows. Consider two such measures $\mu_1,\mu_2$. We write $\mu_1\preceq\mu_2$ if there exist two measurable sets $\hat{A}_1,\hat{A}_2\subset M_f$ with $\hmu_i(\hat{A}_i)>0$ such that for any $(\hx_1,\hx_2)\in \hat{A}_1\times\hat{A}_2$ the sets $W^u(\hx_1)$ and $W^s(\hx_2)$ have a point of transverse intersection. This is equivalent as requiring that there exists two integers $n,m\geq 0$ such that $f^{n+m}(\Wu_{loc}(\hf^{-m}(\hx_1)))$ and $\Ws_{loc}(\hf^n(\hx_2))$ have a point of transverse intersection, where $W^{s/u}_{loc}(\hx)$ is the classical Pesin local unstable/stable manifold at $\hx$. Note that in the case of ergodic hyperbolic measure of saddle type, the Pesin invariant manifolds are which embedded curves. We say that $\mu_1$ and $\mu_2$ are \textit{homoclinically related}, and we denote it by $\mu_1\simh\mu_2$, if $\mu_1\preceq\mu_2$ and $\mu_2\preceq\mu_1$. We denote the measured homoclinic class of some ergodic measure $\mu$ by $\mathcal{M}(\mu)$. Lima, Obata, and Poletti proved in~\cite[Proposition 4.3]{lima2024measures} that the homoclinic relation is an equivalence relation on the set of ergodic hyperbolic measures of saddle type. Their argument is adapted from~\cite[Proposition 2.11]{buzzi2022measures}. The equivalence classes by the homoclinic relation are called the measured homoclinic classes. Buzzi, Crovisier, and Sarig proved in~\cite{buzzi2022measures} that the number of measured homoclinic classes containing a measure with entropy larger than $\chi$ is finite for any $\chi>0$, in the setting of $\cC^{\infty}$ surface diffeomorphisms with positive entropy. In the present work, we obtain a similar result for measured homoclinic classes containing a measure of entropy close enough to the topological entropy. It can be seen as a partial generalization of Buzzi, Crovisier, and Sarig's result to any topological degree.

\begin{mainthm}
	Let $M$ be a closed surface and let $f:M\rightarrow M$ be a $\cC^{\infty}$ local diffeomorphism. Suppose that $h_{top}(f)>\log\deg(f)$. Then, there exists a constant $\epsilon>0$ such that the set of measured homoclinic classes $\mathcal{M}(\mu)$ containing a measure $\nu\in \mathcal{M}(\mu)$ satisfying $h(f,\nu)>h_{top}(f)-\epsilon$ is finite.
	\label{mainthm:B}
\end{mainthm}

Theorem~\ref{mainthm:A} is a corollary of Theorem~\ref{mainthm:B} since measured homoclinic classes contain at most one ergodic measure of maximal entropy, by Lima, Obata, and Poletti's result in~\cite[Theorem 4.1]{lima2024measures}.

\begin{remark} We were not able to treat the homoclinic classes of large entropy. However, we believe that the techniques developed in the present work, combined with a better understanding of homoclinic classes of endomorphisms, permit understanding such homoclinic classes. When considering a measured homoclinic class $\mathcal{M}(\mu)$, one can take the closure of all the union of the supports of the measures contained in $\mathcal{M}(\mu)$. Such a set is called a \textit{topological homoclinic class} and is denoted by $\text{HC}(\mu)$. It is well-known that topological homoclinic classes are $f$-invariant transitive sets. Recall that for a $f$-invariant set $K$, the quantity $h_{top}(f,K)$ denotes the topological entropy of $f$ on $K$.
\end{remark}

\begin{conjecture1}
	Let $M$ be a closed surface and $f$ be a $\cC^{\infty}$ local diffeomorphism. Then, for any $\chi>0$, the number of topological homoclinic classes $\text{HC}(\mu)$ such that $h_{top}(f,\text{HC}(\mu))> \log \deg(f)+\chi$ is finite.
\end{conjecture1}

This conjecture is a generalization to any topological degree of the spectral decomposition proved by Buzzi, Crovisier, and Sarig in~\cite[Theorem 1]{buzzi2022measures}. In the present work, we are able to obtain one side of the homoclinic relation in a more general setting, without the assumption on the entropy of $f$.

\begin{mainthm}
	Let $M$ be a closed surface and let $f:M\rightarrow M$ be a $\cC^{\infty}$ local diffeomorphism. Suppose that $h_{top}(f)>0$ and that any ergodic measure of maximal entropy of $f$ is hyperbolic. Let $(\mu_k)_{k\geq1}$ be a sequence of ergodic, hyperbolic of saddle type, measures of maximal entropy. Then, there exists a finite number of ergodic, hyperbolic of saddle type, $f$-invariant measures $\nu_1,...,\nu_L$, which have positive entropy, displaying the following property. For any $k\geq1$, there exists $i(k)\in \{1,...,L\}$ such that $\mu_k\preceq\nu_{i(k)}$.
	\label{mainthm:C}
\end{mainthm}

\begin{remark} Theorem~\ref{mainthm:C} does not directly give the finiteness of measured homoclinic classes. However, the interest of Theorem~\ref{mainthm:C} lies in the fact that we can apply it in the critical setting, namely when $h_{top}(f)=\log\deg(f)$. The new difficulties appearing in that case will be discussed in sub-Section~\ref{sub-sec:Difficulties}. It gives one direction to treat measures of maximal entropy in a setting where the ergodic m.m.e.'s can be either hyperbolic of saddle type, either hyperbolic of expanding type, or non-hyperbolic. This part is a real new contribution and is the core of the present work.

Let us emphasize that we strongly believe that, under more restrictive entropy hypotheses, a version of Theorem~\ref{mainthm:C} giving the other side of the homoclinic relation should be true. However, it would not give the finiteness of measured homoclinic classes. Indeed, typical $\mu_k$ points would accumulate on some $\nu_i$ in the past and on some $\nu_j$ in the future. However, we cannot homoclinically link $\nu_i$ and $\nu_j$.
\end{remark}

\subsection{Maps with a non-singular attractor}
By a classical argument, the Euler characteristic of any manifold admitting a local diffeomorphism with topological degree strictly greater than one, and which is globally defined, has to be zero. Thus, if $M$ is an orientable closed surface and $f:M\rightarrow M$ is a local diffeomorphism, $M$ has to be the two-torus. The setting of Theorem~\ref{mainthm:A} is then restrictive on the topology of $M$. However, this result extends to $\cC^{\infty}$ maps (not necessarily local diffeomorphisms) defined on two-dimensional manifolds (not necessarily compact) displaying compact non-singular attractors. Let $\mathscr{C}$ be the critical set of $f$. We say that $K$ is a compact non-singular attractor if: $K$ is a compact completely $f$-invariant subset and there exists an open connected neighborhood $U$ of $K$ such that the following holds: $f(\overline{U})\subset U$, for any $x \in U$ we have $d(f^n(x),K)\rightarrow 0$ when $n\rightarrow +\infty$ and $U\cap(\cup_{n\geq 0}f^n(\mathscr{C}))=\emptyset$. Note that under these hypotheses, the number of pre-images is constant on $U$ and then on $K$. The topological degree of $f$ on $K$ is then well-defined and we denote it by $\deg(f_{|K})$. We focus on the following consequence of Theorem~\ref{mainthm:A}.

\begin{maincor}
	Let $M$ be a two-dimensional manifold and let $f:M\rightarrow M$ be a $\cC^{\infty}$ map. Let $K$ be a compact non-singular attractor. Suppose that $h_{top}(f_{|K})>\log\deg(f_{|K})$. Then the set of ergodic $f$-invariant measures $\mu$ supported on $K$ such that $h_{top}(f_{|K})=h(f,\mu)$ is finite.
	\label{maincor}
\end{maincor}

\subsection{New difficulties in the endomorphism case}\label{sub-sec:Difficulties}
To prove the finiteness of measures of maximal entropy for surface diffeomorphisms, two main approaches have been developed recently. Both of them use coding techniques and reduce the problem to showing finiteness of homoclinic classes. The first one is a quantitative approach. Given a sequence $(\mu_k)_k$ of m.m.e.'s converging to some measure $\nu$, it decomposes the measure $\nu$ and analyzes the behavior of iterates of typical $\mu_k$-points, distinguishing iterates seeing expansion with "neutral" parts. This precise and quantitative analysis leads, in the case of surface diffeomorphisms, to proving that m.m.e.'s give uniform measure to a common non-uniformly hyperbolic set (this is called the SPR property, see~\cite{buzzi2025SPR}). Since stable and unstable manifolds have uniform size inside that common set, it leads to the finiteness of homoclinic classes. This approach has been developed in~\cite{buzzi2022continuity}. The second, which goes back to~\cite{buzzi2022measures}, takes a more topological point of view. The idea is that measures of positive entropy are carried by a finite number of small rectangles, bounded by stable and unstable manifolds. Typical $\mu_k$-unstable manifolds must intersect the rectangles. But if $\mu_k$ has large entropy then its typical unstable manifolds cannot be contained in these rectangles, which also leads to the finiteness of homoclinic classes. The strategy of the present work is new and uses a combination of a quantitative analysis and topological techniques. To motivate it, we present here the new main difficulties appearing in the endomorphism case, which make the previous approaches fail.

\parbreak\textbf{Different types of the measures of maximal entropy.}
When the map is invertible, any measure of positive entropy is hyperbolic of saddle type, and measures of expanding (or contracting) type are supported on periodic orbits. When the map becomes non-invertible, this is no longer the case. Three very different behaviors are possible for the ergodic m.m.e.'s: hyperbolic of saddle type, hyperbolic of expanding type, or non-hyperbolic. Note that the hypothesis $h_{top}(f)>\log\deg(f)$ authorizes only the first type.

\parbreak\textbf{No uniform control on the Lyapunov exponents.}
Ruelle's inequality gives a lower bound on the Lyapunov exponents in terms of entropy. In the case of diffeomorphisms, it implies that m.m.e.'s have Lyapunov exponents uniformly bounded away from zero. This is a crucial fact in both approaches. For endomorphisms, only the top Lyapunov exponent displays this property.

\parbreak\textbf{Unstable manifolds can cross.}
The non-invertibility forces unstable manifolds to depend on the past of a point. Thus, the same point in the manifold can have uncountably many unstable manifolds, and two different unstable manifolds can cross. The fact that unstable manifolds form a lamination was absolutely crucial in the topological approach.

\parbreak\textbf{Stable manifolds are not connected and their connected components can have dimension zero.}
When trying to define stable manifolds for local diffeomorphisms, one has to take the pre-image of embedded curves. This implies that stable manifolds become disconnected. Again, connectedness of the stable manifolds was crucial in the topological approach. Another important property that m.m.e.'s display in the diffeomorphism case is positive dimension on their stable and unstable manifolds. Namely, the Hausdorff dimension of typical points along stable and unstable manifolds is positive. When the map is non-invertible, it is not guaranteed that the set of typical points has positive Hausdorff dimension on connected components of stable manifolds, even when the measure is a m.m.e. This leads to real difficulties when trying to construct rectangles, or applying Sard's theorem to find transverse intersections.

\parbreak\textbf{The symbolic "third" stable direction.}
This idea contains, in some sense, all the previous difficulties. When fixing an ergodic measure of positive entropy, Oseledets' theorem gives two invariant directions for almost every point. One has to be expanding, and the other one can be of the three types discussed before: contracting, expanding, or neutral. The different pasts of a given point produce a third symbolic direction. Indeed, since the number of pre-images is finite and constant, choosing a past for a given point is equivalent to choosing an infinite sequence on a finite alphabet. On the set of all possible choices of past for some fixed point, the dynamics becomes equivalent to a shift. Thus, the symbolic direction is stable: indeed, applying $n$ times the dynamics to two different choices of past for the same point implies that the $n$-th first letters of these two infinite sequences become equal. This idea has a geometric interpretation in the natural extension. The lift of any small enough neighborhood, contained in the manifold, to the natural extension is homeomorphic to the product of a disc and a Cantor set. The disc represents the two "differentiable" directions and the Cantor set is the third "symbolic" one. Two unstable manifolds can cross because when lifted to the natural extension, they belong to two different connected components of the Cantor set. One can think of two curves, one above the other, not intersecting in a 3-dimensional space. Lifted to the natural extension, stable manifolds are locally homeomorphic to the product of a curve and a Cantor set and are then disconnected. Iterating backward a fiber in the natural extension gives $\deg(f)$ new fibers, which implies that stable manifolds become more and more disconnected whence iterating backward. Measures of positive entropy can have "bad" disintegration on these sets, for example being supported on the symbolic direction, which will imply that they have zero dimension in the differential direction. In some sense, the difficulties arising in the case of endomorphisms are linked to dimensions and are similar to those appearing for diffeomorphisms in dimension three.

\parbreak The tools developed in the present work overcome some of these difficulties, and we believe that they may be applied in other settings, like three-dimensional diffeomorphisms, and to other problems for surface endomorphisms.

\subsection{A heuristic overview of the proof}\label{subsec:DiscussionProof}
We want to use symbolic dynamics on homoclinic classes and the fact that a measured homoclinic class can carry at most one ergodic m.m.e.~\cite[Theorem 4.1]{lima2024measures}. Proving  finiteness of the set of ergodic m.m.e.'s is then reduced to proving finiteness of the set of measured homoclinic classes carrying an m.m.e. The proof is split into three main parts that we explain here. Let us consider a sequence of ergodic $f$-invariant measures of maximal entropy $(\mu_k)_k$. They then satisfy $h(f,\mu_k) > \log\deg(f)$, which implies in particular that they are hyperbolic of saddle type. The sequence converges to some measure $\nu$. Almost every ergodic component of $\nu$ is an m.m.e. by upper semi-continuity of the measured entropy in $\cC^{\infty}$. In particular hyperbolic of saddle type. Remark that this an important part where the $\cC^{\infty}$ hypothesis is used.

\parbreak\textbf{Part 1: Constructing nice rectangles for the measure $\nu$.}
The goal of the first part of the proof is to construct two finite families of rectangles $R_1,...,R_L$ and $\tilde{R}_1,...,\tilde{R}_L$ which have the following properties.
\begin{itemize}[label={--}]
	\item Their boundary is a Jordan curve, which is the union of four curves such that two of them are contained in local stable manifolds of well-chosen points. We call them the stable boundaries of the $R_i$. The two other curves are the unstable boundaries. They are contained in leaves of a foliation tangent to a natural unstable cone field.
	\item The $\nu$-measure of their union can be chosen arbitrarily close to one.
	\item We have good control on their geometry: they have a small diameter but the distance between their unstable boundaries is bounded from below.
	\item Each rectangle $R_i$ is contained in $\tilde{R}_i$, and the stable boundaries of $\tilde{R}_i$ contain the stable boundaries of $R_i$.
	\item They display a kind of Markov property. If the image by $f$ of some $R_i$ intersects some $R_j$ in a nice way, then the image of $\tilde{R}_i$ has to intersect $\tilde{R}_j$ without crossing the unstable boundaries of $\tilde{R}_j$. To sum up, we want the rectangles to be stretched along the unstable direction by $f$ in a nice way.
\end{itemize}
Only one family would have been enough if we could have built a family of rectangles $(R_i)_i$ displaying the Markov property we described. We were not able to build such a family and we bypassed this difficulty by building the second family $(\tilde{R}_i)_i$ from the first one.

Let us explain the important points of the construction. They rely on some graph transform results for continuous maps tangent to unstable cones. We build the first family of rectangles by taking the intersection between the leaves of some foliation tangent to a well-chosen unstable cone with local stable manifolds of well-chosen points. We build the second family from the first by successive approximations. If $f(R_i)$ intersects some $R_j$ in a way which does not respect the Markov property, we add to $R_j$ a strip so that the intersection becomes legal. The idea is that at each iteration $n$, the images by $f^n$ of the rectangles are exponentially contracted in the stable direction, so the width of the strips we add are exponentially small.

\parbreak\textbf{Part 2: Obtaining transverse intersections between unstable manifolds of typical $\mu_k$-points and stable boundaries of the $\nu$-rectangles.}
This part is the core of the present work, and contains new ideas to treat measures of maximal entropy. We analyze the behavior of typical $\mu_k$-points $x$, inside a given unstable manifold. We show that for $k$ and $n$ large enough, we can decompose the orbit $(x,...,f^n(x))$ into sub-segments of the following two types.
\begin{itemize}[label={--}]
	\item \textbf{Rectangle segment:} these are long segments of a fixed length inside the rectangles $\cup_i R_i$.
	\item \textbf{Wild segments:} these are segments which are highly infrequent, where we do not have any further information.
\end{itemize}
By Ledrappier-Young theory, the entropy inside a typical unstable manifold is equal to the entropy of the measure $\mu_k$. We choose a set $F_k$ of positive $\mu_k$-measure such that any orbit $(x,...,f^n(x))$ of points $x\in F_k$ has a decomposition as above, for a uniform $n$. Fix $x\in F_k$ and consider a curve $\gamma$ contained in the local unstable manifold of $x$. We argue by contradiction and assume the following hypothesis.
\begin{center}
	\textit{For any $n$ and for any rectangle $R_i$, the curve $f^n(\gamma)$ does not intersect the stable manifolds containing the stable boundaries of $R_i$.}
\end{center}
We analyze the entropy of the points of $\gamma\cap F_k$ having the same orbit decomposition. Wild segments can create little entropy since their frequency is arbitrarily small. The main original part of the present work lies in the method used to control the entropy during the rectangle segments. Let us explain here how it is estimated.

The first difficulty appears at the starting iterate $f^n(\gamma)$ of any rectangle segment. We have to cover the intersection between the union of the $R_i$ and $f^n(\gamma)$ by a uniform number of rectangles $R_i$. We first use Yomdin theory to start with a sub-curve $\psi\subset f^n(\gamma)$ with bounded geometry. Since $f$ is $\cC^{\infty}$, the number of such sub-curves has the same magnitude as the entropy. To bound the number of rectangles needed to cover $\psi$, we look at the topology of the intersection between $\psi$ and the $R_i$'s. We exploit the fact that $\psi$ cannot cross the stable manifolds containing the stable boundaries of the $R_i$'s to show that there can be only two types of intersections, which we call \textit{crossing} and \textit{folding} intersections. The first type appears when some sub-curve of $\psi$ crosses a rectangle $R_i$ from the left unstable boundary to the right one. We use the fact that the distance between two unstable boundaries of some $R_i$ is bounded from below to bound the number of crossing intersections. For the folding ones, we were not able to bound the number of such intersections, but we show, using topological arguments, that they are contained in a uniform number of rectangles. We then treat the crossing and the folding intersections separately. 

We use the connectedness of crossing intersections and the Markov property to conclude that a given crossing intersection stays inside a unique rectangle at each iteration of the whole rectangle segment. Indeed, if one takes a crossing intersection $\varphi\subset \psi$, there exists some rectangle such that $\varphi\subset \tilde{R}_i$. Then, since $f(\varphi)$ intersects some $R_j$, the curve $\varphi$ is trapped between the two stable boundaries of $\tilde{R}_j$ and the two unstable boundaries of $f(\tilde{R}_i)$. The Markov property implies that the two unstable boundaries of $f(\tilde{R}_i)$ lie between the two unstable boundaries of $\tilde{R}_j$, which then implies that $\varphi$ is contained in $\tilde{R}_j$. Since the diameter of the rectangles is small, such crossing curves do not contribute to orbit separation. 

For the folding ones, we treat each rectangle containing such intersections separately. Let us fix such a rectangle $R_i$. The union of folding sub-curves of $\psi$ contained in $R_i$ is disconnected. Nevertheless, the stable boundaries $R_i$ are contained in local stable manifolds that are much longer than the diameter of $R_i$ and the length of $\psi$. Using again the fact that $\psi$ cannot meet these stable manifolds, the Markov property of the $R_i$'s allows us again to conclude that each iterate of every folding curve contained in $R_i$ is still contained in a unique rectangle. Again, this shows that folding curves do not create entropy during the rectangle segments.

This argument concludes that the entropy of the unstable manifolds of points in $F_k$ is very close to zero. As a corollary of this argument, and since the $\mu_k$'s are m.m.e.'s, typical $\mu_k$-unstable manifolds have to meet all the stable manifolds containing the stable boundaries of all the rectangles of the family $(R_i)_i$, otherwise their entropy would be strictly smaller than the topological entropy. We conclude that for any $k$ large enough and any $x \in F_k$, the unstable manifold of $x$ intersects topologically all the stable manifolds containing the stable boundaries of the rectangles $R_i$'s. Finally, we use Sard Theorem to find transverse intersections.

\parbreak\textbf{Part 3: Obtaining transverse intersections between stable manifolds of typical $\mu_k$-points and unstable boundaries of some $\nu$-rectangles.}
The last part of the proof is a kind of stable version of the second part. Here, we will analyze the separation, under the map $f^{-1}$, of typical $\mu_k$-points inside stable manifolds. The non-invertibility forces us to work in the natural extension $M_f$. Lifted to $M_f$, stable local manifolds are homeomorphic to the product of a curve and a Cantor set, which corresponds to the symbolic direction. The only changing part compared to Part 2 is the control on the entropy during rectangle segments.

First, we can modify the unstable boundaries of the $\nu$-rectangles to replace them with unstable manifolds of well-chosen points. We denote this new family by $T_1,...,T_L$. Note that we lose the Markov property and the bound from below of the distance between unstable boundaries. Nevertheless, these new rectangles still have small diameter.

Let us now explain how we control the separation of points of a well-chosen stable manifold $V$ during the rectangle segments. Recall that $\pi$ is the projection from $M_f$ to $M$ and that  $\hf$ is the lift of $f$ to $M_f$. When we lift some rectangle $T_i$ to the natural extension, $\pi^{-1}(T_i)$ is now homeomorphic to the product of a disc and a Cantor set. We argue again by contradiction and assume the following.
\begin{center}
	\textit{For any $n$, for any rectangle $\pi^{-1}(T_i)$ and for any connected component $W\subset \hf^{-n}(V)$ intersecting $\pi\inv(T_i)$, the curve $W$ is contained in $\pi^{-1}(T_i)$.}
\end{center}
During rectangle segments, only the symbolic direction can then create entropy. We conclude that the orbit separation during rectangle segments is bounded by $\deg(f)^m$, where $m$ is the length of the segment. This leads to a contradiction since $h(f,\mu_k)>\log\deg(f)$, and implies that, for $k$ large enough, $\mu_k$-typical stable manifolds meet some unstable boundary of a rectangle $T_i$. We again conclude using Sard Theorem.

\parbreak Let us explain how to combine Part 2 and Part 3 to prove that for $k$ large enough, the measure $\mu_k$ are homoclinically related to a finite number of hyperbolic saddle periodic orbits. Firstly, Part 3 implies that typical $\mu_k$-stable manifolds intersect transversely the unstable boundary of some rectangle $T_{i(k)}$. Secondly, Part 2 gives that typical $\mu_k$ unstable manifolds cross the stable manifolds containing the stable boundaries of all the $R_i$'s, and in particular of $R_{(i_k)}$, which can be tough as the same rectangle as $T_{i(k)}$, since only their unstable boundaries differ. Now, using $\lambda$-lemma and Katok Theorem, we conclude that there exists a saddle periodic orbit $\cO_{i(k)}$ such that its stable and unstable manifolds accumulates on the boundaries of $T_{i(k)}$. This proves that $\mu_k$ is homoclinically related to $\cO_{i(k)}$. Since the number of rectangles is finite, this concludes the proof.

\subsection{Comments and questions about finite regularity, entropy assumption and uniqueness of the m.m.e.}\label{subsec:Intro:Comments}
In this paragraph, we discuss the behavior of measures of maximal entropy in some more general cases.

\parbreak\textbf{Regularity.}
Even though our results stand in the $\cC^{\infty}$ setting, we can expect results in lower regularity. Indeed, Proposition~\ref{prop:IntersectionsUnstableMaxEntropy}, which tackles the principal difficulty linked to unstable manifolds, is already stated for $\cC^r$ local diffeomorphisms, under an assumption on the entropy of the map. Proposition~\ref{prop:StableUnstableIntersection} is stated for $\cC^{\infty}$ local diffeomorphisms, but we think it should be true in lower regularity under some stronger entropy assumption as in Proposition~\ref{prop:IntersectionsUnstableMaxEntropy}, with few changes in the proof. The issue is that we do not have any control on the number $\beta$ in Proposition~\ref{prop:IntersectionsUnstableMaxEntropy}, which is, informally, the minimal weight given by $\nu$ to the homoclinic classes supporting the $\nu_i$. We believe that a finer investigation of the behavior of the $\mu_k$-typical orbits could give some control on $\beta$.

\begin{conjecture2}
	Fix $r>1$. Let $f:M\rightarrow M$ be a $\cC^r$ local diffeomorphism. There exists a constant $C:=C(\norm{f}_{\cC^1},r)>0$ depending only on the $\cC^1$ norm of $f$ and $r$ such that if $h_{top}(f)>C$, then $f$ admits a finite number of ergodic measures of maximal entropy.
\end{conjecture2}

\parbreak\textbf{Entropy assumption.}
As we said before, the assumption $h_{top}(f)>\log\deg(f)$ has several consequences. First, every ergodic measure of maximal entropy is hyperbolic of saddle type. Second, it also implies that backward separation in stable typical sets of some measure of maximal entropy does not only come from the symbolic direction. Removing the entropy assumption gives then two difficulties to treat. If we want to study hyperbolic measures of saddle type and of maximal entropy, the first one seems not to be insurmountable. Indeed, if $(\mu_k)_k$ is a sequence of hyperbolic m.m.e.'s of saddle type converging to $\nu$, then $\nu$ has a decomposition $\nu=\alpha\nu_1+(1-\alpha)\nu_0$, where $\alpha \in (0,1]$ and almost every ergodic component of $\nu_1$ is hyperbolic of saddle type with positive entropy. It allows us to apply Proposition~\ref{prop:UnstableStableIntersection}, which gives one side of the intersections desired. When trying to get a statement analogous to Proposition~\ref{prop:StableUnstableIntersection}, the second difficulty becomes a serious problem. Indeed, we were not able to give an upper bound on the backward entropy inside the stable manifolds in that case (Part 3 of the proof). This leads us to the following questions.

\begin{question}
	If $f$ is a $\cC^{\infty}$ local diffeomorphism, which weaker assumption guarantees that hyperbolic measures of saddle type with large entropy have: 
	\begin{itemize}[label={--}]
		\item typical stable sets such that backward exponential separation of points does not uniquely come from the symbolic direction?
		\item a nice disintegration of these measures along connected components of their stable sets?
		\item positive dimension along these curves?
	\end{itemize}
\end{question}

As we said previously, it is easy to find examples in the critical setting where $h_{top}(f)=\log\deg(f)$ and such that $f$ admits an infinite number of ergodic measures of maximal entropy. However, these examples seem to be very rigid and we believe that this cannot happen in general.

\begin{question}
	In the set of $\cC^{\infty}$ local diffeomorphisms $f$ such that $h_{top}(f)=\log\deg(f)$, is the property of having finitely many ergodic hyperbolic m.m.e.'s of saddle type generic?
\end{question}

\parbreak\textbf{Transitivity.}
Buzzi, Crovisier, and Sarig obtained in~\cite{buzzi2022measures} the uniqueness of the measure of maximal entropy in the case where $f$ is transitive. They use a topological argument which we describe here. Given two measures $\mu$ and $\nu$ carrying rectangles, transitivity implies that iterates of the $\mu$-rectangles have to cross the $\nu$-rectangles, without being contained in the $\nu$-rectangles. This implies that $\mu$-typical unstable manifolds intersect transversely $\nu$-typical stable manifolds. Since the argument is symmetric, it implies that $\mu$ and $\nu$ are homoclinically related, and then that the measure of maximal entropy is unique. The argument clearly fails in the endomorphism case since unstable manifolds can cross. However, we do not know an example of a transitive local diffeomorphism having two measures of maximal entropy of saddle type.

\begin{question}
	Is there a $\cC^{\infty}$ local diffeomorphism having more than one measure of maximal entropy of saddle type?
\end{question}

\subsection{Previous finiteness results}
The question of the number of ergodic measures of maximal entropy has been studied a lot since the 1970s. Various classes of systems have finitely many ergodic m.m.e.'s. Let us give the main results which have been obtained in the context of smooth dynamics. Finiteness of m.m.e.'s was first shown for uniformly hyperbolic systems: Anosov diffeomorphisms~\cite{sinai1972gibbs} and Axiom A diffeomorphisms~\cite{bowen1971periodic}.

This was then generalized to non-uniformly expanding maps: piecewise monotonic interval maps \cite{hofbauer1981piecewisemono}, smooth interval maps~\cite{buzzi1997interval} in dimension one, skew products of those~\cite{buzziSubshitQF} and surface maps with singularities~\cite{lima2018symbolicsurface} in dimension two.

Non-uniformly hyperbolic invertible dynamics had also been considered: piecewise affine surface homeomorphisms~\cite{buzzi2009mmepiecewiseaffine}, various derived from Anosov~\cite{newhouse2006dynamicsskew,buzzi2012DA,ures2012intrinsic,buzzi2011entropic,climenhaga2018unique} or some partially hyperbolic diffeomorphisms~\cite{hertz2012maximizing,buzzi2019dichotomy}.

A lot of recent progress has been done using coding techniques, first built by Sarig in dimension 2 in~\cite{sarig2013symbolic} then by Ben Ovadia in greater dimension in~\cite{BenOvadiaCoding}, and developed to study measures of maximal entropy in~\cite{buzzi2022measures}. Sarig proved in~\cite{sarig2013symbolic} that $\cC^r$ surface diffeomorphisms with positive topological entropy have countably many ergodic m.m.e.'s for $r>1$. In~\cite{buzzi2022measures} Buzzi, Crovisier, and Sarig proved the finiteness of ergodic m.m.e.'s for such $\cC^{\infty}$ surface diffeomorphisms. Mongez and Pacifico proved in~\cite{mongez2024finite} and~\cite{mongez2023robustness} finiteness of the number of m.m.e.'s for an open set of partially hyperbolic systems with a one-dimensional center direction, and they obtain uniqueness in certain cases.

For non-invertible dynamics, Araujo, Lima, and Poletti built symbolic extensions for non-uniformly hyperbolic maps in~\cite{araujo2024symbolic}, under some hypotheses on the critical set. Lima, Obata, and Poletti developed it in~\cite{lima2024measures} and proved finiteness of the set of ergodic m.m.e.'s for some partially hyperbolic endomorphisms with one-dimensional center direction and some non-uniformly expanding maps, including Viana maps. Note that in the first setting, they also used the assumption $h_{top}(f)>\log\deg(f)$.

\subsection{Outline of the paper}
In Section~\ref{section:PesinTheory}, we recall some well-known facts about surface endomorphisms and establish some Pesin theory results in our setting.

In Section~\ref{sec:Rectangles}, we build, for a given hyperbolic $f$-invariant measure of saddle type, two families of rectangles displaying a Markov type property.

Section~\ref{sec:CurvesIntersectingFamiliesRectangles} gives a uniform bound on the number of rectangles needed to cover a curve intersecting them, under the hypothesis that such a curve does not cross the stable manifolds containing their stable boundaries. In this context, Section~\ref{sec:CurvesIntersectingFamiliesRectangles} defines the crossing and folding intervals of the curve we consider.

Section~\ref{sec:HomClasses} introduces homoclinic classes for endomorphisms and reminds some of their properties. In particular, it contains the fact that homoclinic classes carry at most one m.m.e.

Section~\ref{sec:EntropyYomdin} recalls some facts about entropy, Yomdin theory, and dimension theory. Some of them are well-known or adaptations to endomorphisms.

Sections~\ref{sec:UnstableStableIntersection} and~\ref{sec:StableUnstableIntersections} contain the entropy argument. Section~\ref{sec:UnstableStableIntersection} shows that sequences of ergodic measures of high entropy cannot accumulate on rectangles without displaying intersections between their unstable manifolds and stable manifolds of the limit measure. Section~\ref{sec:StableUnstableIntersections} proves the symmetric property.

Finally, Section~\ref{sec:ProofMainThm} concludes by deducing the main results, announced in the introduction, from the results proved so far.\\

\textbf{Conventions.}
Through all this work, for any constant $c$, we say that a property holds for $c$ sufficiently small when there exists some $c_0>0$ such that the property holds for any $0<c<c_0$. Moreover our constructions and proofs, especially in Section \ref{section:PesinTheory} and Section \ref{sec:Rectangles}, involve on a lot of constants. In any statement, when we introduce a constant by $c:=c(a,b)$, it means that $c$ depends on $a$ and $b$. It allows then the reader to keep track of the interdependence between each constants.\\

\textbf{Acknowledgments.} The author is deeply grateful to Sylvain Crovisier for many fruitful discussions about the proof in the diffeomorphism case, for his generosity in sharing ideas, and for his careful proofreading. We are also grateful to Yuri Lima, Davi Obata, and Mauricio Poletti for valuable discussions on endomorphisms, and especially for pointing out the difficulties related to stable sets. We further thank Rafael Potrie for insightful discussions about the overall strategy and for providing examples of such endomorphisms. Finally, we thank Juan Carlos Mongez for many helpful discussions regarding this work.

\section{Pesin theory for local diffeomorphisms}\label{section:PesinTheory}
In this section, we collect some Pesin theory results in the case of endomorphisms. We prove a graph transform for curves tangent to some contracted cone field and show that this graph transform is a contraction. A classical presentation of Pesin theory in the diffeomorphism case is contained in \cite[Chapter S]{katok1995introduction}. A more recent presentation, which contains all the ideas present in our work, is developed in \cite{sarig2013symbolic}. Pesin theory for endomorphisms is also treated in \cite{qian2009PesinTheoryEndo}.
In the present section, $M$ is a closed surface and $f:M\rightarrow M$ is a $\cC^2$ local diffeomorphism ($\cC^{1+}$ is sufficient).

\subsection{Preliminaries and well-known facts about surface endomorphisms}
Here, we recall some well-known material about the dynamics of smooth endomorphisms and establish some notation.
We equip the surface $M$ with a $\cC^{\infty}$ Riemannian metric that we denote by $(M, \langle \cdot, \cdot \rangle)$. We denote the distance function induced by the metric by $d(\cdot, \cdot)$. By a local diffeomorphism $f$ (or just endomorphism), we mean that for all $x \in M$, the tangent map $d_xf : T_xM \rightarrow T_{f(x)}M$ is a linear isomorphism. We recall that in this setting, the number of pre-images of each point is constant and equal to the topological degree of $f$, which we denote by $\deg(f)$.
We endow $\bbR^2$ with the standard Euclidean metric. Let us set up some notation:
\begin{itemize}[label = {--}]
	\item For $x \in M$ and $r>0$, we denote by $B(x,r)$ the ball centered at $x$ with radius $r$ for the metric $d(\cdot, \cdot)$.
	\item For any $v \in \bbR^2$ and $r>0$, we also denote by $B(v,r)$ the ball centered at $v$ of radius $r$ when there is no ambiguity with the previous notation.
	\item For $x \in M$, $v \in T_xM$ and $r>0$, we write $B_x(v,r) = \{u \in T_xM, \ \norm{u-v}<r\}$.
	\item For any $v \in \bbR^2$ and $r>0$, we write $R(v,r) = \{u=(u_1,u_2) \in \bbR^2, \ |u_1|,|u_2|<r \}$.
\end{itemize}
Let $\exp : TM \rightarrow M$ be the exponential map. Since $M$ is compact, $\exp_x : T_xM \rightarrow M$ is well-defined for all $x\in M$ and there exists a constant $r(M) >0$ such that the map:
\begin{equation*}
	\exp_x : B_x(0,r(M)) \subset T_xM \rightarrow M
\end{equation*}
is a diffeomorphism onto its image.

\parbreak\textbf{Natural extension.} Let us now introduce an important tool in the study of endomorphisms: the natural extension. We define the following set:
\begin{equation*}
	M_f = \big\{(x_n)_{n\in \bbZ} \in M^{\bbZ}, \ f(x_n) = x_{n+1} \ \forall n \in \bbZ\big\}.
\end{equation*}
For any $\hx \in M_f$, we denote by $x_n$ the image of the projection on the $n$-th coordinate in $M$. We equip $M_f$ with a distance, which we denote by $d_{M_f}(\cdot, \cdot)$, such that for $\hx,\hy \in M_f$, we have:
\begin{equation*}
	d_{M_f}(\hx,\hy) = \sum_{i=0}^{+\infty}\frac{d(x_{-i},y_{-i})}{2^i}.
\end{equation*}
When there is no ambiguity, we will also denote the distance function of $M_f$ by $d$. The space $(M_f,d)$ is a compact metric space. Denote by $\pi$ the projection on the $0$-th coordinate. In this setting, it is known that for all $x \in M$, the set $\pi\inv(x)$ is compact, totally disconnected, and has no isolated points, and hence is homeomorphic to a Cantor set. Let $\hf : M_f \rightarrow M_f$ be the map defined by:
\begin{equation*}
	\hf\big((x_n)_{n\in \bbZ}\big) = (f(x_n))_{n\in \bbZ} = (x_{n+1})_{n \in \bbZ}.
\end{equation*}
Note that $\hf$ is a bilateral shift; its inverse is the map $\hf\inv\big((x_n)_{n\in \bbZ}\big) = (x_{n-1})_{n \in \bbZ}$ and $\hf$ is a homeomorphism of $(M_f,d)$. The system $(M_f,\hf)$ is an extension of $(M,f)$, called the natural extension, such that $\pi\circ\hf=f\circ\pi$.

\parbreak\textbf{Lift of invariant measures.}
Let $\bbP(f)$ denote the set of probability $f$-invariant measures on $M$, and $\bbP_e(f) \subset \bbP(f)$ the ones that are ergodic. It is well known that a measure $\mu \in \bbP_e(f)$ has a unique ergodic lift to $\bbP_e(\hf)$. The projection $\pi$ induces a bijection $\pi_{\star}$ from $\bbP_e(\hf)$ to $\bbP_e(f)$. We usually write $\hmu:=\pi_{\star}\inv\mu$. If $\mu$ is not ergodic, we can lift every ergodic measure of its ergodic decomposition with $\pi$. Two lifts of $\mu$ coincide then on a set of full measure and we can still write $\hmu:=\pi\inv\mu$.

\parbreak\textbf{Lyapunov exponents.}
We can define an invertible cocycle over $\hf$. For each $\hx \in M_f$, define the vector space $T_{\hx}M = T_{\pi(\hx)}M$. Let:
\begin{equation*}
	TM_f = \bigsqcup_{\hx \in M_f} T_{\hx}M.
\end{equation*}
Note that $TM_f$ defines a topological vector bundle over $M_f$. We then introduce the cocycle $d\hf : TM_f \rightarrow TM_f$, defined for $\hx \in M_f$ by:
\begin{equation*}
	d_{\hx}\hf^n =
	\left\{
	\begin{aligned}
		&d_{x_0}f^n \ \text{if} \ n\geq 0\\
		&(d_{x_{-n}}f^n)\inv \ \text{if} \ n<0.
	\end{aligned}
	\right.
\end{equation*}
This is a well-defined continuous invertible cocycle over $\hf$, which coincides with the tangent map of $f$ along every orbit $\hx \in M_f$. We can then apply Oseledets' theorem to this cocycle.
Let $\hmu \in \bbP(\hf)$. For $\hmu$-almost every $\hx$, there exist two $d\hf$-invariant subspaces $E^+(\hx), E^-(\hx) \subset T_{\hx}M$ such that $T_{\hx}M = E^+(\hx) \oplus E^-(\hx)$, and real numbers $\lambda^+(\hx)> \lambda^-(\hx)$ which satisfy:
\begin{equation*}
	\lim_{n\rightarrow +\infty} \frac{1}{n} \log\norm{d_{\hx}\hf^n_{|E^+(\hx)}} = \lambda^+(\hx) \ \text{and} \ \lim_{n\rightarrow +\infty} \frac{1}{n} \log\norm{d_{\hx}\hf^n_{|E^-(\hx)}} = \lambda^-(\hx).
\end{equation*}
The numbers $\lambda^+(\hx)$ and $\lambda^-(\hx)$ are the Lyapunov exponents at the point $\hx$. These are $\hf$-invariant functions, so they are constant almost everywhere when the measure $\hmu$ is ergodic. In that case, we denote them $\lambda^+(\hmu)$ and $\lambda^-(\hmu)$. Note that the measurable sub-bundle $E^-$ does not depend on the choice of the past. Namely, let $\mu:=\pi_{\star}\hmu$. For $\mu$-almost every $x$ and for any $\hx_1,\hx_2\in\pi\inv(x)$, we have $E^-(\hx_1)=E^-(\hx_2)$.
Let us mention a Ruelle inequality proved by Liu in \cite{liu2003ruelle} which links, in the context of local diffeomorphisms, entropy, non-positive Lyapunov exponents, and degree.
\begin{prop}[Ruelle inequality]
	Let $\mu \in \bbP_e(f)$ and $\hmu = \pi_*\inv\mu$. The following inequality holds:
	\begin{equation*}
		h(f,\mu) \leq \sum_{\lambda(\mu)<0} \lambda(\mu) + \log\deg(f).
	\end{equation*}
	\label{prop:Ruelle}
\end{prop}

\parbreak\textbf{Hyperbolic measures.}
Let $\mu \in \bbP(f)$ and $\hmu = \pi_*\inv \mu$. We say that $\mu$ is hyperbolic if for $\hmu$-almost every $\hx \in M_f$, we have $\lambda^+(\hx),\lambda^-(\hx) \ne 0$. If $\mu$ is hyperbolic and $\lambda^+(\hx) > 0 > \lambda^-(\hx)$ for $\hmu$-almost every $\hx \in M_f$, we say that $\mu$ is hyperbolic of saddle type. Remark that in the case where $h_{top}(f)>\log\deg(f)$, Ruelle inequality \ref{prop:Ruelle} implies that any ergodic measure of maximal entropy is hyperbolic saddle.

\subsection{Pesin blocks}
Let us now enter the core of Pesin theory. We start by defining the Pesin blocks, the sets where hyperbolicity looks like uniform.
\begin{defn}
	Let $\chi,\epsilon >0$, $l>1$, with $\epsilon <\chi$. Let $k\leq 2$ be an integer. The Pesin block with constants $\chi,\epsilon,l,k$ is the set $\hNUH^k \subset M_f$ defined by the following. For each $\hx \in \hNUH^k$, there exists a $d\hf$-invariant splitting $T_{\hx}M = \Es(\hx) \oplus \Eu(\hx)$ where for all $n \geq 0, \ m\in \bbZ$:
	\begin{enumerate}
		\item[(PS1)] $\norm{d_{\hf^m(\hx)}\hf^n_{|\Es(\hf^m(\hx))}} \leq le^{-\chi n}e^{\epsilon|m|} \ \text{and} \ \norm{d_{\hf^m(\hx)}\hf^{-n}_{|\Es(\hf^m(\hx))}} \leq le^{-\chi n}e^{\epsilon|m|}$.
		\item[(PS2)] $\sin\measuredangle\Big(\Es\big(\hf^m(\hx)\big),\Eu\big(\hf^m(\hx)\big)\Big) \geq l\inv e^{-\epsilon|m|}$.
		\item[(PS3)] $\dim(\Eu(\hx)) = k$.
	\end{enumerate}
	\label{def:PesinBlock}
\end{defn}

It is well known that the Pesin blocks are compact sets and that the splitting $\hx \in \hNUH^k \rightarrow E^{u/s}(\hx)$ is continuous. The Pesin blocks are not invariant, but we notice that for each $m \in \bbZ$, we have $\hf^m(\hNUH^k) \subset \hat{\Lambda}_{\chi,\epsilon,le^{\epsilon|m|}}^k$. We also notice that if $\chi_1\leq\chi_2$, $\epsilon_2\leq\epsilon_1$, and $l_2\leq l_1$, then $\hat{\Lambda}_{\chi_2,\epsilon_2,l_2}^k \subset \hat{\Lambda}_{\chi_1,\epsilon_1,l_1}^k$. For constants $0<\epsilon<\chi$ and for an integer $k\leq 2$, we then define the corresponding Pesin set by:
\begin{equation*}
	\hat{\Lambda}_{\chi,\epsilon}^k = \bigcup_{l>1} \hNUH^k.
\end{equation*}
Notice that this set is $\hf$-invariant. We also denote the projection of these sets by:
\begin{equation*}
	\NUH^k = \pi\big(\hNUH^k\big) \quad \Lambda_{\chi,\epsilon}^k = \pi\big(\hat{\Lambda}_{\chi,\epsilon}^k\big).
\end{equation*}
It is also well known that for any hyperbolic measure $\mu \in \bbP_e(f)$, there exists a constant $\chi>0$ such that for all $\epsilon<\chi/2$, we have $\hmu(\hat{\Lambda}_{\chi,\epsilon}^k) = 1$, where $\hmu = \pi_*\inv\mu$.

\subsection{Lyapunov change of coordinates}

From now on, we fix $k=1$ and a Pesin set $\hat{\Lambda}^k_{\chi,\epsilon}$. To lighten the notation, we will write $\hat{\Lambda}_{\chi,\epsilon} = \hat{\Lambda}^k_{\chi,\epsilon}$ and $\Lambda^k_{\chi,\epsilon} = \Lambda_{\chi,\epsilon}$. For each $\hx \in \hat{\Lambda}_{\chi,\epsilon}$, define two unit vectors $e^s(\hx),e^u(\hx) \in T_{\hx}M$ such that $E^{u/s}(\hx) = \text{span}\{e^{u/s}(\hx)\}$. We also denote by $\{e_1,e_2\}$ the canonical basis of $\bbR^2$. Given $\hx \in \hat{\Lambda}_{\chi,\epsilon}$, let:
\begin{equation*}
	\begin{aligned}
		&s(\hx) = \Big(\sum_{n=0}^{+\infty} e^{n\chi} \norm{d_{\hx}\hf^n e^s(\hx)}^2\Big)^{\frac{1}{2}}\\
		&u(\hx) = \Big(\sum_{n=0}^{+\infty} e^{n\chi} \norm{d_{\hx}\hf^{-n} e^u(\hx)}^2\Big)^{\frac{1}{2}}.
	\end{aligned}
\end{equation*}
Note that these sums are well defined by property (PS1) from Definition \ref{def:PesinBlock}, and that the maps $\hx \rightarrow s(\hx)$ and $\hx \rightarrow u(\hx)$ are continuous on $\hNUH$.
\begin{defn}
	Given $\hx \in \hNUH$, the Lyapunov change of coordinates at $\hx$ is the linear map $C(\hx): \bbR^2 \rightarrow T_{\hx}M$ such that:
	\begin{equation*}
		C(\hx) e_1 = s(\hx)\inv e^s(\hx) \ \text{and} \ C(\hx) e_2 = u(\hx)\inv e^u(\hx).
	\end{equation*}
	\label{def:LyapunovChangeCoordinates}
\end{defn}
An easy calculation shows the following estimate.
\begin{lemma}
	For all $\hx \in \hNUH$, we have $\norm{C(\hx)} \leq \sqrt{2}$.
	\label{lem:LCCisContraction}
\end{lemma}
The Lyapunov changes of coordinates are also continuous on Pesin blocks in the following way.
\begin{lemma}
	For all $\delta^{'}$ small enough, there exists $\delta$ such that for all $\hx,\hy \in \hNUH$:
	\begin{equation*}
		d(\hx,\hy) < \delta \implies \norm{C(\hx) - P_{y_0,x_0}C(\hy)} < \delta^{'}
	\end{equation*}
	where $P_{y_0,x_0}$ denotes the parallel transport from $y_0:=\pi(\hy)$ to $x_0:=\pi(\hx)$.
\end{lemma}
\begin{thm}
	There exists a constant $C_f$, which only depends on $f$, such that for all $\hx \in \hNUH$:
	\begin{equation*}
		C(\hf(\hx))\inv \circ d_{\hx}\hf \circ C(\hx) = \begin{pmatrix*}
			A(\hx) & 0\\
			0 & B(\hx)
		\end{pmatrix*}
	\end{equation*}
	where $C_f\inv < A(\hx) \leq e^{-\chi/2}$ and $e^{\chi/2} \leq B(\hx) < C_f$.
\end{thm}
\subsection{Pesin charts}

By compactness, there exists a constant $\rho(M)>0$ such that for all $x \in M$, we have:
\begin{itemize}[label = {--}]
	\item $B(x,\rho(M)) \subset \exp_x\big(B_x(0,r(M))\big)$.
	\item For all $z \in M$, the map $(x,y) \rightarrow \exp_x(y)$ is well defined and $2$-Lipschitz in $B(z,\rho(M)) \times B(z,\rho(M))$.
	\item For all $z \in M$ and for all $x,y \in B(z,\rho(M))$, we have $\norm{d_y(\exp_x)\inv} \leq 2$.
\end{itemize}
For every $\hx \in \hNUH$, we define the map
\begin{equation*}
	\Psi_{\hx} = \exp_{x_0} \circ C(\hx).
\end{equation*}
By Lemma \ref{lem:LCCisContraction}, the map $\Psi_{\hx}$ is a diffeomorphism from $R(0,\frac{r(M)}{\sqrt{2}})$ onto its image. Define also the map
\begin{equation*}
	f_{\hx} = \Psi_{\hf(\hx)}\inv \circ f \circ \Psi_{\hx}
\end{equation*}
and note that the tangent map of $f_{\hx}$ at $0$ is the linear hyperbolic matrix
\begin{equation*}
	\begin{pmatrix*}
		A(\hx) & 0\\
		0 & B(\hx)
	\end{pmatrix*}.
\end{equation*}
\begin{thm}[Pesin]
	The following holds for all $\epsilon$ sufficiently small.
	
	There exists a measurable function $Q(\chi,\epsilon) : \hat{\Lambda}_{\chi,\epsilon} \rightarrow \bbR^+$ with the following properties for all $\tilde{\epsilon}\leq 1$.
	\begin{enumerate}
		\item $Q(\chi,\epsilon)$ is continuous on each $\hNUH$, $Q(\chi,\epsilon) < \epsilon/2$ and $e^{-\epsilon} \leq \big(Q(\chi,\epsilon)\circ \hf\big)/Q(\chi,\epsilon) \leq e^{\epsilon}$.
		\item The map $f_{\hx} = \Psi_{\hf(\hx)}\inv \circ f\circ \Psi_{\hx}$ is a well-defined diffeomorphism from $R(0,2Q(\chi,\epsilon)(\hx))$ onto its image and we have:
		\begin{enumerate}
			\item $f_{\hx}(0) = 0$ and $d_0f_{\hx} = \begin{pmatrix*}A(\hx) & 0\\ 0 & B(\hx)\end{pmatrix*}$.
			\item $\norm{f_{\hx} - d_0f_{\hx}}_{\cC^1} < \tilde{\epsilon}$ on $R(0,\tilde{\epsilon}2Q(\chi,\epsilon)(\hx))$.
		\end{enumerate}
		\item The symmetric statement holds for $f_{\hx}\inv$.
	\end{enumerate}
	\label{thm:finPesinChart}
\end{thm}
From now on, we will write $q_{\tilde{\epsilon}}$ for the map $\tilde{\epsilon}Q(\chi,\epsilon)$. We emphasize that $q_{\tilde{\epsilon}}$ does depend on the parameters of the Pesin set, but we choose not to write explicitly the dependence since we will often work on a fixed Pesin set.\\
We can now give a definition of a Pesin chart.
\begin{defn}[Pesin charts]
	Let $\hx \in \hat{\Lambda}_{\chi,\epsilon}$ and let $\kappa \leq Q(\chi,\epsilon)(\hx)$. A Pesin chart at the point $\hx$ is a pair $\big(\Psi_{\hx}, \kappa\big)$ which denotes the map $\Psi_{\hx}$ restricted to $\overline{R(0,\kappa)}$, where $\kappa$ is called the size of the Pesin chart.
	\label{def:PesinChart}
\end{defn}
\subsection{Admissible manifolds and graph transform}
\begin{defn}[$u/s$-manifolds]
	Let $\hx \in \hat{\Lambda}_{\chi,\epsilon}$. Let $\kappa\leq Q(\chi,\epsilon)(\hx)$. A $u$-manifold $V^u$ in the chart $\big(\Psi_{\hx},\kappa\big)$ is a set of the form:
	\begin{equation*}
		\Psi_{\hx}\big(\{(F(t),t), \ |t|\leq q\}\big)
	\end{equation*}
	where $q\leq \kappa$ and $F : [-q,q] \rightarrow \bbR$ is a Lipschitz function such that $\norm{F}_{\infty} \leq \kappa$. We call $F$ the representative function of the $u$-manifold $V^u$ and $q$ its size.\\
	Similarly, a $s$-manifold $V^s$ in the chart $\big(\Psi_{\hx},\kappa\big)$ is a set of the form:
	\begin{equation*}
		\Psi_{\hx}\big(\{(t,G(t)), \ |t|\leq q\}\big)
	\end{equation*}
	where $q\leq \kappa$ and $G : [-q,q] \rightarrow \bbR$ is a Lipschitz function such that $\norm{G}_{\infty} \leq \kappa$. We call $G$ the representative function of the $s$-manifold $V^s$ and $q$ its size.
	\label{def:us-manifolds}
\end{defn}
Let us now define the parameters of a $u/s$-manifold.
\begin{defn}[Parameters of a $u/s$-manifold]
	Let $\hx \in \hat{\Lambda}_{\chi,\epsilon}$. Let $\kappa\leq Q(\chi,\epsilon)(\hx)$. Let $V^{u/s}$ be a $u/s$-manifold in $\big(\Psi_{\hx},\kappa\big)$ and let $F^{u/s}$ be its representative function. We define:
	\begin{itemize}[label={--}]
		\item $\varphi(V^{u/s}) = |F^{u/s}(0)|$,
		\item $\gamma(V^{u/s}) = \text{Lip}(F^{u/s})$,
		\item $q(V^{u/s})$ the size of $V^{u/s}$.
	\end{itemize}
	\label{def:ParametersusManifolds}
\end{defn}
In the following, we will have to use different notions of $u/s$-manifolds, more or less restrictive. Let us define them.
\begin{defn}[$\gamma$-admissible $u/s$-manifolds]
	Let $\hx \in \hat{\Lambda}_{\chi,\epsilon}$. Let $\gamma<\epsilon$ and $\kappa \leq Q(\chi,\epsilon)(\hx)$. Let $V^{u/s}$ be a $u/s$-manifold in $\big(\Psi_{\hx},\kappa\big)$. We say that $V^{u/s}$ is $\gamma$-admissible if:
	\begin{itemize}[label={--}]
		\item $\varphi(V^{u/s}) \leq 10^{-3}\kappa$.
		\item $\gamma(V^{u/s}) \leq \gamma$.
		\item $q(V^{u,s}) = \kappa$.
	\end{itemize}
	Write $\mathcal{L}^{u/s}_{\gamma}\big(\Psi_{\hx},\kappa\big)$ for the set of all $\gamma$-admissible $u/s$-manifolds at $\big(\Psi_{\hx},\kappa\big)$.
	\label{def:Admissible$u/s$-manifold}
\end{defn}
We will also need to define admissible manifolds tangent to a cone field.
\begin{defn}[Cone field]
	Let $K \subset M$. A cone field on $K$ is the data of a subset $\mathscr{C} \subset TM_{|K}$, a line bundle $E\subset TM_{|K}$, and a map $\gamma : K \rightarrow \bbR^+$ such that for all $x \in K$, we have:
	\begin{equation*}
		\mathscr{C}_x \subset T_xM = \big\{v \in T_xM, \ \measuredangle\big(v,E_x\big) \leq \gamma(x)\big\}
	\end{equation*}
	We call $E$ the direction of the cone field and $\gamma$ the opening map of the cone field.
	
	We say that the cone field $\mathscr{C}$ is continuous if $E$ and $\gamma$ are continuous.
	\label{def:ConeFiled}
\end{defn}
We define in the same way a cone field on a set of $\bbR^2$. Let us now define the notion of admissible $u$-manifold tangent to a cone field.

\begin{defn}[$(\mathscr{C},u)$-manifold]
	Let $\hx \in \hat{\Lambda}_{\chi,\epsilon}$. Let $\kappa\leq Q(\chi,\epsilon)(\hx)$. Let $V$ be a $u$-manifold in $\big(\Psi_{\hx},\kappa\big)$ and let $F$ be its representative function. Finally, let $\mathscr{C}$ be a continuous cone field on $\overline{R\big(0,\kappa\big)}$. We say that $V$ is a $(\mathscr{C},u)$-manifold if:
	\begin{itemize}[label={--}]
		\item $|F(0)| < 10^{-3}\kappa$.
		\item $(F^{'}(t),1) \in \text{Int}\big(\mathscr{C}_{(F(t),t)}\big)$ for all $t$ where $F$ is differentiable.
		\item $q(V) \geq 10^{-2}\kappa$.
	\end{itemize}
	
	We denote by $\mathcal{L}^{(\mathscr{C},u)}\big(\Psi_{\hx},\kappa\big)$ the set of all $(\mathscr{C},u)$-manifolds in $\big(\Psi_{\hx},\kappa\big)$.
	\label{def:(C,u)AdmissibleManifold}
\end{defn}
Remark that a $\gamma$-admissible $u$-manifold is a $(\mathscr{C},u)$-manifold for a constant cone field $\mathscr{C}$.\\
We then define a natural way to compare the distance of two $u/s$-manifolds.
\begin{defn}[Pesin distance]
	Let $\hx \in \hat{\Lambda}_{\chi,\epsilon}$. Let $\kappa\leq Q(\chi,\epsilon)(\hx)$. Let $V_1^{u/s},V_2^{u/s}$ be two $u/s$-manifolds in $\big(\Psi_{\hx},\kappa\big)$. Let $F_1^{u/s},F_2^{u/s}$ be their representative functions and let $q_i^{u/s} = q(V_i^{u/s})$ for $i=1,2$. We denote by $q^{u/s} := \min(q^{u/s}_1,q_2^{u/s})$. We define the Pesin distance between $V_1^{u/s}$ and $V_2^{u/s}$ to be:
	\begin{equation*}
		d_{\hx}(V_1^{u/s},V_2^{u/s}) = \norm{F_1^{u/s}-F_2^{u/s}}_{\cC^0([-q^{u/s},q^{u/s}])}.
	\end{equation*}
	\label{def:PesinDistance}
\end{defn}
The next lemma aims to show that, given a continuous cone field, any Lipschitz graph tangent to it has a uniformly bounded Lipschitz constant.
\begin{lemma}
	Let $r_1>0$ and let $\mathscr{C}$ be a continuous cone field on $\overline{R(0,r_1)}$. The following is true for $0<r_2\leq r_1$ small enough. There exists a constant $L := L(\mathscr{C},r_1,r_2)$ with the following property. Any Lipschitz graph $\big\{\big(F(t),t), \ |t|\leq r_2\big\}$, such that for almost every $t$ we have $(F^{'}(t),1) \in \text{Int}\big(\mathscr{C}_{(F(t),t)}\big)$, satisfies that $\text{Lip}(F)\leq L$.
	\label{lem:LipschitzCstAdmissibleCuManifold}
\end{lemma}
\begin{proof}
	Let $F$ be a Lipschitz function. Recall that since $F$ is Lipschitz, it is differentiable almost everywhere. Suppose that for any point $t \in [-r_2,r_2]$ where $F$ is differentiable at $t$, we have $(F^{'}(t),1) \in \text{Int}\big(\mathscr{C}_{(F(t),t)}\big)$.\\
	Up to take $r_2$ small enough, we can suppose by continuity of the cone field and compactness of $[-r_2,r_2]$, there exists a constant cone field $\mathscr{C}^{\star}$ (which means that its direction and opening coefficient are constant maps) such that for all $t \in [-r_2,r_2]$, we have:
	\begin{equation*}
		\mathscr{C}_{(F(t),t)} \subset \mathscr{C}^{\star}_{(F(t),t)}.
	\end{equation*}
	For $v \in \overline{R(0,r_1)}$, let $p_v : T_v\bbR^2 \rightarrow \bbR^2$ be the natural identification such that $p(0) = v$. Let $\tilde{\mathscr{C}}_{(F(t),t)} = p_{(F(t),t)}\big(\mathscr{C}^{\star}_{(F(t),t)}\big)$. Fix $t \in [-r_2,r_2]$ such that $F$ is differentiable at $t$. Since the map $t \rightarrow \tilde{\mathscr{C}}_{(F(t),t)}$ is continuous and $(F^{'}(t),1) \in \text{Int}(\mathscr{C}_{(F(t),t)})$, there exists $I_t$, a neighborhood of $t$, such that:
	\begin{equation*}
		\big\{(F(s),s), \ s \in I_t\big\} \subset \tilde{\mathscr{C}}_{(F(t),t)}.
	\end{equation*}
	Now, because $F$ is differentiable almost everywhere, the union of the intervals $I_t$ covers $[-r_2,r_2]$.
	Moreover, $p_{(F(t),t)}$ is an isometry and the cone $\tilde{\mathscr{C}}_{(F(t),t)}$ has the same constant direction and opening coefficient as $\mathscr{C}^{\star}$ at any $t$. These two properties imply that for any $t \in [-r_2,r_2]$, we have:
	\begin{equation*}
		\big\{\big(F(s),s\big), \ s \in [-r_2,r_2]\big\} \subset \tilde{\mathscr{C}}_{(F(t),t)}.
	\end{equation*}
	This gives us the desired property.
\end{proof}
Remark that the previous lemma implies that, given $\hx \in \hat{\Lambda}_{\chi,\epsilon}$ a continuous cone field $\mathscr{C}$ on $\overline{R\big(0,Q(\chi,\epsilon)\big)}$ and $0<\kappa\leq Q(\chi,\epsilon)$ small enough, any $\big(\mathscr{C},u\big)$-manifold in the chart $\big(\Psi_{\hx},\kappa\big)$ has a Lipschitz constant uniformly bounded by some constant depending only on $\mathscr{C}$ and $\kappa$.\\
Let us now give some classical results in Pesin theory concerning admissible manifolds.
\begin{prop}
	Let $\mathscr{C}$ be a continuous cone field on $\bbR^2$. The following holds for $\gamma := \gamma(\mathscr{C},\epsilon,\chi)$ sufficiently small.
	
	Let $\hx \in \hat{\Lambda}_{\chi,\epsilon}$. Let $\kappa\leq Q(\chi,\epsilon)(\hx)$. Let $V^u,V^s$ be respectively a $(\mathscr{C},u)$-manifold and a $\gamma$-admissible $s$-manifold in $\big(\Psi_{\hx},\kappa\big)$. Let $L$ be the constant given by Lemma~\ref{lem:LipschitzCstAdmissibleCuManifold} for the cone field $\mathscr{C}$ defined on $\overline{R(0,\chi/2)}$. Suppose that $\varphi(V^s)< \min\big(10^{-3}\kappa,\frac{10^{-3}}{L}\kappa\big)$. Then the following properties hold:
	\begin{enumerate}
		\item $V^u$ and $V^s$ intersect in a unique point that we denote by $[V^u,V^s] = \Psi_{\hx}(v,w)$.
		\item $|v|,|w|<2\times10^{-2}\kappa$.
	\end{enumerate}
	\label{prop:IntersectionAdmissiblemanifolds}
\end{prop}
\begin{proof}
	This type of argument is classical in Pesin theory.
	
	Let $V^u = \Psi_{\hx}\big(\big\{(F^u(t),t), \ t \in [-q(V^u), q(V^u)]\big\}\big)$ be a $(\mathscr{C},u)$-manifold and let $V^s=\Psi_{\hx}\big(\big\{(t,F^s(t)), \ t \in [-q(V^s),q(V^s)]\big\}\big)$ be a $\gamma$-admissible $s$-manifold, both in the Pesin chart $\big(\Psi_{\hx},\kappa\big)$, where $\kappa\leq Q(\chi,\epsilon)(\hx)$. By our hypothesis, we have $|F^s(0)|<\min\big(10^{-3}\kappa,\frac{10^{-3}}{L}\kappa\big)$. Consider the map $G=F^s\circ F^u$, which is well defined in $[-q(V^u),q(V^u)]$ since $q(V^s) = \kappa$.\\
	Note that for all $\epsilon<\chi/2$ and for all $\hx \in \hat{\Lambda}_{\chi,\epsilon}$, we have:
	\begin{equation*}
		\kappa \leq Q(\chi,\epsilon)(\hx) < \epsilon < \chi/2.
	\end{equation*}
	This implies in particular that $R(0,\kappa) \subset R(0,\chi/2)$. Let $L$ be the Lipschitz constant from Lemma~\ref{lem:LipschitzCstAdmissibleCuManifold} for the cone field $\mathscr{C}$ on $\overline{R(0,\chi/2)}$ and the constant $Q(\chi,\epsilon)(\hx)$ taken small enough, independently on $\hx$. Note that $L$ depends only on $\mathscr{C}$, $\epsilon$ and $\chi$.
	
	We will first show that $G([-q(V^u),q(V^u)]) \subset \big(-\frac{1}{\max(L,1)}10^{-2}\kappa,\frac{1}{\max(L,1)}10^{-2}\kappa\big)$. Indeed, take any $t \in [-q(V^u),q(V^u)]$, we have:
	\begin{equation*}
		\begin{aligned}
			|G(t)| \leq &|G(t)-G(0)| + |G(0)| \leq \gamma L|t| + |F^s\circ F^u(0)|\\
			&\leq \epsilon Lq(V^u) + |F^s\circ F^u(0) - F^s(0)| + |F^s(0)|\\
			&\leq \gamma Lq(V^u) + \gamma |F^u(0)| + |F^s(0)|\\
			&\leq \kappa \big(\gamma L + \gamma 10^{-3} + \min(10^{-3},10^{-3}/L)\big) < \kappa\min(10^{-2},10^{-2}/L)
		\end{aligned}
	\end{equation*}
	where the last inequality holds true for $\gamma$ small enough, depending only on $L$. Note that this also shows that the image of $G$ is contained in $[-q(V^u),q(V^u)]$ since $q(V^u)\geq 10^{-2}\kappa$.
	
	Secondly, let us prove that $V^s$ and $V^u$ intersect in a unique point. Remark that $G$ is $\gamma L$-Lipschitz and then a contraction of $[-q(V^u),q(V^u)]$ for $\gamma$ sufficiently small, again depending only on $L$. By the Banach-Picard fixed point theorem, $G$ has a unique fixed point $t^*$. Now, notice that $s \in [-q(V^u),q(V^u)]$ is a fixed point for $G$ if and only if $(F^u(s),s)$ is an intersection point of the two graphs $\big\{(F^u(t),t), \ t\in [-q(V^u),q(V^u)]\big\}$ and $\big\{(t,F^s(t)), \ t\in [-q(V^u),q(V^u)]\big\}$. We then have that $V^u$ and $V^s$ intersect in a unique point $\Psi_{\hx}(F^u(t^*),t^*)$ that we denote $[V^u,V^s]$. This finishes the proof of the first point.\\
	Finally, since $G([-q(V^u),q(V^u)]) \subset \big(-\frac{1}{\max(L,1)}10^{-2}\kappa,\frac{1}{\max(L,1)}10^{-2}\kappa\big)$, we have that $G(t^{\star}) = t^{\star} < \frac{1}{\max(L,1)}10^{-2}\kappa$. And then:
	\begin{equation*}
		|F^u(t^{\star})| \leq |F^u(t^{\star})-F^u(0)| + |F^u(0)| < 2\times10^{-2}\kappa.
	\end{equation*}
	which proves the second point and concludes the proof.
\end{proof}
The next proposition is the classical graph transform in our setting.
\begin{prop}[Graph transform]
	The following holds for all $\epsilon$ sufficiently small.
	
	Let $\mathscr{C}$ and $\tilde{\mathscr{C}}$ be two continuous cone fields on $\bbR^2$.
	The following holds for $\gamma:=\gamma(\chi,\epsilon)$ and $\tilde{\epsilon} := \tilde{\epsilon}(\chi,\mathscr{C},\gamma)$ sufficiently small.
	
	Let $\hx \in \hat{\Lambda}_{\chi,\epsilon}$. Let $\kappa\leq q_{\tilde{\epsilon}}(\hx)$ and $\tilde{\kappa}\leq q_{\tilde{\epsilon}}(\hf(\hx))$ such that $e^{-\tilde{\epsilon}}\leq \tilde{\kappa}/\kappa\leq e^{\tilde{\epsilon}}$. Suppose that $df_{\hx} \mathscr{C} \subset \text{Int}(\tilde{\mathscr{C}})$ on $\overline{R(0,q_{\tilde{\epsilon}}(\hx))}$. Then we have:
	\begin{enumerate}
		\item Let $V^u$ be a $(\mathscr{C},u)$-manifold in $\big(\Psi_{\hx},\kappa\big)$. Then $f(V^u)$ contains a unique $(\tilde{\mathscr{C}},u)$-manifold $\tilde{V}^u$ in $\big(\Psi_{\hf(\hx)}, \tilde{\kappa}\big)$ such that $q(\tilde{V}^u) = \max\big(q(V^u)e^{\chi/2 - \sqrt{\tilde{\epsilon}}}, \tilde{\kappa}\big)$.
		\item Let $V^{s}$ be a $\gamma$-admissible $s$-manifold in $\big(\Psi_{\hx},\kappa\big)$. Then $f\inv(V^s)$ contains a unique $\gamma$-admissible $s$-manifold in $\big(\Psi_{\hf\inv(\hx)}, \tilde{\kappa}\big)$.
		\item The symmetric statement holds for admissible $u$-manifolds.
	\end{enumerate}
	\label{prop:GraphTranform}
\end{prop}
\begin{proof}
	We prove the proposition in the case of $(\mathscr{C},u)$-manifolds. The case of $\gamma$-admissible $u/s$-manifolds is classical.
	
	Let $V^u = \Psi_{\hx}\big(\big\{(F^u(t),t), \ t \in [-q(V^u), q(V^u)]\big\}\big)$ be a $(\mathscr{C},u)$-manifold in $\big(\Psi_{\hx},\kappa\big)$ and let $L$ be the constant given by Lemma \ref{lem:LipschitzCstAdmissibleCuManifold} for the cone field $\mathscr{C}$ on $\overline{R(0,\chi/2)}$ and the constant $Q(\chi,\epsilon)(\hx)$ taken small enough uniformly on $\hx$. We first prove that $f(V^u)$ contains a Lipschitz graph in $\big(\Psi_{\hf(\hx)},\tilde{\kappa}\big)$.
	
	Let $f_{\hx}$ be the map defined in Theorem \ref{thm:finPesinChart}. Recall that:
	\begin{equation}
		f_{\hx} (u,v) = \big(Au + h_1(u,v), Bv + h_2(u,v)\big)
		\label{eq:finCharts}
	\end{equation}
	where $|A|,|B\inv| \leq e^{-\chi / 2}$ and $\norm{h_i}_{\cC^1} < \tilde{\epsilon}$ on $R\big(0,\kappa\big)$.
	
	For $t \in [-q(V^u),q(V^u)]$, define:
	\begin{equation}
		\tau(t) = Bt + h_2(F^u(t),t).
		\label{eq:MapTau}
	\end{equation}
	We want to show that $\tau$ is invertible, on a domain which we will evaluate, because it will imply that $f_{\hx}\big(\{(F^u(t),t), t \in [-q(V^u),q(V^u)]\}\big)$ is a graph on such a domain. Then, for $t \in [-q(V^u),q(V^u)]$ and $\tilde{\tau} \in \tau([-q(V^u),q(V^u)])$, define the map:
	\begin{equation*}
		h(t) = B\inv\big(\tilde{\tau}-h_2(F^u(t),t)\big).
	\end{equation*}
	Note that $h$ has a unique fixed point if and only if $\tau(t) = \tilde{\tau}$ has a unique solution.
	
	Let $t_1,t_2 \in [-q(V^u),q(V^u)]$, we then have:
	\begin{equation*}
		|h(t_1)-h(t_2)| = |B\inv \big(h_2(F^u(t_1),t_1) - h_2(F^u(t_2),t_2)\big)| < e^{-\chi/2}\tilde{\epsilon}(1+L)|t_1-t_2| < |t_1-t_2|.
	\end{equation*}
	By the contraction mapping theorem, if $\tilde{\epsilon}$ is small enough, $h$ has a unique fixed point for all $\tilde{\tau} \in \tau([-q(V^u),q(V^u)])$ and then $\tau\inv$ is well defined on the same domain.
	
	Let us now evaluate this domain. Let $t \in [-q(V^u),q(V^u)]$ and remark that:
	\begin{equation*}
		|\tau(t)| > e^{\chi/2}t - \tilde{\epsilon}(1+L)t.
	\end{equation*}
	this implies that $\tau([-q(V^u),q(V^u)]) \supset \big(-q(V^u)(e^{\chi /2} - \tilde{\epsilon}(1+L)),q(V^u) (e^{\chi /2} - \tilde{\epsilon}(1+L))\big) \supset [-q(V^u)e^{\chi/2 - \sqrt{\tilde{\epsilon}}}, q(V^u)e^{\chi/2 - \sqrt{\tilde{\epsilon}}}]$ provided $\tilde{\epsilon}$ is small enough. Note that $q(V^u)e^{\chi/2 - \sqrt{\tilde{\epsilon}}}\geq 10^{-2}\tilde{\kappa}$ for $\tilde{\epsilon}$ small enough since $\tilde{\kappa}/\kappa\geq e^{-\tilde{\epsilon}}$.
	
	We just showed that:
	\begin{equation*}
		f_{\hx}\big(\{(F^u(t),t), t \in [-q(V^u),q(V^u)]\}\big) \supset \{(G^u(t),t), \ t \in [-q(V^u)e^{\chi/2 - \sqrt{\tilde{\epsilon}}}, q(V^u)e^{\chi/2 - \sqrt{\tilde{\epsilon}}}]\}
	\end{equation*}
	for a map $G^u : [-q(V^u)e^{\chi/2 - \sqrt{\tilde{\epsilon}}}, q(V^u)e^{\chi/2 - \sqrt{\tilde{\epsilon}}}] \rightarrow \bbR$.\\
	We will now show that $G^u$ is Lipschitz and tangent to the cone field $\tilde{\mathscr{C}}$ for almost every point.
	
	First, let $t_1,t_2 \in [-q(V^u)e^{\chi/2 - \sqrt{\tilde{\epsilon}}}, q(V^u)e^{\chi/2 - \sqrt{\tilde{\epsilon}}}]$. We have:
	\begin{equation*}
		\begin{aligned}
			|G^u(\tau(t_1)) - G^u(\tau(t_2))| &= |A\big(F^u(t_1)-F^u(t_2)\big) + h_1\big(F^u(t_1),t_1\big) - h_1\big(F^u(t_2),t_2\big)|\\
			&\leq e^{-\chi/2} L|t_1-t_2| + \tilde{\epsilon}(1+L)|t_1-t_2|.
		\end{aligned}
	\end{equation*}
	On the other hand, we have:
	\begin{equation*}
		\begin{aligned}
			|\tau(t_1)-\tau(t_2)| &= |B(t_1-t_2) + h_2(F^u(t_1),t_1) - h_2(F^u(t_2),t_2)|\\
			&\geq e^{\chi /2}|t_1-t_2| - \tilde{\epsilon}(1+L)|t_1-t_2|.
		\end{aligned}
	\end{equation*}
	Note that this implies that $\tau\inv$ is $\big(e^{\chi/2}-\tilde{\epsilon}(1+L)\big)\inv$-Lipschitz on its domain.
	
	Finally, we then have:
	\begin{equation*}
		|G^u(\tau(t_1)) - G^u(\tau(t_2))| \leq \frac{Le^{-\chi /2} + \tilde{\epsilon}(1+L)}{e^{\chi/2} - \tilde{\epsilon}(1+L)} |\tau(t_1)-\tau(t_2)|.
	\end{equation*}
	Taking $\tilde{\epsilon}$ small enough ensures that $G^u$ is $L$-Lipschitz on $[-q(V^u)e^{\chi/2 - \sqrt{\tilde{\epsilon}}}, q(V^u)e^{\chi/2 - \sqrt{\tilde{\epsilon}}}]$. It remains to show that $G^u$ is tangent to $\tilde{\mathscr{C}}$ at almost every point. Let $t \in [-q(V^u),q(V^u)]$ such that $F^u$ is differentiable at $t$. Note that this implies that $G^u$ is differentiable at $\tau(t)$. Remark that:
	\begin{equation*}
		d_{(F^u(t),t)}f_{\hx} \big((F^u)^{'}(t),1\big) = \alpha\big((G^u)^{'}(\tau(t)), 1\big) \in d_{(F^u(t),t)}f_{\hx} \mathscr{C}_{(F^u(t),t)} \subset \text{Int}\big(\tilde{\mathscr{C}}_{(G^u(\tau(t)),\tau(t))}\big)
	\end{equation*}
	for a real number $\alpha$.
	
	To conclude the proof, we just have to estimate $G(0)$:
	\begin{equation*}
		|G(0)| = |AF^u(\tau\inv(0)) + h_1(F^u(\tau\inv(0)),0)| < (e^{-\chi/2}+\tilde{\epsilon} L)|\tau\inv(0)|.
	\end{equation*}
	On the other hand:
	\begin{equation*}
		\begin{aligned}
			|\tau\inv(0)|& = |\tau\inv(0) - \tau\inv(\tau(0))| \leq \big(e^{\chi/2}-\tilde{\epsilon}(1+L)\big)\inv |h_2(F^u(0),0)|\\
			&\leq \big(e^{\chi/2}-\tilde{\epsilon}(1+L)\big)\inv \tilde{\epsilon} |F^u(0)|.
		\end{aligned}
	\end{equation*}
	
	We then have:
	\begin{equation*}
		|G(0)| < (e^{-\chi/2}+\tilde{\epsilon} L) \big(e^{\chi/2}-\tilde{\epsilon}(1+L)\big)\inv \tilde{\epsilon} 10^{-3}q(V^u) < 10^{-3}|F^u(0)| e^{\chi - \sqrt{\tilde{\epsilon}}}<10^{-3}\tilde{\kappa}
	\end{equation*}
	for $\tilde{\epsilon}$ small enough, since $\tilde{\kappa}/\kappa\leq e^{\tilde{\epsilon}}$. Note that the $\tilde{\epsilon}$ that we have to take to make everything work only depends on $L$ and $\chi$. This then concludes the proof.
\end{proof}
This allows us to define the graph transform operator, which acts on all $u/s$-admissible manifolds.
\begin{defn}[Graph transform operator]
	Let $\mathscr{C}$, $\tilde{\mathscr{C}}$ be two continuous cone fields on $\bbR^2$. Let $\tilde{\epsilon}$ and $\gamma$ be small enough so that Proposition \ref{prop:GraphTranform} holds. Let $\hx \in \hat{\Lambda}_{\chi,\epsilon}$. Let $\kappa\leq q_{\tilde{\epsilon}}(\hx)$ and $\tilde{\kappa}\leq q_{\tilde{\epsilon}}(\hf(\hx))$ such that $e^{-\tilde{\epsilon}}\leq \tilde{\kappa}/\kappa\leq e^{\tilde{\epsilon}}$. Suppose that $df_{\hx} \mathscr{C} \subset \text{Int}(\tilde{\mathscr{C}})$ on $\overline{R(0,q_{\tilde{\epsilon}}(\hx))}$. Let $V^u$ be a $(\mathscr{C},u)$-manifold in $\big(\Psi_{\hx},\kappa\big)$. The graph transform operator at $\hx$ is the map:
	\begin{equation*}
		\mathcal{F}^u_{\hx} : \mathcal{L}^{(\mathscr{C},u)}\big(\Psi_{\hx},\kappa\big) \rightarrow \mathcal{L}^{(\tilde{\mathscr{C}},u)}\big(\Psi_{\hf(\hx)},\tilde{\kappa}\big)
	\end{equation*}
	such that $\mathcal{F}^u_{\hx}(V^u)$ is the unique $(\tilde{\mathscr{C}},u)$-manifold in $\big(\Psi_{\hf(\hx)},\tilde{\kappa}\big)$ contained in $f(V^u)$ with $q(\mathcal{F}^u_{\hx}(V^u)) = \max\big(q(V^u)e^{\chi/2 - \sqrt{\tilde{\epsilon}}},\tilde{\kappa}\big)$. We choose not to mark the dependence on the cone fields and on the size of the Pesin charts to lighten the notation. We will then explicitly specify the cone fields and the size we choose in each use of the graph transform.
	
	The operator $\mathcal{F}^u_{\hx}$ also acts naturally on $\mathcal{L}^u_{\gamma}\big(\Psi_{\hx},\kappa\big)$. We similarly define $\mathcal{F}^s_{\hx}$.
	\label{def:GraphTransformOperator}
\end{defn}
\begin{prop}[The graph transform is a contraction]
	Let $\mathscr{C}$ and $\tilde{\mathscr{C}}$ be two continuous cone fields on $\bbR^2$. The following holds for $\epsilon$ and $\tilde{\epsilon}:=\tilde{\epsilon}(\chi,\mathscr{C})$ and $\gamma:=\gamma(\chi,\tilde{\epsilon})$ small enough.
	
	Let $\hx \in \hat{\Lambda}_{\chi,\epsilon}$. Let $\kappa\leq q_{\tilde{\epsilon}}(\hx)$ and $\tilde{\kappa}\leq q_{\tilde{\epsilon}}(\hf(\hx))$ such that $e^{-\tilde{\epsilon}}\leq \tilde{\kappa}/\kappa\leq e^{\tilde{\epsilon}}$. Suppose that $df_{\hx} \mathscr{C} \subset \text{Int}(\tilde{\mathscr{C}})$ on $\overline{R(0,q_{\tilde{\epsilon}}(\hx))}$. Let $V^u_1,V^u_2$ be two $(\mathscr{C},u)$-manifolds in $\big(\Psi_{\hx},\kappa\big)$. Then we have:
	\begin{equation*}
		d_{\hf(\hx)}\big(\mathcal{F}^{u}_{\hx}(V^{u}_1),\mathcal{F}^{u}_{\hx}(V^{u}_2)\big) \leq e^{-\chi/4}d_{\hx}(V^{u}_1,V^{u}_2).
	\end{equation*}
	
	The same statement also holds for graph transforms of $\gamma$-admissible $u/s$-manifolds.
	
	The same statement also holds for any sub-curve of the $(\mathscr{C},u)$-manifolds $V^u_i$.
	\label{prop:ContractionTransGraph}
\end{prop}
\begin{proof}
	We prove the proposition in the case of $(\mathscr{C},u)$-manifolds. The case of $\gamma$-admissible $u/s$-manifolds is classical. Take $\tilde{\epsilon}$ such that the graph transform (Definition \ref{def:GraphTransformOperator} and Proposition \ref{prop:GraphTranform}) is well defined.
	
	Let $V^u_i = \Psi_{\hx}\big(\big\{(F^u_i(t),t), \ t \in [-q(V^u_i), q(V^u_i)]\big\}\big)$ be $(\mathscr{C},u)$-manifolds in $\big(\Psi_{\hx},\kappa\big)$. Denote by $q = \min(q(V^u_1),q(V^u_2))$. Denote by $G^u_1$ and $G^u_2$ the representative functions of $\mathcal{F}^u_{\hx}(V^u_1)$ and $\mathcal{F}^u_{\hx}(V^u_2)$. Recall that by Proposition \ref{prop:GraphTranform}, $q(\mathcal{F}^u_{\hx}(V^u_i)) = \max\big(q(V^u_i)e^{\chi/2 - \sqrt{\tilde{\epsilon}}},\tilde{\kappa}\big)$. Denote by $\tilde{q} = \min\big(q(\mathcal{F}^u_{\hx}(V^u_1)),q(\mathcal{F}^u_{\hx}(V^u_2))\big)$. Recall the definition of the map $f_{\hx}$ from \eqref{eq:finCharts} and the definition of the maps $\tau_i$ from \eqref{eq:MapTau}. Finally, let $L$ be the constant given by Lemma \ref{lem:LipschitzCstAdmissibleCuManifold} for the cone field $\mathscr{C}$ and the constant $Q(\chi,\epsilon)(\hx)$ taken small enough, independently from $\hx$.
	
	Let $t_1,t_2\in [-q,q]$ such that $\tau_1(t_1) = \tau_2(t_2) \in [-\tilde{q},\tilde{q}]$. We have:
	\begin{equation*}
		\begin{aligned}
			|G^u_1(\tau_1(t_1)) - G^u_2(\tau_2(t_2))| &= |A\big(F^u_1(t_1) - F^u_2(t_2)\big) + h_1(F^u_1(t_1),t_1) - h_1(F^u_2(t_2),t_2)|\\
			&\leq e^{-\chi/2} |F^u_1(t_1) - F^u_2(t_2)| + \tilde{\epsilon}(|F^u_1(t_1) - F^u_2(t_2)| + |t_1-t_1|)\\
			&\leq (e^{-\chi/2}+\tilde{\epsilon})(|F^u_1(t_1)-F^u_2(t_1)| + |F^u_2(t_1)-F^u_2(t_2)|) + \tilde{\epsilon} |t_1-t_2|\\
			&\leq(e^{-\chi/2}+\tilde{\epsilon})\norm{F^u_1-F^u_2}_{\cC^0([-q,q])} + \big(L(e^{-\chi/2}+\tilde{\epsilon}) + \tilde{\epsilon}\big)|t_1-t_2|.
		\end{aligned}
	\end{equation*}
	
	On the other hand, we have:
	\begin{equation*}
		\begin{aligned}
			|B||t_1-t_2| &= |h_2(F^u_1(t_1),t_1) - h_2(F^u_2(t_2),t_2)|\\
			&\leq \tilde{\epsilon}\big( |F^u_1(t_1) - F^u_2(t_2)| + |t_1-t_2|\big)\\
			&\leq \tilde{\epsilon} \norm{F^u_1-F^u_2}_{\cC^0([-q,q])} + (L\tilde{\epsilon} + \tilde{\epsilon})|t_1-t_2|.
		\end{aligned}
	\end{equation*}
	This implies that:
	\begin{equation*}
		\big(e^{\chi/2}-(1+L)\tilde{\epsilon}\big)|t_1-t_2| \leq \tilde{\epsilon}\norm{F^u_1-F^u_2}_{\cC^0([-q,q])}
	\end{equation*}
	and then, if $\tilde{\epsilon}$ is small enough:
	\begin{equation*}
		|t_1-t_2| \leq \tilde{\epsilon} \norm{F^u_1-F^u_2}_{\cC^0([-q,q])}.
	\end{equation*}
	
	Putting it all together, we obtain:
	\begin{equation*}
		\begin{aligned}
			|G^u_1(\tau_1(t_1)) - G^u_2(\tau_2(t_2))| &\leq \Big(e^{-\chi/2} + \tilde{\epsilon} + \tilde{\epsilon}\big(L(e^{-\chi/2}+\tilde{\epsilon}) + \tilde{\epsilon}\big)\Big)\norm{F^u_1-F^u_2}_{\cC^0([-q,q])}\\
			&\leq e^{-\chi/4}\norm{F^u_1-F^u_2}_{\cC^0([-q,q])}
		\end{aligned}
	\end{equation*}
	for $\tilde{\epsilon}$ small enough, depending only on $L$ and $\chi$.
	
	Remark that the proof does not use the fact that $q(V^u_i)\geq \kappa$. The same statement then holds for any sub-curve of the $V^u_i$ and this concludes the proof.
\end{proof}
We are now able to define our local stable and unstable manifolds for points in $\hat{\Lambda}_{\chi,\epsilon}$. Let $\tilde{\mathcal{L}}_{\gamma}^{u/s}\big(\Psi_{\hx},\kappa\big)$ be the set of $\gamma$-admissible $u/s$-manifolds $V^{u/s}$ in $\big(\Psi_{\hx},\kappa\big)$ such that $\varphi(V^{u/s}) = 0$. We endow this space with the Pesin distance. The proof of the following theorem is classical using graph transform.
\begin{thm}[Pesin invariant manifold]
	The following holds for any $\epsilon$, $\gamma$ and $\tilde{\epsilon}:=\tilde{\epsilon}(\chi,\gamma)$ sufficiently small.
	
	Let $\hx \in \hat{\Lambda}_{\chi,\epsilon}$. There exist two unique families of $\gamma$-admissible respectively $u$- and $s$-manifolds $(V^u_m)_{m\in \bbZ}$ and $(V^s_m)_{m \in \bbZ}$ with the following properties.
	\begin{enumerate}
		\item For each $m \in \bbZ$, $V^u_m \in \tilde{\mathcal{L}}_{\gamma}^u\big(\Psi_{\hf^m(\hx)},q_{\tilde{\epsilon}}(\hf^m(\hx))\big)$ and $V^s_m \in \tilde{\mathcal{L}}_{\gamma}^s\big(\Psi_{\hf^m(\hx)},q_{\tilde{\epsilon}}(\hf^m(\hx))\big)$.
		\item For each $m\in \bbZ$, we have $\mathcal{F}^u_{\hf^m(\hx)}(V^u_m) = V^u_{m+1}$ and $\mathcal{F}^s_{\hf^m(\hx)}(V^s_m) = V^s_{m-1}$.
	\end{enumerate}
	\label{thm:Localu/sManifolds}
\end{thm}
\begin{defn}[Local invariant manifolds]
	We fix some choices of $\epsilon,\gamma,\tilde{\epsilon}>0$ such that Theorem \ref{thm:Localu/sManifolds} holds. Let $\hx \in \Lambda_{\chi,\epsilon}$. We define the local unstable manifold at $\hf^m(\hx)$ by $\Wu_{loc}(\hf^m(\hx)) = V^u_m$. These are $\cC^{r}$ embedded curves. We define  local stable manifolds similarly.
	\label{def:LocalInvManifolds}
\end{defn}
Note that these curves are $f$-invariant in the following way:
\begin{equation*}
	f\big(\Ws_{loc}(\hx)\big) \subset \Ws_{loc}(\hf(\hx)) \ \text{and} \ \Wu_{loc}(\hf(\hx)) \subset f\big(\Wu_{loc}(\hx)\big).
\end{equation*}
Note also that the map $\hx \rightarrow W^{u/s}_{loc}(\hx)$ is continuous for the $\cC^1$ topology on Pesin blocks. Since the stable bundle depend only on the points in the manifold, the map $x \mapsto \Ws_{loc}(x)$ is continuous on each Pesin block and for any point $x$ belonging to some Pesin block and any $\hx_1,\hx_2 \in \pi\inv(x)$ we have $\Ws_{loc}(\hx_1)=\Ws_{loc}(\hx_2)$.

Finally, we want to compare $\gamma$-admissible manifolds defined in Pesin charts of points in the same Pesin blocks which are close enough. This is done by the next lemma.
\begin{lemma}
	The following holds for $\gamma$ and $\tilde{\epsilon} := \tilde{\epsilon}(\gamma)$ sufficiently small.
	
	Let $l\geq 1$. There exists a positive constant $\delta_1 := \delta_1(\hNUH)$ with the following property. For each $\hx,\hy \in \hNUH$ such that $d(\hx,\hy)<\delta_1$ and for each $\gamma$-admissible $u/s$-manifold $V^{u/s}_{\hx}$ in $\big(\Psi_{\hx}, q_{\tilde{\epsilon}}(\hx)\big)$, there exists a $u/s$-manifold $V^{u/s}_{\hy}$ in $\big(\Psi_{\hy}, q_{\tilde{\epsilon}}(\hy)\big)$ such that:
	\begin{itemize}[label={--}]
		\item $V^{u/s}_{\hx} = V^{u/s}_{\hy}$ on $\Psi_{\hy}\big(R(e^{-4\tilde{\epsilon}}q_{\tilde{\epsilon}}(\hy))\big)$.
		\item $\varphi(V^{u/s}_{\hy}) < e^{\tilde{\epsilon}} \varphi(V^{u/s}_{\hx})$.
		\item $\gamma(V^{u/s}_{\hy}) \leq e^{2\tilde{\epsilon}} \gamma(V^{u/s}_{\hx})$.
	\end{itemize}
	\label{lem:AdmissibleManifoldinCloseCharts}
\end{lemma}
\begin{proof}
	We give the details in the case of $s$-manifolds; the case of $u$-manifolds is completely analogous.
	
	Recall the fact from Theorem \ref{thm:finPesinChart} that the maps $q_{\tilde{\epsilon}}$ and $\hx \mapsto \Psi_{\hx}$ are uniformly continuous on $\hNUH$ (the latter is continuous for the $\cC^1$ topology). Then we can choose $\delta>0$ such that for any $\hx,\hy \in \hNUH$ satisfying $d(\hx,\hy)<\delta$, we have:
	\begin{itemize}[label={--}]
		\item $e^{-\tilde{\epsilon}}\leq q_{\tilde{\epsilon}}(\hx)/q_{\tilde{\epsilon}}(\hy)\leq e^{\tilde{\epsilon}}$,
		\item $\Psi_{\hx}\big(R(0,e^{-2\tilde{\epsilon}}q_{\tilde{\epsilon}}(\hx))\big) \subset \Psi_{\hy}\inv\big(R(0,q_{\tilde{\epsilon}}(\hy)\big)$ and $\Psi_{\hy}\big(R(0,e^{-2\tilde{\epsilon}}q_{\tilde{\epsilon}}(\hy))\big) \subset \Psi_{\hx}\inv\big(R(0,q_{\tilde{\epsilon}}(\hx)\big)$,
		\item $\norm{\Psi_{\hy}\inv\circ\Psi_{\hx}-\id}_{\cC^1}<\tilde{\epsilon}$ on $R(0,e^{-2\tilde{\epsilon}}q_{\tilde{\epsilon}}(\hx))$ and $\norm{\Psi_{\hx}\inv\circ\Psi_{\hy}-\id}_{\cC^1}<\tilde{\epsilon}^2$ on $R(0,e^{-2\tilde{\epsilon}}q_{\tilde{\epsilon}}(\hy))$.
	\end{itemize}
	Let $V^s_{\hx} = \Psi_{\hx}\big(\{(t,F^s(t)), |t|\leq q_{\tilde{\epsilon}}(\hx)\}\big)$ be a $s$-manifold in $\big(\Psi_{\hx},q_{\tilde{\epsilon}}(\hx)\big)$. Let $p_1,p_2 : \bbR^2 \rightarrow \bbR$ be the projections on the first and second coordinate. We will first show that $\Psi_{\hy}\inv(V^s_{\hx})$ is a graph, where it makes sense.\\
	Define the map $h : t\mapsto p_1\circ\Psi_{\hy}\inv\circ\Psi_{\hx}(t,F^s(t))$, which is well defined on $[-e^{-2\tilde{\epsilon}}q_{\tilde{\epsilon}}(\hx),e^{-2\tilde{\epsilon}}q_{\tilde{\epsilon}}(\hx)]$ since $q(V^s_{\hx}) = q_{\tilde{\epsilon}}(\hx)$. As in the proof of Proposition \ref{prop:GraphTranform}, if $h$ is one-to-one, then $\Psi_{\hy}\inv(V^s_{\hx})$ is a graph on the image of $h$. Write $h = \id + \tilde{h}$. It is sufficient to show that $\tilde{h}$ is a contraction of $[-e^{-2\tilde{\epsilon}}q_{\tilde{\epsilon}}(\hx),e^{-2\tilde{\epsilon}}q_{\tilde{\epsilon}}(\hx)]$. Let $t_1,t_2 \in [-e^{-2\tilde{\epsilon}}q_{\tilde{\epsilon}}(\hx),e^{-2\tilde{\epsilon}}q_{\tilde{\epsilon}}(\hx)]$, we have:
	\begin{equation*}
		\begin{aligned}
			|\tilde{h}(t_1)-\tilde{h}(t_2)| &= |t_1-\big(p_1\circ\Psi_{\hy}\inv\circ\Psi_{\hx}\big)(t_1,F^s(t_1)) - t_2 + \big(p_1\circ\Psi_{\hy}\inv\circ\Psi_{\hx}\big)(t_2,F^s(t_2))|\\
			&=|p_1\circ\big(\id - \Psi_{\hy}\inv\circ\Psi_{\hx}\big)(t_1,F^s(t_1)) - p_1\circ\big(\id - \Psi_{\hy}\inv\circ\Psi_{\hx}\big)(t_2,F^s(t_2))|\\
			&\leq |\big(\id - \Psi_{\hy}\inv\circ\Psi_{\hx}\big)(t_1,F^s(t_1)) - \big(\id - \Psi_{\hy}\inv\circ\Psi_{\hx}\big)(t_2,F^s(t_2))|\\
			&\leq\sup_{v \in R(0,e^{-2\tilde{\epsilon}}q_{\tilde{\epsilon}}(\hx))}\norm{d_{v}\big(\id - \Psi_{\hy}\inv\circ\Psi_{\hx}\big)} \norm{(t_1,F^s(t_1)) - (t_2,F^s(t_2))}\\
			&<\tilde{\epsilon}^2\big(|t_1-t_2|+|F^s(t_1)-F^s(t_2)|\big)\\
			&<\tilde{\epsilon}^2(1+\gamma)|t_1-t_2|.
		\end{aligned}
	\end{equation*}
	This then proves that $\tilde{h}$ is a contraction for a good choice of $\tilde{\epsilon}$. We now want to estimate the size of the image by $h$ of $[-e^{-2\tilde{\epsilon}}q_{\tilde{\epsilon}}(\hx),e^{-2\tilde{\epsilon}}q_{\tilde{\epsilon}}(\hx)]$. Let $t \in [-e^{-2\tilde{\epsilon}}q_{\tilde{\epsilon}}(\hx),e^{-2\tilde{\epsilon}}q_{\tilde{\epsilon}}(\hx)]$, we have:
	\begin{equation*}
		\begin{aligned}
			|h(t)| &= |t+\tilde{h}(t)|\geq |t| - |\tilde{h}(t)|\\
			&\geq |t| - \big(|\tilde{h}(t)-\tilde{h}(0)| + |\tilde{h}(0)|\big)\\
			&> |t| - \big(\tilde{\epsilon}^2(1+\gamma)|t| + \norm{(\id-\Psi_{\hy}\inv\circ\Psi_{\hx})(0,F^s(0))}\big)\\
			&>|t|-\big(\tilde{\epsilon}^2(1+\gamma)|t| + \tilde{\epsilon}^2\big)\\
			&> |t|\big(1-\tilde{\epsilon}^2(1+\gamma)\big) - \tilde{\epsilon}^2\\
			&>|t|(1-2\tilde{\epsilon}^2)-\tilde{\epsilon}^2
		\end{aligned}
	\end{equation*}
	for $\gamma$ sufficiently small. This shows that $h([-e^{-2\tilde{\epsilon}}q_{\tilde{\epsilon}}(\hx),e^{-2\tilde{\epsilon}}q_{\tilde{\epsilon}}(\hx)]) \supset [-(1-2\tilde{\epsilon})e^{-2\tilde{\epsilon}}q_{\tilde{\epsilon}}(\hx) + \tilde{\epsilon}, (1-2\tilde{\epsilon})e^{-2\tilde{\epsilon}}q_{\tilde{\epsilon}}(\hx) - \tilde{\epsilon}]\supset [-e^{-4\tilde{\epsilon}}q_{\tilde{\epsilon}}(\hy),e^{-4\tilde{\epsilon}}q_{\tilde{\epsilon}}(\hy)]$ for $\tilde{\epsilon}$ small enough.
	We just showed that $\Psi_{\hy}\inv(V^s_{\hx}) \cap R(0,e^{-4\tilde{\epsilon}}q_{\tilde{\epsilon}}(\hy)) = \big\{(t,G^s(t)), |t|\leq e^{-4\tilde{\epsilon}}q_{\tilde{\epsilon}}(\hy)\big\}$ for some map $G^s:[-e^{-4\tilde{\epsilon}}q_{\tilde{\epsilon}}(\hy),e^{-4\tilde{\epsilon}}q_{\tilde{\epsilon}}(\hy)] \rightarrow \bbR$.\\
	We will now show that $G^s$ is Lipschitz and estimate its constant. Let $\tau_1,\tau_2 \in [-e^{-4\tilde{\epsilon}}q_{\tilde{\epsilon}}(\hy),e^{-4\tilde{\epsilon}}q_{\tilde{\epsilon}}(\hy)]$. Let also $t_1,t_2 \in [-e^{-2\tilde{\epsilon}}q_{\tilde{\epsilon}}(\hx),e^{-2\tilde{\epsilon}}q_{\tilde{\epsilon}}(\hx)]$ such that $h(t_i)=\tau_i$. We have:
	\begin{equation*}
		\begin{aligned}
			|G^s(\tau_1)-G^s(\tau_2)|&= |p_2\circ\big(\id + (\Psi_{\hy}\inv\circ\Psi_{\hx} - \id)\big)(t_1,F^s(t_1))\\
			& - p_2\circ\big(\id + (\Psi_{\hy}\inv\circ\Psi_{\hx} - \id)\big)(t_2,F^s(t_2))|\\
			&\leq |F(t_1)-F(t_2)| + |(\Psi_{\hy}\inv\circ\Psi_{\hx} - \id)(t_1,F^s(t_1)) - (\Psi_{\hy}\inv\circ\Psi_{\hx} - \id)(t_2,F^s(t_2))|\\
			&< \text{Lip}(F^s)|t_1-t_2| + \tilde{\epsilon}^2\big(1+\text{Lip}(F^s)\big)|t_1-t_2|\\
			&\leq\big(\text{Lip}(F^s)+\tilde{\epsilon}^2(1+\text{Lip}(F^s)\big)|t_1-t_2|.
		\end{aligned}
	\end{equation*}
	On the other hand, we have:
	\begin{equation*}
		\begin{aligned}
			|\tau_1-\tau_2| &\geq |t_1-t_2| - |\big(\id-\Psi_{\hy}\inv\Psi_{\hx}\big)(t_1,F^s(t_1))-\big(\id-\Psi_{\hy}\inv\Psi_{\hx}\big)(t_2,F^s(t_2))|\\
			&> \big(1-\tilde{\epsilon}^2(1+\gamma)\big)|t_1-t_2|.
		\end{aligned}
	\end{equation*}
	Putting it all together, we finally have:
	\begin{equation*}
		|G^s(\tau_1)-G^s(\tau_2)|\leq \frac{\text{Lip}(F^s)+\tilde{\epsilon}^2\big(1+\text{Lip}(F^s)\big)}{1-\tilde{\epsilon}^2\big(1+\gamma\big)}|\tau_1-\tau_2|\leq e^{2\tilde{\epsilon}}\text{Lip}(F^s)|\tau_1-\tau_2|
	\end{equation*}
	the last inequality being true for $\tilde{\epsilon}$ small enough. This gives us the desired $s$-manifold $V^s_{\hy}$ from the statement.\\
	We still have to estimate $\varphi(V^s_{\hy})$. We have:
	\begin{equation*}
		|G^s(0)| = |p_2\circ\big(\id + (\Psi_{\hy}\inv\circ\Psi_{\hx} - \id)\big)(0,F^s(0))| < (1+\tilde{\epsilon}^2)|F^s(0)|<e^{\tilde{\epsilon}}|F^s(0)|.
	\end{equation*}
	for $\tilde{\epsilon}$ small enough. This then concludes the proof.
\end{proof}

\section{Two families of $us$-rectangles having a Markov-type property}\label{sec:Rectangles}

In this section, we fix $\mu \in \bbP(f)$, a saddle hyperbolic measure, possibly non-ergodic. Write $\hmu = \pi_*\inv \mu$. The goal of the present section is to construct two families of rectangles with a Markov-type property. Let us fix a Pesin set $\hat{\Lambda}_{\chi,\epsilon}$. We take $\epsilon$ such that all the results from Section \ref{section:PesinTheory} hold. Fix then a Pesin block $\hNUH \subset \hat{\Lambda}_{\chi,\epsilon}$ that we will simply write as $\hat{\Lambda}$ to simplify the notation. Let also $\tilde{\epsilon}<\epsilon$.

\subsection{Natural unstable cone fields}
In this paragraph, we will construct a natural unstable cone field $\mathscr{C}^u$ on $\Lambda := \pi(\hat{\Lambda})$, which does not depend on the fibers of the bundle $M_f$. We will then study the properties of admissible manifolds tangent to $\mathscr{C}^u$.

\begin{prop}[Canonical unstable cone field]
	There exists a continuous cone field $\mathscr{C}^u$ on $\Lambda$ with the following properties:
	\begin{enumerate}
		\item For every $\hx \in \hat{\Lambda}$, $\Eu(\hx)\subset \text{Int}\big(\mathscr{C}^u_{x}\big)$, where $x=\pi(\hx)$.
		\item There exists a neighborhood $V$ of $\Lambda$ such that $\mathscr{C}^u$ can be extended to a continuous cone field $\mathscr{C}^{u,ext}$ on $V$.
		\item There exists an integer $N \geq 1$ such that for all $n \geq N$ and all $x \in \pi\big(\hat{\Lambda} \cap \hf^{-n}(\hat{\Lambda})\big)$, we have:
		\begin{equation*}
			df^n \mathscr{C}^u_x \subset \text{Int}\big(\mathscr{C}^u_{f^n(x)}\big).
		\end{equation*}
	\end{enumerate}
	\label{prop:UnstableConeField}
\end{prop}
\begin{proof}
	Recall that for every $x \in \Lambda$ and every $\hx,\hy \in \pi\inv(x)$, we have $\Es(\hx) = \Es(\hy) = \Es(x)$. In particular, this implies that the map $x \rightarrow \Es(x)$ is continuous on $\Lambda$. By the definition of the Pesin block \ref{def:PesinBlock}, we know that there exists a constant $\gamma_s >0$ such that for all $x \in \Lambda$, the cone:
	\begin{equation*}
		\mathscr{C}^s_x = \{v = v_s+v_c \in \Es(x) \oplus (\Es(x))^{\perp}, \ \norm{v_c}\leq\gamma_s\norm{v_s}\}
	\end{equation*}
	does not intersect any of the unstable directions $\Eu(\hx)$ for $\hx \in \hat{\Lambda}\cap \pi\inv(x)$.
	
	Define then $\mathscr{C}^u_x$ to be the complement of the cone $\mathscr{C}^s_x$ in $T_xM$, removing the zero vector. Since the continuity of $x\mapsto \mathscr{C}^s_x$ is implied by the continuity of $x \mapsto \Es(x)$, the map $x \mapsto \mathscr{C}^u_x$ is also continuous. This proves the first item.
	
	To prove the second item, consider the map $x \mapsto \Es(x)$ from $\Lambda$ to $PTM$, the projective tangent bundle of $M$. It is standard that, given a vector bundle over a compact subset, one can extend it to a neighborhood. For completeness, let us sketch the argument. Cover $\Lambda$ by a finite number of charts $U_i$ such that $E^s$ is homeomorphic to $(U_i\cap\Lambda)\times \bbR$. Extend $E^s$ trivially on each chart. Now, up to reducing the $U_i$, one can extend the transition functions to $U_i$ by Tietze extension Theorem. This allows us to build a vector bundle via the usual way on $V:=\cup U_i$ which extends $E^s$.
	
	It remains to prove the last point. Take $\hx \in \hat{\Lambda}$ and consider the Pesin chart $\big(\Psi_{\hx},q_{\tilde{\epsilon}}(\hx)\big)$. Recall that by Theorem \ref{thm:finPesinChart}, for all $n \geq 1$, the map $f^n_{\hx} = \Psi_{\hf^n(\hx)}\inv\circ f^n \circ \Psi_{\hx}$ is a well-defined diffeomorphism from $R(0,q_{\tilde{\epsilon}}(\hx))$ onto its image. Recall that $d_0f^n_{\hx}$ is the following hyperbolic matrix:
	\begin{equation*}
		\begin{pmatrix*} A^n & 0\\ 0 & B^n\end{pmatrix*},
	\end{equation*}
	where $|A|,|B|\inv \leq e^{-\chi/2}$.
	
	Let $\gamma>0$ and define the $\gamma$-horizontal cone:
	\begin{equation*}
		\mathscr{C}^H_{\gamma} = \{(u_1,u_2) \in \bbR^2, \ |u_2|\leq\gamma|u_1|\}.
	\end{equation*}
	We are going to show that for $n\geq 1$, the map $(d_0f^n_{\hx})\inv$ contracts exponentially the $\gamma$-horizontal cone. Let $u = u_1 + u_2 \in \mathscr{C}^H_{\gamma}$. We then have:
	\begin{equation*}
		\begin{aligned}
			|(d_0f^n_{\hx})\inv u_2| &= |B^n|\inv|u_2| \leq e^{-n\chi/2}|u_2|\\
			&\leq \gamma e^{-n\chi/2}|u_1|\\
			&\leq \gamma e^{-n\chi/2} e^{-n\chi/2}|(d_0f^n_{\hx})\inv u_1|.
		\end{aligned}
	\end{equation*}
	This proves that $(d_0f^n_{\hx})\inv \mathscr{C}^H_{\gamma} \subset \mathscr{C}^H_{\gamma e^{-n\chi}}$.
	
	Let now $n\geq 1$. Recall the definition of the Lyapunov change of coordinates $C(\hx)$ from Definition \ref{def:LyapunovChangeCoordinates}. Take $\gamma_0$ such that $C(\hx)\inv \mathscr{C}^s_x \supset \mathscr{C}^H_{\gamma_0}$, which exists by continuity of $\hx \mapsto C(\hx)$ on the compact set $\hat{\Lambda}$. Define $\gamma_n = e^{n\chi}\gamma_0$.
	
	Take $N\geq 1$ such that for all $n\geq N$ and for all $\hx \in \hat{\Lambda}\cap\hf^{-n}(\Lambda)$, we have:
	\begin{equation*}
		C(\hf^n(\hx))\inv \mathscr{C}^s_{f^n(\pi(x))} \subset \text{Int}\big(\mathscr{C}^H_{\gamma_n}\big).
	\end{equation*}
	Such an $N$ exists because since $\hf^n(\hx) \in \hat{\Lambda}$, the norm of the map $C(\hf^n(\hx))$ and the opening coefficient of the cone $\mathscr{C}^s_{f^n(x_0)}$ are bounded by constants depending only on $\hat{\Lambda}$.
	
	Take then $x \in \pi\big(\hat{\Lambda} \cap \hf^{-n}(\hat{\Lambda})\big)$ and take any $\hx \in \hat{\Lambda} \cap \hf^{-n}(\hat{\Lambda})\cap\pi\inv(x)$.
	
	Now, we have:
	\begin{equation*}
		\begin{aligned}
			(d_0f^n_{\hx})\inv \mathscr{C}^H_{\gamma_n} &= C(\hx)\inv \circ (d_xf^n)\inv \circ C(\hf^n(\hx)) \mathscr{C}^H_{\gamma_n} \\
			&\subset \mathscr{C}^H_{\gamma_1} \subset C(\hx)\inv \mathscr{C}^s_x.
		\end{aligned}
	\end{equation*}
	But then this implies:
	\begin{equation*}
		\text{Int}\big(d_xf^n \mathscr{C}^s_x\big) \supset C(\hf^n(\hx)) \text{Int}\big(\mathscr{C}^H_{\gamma_n}\big) \supset \mathscr{C}^s_{f^n(x)}.
	\end{equation*}
	
	This concludes the proof since $\mathscr{C}^u_x = (\mathscr{C}^s_x \cup \{0\})^c$.
\end{proof}

At some point, we may have to use thinner cones to have better control over geometry.

\begin{defn}[$\alpha$-stable/unstable cone field]
	Let $\alpha>0$. Let $\mathscr{C}^{u,ext},V$ be the cone field and the neighborhood constructed in Proposition \ref{prop:UnstableConeField}. Let $\gamma_u, \gamma_s : V \rightarrow \bbR$ and $F_u,F_s : V \rightarrow TV$ be the opening map and the direction of the cone fields $\mathscr{C}^{u,ext}$ and respectively $(\mathscr{C}^{u,ext})^c$. We then define $\mathscr{C}^{u,ext,\alpha}$ and $\mathscr{C}^{s,ext,\alpha}$, the $\alpha$-unstable/stable cone fields, to be the ones with direction $F_u$, respectively $F_s$, and opening map $\alpha\gamma_u$, respectively $\alpha\gamma_s$.
	\label{def:alphaStableConeField}
\end{defn}

We are now able to define a natural cone field in the Pesin charts of points in $\hat{\Lambda}$.

\begin{defn}[Canonical unstable cone field in the charts]
	Let $\mathscr{C}^u, \mathscr{C}^{u,ext}$ and $V$ be the two cone fields and the neighborhood constructed in Proposition \ref{prop:UnstableConeField}. Let $x \in \Lambda$ and let $\hx \in \hat{\Lambda}_{\chi,\epsilon} \cap \pi\inv(x)$. Up to reducing the size of the chart, we can ensure that $\Psi_{\hx}\big(R(0,q_{\tilde{\epsilon}}(\hx))\big) \subset V$. We then define the canonical unstable cone field $\mathscr{C}^{u,\hx}$ in the chart $\big(\Psi_{\hx},q_{\tilde{\epsilon}}(\hx)\big)$ to be:
	\begin{equation}
		\mathscr{C}^{u,\hx} = d\Psi_{\hx}\inv\big(\mathscr{C}^{u,ext}\big).
		\label{eq:UnstableConeFieldPesinChart}
	\end{equation}
	\label{def:UnstableConeFieldPesinChart}
\end{defn}

\begin{lemma}
	There exists $q^* := q^*(\hat{\Lambda})$ with the following property. For all $\hx \in \hat{\Lambda}$, the cone field $\mathscr{C}^{u,\hx}_{|R(0,q^*)}$ contains the vertical direction and does not contain the horizontal one.
	\label{lem:UnstableConeFieldInChartIsVertical}
\end{lemma}
\begin{proof}
	We keep the notation of the proof of Proposition \ref{prop:UnstableConeField}.
	
	Take $x \in \Lambda$ and any $\hx \in \hat{\Lambda}\cap\pi\inv(x)$. Recall that $\mathscr{C}^{u,\hx}_0 = C(\hx)\inv \mathscr{C}^u_x$. By Proposition \ref{prop:UnstableConeField}, the unstable direction $\Eu(\hx)$ belongs to $\text{Int}\big(\mathscr{C}^u_x\big)$ and since $C(\hx)\inv \Eu(\hx) = \text{span}\{e_2\}$, we deduce that $\mathscr{C}^{u,\hx}_0$ contains the vertical direction in its interior. A similar argument shows that $\mathscr{C}^{u,\hx}_0$ does not contain the horizontal direction $\text{span}\{e_1\}$. By continuity of the cone field $\mathscr{C}^{u,ext}$, of the Pesin charts, and compactness of $\hat{\Lambda}$, we can find $q^*$ such that $\mathscr{C}^{u,\hx}_{|R(0,q^*)}$ contains the vertical direction and does not contain the horizontal one.
\end{proof}

\begin{lemma}
	For all $n\geq 1$ large enough, there exists an open neighborhood $\hat{U}:= \hat{U}(n)$ of $\hat{\Lambda}$ and constants $\delta_1,\delta_2 := \delta_1(n,\hat{U}), \delta_2(n,\hat{U})$ such that:
	\begin{enumerate}
		\item $\pi(\hat{U}) \subset V$, where $V$ is the neighborhood from Proposition \ref{prop:UnstableConeField}.
		\item For all $\hx \in \hat{\Lambda}$, if there exists $\hz \in \hat{U}\cap\hf^{-n}(\hat{U})$ such that $d(\pi(\hz),\pi(\hx)) < \delta_1$, then $\mathscr{C}^{u,ext}_{f^n(y)}$ is well defined and we have:
		\begin{equation*}
			d_yf^n\mathscr{C}^{u,ext}_y \subset \text{Int}(\mathscr{C}^{u,ext}_{f^n(y)})
		\end{equation*}
		for all $y \in B(x,\delta_2)$.
	\end{enumerate}
	\label{lem:ContractionExtendedUnstableConeField}
\end{lemma}
\begin{proof}
	First of all, take $n\geq 1$ large enough so that for all $y \in \pi\big(\hat{\Lambda}\cap \hf^{-n}(\hat{\Lambda})\big)$, we have:
	\begin{equation*}
		d_yf^n \mathscr{C}^u_y \subset \text{Int}\big(\mathscr{C}^u_{f^n(y)}\big).
	\end{equation*}
	This can be done by Proposition \ref{prop:UnstableConeField}. From now on, we fix such an $n$ large enough. By continuity of the maps $x \mapsto d_xf^n$ and $z \mapsto \mathscr{C}^{u,ext}_z$, and compactness of $\pi\big(\hat{\Lambda}\cap\hf^{-n}(\hat{\Lambda})\big)$, there exists a constant $\tilde{\delta}$ such that for all $y \in \pi\big(\hat{\Lambda}\cap\hf^{-n}(\hat{\Lambda})\big)$ and all $z \in B(y,\tilde{\delta})$, we have:
	\begin{equation*}
		d_zf^n \mathscr{C}^{u,ext}_z \subset \text{Int}\big(\mathscr{C}^{u,ext}_{f^n(z)}\big).
	\end{equation*}
	Note that taking $\tilde{\delta}$ small enough, depending on $n$, guarantees that $z,f^n(z) \in V$.
	
	Now, by compactness of $\hat{\Lambda}\cap\hf^{-n}(\hat{\Lambda})$, we can find a neighborhood $\hat{U}$ of $\hat{\Lambda}$, depending on $n$, such that for all $\hz \in \hat{U} \cap \hf^{-n}(\hat{U})$, there exists $\hy \in \hat{\Lambda}\cap\hf^{-n}(\hat{\Lambda})$ such that $d(\pi(\hz),\pi(\hy)) < \tilde{\delta}/2$. Take $\delta_1$ such that for any $\hx \in \hat{\Lambda}$, if there exists $\hz \in \hat{U}\cap\hf^{-n}(\hat{U})$ with $d(\pi(\hx),\pi(\hz))<\delta_1$, then there exists $\hy \in \hat{\Lambda}\cap\hf^{-n}(\hat{\Lambda})$ with $d(\pi(\hx),\pi(\hy))<\tilde{\delta}$.
	
	Again, by compactness of $\hat{\Lambda}$ and continuity of $x \mapsto d_xf^n$ and $z \mapsto \mathscr{C}^{u,ext}_z$, there exists $\delta_2$ such that for all $\hx \in \hat{\Lambda}$ such that there exists $\hz \in \hat{U}\cap\hf^{-n}(\hat{U})$ with $d(\pi(\hx),\pi(\hz))<\delta_1$, the derivative $d_yf^n$ contracts $\mathscr{C}^{u,ext}$ for all $y \in B(\pi(\hx),\delta_2)$. Taking $\delta_2$ small enough guarantees that $y, f^n(y) \in V$ for all $y \in B(\pi(\hx),\delta_2)$, and this concludes the proof.
\end{proof}

We will work with foliations tangent to these canonical unstable cone fields.

\begin{lemma}
	Let $\mathscr{C}^u, \mathscr{C}^{u,ext}$ and $V$ be the two cone fields and the neighborhood constructed in Proposition \ref{prop:UnstableConeField}. Let $\alpha>0$. Let $\mathscr{F}$ be a smooth foliation of an open set $U\subset V$ tangent to $\mathscr{C}^{u,ext,\alpha}$. The following holds for $\tilde{\epsilon} := \tilde{\epsilon}(\mathscr{F},\hat{\Lambda})$ and $\alpha := \alpha(\hat{\Lambda})$ sufficiently small. There exists $q \leq \min_{\hat{\Lambda}} q_{\tilde{\epsilon}}(\hx)$ with the following property. Let $x \in \Lambda$ and $\hx \in \hat{\Lambda}\cap\pi\inv(x)$. For any $y \in \Psi_{\hx}\big(R(0,q)\big) \cap U$, the leaf $\mathscr{F}(y)$ contains a $(\mathscr{C}^{u,\hx},u)$-manifold through $y$ in $\big(\Psi_{\hx},q_{\tilde{\epsilon}}(\hx)\big)$ of size $q_{\tilde{\epsilon}}(\hx)$.
	\label{lem:FoliationTangentCone}
\end{lemma}
\begin{proof}
	First, for each $\hx \in \hat{\Lambda}$, we can take $\tilde{\epsilon}$ small enough so that every leaf of $\mathscr{F}$ intersecting $\Psi_{\hx}\big(R(0,q_{\tilde{\epsilon}}(\hx))\big)$ is the graph of a Lipschitz function in the chart. This can be done since the cone field $\mathscr{C}^{u,ext,\alpha}$ is continuous, so reducing $\tilde{\epsilon}$ ensures that $\mathscr{C}^{u,\hx}$ is almost constant. By continuity of the foliation and the Pesin charts on $\hat{\Lambda}$, we can take the same $\tilde{\epsilon}$ for all $\hx \in \hat{\Lambda}$. For each $\hx \in \hat{\Lambda}$, we can take $\tilde{\epsilon}$, $q$, and $\alpha$ small enough so that for each $y \in \Psi_{\hx}\big(R(0,q)\big) \cap U$, the representing function $G_y$ of the leaf $\mathscr{F}(y)$ in the chart $\big(\Psi_{\hx},q_{\tilde{\epsilon}}(\hx)\big)$ is well defined on $[-q_{\tilde{\epsilon}}(\hx),q_{\tilde{\epsilon}}(\hx)]$. Indeed, since each leaf is smooth, reducing $q$ ensures that each $G_y$ is close to the line of slope $|G_x^{'}(0)|$ passing through $0$, where $x:=\pi(\hx)$. Taking $\alpha$ small enough guarantees that $|G_x^{'}(0)|<1$, and taking $\tilde{\epsilon}$ small enough will guarantee that $G_x$ is close enough to the straight line of slope $|G_x^{'}(0)|$ passing through $0$, so all the $G_y$ for $y$ close will be Lipschitz graphs defined on the whole chart. Still by a compactness argument, we can choose $\tilde{\epsilon}$, $q$, and $\alpha$ uniformly. Choosing $q$ such that for any $y \in \Psi_{\hx}\big(R(0,q)\big) \cap U$, we have $|G_y(0)|<10^{-3}q_{\tilde{\epsilon}}(\hx)$ concludes the proof. Again, it is easy to see that for any $\hx \in \hat{\Lambda}$, such a $q$ exists, then a compactness argument gives a uniform $q$.
\end{proof}

\subsection{$us$-rectangles}
Recall the cone fields $\mathscr{C}^u, \mathscr{C}^{u,ext}$ and the neighborhood $V$ constructed in Proposition~\ref{prop:UnstableConeField}, for the chosen Pesin block $\hat{\Lambda}$. Let $\tilde{\epsilon}>0$ such that $\Psi_{\hx}\big(R(0,q_{\tilde{\epsilon}}(\hx))\big) \subset V$ for any $\hx\in \hat{\Lambda}$.

Let us first clarify what we mean by rectangles.

\begin{defn}[$us$-rectangle]
	A $us$-rectangle $R\subset M$ is an open set such that there exists a Pesin chart $\big(\Psi_{\hx},q_{\tilde{\epsilon}}(\hx)\big)$, with $\hx \in \hat{\Lambda}$, for which the following properties hold.
	\begin{enumerate}
		\item[(R1)] The boundary of $R$ is a Jordan curve which decomposes as the union of four compact simple curves that we denote by: $\partial R = \partial^{s,1} R \cup \partial^{u,1} R \cup \partial^{s,2} R \cup \partial^{u,2} R$.
		
		We call $\partial^{s} R = \partial^{s,1} R\cup \partial^{s,2} R$ the stable boundary. We define similarly $\partial^u R$ as the unstable boundary.
		
		\item[(R2)] Each of the curves of the stable boundary is contained in $s$-manifolds in the chart $\big(\Psi_{\hx},q_{\tilde{\epsilon}}(\hx))\big)$.
		
		\item[(R3)] Each of the curves of the unstable boundary is contained in $(\mathscr{C}^{u,\hx},u)$-manifolds in the chart $\big(\Psi_{\hx},q_{\tilde{\epsilon}}(\hx))\big)$.
		
		We call these manifolds the representative $s$-, respectively $u$-, manifolds of the stable, respectively unstable, boundaries.
	\end{enumerate}
	We say that the Pesin chart $\big(\Psi_{\hx},q_{\tilde{\epsilon}}(\hx)\big)$ is adapted for the $us$-rectangle $R$.
\end{defn}

We are going to define some notions to control the geometry of rectangles. For some cone field $\mathscr{C}$ defined on an open set of $M$, let $\mathcal{T}(\mathscr{C})$, respectively $\mathcal{T}^1(\mathscr{C})$, be the set of all the piecewise $\cC^1$ curves, respectively all the $\cC^1$ curves, parameterized by $[0,1]$, tangent to $\mathscr{C}$ on every piece, respectively everywhere.

\begin{defn}[Parameters of $us$-rectangles]
	Let $R$ be a $us$-rectangle. We define the following parameters of $R$:
	\begin{itemize}[label={--},itemsep=0.5cm]
		\item $\sigma^s_{\max}(R) = \sup\big\{l(\gamma), \ \gamma \in \mathcal{T}(\mathscr{C}^{u,ext}), \ \gamma([0,1]) \subset R, \ \gamma(0)\in \partial^{s,1}R, \ \gamma(1) \in \partial^{s,2}R\big\}$.
		\item $\sigma^u_{\min}(R) = \inf\big\{l(\gamma), \ \gamma\in \cC^1([0,1],M), \ \gamma(0) \in \partial^{u,1}R, \ \gamma(1) \in \partial^{u,2}R \big\}$.
		\item $\sigma^{u,\alpha}_{\max}(R) = \sup\big\{l(\gamma), \ \gamma \in \mathcal{T}^1(\mathscr{C}^{s,ext,\alpha}), \ \gamma([0,1]) \subset R, \ \gamma(0)\in \partial^{u,1}R, \ \gamma(1) \in \partial^{u,2}R\big\}$.
	\end{itemize}
	\label{def:ParametersRectangles}
\end{defn}

The previous definition extends trivially to $u/s$-manifolds. Let $V^u_1,V^u_2$ be two $u$-manifolds in a chart $\big(\Psi_{\hx},\kappa\big)$. We define:
\begin{equation*}
	\sigma^{u,\alpha}_{\max}(V^u_1,V^u_2) = \sup\big\{l(\gamma), \ \gamma \in \mathcal{T}^1(\mathscr{C}^{s,ext,\alpha}), \ \gamma(0)\in V^u_1, \ \gamma(1) \in V^u_2\big\}.
\end{equation*}

The following lemma asserts that controlling the parameters of a $us$-rectangle allows us to control its diameter.

\begin{lemma}
	Let $R$ be a $us$-rectangle. Let $\alpha,\sigma_1,\sigma_2 >0$. Suppose that $\partial^{s,i}R \in \mathcal{T}(\mathscr{C}^{s,ext,\alpha})$ for $i \in \{1,2\}$. Suppose also that we have the following estimates:
	\begin{equation*}
		\sigma^{u,\alpha}_{\max}(R) \leq \sigma_1 \ \text{and} \ \sigma^{s}_{\max}(R) \leq \sigma_2.
	\end{equation*}
	Then we have the following control on the diameter:
	\begin{equation*}
		\text{diam}(R) \leq \sigma_1+\sigma_2.
	\end{equation*}
	\label{lem:DiameterRectangle}
\end{lemma}
\begin{proof}
	Let any $y,z \in R$. Let $\hx \in \hat{\Lambda}$ and $(\Psi_{\hx},q_{\tilde{\epsilon}})$ be an adapted chart for $R$. Take two Lipschitz maps $F,G$ which satisfy the following properties:
	\begin{itemize}[label={--}]
		\item $\Psi_{\hx}\big(\{(t,F(t)), \ |t|\leq q_{\tilde{\epsilon}}(\hx)\}\big) \in \mathcal{T}(\mathscr{C}^{s,ext,\alpha})$ and $\Psi_{\hx}\big(\{(G(t),t), \ |t|\leq q_{\tilde{\epsilon}}(\hx)\}\big) \in \mathcal{T}(\mathscr{C}^{u,ext})$.
		\item $y \in \Psi_{\hx}\big(\{(t,F(t)), \ |t|\leq q_{\tilde{\epsilon}}(\hx)\}\big)$.
		\item $z \in \Psi_{\hx}\big(\{(G(t),t), \ |t|\leq q_{\tilde{\epsilon}}(\hx)\}\big)$.
	\end{itemize}
	The map $F$ represents then a $s$-manifold in $(\Psi_{\hx},q_{\tilde{\epsilon}})$ tangent to $\mathscr{C}^{s,ext,\alpha}$ that we denote by $V^s$. Similarly, $G$ represents a $(\mathscr{C}^{u,\hx},u)$-manifold in $(\Psi_{\hx},q_{\tilde{\epsilon}})$ that we denote by $V^u$. Since $\partial^{s,i}R$ are tangent to $\mathscr{C}^{s,ext,\alpha}$, $F$ can be taken so that $V^s$ does not intersect any of $\partial^{s,i}R$. Then, $V^s$ intersects both $\partial^{u,1}R$ and $\partial^{u,2}R$. Similarly, we can take $G$ so that $V^u$ intersects both $\partial^{s,1}R$ and $\partial^{s,2}R$. Denote by $c^s \subset V^s$ the curve going from $\partial^{u,1}R$ to $\partial^{u,2}R$. Denote also $c^u \subset V^u$ the curve going from $\partial^{s,1}R$ to $\partial^{s,2}R$. Note that $y \in c^s$ and $z \in c^u$. We then have:
	\begin{equation*}
		d(y,z) \leq l(c^s) + l(c^u) \leq \sigma_1 + \sigma_2
	\end{equation*}
	which concludes the proof.
\end{proof}

The following lemma discusses the geometry of the rectangles in charts and in the manifold. In particular, it allows us to compare the distances in the Pesin charts and in the manifold.

\begin{lemma}
	Let $\alpha < 1$. The following holds for $\tilde{\epsilon} := \tilde{\epsilon}(\alpha, \hat{\Lambda})$ small enough.
	
	There exists a constant $C_1:=C_1(\hat{\Lambda})$ with the following properties. Let $\hx \in \hat{\Lambda}$. Let $\kappa \leq q_{\tilde{\epsilon}}(\hx)$. Let $V^u_1,V^u_2$ be two $u$-manifolds in $\big(\Psi_{\hx},\kappa\big)$ which are almost everywhere tangent to the cone field $\mathscr{C}^{u,\hx}$. Then we have:
	\begin{equation*}
		\frac{1}{C_1(\Lambda)} \sigma^{u,\alpha}_{\max}(V^u_1,V^u_2) \leq d_{\hx}(V^u_1,V^u_2) \leq C_1(\Lambda) \sigma^{u,\alpha}_{\max}(V^u_1,V^u_2).
	\end{equation*}
	\label{lem:DistanceChartsManifold}
\end{lemma}
\begin{proof}
	Let $V$ be the neighborhood defined in Proposition \ref{prop:UnstableConeField}. Let $V^u_1,V^u_2$ be two $(\mathscr{C}^{u,\hx},u)$-manifolds in $(\Psi_{\hx},\kappa)$ and write $F^u_1,F^u_2$ for their representative maps. For any $\hy \in \hat{\Lambda}$, let $L^u(\hy)$ be the constant from Lemma \ref{lem:LipschitzCstAdmissibleCuManifold} for the cone field $\mathscr{C}^{u,\hy}$ defined on $\overline{R(0,\sup_{\hat{\Lambda}}q_{\tilde{\epsilon}})}$. Recall that $F^u_1$ and $F^u_2$ are $L^u(\hx)$-Lipschitz by Lemma \ref{lem:LipschitzCstAdmissibleCuManifold}. Write also $q = \min(q(V^u_1),q(V^u_2))$. Define also, as for the unstable one, the $\alpha$-stable cone field in charts $\mathscr{C}^{s,\alpha,\hx} = d\Psi_{\hx}\inv \mathscr{C}^{s,ext,\alpha}$. As for the unstable cone fields, define $L^{s,\alpha}(\hy)$ for any $y \in \hat{\Lambda}$. Note that $\alpha<1$ implies, taking $\tilde{\epsilon}$ small enough, that for every $\hy \in \hat{\Lambda}$, the two cone fields $\mathscr{C}^{u,\hy}$ and $\mathscr{C}^{s,\hy,\alpha}$ are contained in disjoint constant cone fields. In particular, this implies $L^u(\hy)L^{s,\alpha}(\hy)<1$. Up to reducing $\tilde{\epsilon}$ again, we can suppose that for all $\hy \in \hat{\Lambda}$, the cone field $\mathscr{C}^{s,\hy}$ contains the horizontal direction.\\
	Let us first prove that there exists a constant $K_1$, depending only on $\hat{\Lambda}$, such that:
	\begin{equation*}
		d_{\hx}(V^u_1,V^u_2) \leq K_1\sigma^{u,\alpha}_{\max}(V^u_1,V^u_2).
	\end{equation*}
	Let $t \in [-q,q]$. Let $\gamma_t$ be the straight line joining $F^u_1$ and $F^u_2$ at $t$. Write $\tilde{\gamma}_t := \Psi_{\hx}\circ\gamma_t$.
	First, write $q = \min(q(V^u_1),q(V^u_2))$. Remark that, since $\mathscr{C}^{s,\hx}$ contains the horizontal direction, the curve $\tilde{\gamma}_t \in \mathcal{T}^1(\mathscr{C}^{s,ext})$ and verifies, up to reparameterizing, $\tilde{\gamma}_t(0) \in V^u_1$ and $\tilde{\gamma}_t(1) \in V^u_2$. We then have:
	\begin{equation*}
		l(\gamma_t) \leq \sup_{\hy \in \hat{\Lambda}} \norm{d\Psi\inv_{\hy}} l(\tilde{\gamma}_t) \leq \sup_{\hy \in \hat{\Lambda}} \norm{d\Psi\inv_{\hy}} \sigma^u_{\max}(V^u_1,V^u_2).
	\end{equation*}
	We then have the desired inequality by taking the supremum over $t$.\\
	We now want to prove the other inequality, showing that there exists a constant $K_2$, depending only on $\hat{\Lambda}$, such that:
	\begin{equation*}
		\sigma^{u,\alpha}_{\max}(V^u_1,V^u_2) \leq K_2d_{\hx}(V^u_1,V^u_2).
	\end{equation*}
	Let $\tilde{\gamma} \in \mathcal{T}(\mathscr{C}^{s,ext,\alpha})$ such that $\tilde{\gamma}(0) \in V^u_1$ and $\tilde{\gamma}(1) \in V^u_2$. Write $\gamma := \Psi_{\hx}\inv \circ \tilde{\gamma}$. Let $t_1 \in [-q(V^u_1),q(V^u_1)]$ and $t_2 \in [-q(V^u_2),q(V^u_2)]$ such that $\gamma(0) =(F^u_1(t_1),t_1)$ and $\gamma(1) = (F^u_2(t_2),t_2)$. Note that for almost every $t \in [0,1]$, we have that $\gamma^{'}(t) \in \mathscr{C}^{s,\hx,\alpha}$. Recall that $\{e_1,e_2\}$ denote the canonical basis of $\bbR^2$. By our assumption on $\tilde{\epsilon}$, we have that $e_2 \notin \mathscr{C}^{s,\hx,\alpha}_{\tilde{\gamma}(t)}$ for all $t \in [0,1]$. This implies that there exists a map $G : [F^u_1(t_1),F^u_2(t_2)] \rightarrow \bbR$, which is almost everywhere differentiable, such that $\gamma([0,1]) = \{(t,G(t)), t \in [F^u_1(t_1),F^u_2(t_2)]\}$.
	
	Note that:
	\begin{equation*}
		l(\tilde{\gamma}) \leq \sup_{\hy \in \hat{\Lambda}}\norm{d\Psi_{\hy}} l(\gamma).
	\end{equation*}
	
	Since $\gamma$ is tangent to $\mathscr{C}^{s,\hx,\alpha}$ almost everywhere, there exists a constant $C(\hx)$ such that $|G^{'}(t)|\leq C(\hx)\alpha$ for almost every $t \in [F^u_1(t_1),F^u_2(t_2)]$. We can then compute $l(\gamma)$:
	\begin{equation*}
		\begin{aligned}
			l(\gamma) &= \int_{F^u_1(t_1)}^{F^u_2(t_2)} \norm{(1,G^{'}(s))}ds \leq |F^u_1(t_1)-F^u_2(t_2)|(1+\alpha C(\hx))\\
			&\leq (1+C(\hx))\big(d_{\hx}(V^u_1,V^u_2) + L^u(\hx)|t_1-t_2|\big).
		\end{aligned}
	\end{equation*}
	But note also that since $L^{s,\alpha}(\hx)L^u(\hx) <1$, we get:
	\begin{equation*}
		\begin{aligned}
			|t_1-t_2| &= |G(F^u_2(t_2))-G(F^u_1(t_1))| \leq L^{s,\alpha}(\hx)|F^u_1(t_1)-F^u_2(t_2)|\\
			&\leq L^{s,\alpha}(\hx) d_{\hx}(V^u_1,V^u_2) + L^{s,\alpha}(\hx)L^u(\hx)|t_1-t_2|.
		\end{aligned}
	\end{equation*}
	And hence:
	\begin{equation*}
		|t_1-t_2|\leq \frac{L^{s,\alpha}(\hx)}{1-L^u(\hx)L^{s,\alpha}(\hx)} d_{\hx}(V^u_1,V^u_2).
	\end{equation*}
	So finally we get:
	\begin{equation*}
		l(\gamma) \leq \Big(\big(\frac{L^{s,\alpha}(\hx)}{1-L^{s,\alpha}(\hx)L^u(\hx)}\big)L^u(\hx)(1+C(\hx)) + 1+C(\hx)\Big)d_{\hx}(V^u_1,V^u_2).
	\end{equation*}
	By compactness, continuity of the Pesin charts and the cone fields, there exists $K_2>0$ depending only on $\hat{\Lambda}$ such that:
	\begin{equation*}
		\Big(\big(\frac{L^{s,\alpha}(\hx)}{1-L^{s,\alpha}(\hx)L^u(\hx)}\big)L^u(\hx)(1+C(\hx)) + 1+C(\hx)\Big) \leq K_2.
	\end{equation*}
	This concludes the proof.
\end{proof}

\subsection{The two families of rectangles}
The boundaries of the rectangles we are going to construct will be contained in stable manifolds of well-chosen points. Let us define the set of points in which we are interested.

Let $\hat{Y}^{\#}\subset M_f$ be the set of points $\hx$ such that the following holds.
\begin{enumerate}
	\item There exist two sequences of positive integers $(n_k)_k$ and $(m_k)_k$ converging to $+\infty$ such that $\hf^{n_k}(\hx)$ and $\hf^{-m_k}(\hx)$ belong to the same Pesin block and converge to $\hx$.
	\item There exists a set of points $\hat{\Lambda}_{\hx}$ with well-defined one-dimensional stable manifolds, which accumulate $\hx$ by points projecting on both sides of $\Wu_{loc}(\hx)$ and such that the union of local stable manifolds of the points in $\hat{\Lambda}_{\hx}$ forms a continuous lamination with Lipschitz holonomies and with $\cC^r$ leaves. For more details and precise definitions, see sub-Section \ref{subsec:Lamination}.
	\item For any point $\hy \in \hat{\Lambda}_{\hx}$, for any open neighborhood $\hat{U}$ of $\hy$, there exist two sequences of integers $(n_k)_k$ and $(m_k)_k$ with positive density, i.e., $\lim_k n_k/k>0$, such that $\hf^{n_k}(\hy)$ and $\hf^{-m_k}(\hy)$ belong to the intersection of $\hat{U}$ and some common Pesin block, independent of $\hy$.
\end{enumerate}
The first item forms a set of full measure for any ergodic measure by the Birkhoff Theorem. By Theorem \ref{thm:LipschitzHolonomiesLamination}, any ergodic measure which is hyperbolic of saddle type and such that almost every of its ergodic components have positive entropy gives full measure to the set of points satisfying the second and the third items. Remark that the last item guarantees that, given any ergodic, hyperbolic, $\hf$-invariant measure $\hmu$ of saddle type, for $\hmu$-almost every $\hy$, the set $\hat{\Lambda}_{\hx}$ is included in the support of $\hmu$ restricted to some Pesin block.

The goal of this section is to prove the following proposition.

\begin{prop}[Existence of family of $us$-rectangle with a Markov-type property]
	Let $\nu$ be a $f$-invariant measure, possibly non-ergodic. Let $\hnu := \pi\inv_{\star}\nu$. Suppose that $\hnu$ satisfies the following. 
	\begin{itemize}[label={--}]
		\item For $\hnu$-almost every $\hx$, we have $\lambda^u(\hx)>0>\lambda^s(\hx)$, where $\lambda^{u/s}(\hx)$ are the two Lyapunov exponents at $\hx$.
		\item $\hnu$-almost every ergodic component of $\hnu$ has positive entropy.
	\end{itemize}
	Let $K>8$ and let $\eta \in (0,1)$. There exists a Pesin block $\hNUH := \hNUH(\eta)$ such that the following holds for an integer $N := N(\hNUH,K)$ sufficiently large and $\alpha^{\prime}$, $\sigma:=\sigma(\hNUH,n)$ sufficiently small.
	
	There exist two families of $us$-rectangles $\mathcal{R} = \big(R_1,...,R_L\big)$ and $\tilde{\mathcal{R}} = \big(\tilde{R}_1,...,\tilde{R}_L\big)$, and a foliation $\mathscr{F}^u$ of an open set $V\supset \NUH$ such that:
	\begin{enumerate}[label=(MR\arabic*)]
		\item\label{property:MR1} For any $i \in \{1,...,L\}$, we have that $R_i$ and $\tilde{R}_i$ are two $us$-rectangles such that:
		\begin{itemize}[label={--}]
			\item $R_i \subset \tilde{R}_i$;
			\item there exists $\hx_i \in \hNUH$ such that $\big(\Psi_{\hx_i},q_{\tilde{\epsilon}}(\hx_i)\big)$ is an adapted chart for $R_i$ and $\tilde{R}_i$;
			\item $\partial^{s,j} R_i \subset \partial^{s,j} \tilde{R}_i \subset \Ws_{loc}(\hy_i^j)$ where $\hy_i^j \in \hNUH\cap \hat{Y}^{\#}$ for $j \in \{1,2\}$.
		\end{itemize}
		\item\label{property:MR2} Let $\hat{U}$ be the neighborhood of $\hNUH$ constructed in Lemma \ref{lem:ContractionExtendedUnstableConeField}.
		\begin{equation*}
			\hnu\Big(\hat{U}\Big) > 1-\eta \ \text{and} \ \pi(\hat{U}) \subset \bigcup_{i=1}^L R_i
		\end{equation*}
		\item\label{property:MR3} For any $i \in \{1,...,L\}$, each unstable boundary of $R_i$ is contained in a leaf of $\mathscr{F}^u$. Moreover, they are not the same leaf: $\partial^{u,l} R_i \cap \partial^{u,k} R_j = \emptyset$ for all $k,l\in \{1,2\}$ whenever $i\ne j$ or $i=j$ and $k\ne l$.
		\item\label{property:MR4} For any $i \in \{1,...,L\}$, we have the following:
		\begin{equation*}
			\frac{1}{2}\sigma \leq \sigma_{\min}^u(R_i) \leq \sigma^{u,\alpha^{\prime}}_{\max}(R_i) \leq \sigma; \quad \sigma^s_{\max}(R_i) \leq \sigma; \quad \sigma^s_{\max}(\tilde{R}_i),\sigma^{u,\alpha^{\prime}}_{\max}(\tilde{R}_i) \leq K\sigma.
		\end{equation*}
		\item\label{property:MR5} For any $i,j \in \{1,...,L\}$ such that $f^N(R_i) \cap R_j \ne \emptyset$, if there exists $\hz \in \hat{U} \cap \hf^{-N}(\hat{U})$ such that $\pi(z) \in R_i$, then:
		\begin{equation*}
			f^N(\tilde{R}_i) \cap \partial^u\tilde{R}_j = \emptyset.
		\end{equation*}
	\end{enumerate}
	\label{prop:MarkovRectangles}
\end{prop}

The last property is a type of Markov property. We ask that the rectangles intersect in a nice way, such that the image by $f^n$ of some rectangle does not intersect the unstable boundary of the intersected rectangle. One can compare it to the classical Markov property on Figure \ref{fig:Examples_Intersection_rectangles}.

\begin{figure}[h]
	\centering
	\includegraphics{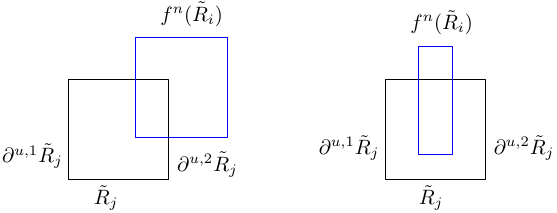}
	\caption{On the left, an example of an intersection between $\tilde{R}_j$ and $f^n(\tilde{R}_i)$ which does not satisfy Property \ref{property:MR5}. On the right, an example of an intersection satisfying \ref{property:MR5}.}
	\label{fig:Examples_Intersection_rectangles}
\end{figure}

Before going into the technical details, let us explain the main steps and ideas of the proof.

\parbreak\textbf{Step 1: Construction of a foliation tangent to $\mathscr{C}^{u,ext}$.} First of all, we fix a Pesin block $\hat{\Lambda}$ such that $\hnu(\hat{\Lambda}) > 1-\eta$. We construct a smooth foliation $\mathscr{F}^u$ tangent to the cone field $\mathscr{C}^{u,ext}$. The plaques will be $(\mathscr{C}^{u,\hx},u)$-manifolds in the Pesin charts.

\parbreak\textbf{Step 2: Construction of the family $\mathcal{R}$.} We take $x \in \Lambda$ and $\hx \in \hat{Y}^{\#}\cap \hat{\Lambda}\cap\pi\inv(\hx)$. We can then suppose that there exists points of $\hat{\Lambda}$ on each side of the local unstable set of $\hx$. These points have Pesin charts close to the Pesin chart of $\hx$. Looking at the stable manifolds of these points, they will intersect each plaque of the foliation $\mathscr{F}^u$ on a uniform neighborhood. Taking two plaques at a good scale, this will form a rectangle. Covering $\Lambda$ by finitely many of these, we obtain the first family of $us$-rectangles $\mathcal{R}$.

\parbreak\textbf{Step 3: Obtaining $\tilde{\mathcal{R}}$ by applying successive approximations to $\mathcal{R}$.} The idea is similar to Sinai's construction of Markov partitions for Axiom A diffeomorphisms. We start with the rectangles of the first family $\mathcal{R}$ and look at the rectangles $R_i, R_j$ whose iterates intersect in a "bad" way: $f^n(R_i) \cap \partial^uR_j \ne \emptyset$, but such that there exists a "good" Pesin chart: there exists $\hz \in \hat{U}\cap\hf^{-n}(\hat{U})$ with $\pi(\hz) \in R_i$. We then enlarge the rectangle $R_j$ such that the intersection with $f^n(R_i)$ becomes nice. Doing this for all the $R_i$ which intersect $R_j$ in a bad way but with a good chart, we obtain $\tilde{R}_j^1$. We then repeat this construction, but just looking at the bad intersections of rectangles $\tilde{R}^1_i, \tilde{R}^1_j$ such that their "base" rectangles $R_i,R_j$ intersect. Note that at step $k$, the portions we add are delimited by $(\mathscr{C}^{u,\hx},u)$-manifolds which had been contracted $k$ times by the graph transform operator. By the choice of $n$, large enough ensuring that the graph transform is sufficiently contracting, we can then control the size of the portion we add.

\subsection{The key lemma.}
We are first going to prove a lemma which will be the key step of the proof of Proposition \ref{prop:MarkovRectangles}. This lemma allows us to iterate the successive approximations we explained.

\begin{lemma}
	Let $\hNUH$ be a Pesin block. For any $K>4$, the following holds for $N := N(\hNUH,K)$ sufficiently large and $\alpha$, $\sigma := \sigma(\hNUH,n)$ and $\tilde{\epsilon}:=\tilde{\epsilon}(\hNUH,n,\alpha)$ sufficiently small. Let $R_1$ and $R_2$ be two $us$-rectangles such that $\partial^{s,i}R_l \subset \Ws_{loc}(\hy_{l,i})$ for some $\hy_{l,i} \in \hNUH \cap \partial^{s,i}R_l$. Suppose that they satisfy the following:
	\begin{itemize}[label={--}]
		\item There exists an adapted chart $\big(\Psi_{\hx_i},q_{\tilde{\epsilon}}(\hx_i)\big)$ to the rectangle $R_i$ such that $\hx_i \in \hNUH$.
		\item Let $\hat{U}:=\hat{U}(N)$ be the neighborhood of $\hNUH$ constructed in Lemma \ref{lem:ContractionExtendedUnstableConeField}. There exists $\hz \in \hat{U}\cap\hf^{-N}(\hat{U})$ such that $\pi(\hz) \in R_1$.
		\item $\sigma^{u,\alpha}_{\max}(R_1)\leq K\sigma, \ \sigma^{u,\alpha}_{\max}(R_2) \leq \sigma$ and $\sigma^s_{\max}(\Ws_{loc}(\hy_{i,1}),\Ws_{loc}(\hy_{i,2})) \leq \sigma$.
		\item $f^N(R_1) \cap \partial^uR_2 \ne \emptyset$.
	\end{itemize}
	Then there exists a $us$-rectangle $\tilde{R}_2$ with the following properties:
	\begin{itemize}[label={--}]
		\item $R_2\subset \tilde{R}_2$ and $\partial^{s,i}R_2\subset \partial^{s,i}\tilde{R}_2\subset \Ws_{loc}(\hy_i)$.
		\item $\sigma^{u,\alpha}_{\max}(\tilde{R}_2) \leq K\sigma$.
		\item $f^N(R_1) \cap \partial^{u}\tilde{R}_2 = \emptyset$.
	\end{itemize}
	\label{lem:KeyStepRectangles}
\end{lemma}
\begin{proof}
	Let us fix a Pesin block $\hNUH$ and a constant $K>4$. Let $\mathscr{C}^{u,ext}$ and $V$ be the cone field and the neighborhood of $\NUH$ constructed in Proposition \ref{prop:UnstableConeField}. Take any $\hx \in \hNUH$ and $V^u_1,V^u_2$ two $u$-manifolds almost everywhere tangent to $\mathscr{C}^{u,\hx}$ in the chart $\big(\Psi_{\hx},q_{\tilde{\epsilon}}(\hx)\big)$. For any $\alpha<1$ and $n\geq1$, combining Proposition \ref{prop:ContractionTransGraph} and Lemma \ref{lem:DistanceChartsManifold} whenever it makes sense, we have:
	\begin{equation}
		\sigma^{u,\alpha}_{\max}\big((\mathcal{F}^u_{\hx})^n(V^u_1),(\mathcal{F}^u_{\hx})^n(V^u_2)\big) \leq C(\hNUH)^2e^{-n\chi/4} \sigma^{u,\alpha}_{\max}(V^u_1,V^u_2)
		\label{eq:keylem:ContractionUManifolds}
	\end{equation}
	where $(\mathcal{F}^u_{\hx})^n$ is the graph transform operator (Definition \ref{def:GraphTransformOperator}) between the charts $\big(\Psi_{\hx},q_{\tilde{\epsilon}}(\hx)\big)$ and $\big(\Psi_{\hf^n(\hx)},q_{\tilde{\epsilon}}(\hf^n(\hx))\big)$ for the map $f^n$ (it can also be seen as the composition $n$ times of the graph transform operator for $f$), and $C(\hNUH)$ is the constant given by Lemma \ref{lem:DistanceChartsManifold}. Choose $\lambda\in (0,1)$ such that $\lambda K <1$ and take $n$ large enough so that:
	\begin{equation}
		C(\hNUH)^2e^{-n\chi/4}<\lambda.
		\label{eq:RectangleLambda}
	\end{equation}
	From now on, we fix such an $n$ and we will work with the map $f^n$. Let $\hat{U}$ be the neighborhood of $\hNUH$ and $\delta_1,\delta_2$ be the constants given by Lemma \ref{lem:ContractionExtendedUnstableConeField}. Take $\gamma := \gamma(\hNUH,\alpha)$ and $\tilde{\epsilon} := \tilde{\epsilon}(\hNUH,n,\delta_2,\gamma)$ small enough so that:
	\begin{itemize}[label = {--}]
		\item For any $\hx \in \hNUH$, we have $\Psi_{\hx}\big(R(0,q_{\tilde{\epsilon}}(\hx))\big) \subset B(\pi(\hx),\delta_2)$.
		\item Propositions \ref{prop:GraphTranform} and \ref{prop:ContractionTransGraph} hold for any $\hx \in \hNUH$ such that there exists $\hz \in \hat{U}\cap\hf^{-n}(\hat{U})$ satisfying $d(\pi(\hx),\pi(\hz)) < \delta_1$, with the cone field $\mathscr{C}^{u,\hx}$ and the map $f^n$.
		\item Proposition \ref{prop:IntersectionAdmissiblemanifolds} holds for any $\hx \in \hNUH$, with the cone field $\mathscr{C}^{u,\hx}$.
		\item For any $\hx \in \hNUH$, the local stable manifold $\Ws_{loc}(\hx)$ is a $\gamma/2$-admissible $s$-manifold in $\big(\Psi_{\hx},q_{\tilde{\epsilon}}(\hx)\big)$ and is tangent to $\text{Int}(\mathscr{C}^{s,ext,\alpha})$.
	\end{itemize}
	Indeed, for all $\hx \in \hNUH$, such a $\gamma$ exists. Now, by the continuity of the Pesin charts and the cone field $\mathscr{C}^{u,ext}$, and the compactness of $\hNUH$, it can be taken uniformly in $\hx$. The same argument applies to $\tilde{\epsilon}$, with $\gamma$ fixed. Note that the second condition holds since, under the mentioned hypothesis, the cone filed $\mathscr{C}^{u,\hx}$ is contracted by $f^n$ in the right Pesin charts.
	
	Take $\big(\Psi_{\hx_1},q_{\tilde{\epsilon}}(\hx_1)\big)$ and $\big(\Psi_{\hx_2},q_{\tilde{\epsilon}}(\hx_2)\big)$ two adapted charts for $R_1$ and $R_2$ such that $\hx_1,\hx_2 \in \hNUH$. The following construction will hold for any $\sigma := \sigma(n,\hNUH)$ sufficiently small. We assume that $f^n(R_1) \cap \partial^uR_2 \ne \emptyset$. Denote by $V^u_{i,j}$ the $\big(\mathscr{C}^{u,\hx_i},u\big)$-manifold in $\big(\Psi_{\hx_i},q_{\tilde{\epsilon}}(\hx_i)\big)$ containing $\partial^{u,j}R_i$, for $j \in \{1,2\}$. By definition, $R_i \subset \Psi_{\hx_i}\big(R(0,q_{\tilde{\epsilon}}(\hx_i))\big)$. By Lemma \ref{lem:ContractionExtendedUnstableConeField}, we can then take the graph transform between the cone fields $\mathscr{C}^{u,\hx_1}$ and $\mathscr{C}^{u,\hf^n(\hx_1)}$. We get: $W^u_{1,i} := (\mathcal{F}^u_{\hx_1})^n(V^u_{1,i})$. By Proposition \ref{prop:GraphTranform}, the curves $W^u_{1,i}$ are $\big(\mathscr{C}^{u,\hf^n(\hx_1)},u)$-manifolds in $\big(\Psi_{\hf^n(\hx_1)},q_{\tilde{\epsilon}}(\hf^n(\hx_1))\big)$.
	
	\begin{claim} The curves $f^n(\partial^{u,i}R_1)$ are contained in $\big(\mathscr{C}^{u,\hx_2},u\big)$-manifolds in $\big(\Psi_{\hx_2},q_{\tilde{\epsilon}}(\hx_2)\big)$ that we denote by $\tilde{W}^u_{1,i}$.
	\label{claim:Rectangles1}
	\end{claim}
	\begin{proof}
		By definition, each $W^u_{1,i}$ is almost everywhere tangent to $\mathscr{C}^{u,ext}$. This then implies that the curve $w^u_{1,i}:=\Psi_{\hx_2}\inv\Big(W^u_{1,i}\cap \Psi_{\hx_2}\big(R(0,q_{\tilde{\epsilon}}(\hx_2))\big)\Big) \subset \bbR^2$ is almost everywhere tangent to $\mathscr{C}^{u,\hx_2}$. Note that taking $\sigma$ small enough guarantees that $f^n(\partial^{u,i}R_1) \subset (W^u_{1,i}\cap \Psi_{\hx_2}\big(R(0,q_{\tilde{\epsilon}}(\hx_2))\big)$. Since $\mathscr{C}^{u,\hx_2}$ does not contain the horizontal direction, $w^u_{1,i}$ is a well-defined Lipschitz graph. Call its representative function $F_i$. Since by \eqref{eq:keylem:ContractionUManifolds} and \eqref{eq:RectangleLambda}, we have:
		\begin{equation*}
			\sigma^{u,\alpha}_{\max}(f^n(R_1)) \leq \lambda\sigma^{u,\alpha}_{\max}(R_1) \leq \lambda K\sigma
			\label{eq:keylem:EstimationIterateUnstable}
		\end{equation*}
		and since $f^n(R_1) \cap R_2 \ne \emptyset$, it follows that taking $\sigma$ small enough guarantees that $|F_i(0)| < 10^{-3}q_{\tilde{\epsilon}}(\hx_2)$.
		
		Now, if $q\Big(W^u_{1,i}\cap \Psi_{\hx_2}\big(R(0,q_{\tilde{\epsilon}}(\hx_2))\big)\Big)<10^{-2}q_{\tilde{\epsilon}}(\hx_2)$ (in the chart $\big(\Psi_{\hx_2},q_{\tilde{\epsilon}}(\hx_2)\big)$), then we can just extend $w^u_{1,i}$, in an arbitrary way, to a $\big(\mathscr{C}^{u,\hx},u\big)$-manifold in the chart $\big(\Psi_{\hx_2},q_{\tilde{\epsilon}}(\hx_2)\big)$, which proves the claim.
	\end{proof}
	Taking $\sigma$ small enough and up to extending it, by continuity of local stable manifolds on $\NUH$, $\Ws_{loc}(\hy_{2,j})$ are $\gamma$-admissible $s$-manifolds in $\big(\Psi_{\hx_2},q_{\tilde{\epsilon}}(\hx_2)\big)$, tangent to $\mathscr{C}^{s,ext,\alpha}$ for $j=1,2$. Write $F^2_i$ for the representative function of $V^u_{2,i}$ and $F^1_i$ the representative function of $\tilde{W}^u_{1,i}$. Denote by $U^u_1$ the $\big(\mathscr{C}^{u,\hx_2},u\big)$-manifold defined by the representative function $U^u_1 = \min(F^j_i)$, for $i,j \in \{1,2\}$, where the min is taken on $[-\min\big(q(V^u_{i,j}), q(\tilde{W}^u_{i,j})\big), \min\big(q(V^u_{i,j}),q(\tilde{W}^u_{i,j})\big)]$. Define similarly $U^u_2$ by taking the max of these four functions on the same domain. Then, by Proposition \ref{prop:IntersectionAdmissiblemanifolds}, for each $j \in \{1,2\}$, the curve $\Ws_{loc}(\hy_{2,j})$ intersects in a unique point each $U^u_i$. Note that Proposition \ref{prop:IntersectionAdmissiblemanifolds} guarantees that this intersection point belongs to $\Ws_{loc}(\hy_{2,j})$, taking $\sigma$ small enough. This defines naturally a Jordan curve, which bounds a topological disk $\tilde{R_2}$. Figure \ref{fig:Rectangles_Construction} shows how the construction works. We sum up the properties of $\tilde{R}_2$ in the following claim.
	
	\begin{claim} $\tilde{R}_2$ is a topological disk with the following properties:
	\begin{itemize}[label={--}]
		\item $\tilde{R}_2$ is a $us$-rectangle in the chart $\big(\Psi_{\hx_2},q_{\tilde{\epsilon}}(\hx_2)\big)$.
		\item $R_2 \subset \tilde{R}_2$ and $\partial^{s,i}R_2 \subset \partial^{s,i}\tilde{R}_2 \subset \Ws_{loc}(\hy_{2,i})$.
		\item $f^n(R_1) \cap \partial^u\tilde{R}_2 = \emptyset$.
	\end{itemize}
	\label{claim:Rectangles2}
	\end{claim}
	\begin{proof}
		The first two points follow directly from the construction. For the last point, suppose that $f^n(R_1) \cap \partial^u\tilde{R}_2 \ne \emptyset$. It then implies, without loss of generality, that there is a point $z \in \partial^{u,1}\tilde{R}_2$ such that $z \in f^n(R_1)$. Recall that $F^1_i$ is the representative function of $\tilde{W}^u_{i,1}$ and $U^u_1$ the one of $\partial^{u,1}\tilde{R}_2$. Write $\Psi_{\hx_2}\inv(z) = (z_1,z_2)$. We then have $F^1_1(z_2)<z_1<F^1_2(z_2)$, which is absurd since $z_1 = U^u_1(z_2)$ and $U^u_1$ is the min (or the max) of all the $F^j_i$.
	\end{proof}
	
	\begin{figure}[h]
		\centering
		\begin{subfigure}{.5\textwidth}
			\centering
			\includegraphics[width=0.9\linewidth]{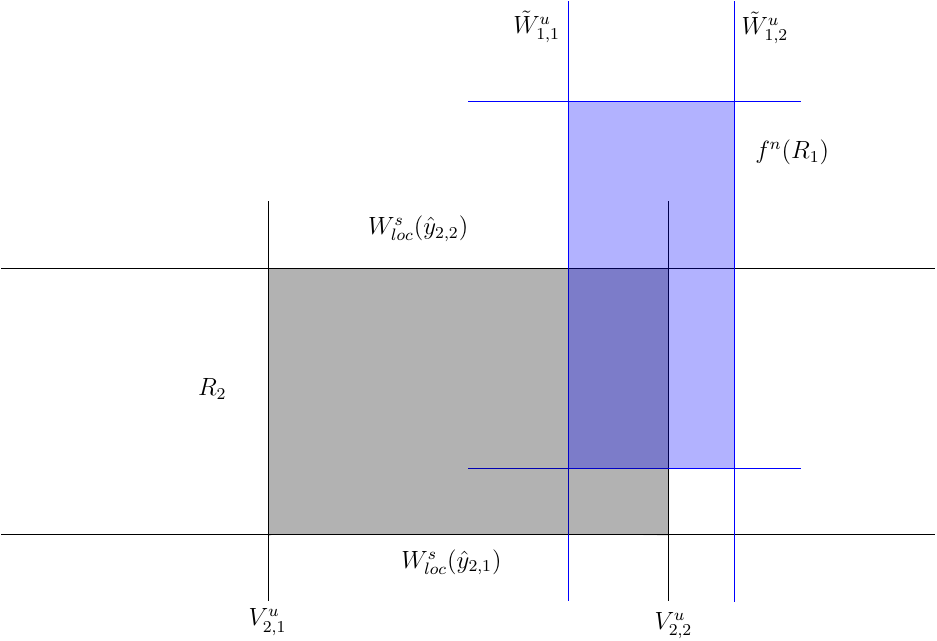}
			\caption{The intersection between $R_2$ (in gray) and $f^n(R_1)$ (in blue).}
			\label{subfig:RectanglesConstruction1}
		\end{subfigure}%
		\begin{subfigure}{.5\textwidth}
			\centering
			\includegraphics[width=0.9\linewidth]{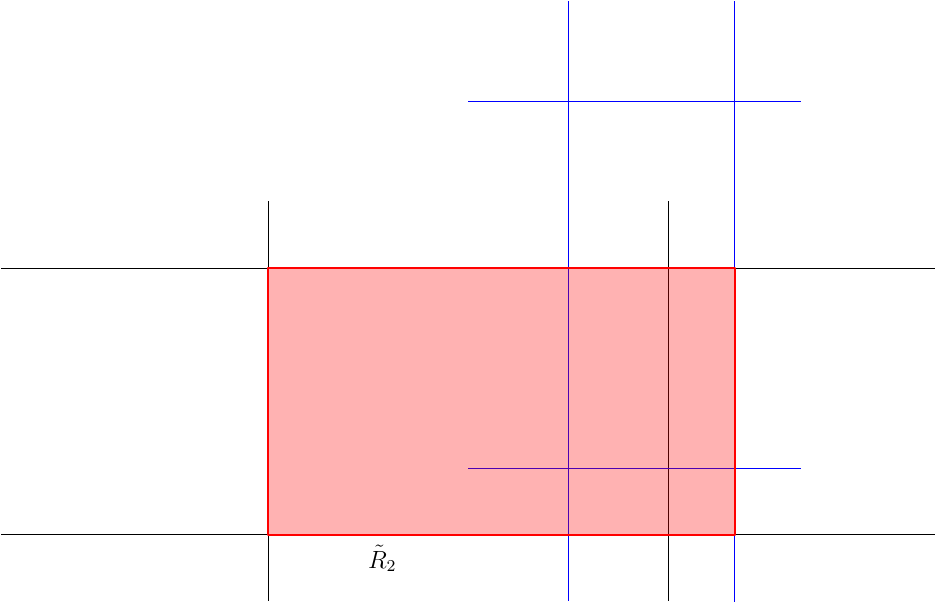}
			\caption{The $us$-rectangle $\tilde{R}_2$ in red.}
			\label{subfig:RectanglesConstruction2}
		\end{subfigure}
		\caption{The construction of the $us$-rectangle $\tilde{R}_2$.}
		\label{fig:Rectangles_Construction}
	\end{figure}
	
	To complete the proof, we still have to estimate the unstable size of $\tilde{R}_2$. For that, we first prove that we can approach the unstable size $\sigma_{\max}^{u,\alpha}(\tilde{R}_2)$ by a geodesic.
	
	\begin{claim} The following holds for $\tilde{\epsilon} := \tilde{\epsilon}(\hNUH)$ sufficiently small. For any constant $C>1$, there exists $\alpha:=\alpha(C)$ and a geodesic curve $c := c(C)$ going from $\partial^{u,1}\tilde{R}_2$ to $\partial^{u,2}\tilde{R}_2$ such that:
	\begin{equation*}
		\sigma_{\max}^{u,\alpha}(\tilde{R})\leq C^2l(c).
	\end{equation*}
	\label{claim:Rectangles3}
	\end{claim}
	\begin{proof}
		Let $C>1$. We will first show that there exists $\alpha>0$ such that for any two curves $c_1,c_2$ tangent to $\mathscr{C}^{s,ext,\alpha}$ such that $c_1(0)=c_2(0) \in \partial^{u,1}\tilde{R}_2$ and $c_1(1)=c_2(1) \in \partial^{u,2}\tilde{R}_2$, we have:
		\begin{equation*}
			l(c_1) \leq Cl(c_2).
		\end{equation*}
		Taking exponential charts and composing with a homothety, we can suppose that $c_1$ and $c_2$ are graphs of $\cC^1$ maps $F_1$ and $F_2$ defined on $[0,1]$, tangent to $\mathscr{C}^{s,ext,\alpha}$. Since $\tilde{R}_2 \subset \Psi_{\hx_2}\big(R(0,q_{\tilde{\epsilon}}(\hx_2))\big)$, taking $\tilde{\epsilon}$ small enough allows us to suppose that $\mathscr{C}^{s,ext,\alpha}$ is contained in a constant cone field $\tilde{\mathscr{C}}$, of horizontal direction and with opening coefficient $2\alpha$. We can then compute:
		\begin{equation*}
			l\big(\text{Graph}(F_1)\big) = \int_0^1 \norm{(1,F_1^{'}(t))}dt \leq 1 + \int_0^1 |F^{'}_1(t)|dt \leq 1 + 2\alpha.
		\end{equation*}
		Moreover, we trivially have that $l\big(\text{Graph}(F_2)\big) \geq 1$. Going back to the manifold, there exists a constant $K$ depending on $\tilde{\epsilon}$ and on the exponential map such that:
		\begin{equation*}
			\frac{l(c_1)}{l(c_2)} \leq K\frac{l\big(\text{Graph}(F_1)\big)}{l\big(\text{Graph}(F_2)\big)} \leq K(1+2\alpha).
		\end{equation*}
		Since $K$ is given essentially by the differential of the exponential map, taking $\tilde{\epsilon}$ small enough allows us to take $K$ as close to $1$ as we want. We then obtain:
		\begin{equation*}
			\frac{l(c_1)}{l(c_2)}\leq C
		\end{equation*}
		for $\alpha$ taken small enough. From now on, we fix such an $\alpha$.
		
		Fix a curve $\gamma \in \mathcal{T}^1(\mathscr{C}^{s,ext,\alpha})$ such that $\gamma(0) \in \partial^{u,1} \tilde{R}_2$ and $\gamma(1) \in \partial^{u,2}\tilde{R}_2$ which satisfies:
		\begin{equation*}
			\sigma^{u,\alpha}_{\max}(\tilde{R}_2)\leq C l(\gamma).
		\end{equation*}
		Let $c$ be the geodesic between $\gamma(0)$ and $\gamma(1)$. Parameterize it such that $c(0) = \gamma(0)$ and $c(1) = \gamma(1)$. Up to taking $\tilde{\epsilon}$ small enough, we can lift $\tilde{R}_2$, the curve $\gamma$, and the cone field $\mathscr{C}^{s,ext,\alpha}$ on $\Psi_{\hx_2}\big(R(0,q_{\tilde{\epsilon}}))\big)$ to $T_{\gamma(0)}M$ by the exponential map. Denote the lift of $\mathscr{C}^{s,ext,\alpha}$ by $\mathscr{C}^{\star}$. Again, taking $\tilde{\epsilon}$ small enough allows us to suppose that $\mathscr{C}^{\star}$ is constant. There exists some $v^{\star} \in T_{\gamma(0)}M$ and some $r^{\star}>0$ such that $c([0,1]) = \exp_{\gamma(0)}(\{tv^{\star}, \ t \in [0,r^{\star}]\})$. If $v^{\star} \notin \mathscr{C}^{\star}$, since the straight line $t \mapsto tv^{\star}$ is the only one joining the lift of $\gamma(0)$ to the lift of $\gamma(1)$, and since $\mathscr{C}^{\star}$ is constant, then it is impossible that there exists a curve tangent to $\mathscr{C}^{\star}$ joining these two points. We then conclude that $v^{\star} \in \mathscr{C}^{\star}$ and then that $c$ is tangent to $\mathscr{C}^{s,ext,\alpha}$.
		
		We finally have:
		\begin{equation*}
			\sigma_{\max}^{u,\alpha}(\tilde{R}_2) \leq C l(\gamma) \leq C^2 l(c)
		\end{equation*}
		which concludes the proof of the claim.
	\end{proof}
	
	Now, we are going to investigate two cases to estimate $\sigma^{u,\alpha}_{\max}(\tilde{R}_2)$.
	
	\textbf{Case 1:} $\tilde{W}^u_{1,i} \cap R_2 = \emptyset$ for $i \in \{1,2\}$.
	
	\begin{figure}[h]
		\centering
		\includegraphics[scale=1]{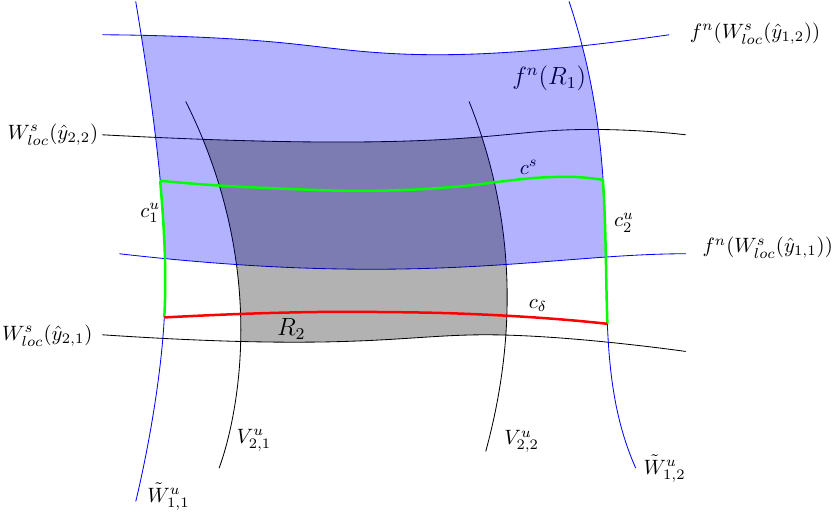}
		\caption{The estimation of $\sigma^{u,\alpha}_{\max}(\tilde{R}_2)$ in Case 1. In blue, the rectangle $f^n(R_1)$. In gray, the rectangle $R_2$. In red, the geodesic curve $c_{\delta}$. In green, the curves $c^u_1,c^u_2,c^s$ used to estimate the length of $c_{\delta}$.}
		\label{fig:RectanglesEstimationCase1}
	\end{figure}
	
	See Figure \ref{fig:RectanglesEstimationCase1} for a visual interpretation of the proof in Case 1. Remark that this implies that $\partial^{u,i}\tilde{R}_2 \subset \tilde{W}^u_{1,i}$ (up to changing the numeration). Indeed, since $f^n(R_1)\cap R_2 \ne \emptyset$, if there exists $z \in \partial^{u,1} \tilde{R}_2$ such that $z \notin \tilde{W}^u_{1,1}$, then we must have $z\in \partial^{u,1}R_2$. By a similar argument as before, this must imply that $|F^2_1(v)|>|F^1_1(v)|$ for some point $v$. But this means that $\tilde{W}^u_{1,1} \cap R_2 \ne \emptyset$.
	
	Let $C>1$ such that $4C^2<K$. Let $c$ be the geodesic given by Claim \ref{claim:Rectangles3}. Parameterize it such that $c(0) \in \partial^{u,1}\tilde{R}_2$ and $c(1)\in \partial^{u,2}\tilde{R}_2$. Take any curve $c^s$ going from $f^n(\partial^{u,1}R_1)$ to $f^n(\partial^{u,2}R_1)$ tangent to $\mathscr{C}^{s,ext,\alpha}$. Such a curve exists because, since we fixed $n$, the curves $f^n\big(\Ws_{loc}(\hy_{1,i})\big)$ are close enough to $\Ws_{loc}(\hx_1)$ and then are tangent to $\mathscr{C}^{s,ext,\alpha}$, taking $\tilde{\epsilon}$ small enough. Denote by $c^u_1$ the subcurve of $\tilde{W}^u_{1,1}$ going from $c(0)$ to the intersection between $c^s$ and $\tilde{W}^u_{1,1}$. Define similarly $c_2^u$ to be the subcurve of $\tilde{W}^u_{1,2}$ going from $c(1)$ to the intersection between $c^s$ and $\tilde{W}^u_{1,2}$. Now we have:
	\begin{equation*}
		\sigma^{u,\alpha}_{\max}(\tilde{R}_2) \leq C^2 l(c) \leq C^2 (l(c^u_1)+l(c^s)+l(c_2^u)) \leq C^2(\sigma + \lambda\sigma K + \sigma).
	\end{equation*}
	Indeed, $c^u_1$ and $c_u^2$ are tangent to $\mathscr{C}^{u,ext}$ and go from $\Ws_{loc}(\hy_{2,1})$ to $\Ws_{loc}(\hy_{2,2})$, then $l(c_1^u),l(c_2^u)\leq\sigma$ by hypothesis. Moreover, by \eqref{eq:keylem:ContractionUManifolds} and \eqref{eq:RectangleLambda}, we obtain $l(c^s)\leq \sigma^{u,\alpha}_{\max}\big(f^n(\partial^{u,1}R_1),f^n(\partial^{u,2}R_1)\big)\leq\lambda\sigma K$.\\
	
	\textbf{Case 2:} $\tilde{W}^u_{1,i}\cap R_2 \ne \emptyset$ for some $i\in \{1,2\}$.
	
	\begin{figure}[h]
		\centering
		\begin{subfigure}{.5\textwidth}
			\centering
			\includegraphics[width=1.1\linewidth]{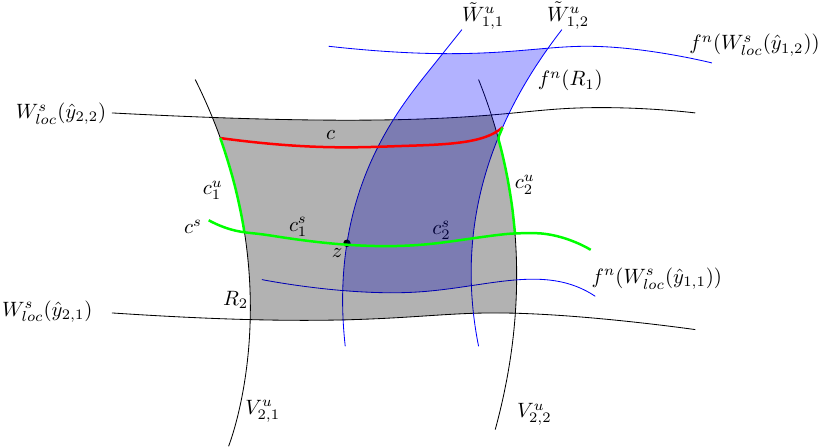}
			\caption{Case 2a.}
			\label{subfig:RectanglesEstimationCase2a}
		\end{subfigure}%
		\begin{subfigure}{.5\textwidth}
			\centering
			\includegraphics[width=1.1\linewidth]{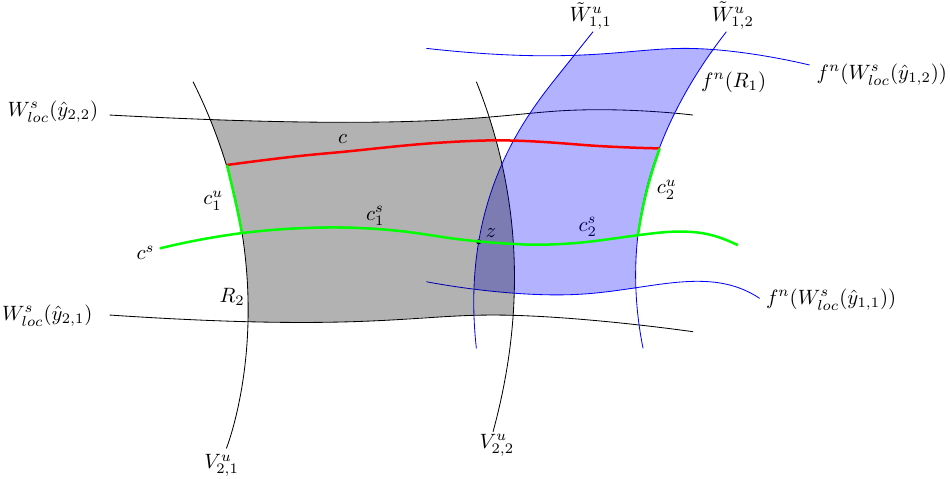}
			\caption{Case 2b.}
			\label{subfig:RectanglesEstimationCase2b}
		\end{subfigure}
		\caption{The estimation of $\sigma^{u,\alpha}_{\max}(\tilde{R}_2)$ in Case 2. In blue, the rectangle $f^n(R_1)$. In gray, the rectangle $R_2$. In red, the geodesic curve $c$. In green, the curves $c^u_1,c^u_2,c^s_1,c^s_2$ used to estimate the length of $c$. $c^s$ is the whole horizontal green curve while $c^s_1$ and $c^s_2$ are the subcurves going from $V^u_{2,1}$ to $z$ and $z$ to $V^u_{2,2}$ in Case 2a, or $\tilde{W}^u_{1,2}$ in Case 2b.}
		\label{fig:RectanglesEstimationCase2}
	\end{figure}
	
	See Figure \ref{fig:RectanglesEstimationCase2} for a visual interpretation of the proof. Suppose without loss of generality that $\tilde{W}^u_{1,1}\cap R_2 \ne \emptyset$. Let us first notice that there exists $z \in f^n(\partial^{u,1}R_1) \cap R_2$. Indeed, suppose that it was not the case. Since $\tilde{W}^u_{1,1}$ crosses $\partial^sR_2$, we can assume without loss of generality that $f^n(\partial^{u,1}R_1)$ is above $\Ws_{loc}(\hy_{2,1})$ (which has a sense in the Pesin chart), which contains the upper boundary of $R_2$. Since $f^n(\partial^sR_1)$ is contained in some local stable manifold, it cannot cross transversally $\partial^sR_2$. This implies that $f^n(\partial^sR_1)$ is above $\partial^sR_2$ and then $f^n(R_1) \cap R_2 = \emptyset$.
	
	Consider then such a $z \in f^n(\partial^{u,1}R_1) \cap R_2$. Take a curve $c^s$ tangent to $\mathscr{C}^{s,ext,\alpha}$, passing through $z$ which has the following properties:
	\begin{itemize}[label={--}]
		\item $c^s$ stays between $\Ws_{loc}(\hy_{2,1})$ and $\Ws_{loc}(\hy_{2,2})$,
		\item $c^s$ does not cross $f^n(\partial^sR_1)$,
		\item $c^s$ intersects both of $\partial^{u,k} \tilde{R}_2$ and $f^n(\partial^{u,k}R_1)$ uniquely.
	\end{itemize}
	To do this, notice that $z$ can be taken as close as we want from the intersection of $f^n(\partial^{u,1}R_1)$ with a local stable manifold ($\Ws_{loc}(\hy_{2,k})$ or with the local stable manifolds containing $f^n(\partial^sR_1)$). In the Pesin chart, translate this local stable manifold as a curve passing through $z$ and extend it so that the size of this curve, seen as a graph, is the size of the chart. This gives us the desired curve.
	
	Again, take $C>1$ such that $4C^2<K$ and let $c$ be the geodesic given by Claim 3. Parameterize it such that $c(0) \in \partial^{u,1}\tilde{R}_2$ and $c(1) \in \partial^{u,2}\tilde{R}_2$. Consider again $c^u_1$ and $c^u_2$ the curves contained in $\partial^{u,k}\tilde{R}_2$ going from $c(0)$, and respectively $c(1)$, to the intersection with $c^s$. Denote by $c^s_1$ the piece of curve contained in $c^s$ going from the intersection with $\partial^{u,1}\tilde{R}_2$ to $z$ (this is non-empty since $z \in R_2$). Denote by $c_2^s$ the piece of curve contained in $c^s$ going from $z$ to the intersection of $c^s$ and $\partial^{u,2}\tilde{R}_2$. We then distinguish two subcases.\\
	\textbf{Case 2a:} $c^s\cap \partial^{u,2}\tilde{R}_2 \in \partial^uR_2$.
	
	Then we have the following estimate:
	\begin{equation*}
		\begin{aligned}
			\sigma^{u,\alpha}_{\max}(\tilde{R}_2) &\leq C^2l(c)\\
			&\leq C^2 \big(l(c^u_1) + l(c^s_1 \cup c^s_2) + l(c^u_2)\big)\\
			&\leq C^2 3\sigma.
		\end{aligned}
	\end{equation*}
	Indeed, by the same argument as previously, $l(c^u_1),l(c^u_2) \leq \sigma$. Now notice that $c^s_1\cup c^s_2$ is contained in $c^s$ and is then tangent to $\mathscr{C}^{s,ext,\alpha}$, and goes from $\partial^{u,1}R_2$ to $\partial^{u,2}R_2$.\\
	\textbf{Case 2b:} $c^s\cap \partial^{u,2}\tilde{R}_2 \in f^n(\partial^{u,2}R_1)$.
	
	Then we have the following estimate:
	\begin{equation*}
		\begin{aligned}
			\sigma^{u,\alpha}_{\max}(\tilde{R}_2) &\leq C^2 l(c)\\
			&\leq C^2 \big(l(c^u_1) + l(c^s_1)+  l(c^s_2) + l(c^u_2)\big)\\
			&\leq C^2 (3\sigma + \lambda\sigma K).
		\end{aligned}
	\end{equation*}
	Indeed, here the estimates are the same as before for $c^u_1$ and $c^u_2$. The curve $c^s_1$ is contained in a curve tangent to $\mathscr{C}^{s,ext,\alpha}$ going from $\partial^{u,1}R_2$ to $\partial^{u,2}R_2$. The curve $c^s_2$ goes from $z \in f^n(\partial^{u,1}R_1)$ to $f^n(\partial^{u,2}R_1)$.\\
	Finally, taking into account all the cases, since $\lambda K < 1$, we have:
	\begin{equation*}
		\sigma^{u,\alpha}_{\max}(\tilde{R}_2) \leq C^2 4\sigma \leq K\sigma
	\end{equation*}
	by the choice of $C$. This concludes the proof of Lemma \ref{lem:KeyStepRectangles}.
\end{proof}

\subsection{Proof of Proposition \ref{prop:MarkovRectangles}.}
\subsubsection{Choice of constants.}
The proof involves many constants that we have to choose carefully. The order of the choice and the dependency between the constants plays a very important role. The goal of this paragraph is to try to give a clear description of the role and the dependency of each of these constants.

We first fix $\eta \in (0,1)$ and $K>8$. We start by choosing a Pesin block $\hNUH$ such that $\hmu(\hNUH) > 1-\eta$. For the sake of simplicity, denote it by $\hat{\Lambda}$. Let also $\Lambda = \pi(\hat{\Lambda})$. We denote by $\mathscr{C}^{u,ext}$ the canonical extended unstable cone field on $V$, a neighborhood of $\Lambda$ (see Proposition \ref{prop:UnstableConeField}). Choose $n := n(\hNUH,K)$ large enough and $\alpha$, $\sigma:=\sigma(\hNUH,n)$, and $\tilde{\epsilon} := \tilde{\epsilon}(\hNUH,n,\alpha)$ so that Lemma \ref{lem:KeyStepRectangles} holds. Let $\hat{U} := \hat{U}(n)$ be the neighborhood of $\hat{\Lambda}$ constructed in Lemma \ref{lem:ContractionExtendedUnstableConeField}. From now on, we will work with the map $f^n$.

\subsubsection{Construction of a foliation tangent to $\mathscr{C}^{u,ext}$.}
We first construct a foliation tangent to the cone field $\mathscr{C}^{u,ext}$ on a neighborhood of $\Lambda$. For each $x \in \Lambda$, define as before $e^s(x)$ the unit vector such that $\Es(x) = \text{span}\{e^s(x)\}$. Note that we have $e^s(x)^{\perp} \in \text{Int}\big(\mathscr{C}^{u,ext,\alpha}_x\big)$.

As in the proof of Proposition \ref{prop:UnstableConeField}, one can extend the direction $\Es(x)$ continuously on the neighborhood $V$. This gives us a continuous sub-bundle on an open neighborhood $V$ of $\Lambda$. Up to taking an orientable cover, we can assume that there exists a continuous vector field $X : V \rightarrow TV$ which generates the line field $(\Es)^{\perp}$. Approximate $X$ by a $\cC^{\infty}$ vector field $Y$ such that we still have $Y(y) \in \text{Int}\big(\mathscr{C}^{u,ext,\alpha}\big)$. This gives a $\cC^{\infty}$ line field on $V$. Integrating this line field gives then a $\cC^{\infty}$ foliation $\mathscr{F}^u$ on $V$ such that for all $y \in V$, we have:
\begin{equation*}
	T_y\mathscr{F}^u(y) = \bbR.Y(y) \in \text{Int}\big(\mathscr{C}^{u,ext,\alpha}\big).
\end{equation*}
This will be the foliation tangent to the unstable cone field with which we will work.

\subsubsection{Construction of the family $\mathcal{R}$}
In this paragraph, we construct the first family of $us$-rectangles $\mathcal{R}$. Take any $x \in \Lambda$. Let $q$ be the constant given by Lemma \ref{lem:FoliationTangentCone} applied for the open set $V$ and the foliation $\mathscr{F}^{u}$. Up to reducing $\sigma$, we can assume that for all $\hy \in \hat{\Lambda}$, we have: $B(y,\sigma) \subset \Psi_{\hy}(R(0,q))$. Recall the definition of $\hat{Y}^{\#}$ before the statement of Proposition \ref{prop:MarkovRectangles}. The set of such $\hx$ which are accumulated by points of $\hat{\Lambda}$ projecting on $\Wu(\hx)$, on both sides of their unstable manifold, have the same $\hnu$-measure has $\hat{\Lambda}$. Indeed, we supposed that $\hnu$-almost every ergodic component of $\hnu$ has positive entropy and is hyperbolic of saddle type. We will the suppose that for any $x\in \Lambda$, there exists $\hx \in \hat{\Lambda}$ displaying this property.

Applying a similar argument as in Lemma \ref{lem:KeyStepRectangles}, reducing $\tilde{\epsilon}$ and $\sigma$ gives that for any $y \in B(x,\sigma)\cap \Lambda$, the local stable manifold $\Ws_{loc}(y)$ is a $\gamma$-admissible $s$-manifold in $\big(\Psi_{\hx},q_{\tilde{\epsilon}}(\hx)\big)$. By Lemma \ref{lem:FoliationTangentCone}, for any $z \in B(x,\sigma)$, the leaf $\mathscr{F}^{u}(z)$ contains a $(\mathscr{C}^{u,\hx},u)$-manifold $V^u_z$ of size $q_{\tilde{\epsilon}}(\hx)$. Then, by Proposition \ref{prop:IntersectionAdmissiblemanifolds}, $V^u_z$ intersects $\Ws_{loc}(y)$ in a unique point.

Now, pick two points $z^+,z^- \in \Ws_{loc}(x)$ such that $d_{\Ws_{loc}(\hx)}(z^+,z^-) = \frac{3}{4}\sigma$ and $\Psi_{\hx}\inv(z^{+/-}) \in \bbR^{+/-}\times \bbR$. Note that this implies $z^+,z^- \in B(x,\sigma)$.

\begin{claim} The following holds for $\alpha$ and $\sigma$ sufficiently small:
\begin{equation*}
	\begin{aligned}
		&\sigma^u_{\min}\big(\mathscr{F}^u(z^{-})\cap B(x,2\sigma), \mathscr{F}^u(z^{+})\cap B(x,2\sigma)\big) \geq \sigma/2;\\
		&\sigma^{u,\alpha}_{\max}\big(\mathscr{F}^u(z^{-})\cap B(x,2\sigma), \mathscr{F}^u(z^{+})\cap B(x,2\sigma)\big) \leq \sigma.
	\end{aligned}
\end{equation*}
\label{claim:RectangleProofMain1}
\end{claim}
\begin{proof}
	The argument is very similar to arguments in the proof of Lemma \ref{lem:KeyStepRectangles}. By lifting $B(x,2\sigma)$ to $T_xM$ by the exponential map, and since taking $\sigma\rightarrow0$ implies $d_y\exp_x \rightarrow Id$ for any $y \in B(x,2\sigma)$, it is enough to show these inequalities for a foliation $\mathscr{F}^u$ of a ball around $0$ in $\bbR^2$. Each leaf of the foliation $\mathscr{F}^u$ on $B(0,2\sigma)$ converges uniformly to a straight line taking $\alpha\rightarrow 0$. Taking $\sigma$ small enough guarantees also that the curve $c^s \subset \Ws_{loc}(\hx)$ going from $z^-$ to $z^+$ is close to the geodesic between $z^-$ and $z^+$. Since the foliation $\mathscr{F}^u$ is as close as we want to straight lines, the length of $c^s$ approaches $\sigma^u_{\min}\big(\mathscr{F}^u(z^{-})\cap B(0,2\sigma), \mathscr{F}^u(z^{+})\cap B(0,2\sigma)\big)$. This proves the first inequality. For the second, taking $\alpha$ small enough guarantees that the length of any curve going from $\mathscr{F}^u(z^{-})\cap B(0,2\sigma)$ to $\mathscr{F}^u(z^{+})\cap B(0,2\sigma)$ tangent to $\mathscr{C}^{s,ext,\alpha}$ is very close to the length of $c^s$. This gives the second inequality. Note that all these choices of constants can be done uniformly by compactness of $\Lambda$ and continuity of $\mathscr{F}^u$.
\end{proof}

\begin{claim} The following holds for $\tilde{\epsilon}$ and $\sigma$ sufficiently small. There exist two points $y^+,y^- \in \Wu_{loc}(\hx)$ such that:
\begin{equation*}
	\sigma^s_{\max}\big(\Ws_{loc}(y^-),\Ws_{loc}(y^+)\big) \leq \sigma.
\end{equation*}
\label{claim:RectangleproofMain2}
\end{claim}
\begin{proof}
	Repeating the arguments of the proof of Claim \ref{claim:Rectangles3} in the proof of Lemma \ref{lem:KeyStepRectangles}, we can show that taking $\tilde{\epsilon}$ small enough guarantees that there exists a constant $K>1$ depending only on $\mathscr{C}^{u,ext}$ with the following property. For any $s$-manifolds $V^s_1,V^s_2$ in $\big(\Psi_{\hx},q_{\tilde{\epsilon}}(\hx)\big)$, there exists a geodesic curve $c$ going from $V^s_1$ to $V^s_2$ such that:
	\begin{equation*}
		\sigma^s_{\max}(V^s_1,V^s_2) \leq Kl(c).
	\end{equation*}
	Since $\hx$ is accumulate by points of $\hat{\Lambda}$ projecting on each sides of $\Wu_{loc}(\hx)$, we can find two points $\hy^+,\hy^- \in \hat{\Lambda}\cap\pi\inv(\Wu_{loc}(\hx))$ such that $d_{\Wu_{loc}(\hx)}(y^+,y^-) \leq \sigma/2K$ and $y^+,y^- \in B(x,\sigma)$. Recall that taking $\tilde{\epsilon}$ and $\sigma$ sufficiently small guarantees $\Ws_{loc}(y^{+/-})$ is a $\gamma$-admissible $s$-manifold in $\big(\Psi_{\hx},q_{\tilde{\epsilon}}(\hx)\big)$. Let $c$ be the geodesic curve going from $\Ws_{loc}(y^-)$ to $\Ws_{loc}(y^+)$ such that:
	\begin{equation*}
		\sigma_{\max}^s\big(\Ws_{loc}(y^-),\Ws_{loc}(y^+)\big) \leq Kl(c).
	\end{equation*}
	Now, arguments very similar to the proof of Claim \ref{claim:RectangleProofMain1} give us that the length of $c$ is close to the distance between $y^-$ and $y^+$ in $\Wu_{loc}(\hx)$. We then conclude that:
	\begin{equation*}
		\sigma_{\max}^s\big(\Ws_{loc}(y^-),\Ws_{loc}(y^+)\big) \leq Kl(c) \leq 2Kd_{\Wu_{loc}(\hx)}(y^-,y^+)\leq \sigma.
	\end{equation*}
	Again, note that all the constants can be taken uniformly in $x$, which concludes the proof of the claim.
\end{proof}
Until now, we fix $\alpha$, $\tilde{\epsilon}$, and $\sigma$ such that the two claims hold. By the Jordan theorem, the union of the four curves given by:
\begin{itemize}[label={--}]
	\item The subcurve of $\Ws_{loc}(y^+)$ going from the intersection of $\Ws_{loc}(y^+)$ with $V^u_{z^-}$ to the intersection of $\Ws_{loc}(y^+)$ with $V^u_{z^+}$.
	\item The subcurve of $\Ws_{loc}(y^-)$ going from the intersection of $\Ws_{loc}(y^-)$ with $V^u_{z^-}$ to the intersection of $\Ws_{loc}(y^+)$ with $V^u_{z^+}$.
	\item The subcurve of $\mathscr{F}^u(z^+)$ going from the intersection of $V^u_{z^+}$ with $\Ws_{loc}(y^+)$ to the intersection of $V^u_{z^+}$ with $\Ws_{loc}(y^-)$.
	\item The subcurve of $\mathscr{F}^u(z^-)$ going from the intersection of $V^u_{z^-}$ with $\Ws_{loc}(y^-)$ to the intersection of $V^u_{z^-}$ with $\Ws_{loc}(y^+)$,
\end{itemize}
is a closed simple curve that bounds a topological disk that we call $R_x$. See Figure \ref{fig:RectanglesFamily1}. The first two curves are the stable boundaries of $R_x$ and we denote them by $\partial^{s,l}R_x$. The last two curves are the unstable boundaries of $R_x$ and we denote them by $\partial^{u,l}R_x$. Note that we trivially have $\Lambda\subset \bigcup_x R_x$. We can then extract $R_1,...,R_L$ such that $\Lambda \subset \bigcup R_i$.\\
Let us now check that the family $\{R_i\}_{i \in \{1,...,L\}}$ satisfies Properties \ref{property:MR1}, \ref{property:MR2}, \ref{property:MR3}, and \ref{property:MR4} from Proposition \ref{prop:MarkovRectangles}. Fix some $i \in \{1,...,L\}$. To simplify the notation, we still denote by $x,\hx,z^-,z^+$ the points we used to construct the rectangle.
\begin{itemize}[label={--}]
	\item The curves $\partial^{s,l}R_i$ are contained in some $\Ws_{loc}(\hy_{i,l})$ for some $\hy_{i,l}\in \hat{\Lambda}$. The curves $\partial^{u,l}R_i$ are contained in some $(\mathscr{C}^{u,\hx},u)$-manifold in the chart $\big(\Psi_{\hx},q_{\tilde{\epsilon}}(\hx)\big)$ for $\hx \in \hat{\Lambda}$. Then $R_i$ is a $us$-rectangle and $\big(\Psi_{\hx},q_{\tilde{\epsilon}}(\hx)\big)$ is an adapted chart. This proves the part of Property \ref{property:MR1} concerning the first family $\mathcal{R} = (R_1,...,R_L)$.
	\item By definition, we have $\Lambda \subset \bigcup R_i$, which proves Property \ref{property:MR2}.
	\item Property \ref{property:MR3} follows from the fact that $\mathscr{F}^u$ is a foliation on $V$ and since the number of $us$-rectangles is finite, up to taking very close leaves, we obtain that all their unstable boundaries are contained in different leaves.
	\item Claim \ref{claim:RectangleproofMain2} gives that $\sigma^s_{\max}(R) \leq \sigma$. Pick any point $y \in \partial^{u,1}R_i$. Let $c_1$ be the curve contained in $\partial^{u,1}R_i$ going from $y$ to $z^-$ (or $z^+$ depending on the numeration). By Claim \ref{claim:RectangleproofMain2}, $l(c_1)\leq \sigma$. Let $c_2$ be the curve contained in $\Ws_{loc}(x)$ going from $z^-$ to $x$. By construction, $l(c_2)< \sigma$. We have:
	\begin{equation*}
		d(x,y) \leq l(c_1)+l(c_2) < 2\sigma.
	\end{equation*}
	We conclude that $\partial^{u,1}R_i \subset \mathscr{F}^u(z^-)\cap B(x,2\sigma)$ and the same for $\partial^{u,2}R_i$. Claim \ref{claim:RectangleProofMain1} then gives $\sigma^u_{\min}(R_i) \geq \sigma/2$ and $\sigma^{u,\alpha}_{\max}(R_i) \leq \sigma$. This proves Property \ref{property:MR4}.
\end{itemize}
See Figure \ref{fig:RectanglesFamily1} for a representation of the $us$-rectangle $R_x$. This then finishes the construction of the first family of $us$-rectangles $\mathcal{R}$.

\begin{figure}
	\centering
	\includegraphics{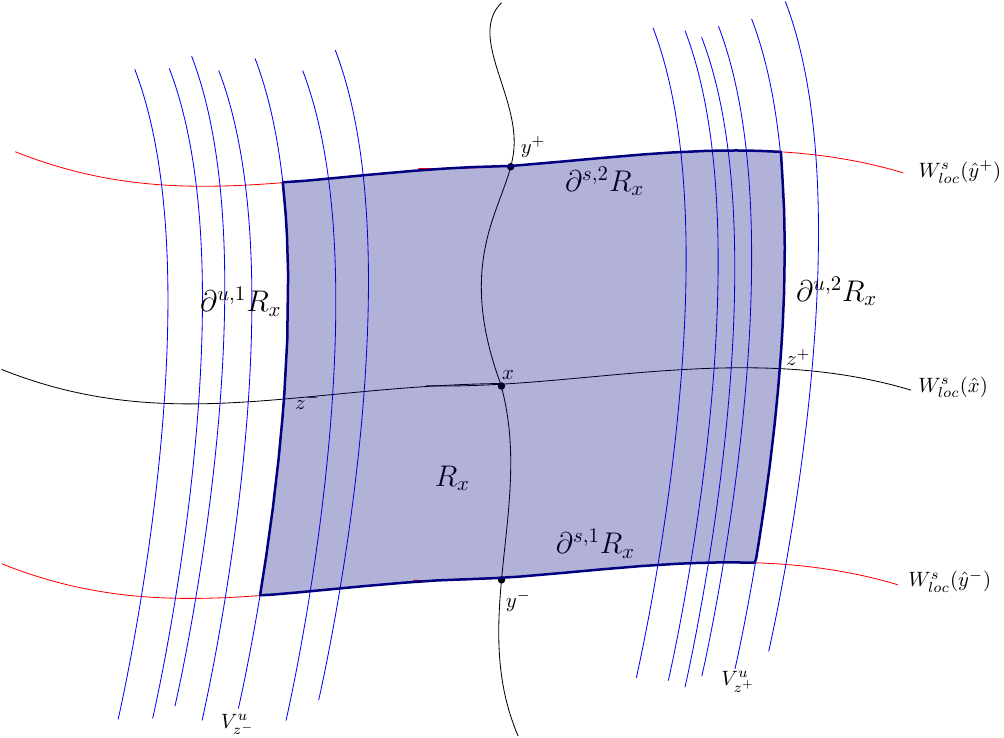}
	\caption{The construction of the first family of rectangles: in light blue, the foliation $\mathscr{F}^u$; in dark blue, the $us$-rectangle.}
	\label{fig:RectanglesFamily1}
\end{figure}

\subsubsection{Obtaining $\tilde{\mathcal{R}}$ by applying successive approximations to $\mathcal{R}$.}
This step goes by successive approximations. At each step, we will construct a family of $us$-rectangles $\tilde{\mathcal{R}}^l$ displaying Property \ref{property:MR5} with the rectangles of the family $\tilde{R}^{l-1}$. Letting $l$ go to infinity will give us the desired family.\\
The construction goes by induction. Fix $i \in \{1,...,L\}$. Define $I(i) \subset \{1,...,L\}$ to be the set of $j \in \{1,...,L\}$ such that there exists $\hz \in \hat{U}\cap \hf^{-n}(\hat{U})$, with $\pi(\hz) \in R_j$, and $f^n(R_j)\cap R_i \ne \emptyset$.

\parbreak\textbf{Step 1.} Let then $i \in \{1,...,L\}$, we will define a new $us$-rectangle $\tilde{R}_i^1$. Denote by $\big(\Psi_{\hx_i},q_{\tilde{\epsilon}}(\hx_i)\big)$ the adapted chart to $R_i$ constructed in the previous paragraph. If for all $j \in I(i)$, we have $f^n(R_j)\cap \partial^uR_i = \emptyset$, then define $\tilde{R}^1_i= R_i$. Suppose now that there exists $j \in I(i)$ such that $f^n(R_j)\cap \partial^uR_i \ne \emptyset$. Recall that $R_i$ and $R_j$ are two $us$-rectangles such that for each $l \in \{1,2\}$, we have $\partial^{s,l}R_k \subset \Ws_{loc}(\hy_{k,l})$ for some $\hy_{k,l} \in \hat{\Lambda} \cap \pi\inv(\partial^{s,l}R_k)$. Recall also that we have the following estimate:
\begin{equation*}
	\sigma^{u,\alpha}_{\max}(R_j) \leq K\sigma/2, \ \sigma^{u,\alpha}_{\max}(R_i) \leq \sigma \ \text{and} \ \sigma^s_{\max}(\Ws_{loc}(\hy_{k,1}),\Ws_{loc}(\hy_{k,2})) \leq \sigma.
\end{equation*}
We can then apply Lemma \ref{lem:KeyStepRectangles} with the constant $K/2>4$ and with the rectangles $R_j$ and $R_i$, where $R_j$ plays the role of $R_1$ and $R_i$ of $R_2$, following the notation of Lemma \ref{lem:KeyStepRectangles}. This gives us a new $us$-rectangle $\tilde{R}_{i,j}^1$ with the following properties:
\begin{itemize}[label={--}]
	\item $R_i \subset \tilde{R}_{i,j}^1$ and $\partial^{s,l}R_i \subset \partial^{s,l}\tilde{R}_{i,j}^1 \subset \Ws_{loc}(\hy_{i,l})$.
	\item $\sigma^{u,\alpha}_{\max}(\tilde{R}^1_{i,j}) \leq K\sigma/2$.
	\item $f^n(R_j)\cap \partial^u\tilde{R}_{i,j}^1=\emptyset$.
\end{itemize}
Do the same construction for each $j \in I(i)$ and define:
\begin{equation*}
	\tilde{R}_i^1 = \bigcup_{j\in I(i)} \tilde{R}_{i,j}^1.
\end{equation*}

\begin{claim} The open set $\tilde{R}_i^1$ has the following properties:
\begin{itemize}[label={--}]
	\item $\tilde{R}_i^1$ is a $us$-rectangle and $\big(\Psi_{\hx_i},q_{\tilde{\epsilon}}(\hx_i)\big)$ is an adapted chart.
	\item $R_i \subset \tilde{R}_i^1$ and $\partial^{s,l}R_i\subset \partial^{s,l} \tilde{R}_i^1\subset \Ws_{loc}(\hy_{i,l})$.
	\item $\sigma/2 \leq \sigma^u_{\min}(\tilde{R}_i^1)\leq \sigma^{u,\alpha}_{\max}(\tilde{R}_i^1)\leq K\sigma$ and $\sigma^s_{\max}(\tilde{R}_i^1) \leq \sigma$.
	\item For all $j \in \{1,...,L\}$ such that $f^n(R_j)\cap R_i \ne \emptyset$, if there exists $\hz \in \hat{U}\cap\hf^{-n}(\hat{U})$ such that $\pi(\hz) \in R_j$, then we have:
	\begin{equation*}
		f^n(R_j) \cap \partial^u\tilde{R}^1_i = \emptyset.
	\end{equation*}
\end{itemize}
\label{claim:RectangleProofMain3}
\end{claim}

\begin{proof}
	For each $j \in I(i)$, write $F^l_{i,j}$ for the representative function of $V^l_{i,j}$, the $\big(\mathscr{C}^{u,\hx_i},u\big)$-manifold in $\big(\Psi_{\hx_i},q_{\tilde{\epsilon}}(\hx_i)\big)$ containing $\partial^{u,l}\tilde{R}^1_{i,j}$. Let $q_i = \min_{j,l} q(V^l_{i,j})$. Define the following functions on the interval $[-q_i,q_i]$:
	\begin{equation*}
		G^1 = \min_{j,l} F^l_{i,j} \ \text{and} \ G^2 = \max_{j,l} F^l_{i,j}.
	\end{equation*}
	Remark that $G^1$ and $G^2$ are representing functions of $\big(\mathscr{C}^{u,\hx_i},u\big)$-manifolds in $\big(\Psi_{\hx},q_{\tilde{\epsilon}}(\hx)\big)$ that we denote by $W^1$ and $W^2$. Indeed, since $G^1$ is the minimum of a finite number of functions tangent almost everywhere to $\mathscr{C}^{u,\hx_i}$, it satisfies also this property. We also have $q(W^1) = \min_{j,l} q(V^l_{i,j}) \geq 10^{-2}q_{\tilde{\epsilon}}(\hx)$ and $\varphi(W^l) \leq \max_{j,l} \varphi(V^l_{i,j}) < 10^{-3}q_{\tilde{\epsilon}}(\hx)$ (recall the definition \ref{def:ParametersusManifolds} of the parameter $\varphi$ of a $u$-manifold). By Proposition \ref{prop:IntersectionAdmissiblemanifolds}, each $W^l$ intersects in a unique point each $\Ws_{loc}(\hy_{i,k})$. This defines a $us$-rectangle $Q$ for which $\big(\Psi_{\hx_i},q_{\tilde{\epsilon}}(\hx_i)\big)$ is an adapted chart. We are going to show that $Q = \tilde{R}^1_i$.
	
	First, we have by construction that $\tilde{R}^1_{i,j}\subset Q$ for each $j \in I(i)$, which implies that $\tilde{R}^1_i \subset Q$. Now suppose that there exists $z \in Q$ such that $z \notin \tilde{R}^1_i$. Write $z = \Psi_{\hx}(z_1,z_2)$. Since $z \in Q$, we have $G^1(z_2) < z_1 < G^2(z_2)$. But since $z \notin \tilde{R}^1_i$, then for all $j$, we have $z \notin \tilde{R}^1_{i,j}$ and then $z_1 > F^l_{i,j}$ for all $l \in \{1,2\}$ and $j \in I(i)$. But this is absurd by the definition of $G^1$ and $G^2$. This proves the first two points of the claim by construction of $Q$.
	
	Since $R_i \subset \tilde{R}_i^1$ and $\partial^{s,l}\tilde{R}^1_i \subset \Ws_{loc}(\hy_{i,l})$, we trivially have:
	\begin{equation*}
		\sigma^u_{\min}(\tilde{R}^1_i)\geq \sigma/2 \ \text{and} \ \sigma^s_{\max}(\tilde{R}_i^1)\leq \sigma.
	\end{equation*}
	Let now $\gamma \in \mathcal{T}^1(\mathscr{C}^{s,ext,\alpha})$ such that $\gamma([0,1]) \subset \tilde{R}_i^1$, $\gamma(0) \in \partial^{u,1}\tilde{R}^1_i$ and $\gamma(1) \in \partial^{u,2}\tilde{R}^1_i$. There exist $j_1,j_2 \in I(i)$ such that $\gamma(0) \in \partial^{u,1}\tilde{R}^1_{i,j_1}$ and $\gamma(1) \in \partial^{u,2} \tilde{R}^1_{i,j_2}$. If $j_1=j_2$, then we have:
	\begin{equation*}
		l(\gamma)\leq \sigma^{u,\alpha}_{\max}(\tilde{R}^1_{i,j_1}) \leq K\sigma/2.
	\end{equation*}
	Suppose then $j_1\ne j_2$. Using the Pesin chart $\big(\Psi_{\hx_i},q_{\tilde{\epsilon}}(\hx_i)\big)$, we can order any $us$-rectangle $R$ in the chart such that $\partial^{u,1}R$ is on the left of $\partial^{u,2}R$ as graphs. Define $t_1,t_2 \in [0,1]$ such that $\gamma(t_1) \in \partial^{u,2}\tilde{R}^1_{i,j_1}$ and $\gamma(t_2) \in \partial^{u,1}\tilde{R}^1_{i,j_2}$. Note that $t_1,t_2$ must exist. Indeed, since for any $j$ and $k,l \in \{1,2\}$, we have $|F_{i,j}^l|\leq |G_k|$, the curve $\gamma$ has to cross the graph of $F^l_{i,j_1}$ and $F^l_{i,j_2}$ for $l\in \{1,2\}$. Since the rectangles $\tilde{R}^1_{i,j_1}$ and $\tilde{R}^1_{i,j_2}$ share the same stable boundary and $\gamma([0,1]) \subset \tilde{R}_i^1$, then $t_1$ and $t_2$ are well defined. Observe also that we must have $t_1>t_2$. Indeed, $\partial^{u,1}\tilde{R}_{i,j_2}$ must be on the left of $\partial^{u,2}\tilde{R}_{i,j_1}$, otherwise we are in the situation where $j_1=j_2$. We then have:
	\begin{equation*}
		l(\gamma) \leq l(\gamma_{|[0,t_1]}) + l(\gamma_{|[t_2,1]}) \leq \sigma^{u,\alpha}_{\max}(\tilde{R}^1_{i,j_1}) + \sigma^{u,\alpha}_{\max}(\tilde{R}^1_{i,j_2}) \leq K\sigma.
	\end{equation*}
	This proves the third point of the claim.
	
	For the last point, let $j \in I(i)$. We know that $f^n(R_j) \cap \partial^u\tilde{R}^1_{i,j} = \emptyset$ and $\tilde{R}^1_{i,j}\subset \tilde{R}^1_i$. In the Pesin chart $\big(\Psi_{\hx_i},q_{\tilde{\epsilon}}(\hx_i)\big)$, the two curves $\Ws_{loc}(\hy_{i,l})$, for $l \in \{1,2\}$, are $\gamma$-admissible $s$-manifolds of size $q_{\tilde{\epsilon}}(\hx_i)/2$, up to reducing $\sigma$. These curves are horizontal graphs which do not intersect and we can suppose without loss of generality that $\Ws_{loc}(\hy_{i_1})$ is below $\Ws_{loc}(\hy_{i,2})$ in $(\Psi_{\hx_i},q_{\tilde{\epsilon}}(\hx_i))$. Let $P$ be the Jordan domain delimited by the $\Ws_{loc}(\hy_{i,1})$ and the vertical bands $\{\pm q_{\tilde{\epsilon}}(\hx_i)\} \times [-q_{\tilde{\epsilon}}(\hx_i),q_{\tilde{\epsilon}}(\hx_i)]$. Note that $P$ contains both $\tilde{R}^1_{i,j}$ and $\tilde{R}^1_i$. Up to reduce $\sigma$ we can ensure that $f^n(R_j) \subset \Psi_{\hx_i}(R(0,q_{\tilde{\epsilon}}(\hx_i)))$. The curves $f^n(\partial^{u,l}R_j)$ are then contained in $(\mathscr{C}^{u,\hx_i},u)$-manifolds $V^{u,l}$ in $(\Psi_{\hx_i},q_{\tilde{\epsilon}}(\hx_i))$ for $l=1,2$. By Proposition \ref{prop:IntersectionAdmissiblemanifolds}, they intersect each $\Ws_{loc}(\hy_{i,l})$ in a unique point. The intersection between each $f^n(\partial^{u,k}R_j)$ and $P$ is then composed by a single connected curve $c_k$. Now, since $f^n(R_j)\cap \partial^{u,l}\tilde{R}^1_{i,j}=\emptyset$, the curve $c_k$ is contained in $\tilde{R}^1_{i,j}$ for $k=1,2$, and then in $\tilde{R}^1_i$. This implies that $f^n(R_j)\cap\partial^{u,l}\tilde{R}^1_i=\emptyset$ for $l=1,2$.
\end{proof}

Repeating this construction for all $i \in \{1,...,L\}$ gives the family $\mathcal{\tilde{R}}^1 = (\tilde{R}_1^1,...,\tilde{R}^1_L)$.

\parbreak\textbf{Step k.} We now have to iterate this construction. Suppose that we have families $\mathcal{\tilde{R}}^1,...,\mathcal{\tilde{R}}^k$ of exactly $L$ $us$-rectangles which satisfy the following properties:
\begin{itemize}[label={--}]
	\item For each $i \in \{1,...,L\}$ and each $l \in \{1,...,k\}$, $\tilde{R}^l_i$ is a $us$-rectangle and $\big(\Psi_{\hx_i},q_{\tilde{\epsilon}}(\hx_i)\big)$ is an adapted chart for $\tilde{R}^l_i$ for some $\hx_i \in \hat{\Lambda}$.
	\item For each $i \in \{1,...,L\}$, all the rectangles are nested:
	\begin{equation*}
		R_i\subset \tilde{R}^1_i\subset...\subset \tilde{R}^k_i \ \text{and} \ \partial^{s,l}R_i\subset \partial^{s,l}\tilde{R}^1_i\subset...\subset \partial^{s,l}\tilde{R}^k_i\subset \Ws_{loc}(\hy_{i,l})
	\end{equation*}
	for some $\hy_{i,l} \in \hat{\Lambda}$ and for $l \in \{1,2\}$.
	\item We have the following control of the geometry for each $i,l$:
	\begin{equation*}
		\sigma^{u,\alpha}_{\max}(\tilde{R}_i^l) \leq K\sigma \ \text{and} \ \sigma^s_{\max}(\Ws_{loc}(\hy_{i,1}),\Ws_{loc}(\hy_{i,2})) \leq \sigma.
	\end{equation*}
	\item For each $i,l$, there exists a finite set $J^l(i) \subset\{1,...,L\}$ such that:
	\begin{equation*}
		\tilde{R}^l_i = \bigcup_{j \in J^l(i)} \tilde{R}^{l}_{i,j}
	\end{equation*}
	where $\tilde{R}^{l}_{i,j}$ is a $us$-rectangle satisfying the conclusion of Lemma \ref{lem:KeyStepRectangles}. In particular, we have $\sigma^{u,\alpha}_{\max}(\tilde{R}^l_{i,j}) \leq K\sigma/2$.
	\item For each $l_1 \in \{1,...,k\}$, for each $i,j \in \{1,...,L\}$, and for each $l_2< l_1$, if $f^n(R_j)\cap R_i \ne \emptyset$ and if there exists $\hz \in \hat{U}\cap\hf^{-n}(\hat{U})$ such that $\pi(\hz) \in R_j$, then we have:
	\begin{equation*}
		f^n(\tilde{R}^{l_2}_j) \cap \partial^u\tilde{R}^{l_1}_i = \emptyset.
	\end{equation*}
\end{itemize}
Note that the families $\tilde{\mathcal{R}}^0:=\mathcal{R}$ and $\mathcal{\tilde{R}}^1$ satisfy these properties. The goal of this step is to construct a family $\mathcal{\tilde{R}}^{k+1}$ such that the families $\mathcal{\tilde{R}}^1,...,\mathcal{\tilde{R}}^{k+1}$ will satisfy these previous properties.

Let then $i \in \{1,...,L\}$. If for all $j \in I(i)$, we have $f^n(\tilde{R}^k_j) \cap \partial^u\tilde{R}^k_i = \emptyset$, then define $\tilde{R}^{k+1}_i = \tilde{R}^k_i$. Suppose now that there exists $j \in I(i)$ such that $f^n(\tilde{R}^k_j) \cap \partial^u\tilde{R}^k_i \ne \emptyset$. Let then $v \in J^k(j)$ such that $f^n(\tilde{R}^k_{j,v}) \cap \partial^u\tilde{R}^k_i \ne \emptyset$. Denote by $J(j,i)$ the set of such $v$. Apply then Lemma \ref{lem:KeyStepRectangles} with the rectangles $\tilde{R}^k_{j,v}$ and $R_i$, and with the constant $K/2$. It gives us a $us$-rectangle $\tilde{R}^{k+1}_{i,j,v}$. Define:
\begin{equation*}
	\tilde{R}^{k+1}_{i,j} = \bigcup_{v \in J(j,i)} \tilde{R}^{k+1}_{i,j,v} \ \text{and} \ \tilde{R}^{k+1}_i = \bigcup_{j \in I(i)} \tilde{R}^{k+1}_{i,j}.
\end{equation*}

\begin{claim} The family $\mathcal{\tilde{R}}^{k+1} = (\tilde{R}^{k+1}_1,...,\tilde{R}^{k+1}_L)$ has the desired properties.
\end{claim}
\begin{proof}
	The proof of the first two desired properties is the same as in Claim \ref{claim:RectangleProofMain3}. Note that we can write:
	\begin{equation*}
		\tilde{R}^{k+1}_i = \bigcup_{j\in I(i)}\bigcup_{v \in J(j,i)} \tilde{R}^{k+1}_{i,j,v}.
	\end{equation*}
	This shows that $\tilde{R}^{k+1}_i$ is the union of a finite number of $us$-rectangles satisfying the conclusion of Lemma \ref{lem:KeyStepRectangles}, and in particular $\sigma^{u,\alpha}_{\max}(\tilde{R}^{k+1}_{i,j,v}) \leq K\sigma/2$. Let us denote this finite set by $J^{k+1}(i)$. This gives us the fourth property. Once we can write:
	\begin{equation*}
		\tilde{R}^{k+1}_i = \bigcup_{j \in J^{k+1}(i)} \tilde{R}^{k+1}_{i,j}
	\end{equation*}
	then the proof of the third property is the same as in Claim \ref{claim:RectangleProofMain3}. For the last property, take any $j \in I(i)$ and any $v \in J^k(j)$ such that $f^n(\tilde{R}^k_{j,v}) \cap \tilde{R}^{k+1}_i \ne \emptyset$. By construction, there exists $s \in J^{k+1}(i)$ such that $f^n(\tilde{R}^k_{j,v}) \cap \partial^u\tilde{R}^{k+1}_{i,s} = \emptyset$ (the rectangle $\tilde{R}^{k+1}_{i,s}$ is $\tilde{R}^{k+1}_{i,j,v}$ constructed earlier). By applying the same argument as the proof of the last point in Claim \ref{claim:RectangleProofMain3}, we obtain that $f^n(\tilde{R}^k_{j,v}) \cap \partial^u\tilde{R}^{k+1}_i = \emptyset$. Since this is true for any $v \in J^k(j)$, this concludes the proof of the last property.
\end{proof}
This concludes the construction of the family $\tilde{\mathcal{R}}^{k+1}$. Combining Step 1 and Step k, we are then able to build a family $\tilde{\mathcal{R}}^k$ satisfying the properties listed above for any $k \in \bbN$.

\parbreak\textbf{Conclusion.} Let us define the following set:
\begin{equation*}
	\tilde{R}_i = \bigcup_{l\geq1}\tilde{R}_i^l.
\end{equation*}
We claim that this set has the following properties:
\begin{itemize}[label = {--}]
	\item For each $i \in \{1,...,L\}$, $\tilde{R}_i$ is a $us$-rectangle and $\big(\Psi_{\hx_i},q_{\tilde{\epsilon}}(\hx_i)\big)$ is an adapted chart.
	\item For each $i\in\{1,...,L\}$, we have the nested property: $R_i \subset \tilde{R}_i$ and $\partial^{s,k}(R_i) \subset \partial^{s,k}\tilde{R}_i\subset \Ws_{loc}(\hy_{i,k})$ for $k \in \{1,2\}$.
	\item $\sigma/2 \leq \sigma^u_{\min}(\tilde{R}_i)\leq \sigma^u_{\max}(\tilde{R}_i) \leq K\sigma$ and $\sigma^s_{\max}(\tilde{R}_i) \leq \sigma$.
	\item For each $i,j \in \{1,...,L\}$, if $f^n(R_j)\cap R_i \ne \emptyset$ and if there exists $\hz \in \hat{U} \cap \hf^{-n}\hat{U}$ such that $\pi(\hz) \in R_j$, then we have:
	\begin{equation*}
		f^n(\tilde{R}_j) \cap \partial^u\tilde{R}_i = \emptyset.
	\end{equation*}
\end{itemize}

Let us denote by $V^l_{i,k}$ the $\big(\mathscr{C}^{u,\hx},u\big)$-manifold in $\big(\Psi_{\hx_i},q_{\tilde{\epsilon}}\big)$ containing $\partial^{u,k} \tilde{R}^l_i$. Denote by $F^l_{i,k}$ their representative function. Let us fix $i$ and $k$. Let $q = \inf_{l} q(V^l_{i,k})$. Note that by definition, $q\geq 10^{-2}q_{\tilde{\epsilon}}(\hx_i)$. Since we have $\sigma^{u,\alpha}_{\max} (\tilde{R}^l_i) \leq K\sigma$ for all $l\geq1$, Lemma \ref{lem:DistanceChartsManifold} implies that the distance between $F^{l_1}_{i,k}$ and $F^{l_2}_{i,k}$ on $[-q,q]$ is uniformly bounded. Using the Arzéla-Ascoli theorem, up to taking a subsequence, $F^l_{i,k}$ converges then to a continuous map $F_{i,k}$ on $[-q,q]$. Note that $F^l_{i,k}$ is still almost everywhere tangent to $\mathscr{C}^{u,\hx}$. Denote by $V_{i,k}$ the $\big(\mathscr{C}^{u,\hx},u\big)$-manifold represented by $F_{i,k}$. Each $V_{i,k}$ intersects each $\Ws_{loc}(\hy_{i,j})$ in a unique point. This defines a $us$-rectangle. A similar argument as in Claim \ref{claim:RectangleProofMain3} shows that this $us$-rectangle is exactly $\tilde{R}_i$. This proves the first point.

The second and the third items are direct from the construction. Let us prove the last item by contradiction. Assume that there exists $j \in I(i)$ such that $f^n(R_j)\cap R_i \ne \emptyset$ but $f^n(\tilde{R}^j)\cap \partial^u\tilde{R}_i \ne \emptyset$. Since the $\tilde{R}^l_j$ are nested open sets, there exists then some $l\geq 1$ such that $f^n(\tilde{R}^l_j) \cap \partial^u\tilde{R}_i \ne \emptyset$. Now, using a similar argument as in the proof of Claim \ref{claim:RectangleProofMain3} shows that this cannot be possible. This concludes the proof of Proposition \ref{prop:MarkovRectangles}.

\section{Curves intersecting families of $us$-rectangles}\label{sec:CurvesIntersectingFamiliesRectangles}
In this section, we study the intersections of a curve $\gamma$ with a family of $us$-rectangles $\mathcal{R} = \big(R_1,\dots,R_L\big)$ that satisfy the properties of the first family in Proposition~\ref{prop:MarkovRectangles}. More precisely, if we assume that $\gamma$ does not intersect any stable boundary of the rectangles, we will show that the number of rectangles needed to cover the intersection between the curve $\gamma$ and the union of the rectangles is bounded, and the bound does not depend on the number of rectangles.

\subsection{Statement of the curve-covering Proposition}
Throughout this section, we consider $f:M\to M$ a $\cC^r$ local diffeomorphism, with $r>1$, and $\mu$ a saddle hyperbolic measure, possibly non-ergodic. We write $\hmu = \pi_{\star}^{-1}\mu$. We assume that $\mu$ satisfies the hypothesis of Proposition~\ref{prop:MarkovRectangles}. Let us fix $\mathcal{R} = (R_1,\dots,R_L)$, a family of $us$-rectangles satisfying the same properties as the first family given by Proposition~\ref{prop:MarkovRectangles}. We recall here their properties. There exists a Pesin block $\hNUH$ and two parameters $\sigma, \alpha>0$ such that:
\begin{itemize}[label={--}]
	\item For any $i \in \{1,\dots,L\}$, there exists $\hx_i \in \hNUH$ such that $\big(\Psi_{\hx_i},q_{\tilde{\epsilon}}(\hx_i)\big)$ is an adapted chart for $R_i$.
	\item $\partial^{s,l}R_i \subset \Ws_{loc}(\hy_i^l)$ where $\hy_i^l \in \hNUH$.
	\item There exists a foliation $\mathscr{F}^u$ of an open set $V\supset \pi(\hNUH)$ tangent to $\mathscr{C}^{u,ext,\alpha}$ such that for any $i \in \{1,\dots,L\}$ the unstable boundary $\partial^{u,l} R_i$ is contained in some leaf of $\mathscr{F}^u$. Moreover, $\partial^{u,l}R_i \ne \partial^{u,k}R_j$ whenever $i \ne j$ or $i=j$ and $k \ne l$.
	\item $\sigma/2 \leq \sigma^u_{\min}(R_i) \leq \sigma^{u,\alpha}_{\max}(R_i) \leq \sigma$.
\end{itemize}
Note that the family $\mathcal{R}$ depends on the parameters $\sigma$ and $\alpha$. Let us first define which curves we are considering. If $N$ is a compact manifold and $\varphi : N \to M$ is a $\cC^r$ map, we define the $\cC^r$-size of $\varphi$ to be the maximum of the norms of all the $k$-derivatives of $\varphi$ in any chart, for $k=1,\dots,r$. We emphasize that the maximum is taken over all the charts of some fixed finite atlas on $M$ and $N$. Note that the $\cC^r$-size depends on the choice of atlases, but by compactness, changing atlases induces equivalent $\cC^r$-sizes. In the present paper, $N$ will often be $[0,1]$ and the atlas we choose will be the trivial one.

\begin{defn}[$(r,\epsilon)$-curves]
	Let $r>1$ and $\epsilon>0$. Let $\gamma : [0,1] \to M$ be a $\cC^r$ curve. We say that $\gamma$ is a $(r,\epsilon)$-curve if $\norm{\gamma}_{\cC^r} < \epsilon$.
	\label{def:Curve}
\end{defn}

In this section, we fix a $(r,C\sigma)$-curve $\gamma$. We will often denote by $\gamma$ the image of $[0,1]$ under the map $\gamma$. We assume that $\gamma$ intersects the rectangles:
\begin{equation*}
	\gamma \cap \big(\bigcup_i R_i\big) \ne \emptyset.
\end{equation*}
To study the intersection between the curve $\gamma$ and the rectangles, we will introduce two different types of intersecting intervals. See Figure~\ref{fig:CrossTurnInterval}.

\begin{defn}[Crossing/folding intervals]
	Let $[s,t] \subset [0,1]$ and let $i \in \{1,\dots,L\}$.
	
	We say that $[s,t]$ is a crossing $(\gamma,R_i)$-interval if:
	\begin{itemize}[label={--}]
		\item $\gamma(s) \in \partial^{u,1} R_i$ and $\gamma(t) \in \partial^{u,2}R_i$ or vice versa.
		\item $\gamma((s,t)) \subset R_i$.
	\end{itemize}
	We say that $[s,t]$ is a crossing $(\gamma,\mathcal{R})$-interval if there exists some $i$ such that $[s,t]$ is a crossing $(\gamma,R_i)$-interval.\\
	We say that $[s,t]$ is a folding $(\gamma,R_i)$-interval if:
	\begin{itemize}[label={--}]
		\item $\gamma(s), \gamma(t) \in \partial^{u,l} R_i$ for some $l \in \{1,2\}$.
		\item $\gamma((s,t)) \subset R_i$.
	\end{itemize}
	We say that $[s,t]$ is a folding $(\gamma,\mathcal{R})$-interval if there exists some $i$ such that $[s,t]$ is a folding $(\gamma,R_i)$-interval.\\
	We denote by $\text{Cross}(\gamma,R_i)$ the collection of all crossing $(\gamma,R_i)$-intervals and $\text{Cross}(\gamma,\mathcal{R})$ the collection of all crossing $(\gamma,\mathcal{R})$-intervals. We denote by $\text{Fold}(\gamma,R_i)$ the collection of all folding $(\gamma,R_i)$-intervals and $\text{Fold}(\gamma,\mathcal{R})$ the collection of all folding $(\gamma,\mathcal{R})$-intervals.
	
	When there is no ambiguity about the family of $us$-rectangles $\mathcal{R}$ we consider, we will often omit the dependence on $\mathcal{R}$.
	\label{def:Crossing/TurningIntervals}
\end{defn}

\begin{figure}[h]
	\centering
	\begin{subfigure}{.5\textwidth}
		\centering
		\includegraphics[width=0.9\linewidth]{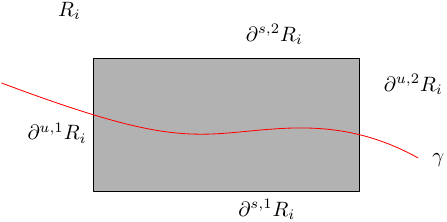}
		\caption{A crossing $\gamma$-interval.}
		\label{subfig:CrossingInterval}
	\end{subfigure}%
	\begin{subfigure}{.5\textwidth}
		\centering
		\includegraphics[width=0.9\linewidth]{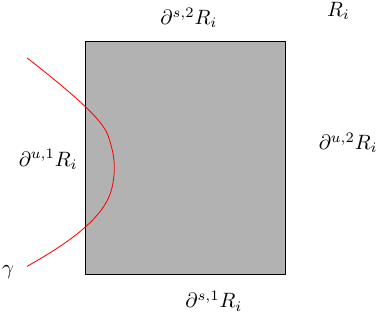}
		\caption{A folding $\gamma$-interval.}
		\label{subfig:TurninInterval}
	\end{subfigure}
	\caption{}
	\label{fig:CrossTurnInterval}
\end{figure}

The goal of this section is to prove the following proposition.

\begin{prop}
	Fix $r>1$. Let $\sigma,\alpha>0$. Let $\mathcal{R} := (R_1,\dots,R_L)$ be a family of $us$-rectangles satisfying the setting given before Definition~\ref{def:Curve}, i.e., they satisfy the properties of the first family in Proposition~\ref{prop:MarkovRectangles}. Fix $\gamma:[0,1]\to M$ a $(r,C\sigma)$-curve such that $\gamma \cap \cup_i R_i \ne \emptyset$. Suppose the following:
	\begin{equation}
		\gamma \cap \Ws_{loc}(\hy_i^l) = \emptyset
		\label{eq:HypothesisSecCurveInterRectangles}
	\end{equation}
	where $\partial^{s,l}R_i \subset \Ws_{loc}(\hy_i^l)$ for $\hy_i^l \in \hNUH$. Then, the following holds for $\sigma := \sigma(\hNUH,C)$ sufficiently small.
	
	There exist two collections $\mathcal{C}_1 \subset \text{Cross}(\gamma)$ and $\mathcal{C}_2 \subset \{1,\dots,L\}$ with the following properties:
	\begin{enumerate}
		\item $\text{Cross}(\gamma) = \bigcup_{[s,t] \in \mathcal{C}_1} [s,t]$. Here we make the slight abuse of notation by considering $\text{Cross}(\gamma)$ as a union of intervals.
		\item $|\mathcal{C}_1| \leq 4C$.
		\item $\gamma\big(\text{Fold}(\gamma)\big) = \gamma\big(\text{Fold}(\gamma)\big) \cap \bigcup_{i \in \mathcal{C}_2} R_i$. Here we also make the slight abuse of notation by considering $\text{Fold}(\gamma)$ as a union of intervals.
		\item $|\mathcal{C}_2| \leq 8C$.
		\item There exist $s_1,s_2 \in [0,1]$ such that:
		\begin{equation*}
			\gamma([0,1]) \cap \bigcup_i R_i = \gamma\big([0,s_1]\big) \cup \gamma\big(\text{Fold}(\gamma)\big) \cup \gamma\big(\text{Cross}(\gamma)\big) \cup \gamma\big([s_2,1]\big)
		\end{equation*}
		where there exist $i,j \in \{1,\dots,L\}$ such that $\gamma([0,s_1)) \subset R_i$ and $\gamma((s_2,1]) \subset R_j$. If $s_1 = 0$ or $s_2 = 1$, we consider $[0,s_1]$ or $[s_2,1]$ to be empty.
	\end{enumerate}
	\label{prop:IntersectionCurveRectangle}
\end{prop}

Let us make some comments on the previous statement. The idea is that we can choose a small number of crossing $\gamma$-intervals to cover the whole set of such intervals by controlling their lengths. In the case of folding intervals, this is no longer possible because we are not able to control their lengths. Nevertheless, we can pick a small number of $us$-rectangles from the family $\mathcal{R}$ to cover the whole set of folding $\gamma$-intervals. In both cases, it is very important that the bound is uniform, in the sense that it depends only on the constant $C$. The rest of this section is devoted to the proof of Proposition~\ref{prop:IntersectionCurveRectangle}.

\subsection{Key lemmas for treating crossing and folding intervals}
Throughout this paragraph, we assume that we are in the same setting as Proposition~\ref{prop:IntersectionCurveRectangle}. We begin by estimating the size of a crossing $\gamma$-interval.

\begin{lemma}
	Let $[s,t] \in \text{Cross}(\gamma)$. Then we have:
	\begin{equation*}
		|t-s| \geq \frac{1}{2C}.
	\end{equation*}
	\label{lem:EstimationSizeCrossing}
\end{lemma}

\begin{proof}
	Let $i \in \{1,\dots,L\}$ such that $[s,t] \in \text{Cross}(\gamma,R_i)$. Recall that we then have:
	\begin{itemize}[label={--}]
		\item $\gamma(s) \in \partial^{u,1}R_i$,
		\item $\gamma(t) \in \partial^{u,2}R_i$,
		\item $\gamma((s,t)) \subset R_i$.
	\end{itemize}
	Let $\tilde{\gamma} : [0,1] \to M$ be the curve defined by:
	\begin{equation*}
		\tilde{\gamma}(x) = \gamma\big((t-s)x + s\big).
	\end{equation*}
	Then, the curve $\tilde{\gamma}$ is a $\cC^1$ curve such that $\tilde{\gamma}(0) \in \partial^{u,1}R_i$ and $\tilde{\gamma}(1) \in \partial^{u,2}R_i$. We then have:
	\begin{equation*}
		l(\tilde{\gamma}) \geq \sigma^u_{\min}(R_i) \geq \sigma/2.
	\end{equation*}
	On the other hand, since $\gamma$ is a $(r,C\sigma)$-curve, we have:
	\begin{equation*}
		l(\tilde{\gamma}) = \int_0^1 \norm{\tilde{\gamma}'(x)} dx = \int_s^t \norm{\gamma'(x)} dx < |t-s|C\sigma.
	\end{equation*}
	Then, putting everything together leads to:
	\begin{equation*}
		|t-s| > \frac{1}{2C}.
	\end{equation*}
\end{proof}

Let us continue with a useful lemma on the geometry of the intersection between the curve $\gamma$ and the family of $us$-rectangles $\mathcal{R}$. The idea is that, using~\eqref{eq:HypothesisSecCurveInterRectangles}, the curve $\gamma$ must be contained in a domain delimited by two stable manifolds containing the stable boundaries of rectangles intersecting $\gamma$. See Figure~\ref{fig:ProofLemmaGoodChartCurveRectangle}.

\begin{lemma}
	The following holds for $\sigma := \sigma(\hNUH)$ sufficiently small.
	
	There exists a chart $\Psi : (0,1)^2 \to M$ that displays the following properties.
	\begin{itemize}[label={--}]
		\item $\gamma([0,1]) \subset \Psi((0,1)^2)$.
		\item For any $R \in \mathcal{R}$ such that $R$ intersects $\gamma$, the image under $\Psi^{-1}$ of each unstable boundary of $R$ is a vertical Lipschitz graph that separates $(0,1)^2$ into two disjoint open connected components. This gives a natural meaning to being on the left (or the right) of such a curve.
	\end{itemize}
	In particular, for any $R \in \mathcal{R}$, if there exist $s,t \in [0,1]$ such that $\gamma(s)$ is on the left of the left unstable boundary of $R$ and $\gamma(t)$ is on the right of the right unstable boundary of $R$, then the interval $[s,t]$ contains a crossing $\gamma$-interval.
	\label{lem:GoodChartCurveRectangle}
\end{lemma}

\begin{figure}[h]
	\centering
	\includegraphics[scale=0.8]{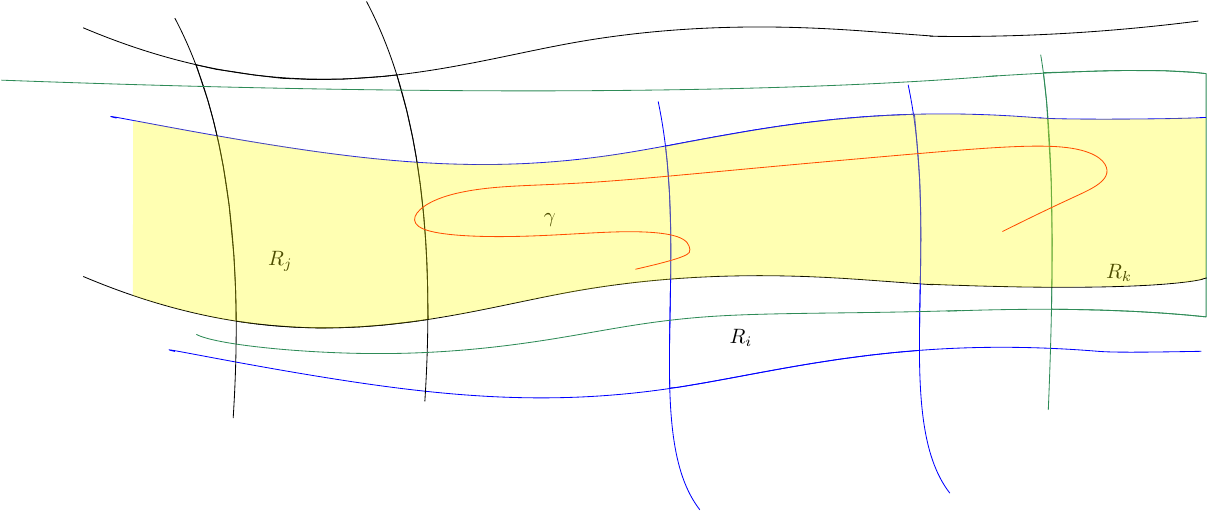}
	\caption{The proof of Lemma~\ref{lem:GoodChartCurveRectangle} in the case where $\gamma$ intersects three rectangles $R_i, R_j$ and $R_k$. Horizontals represent the stable direction. The chart $\Psi$ is given by the colored domain.}
	\label{fig:ProofLemmaGoodChartCurveRectangle}
\end{figure}

\begin{proof}
	If $\gamma$ does not intersect any $R \in \mathcal{R}$, there is nothing to prove. Let $R \in \mathcal{R}$ such that $\gamma([0,1]) \cap R \ne \emptyset$. Let $\hx \in \hNUH$ and let $\big(\Psi_{\hx},q_{\tilde{\epsilon}}(\hx)\big)$ be an adapted chart for $R$. Recall that since $\gamma$ is a $(r,C\sigma)$-curve, taking $\sigma$ uniformly small enough implies that $\gamma([0,1]) \subset \Psi_{\hx}\big(R(0,q_{\tilde{\epsilon}}(\hx))\big)$. Recall that for any $R' \in \mathcal{R}$, each of the stable boundaries of $R'$ is contained in the local stable manifold of some point of $\hNUH$. Let $R' \in \mathcal{R}$ be a $us$-rectangle intersecting $\gamma$, let $l\in \{1,2\}$ and let $\hy \in \hNUH$ such that $\partial^{s,l}R' \subset \Ws_{loc}(\hy)$. Up to reducing the size of the chart $\big(\Psi_{\hx},q_{\tilde{\epsilon}}(\hx)\big)$ by a uniform factor, the curve $\Psi_{\hx}^{-1}(\Ws_{loc}(\hy))$ is a horizontal Lipschitz graphs of size $q_{\tilde{\epsilon}}(\hx)$. Such Lipschits graphs separates the chart $\big(\Psi_{\hx},q_{\tilde{\epsilon}}(\hx)\big)$ into two disjoints open connected component and \eqref{eq:HypothesisSecCurveInterRectangles} implies that $\Psi_{\hx}\inv(\gamma)$ is contained in one of these two connected components. This implies that there exist $\hy^+,\hy^- \in \hNUH$ and a Jordan domain $Q$ displaying the following properties.
	\begin{itemize}[label={--}]
		\item The curve $\Psi_{\hx}^{-1}(\gamma)$ is contained in $Q$.
		\item Its boundary is the union of the four following curves.
		\begin{itemize}[label=$\bullet$]
			\item The two Lipschitz horizontal graphs of size $q_{\tilde{\epsilon}}(\hx)$ that are the images of $\Ws_{loc}(\hy^+)$ and $\Ws_{loc}(\hy^-)$ under $\Psi_{\hx}^{-1}$.
			\item The left and right vertical boundaries of the chart $\big(\Psi_{\hx},q_{\tilde{\epsilon}}(\hx)\big)$, i.e., $\{\pm q_{\tilde{\epsilon}}(\hx)\} \times [-q_{\tilde{\epsilon}}(\hx),q_{\tilde{\epsilon}}(\hx)]$.
		\end{itemize}
		\item For any $R' \in \mathcal{R}$ and any $\hy$ such that $\partial^{s,l}R' \subset \Ws_{loc}(\hy)$, the curve $\Ws_{loc}(\hy)$ does not intersect the domain $Q$. Note that $\Ws_{loc}(\hy)$ has to be above $\Ws_{loc}(\hy^+)$ or below $\Ws_{loc}(\hy^-)$ (up to changing the notation) in the chart $(\Psi_{\hx},q_{\tilde{\epsilon}}(\hx))$.
	\end{itemize}
	Note that for any rectangle $R' \in \mathcal{R}$ that intersects $\gamma$, we have that $\partial^{u,l}R'$ separates $Q$ into two disjoint open connected components. Indeed,~\eqref{eq:HypothesisSecCurveInterRectangles} implies that $\Psi_{\hx}^{-1}(\gamma)$ must be contained in the Jordan domain formed by the left and right vertical boundaries of the chart and the two horizontal graphs defined by the local stable manifolds containing $\partial^sR'$. Remark that this domain is separated into two disjoint open connected components by each of the $\partial^{u,l}R'$. Since local stable manifolds are either equal or disjoint, we also obtain that $Q$ must be contained inside this domain. We conclude that $Q$ is also separated into two disjoint open connected components by each of the $\partial^{u,l}R'$.
	
	Take now a diffeomorphism $\varphi : [0,1]^2 \to \overline{Q}$ that sends each horizontal, respectively vertical, boundary of $(0,1)^2$ to each horizontal, respectively vertical, boundary of $Q$. Since $\varphi$ is a diffeomorphism, it preserves the properties of $Q$ discussed above. Taking $\Psi = \Psi_{\hx} \circ \varphi$ then concludes the proof, since the last statement is obvious by connectedness of $\gamma$.
\end{proof}

\subsection{Proof of the curve-covering Proposition~\ref{prop:IntersectionCurveRectangle}}
We split the proof into two distinct parts. In the first part, we treat the crossing intervals using Lemma~\ref{lem:EstimationSizeCrossing}. In the second part, we study the folding intervals. Before going into the details, let us quickly explain the idea of the proof. First, in the case of crossing intervals, Lemma~\ref{lem:EstimationSizeCrossing} tells us that crossing $\gamma$-intervals are long. Then, using the fact that the unstable boundaries of the $us$-rectangles of the family $\mathcal{R}$ are disjoint, we conclude that if some crossing $\gamma$-intervals intersect, then taking the first and the last intervals of the intersection is enough to cover all the intersecting intervals. The case of folding intervals is more delicate. We use the fact that $\gamma$ cannot cross the stable manifolds containing the stable boundaries of the rectangles to conclude that $\gamma$ is trapped in a rectangular domain where each unstable boundary of each rectangle of $\mathcal{R}$ separates it into two disjoint connected components. We use this fact to select a uniformly small number of rectangles containing the image of all the folding $\gamma$-intervals between two crossing $\gamma$-intervals.

\parbreak\textbf{Step 1: A bound on the number of crossing $\gamma$-intervals.}
Let us start with the crossing intervals. Since the number of crossing $\gamma$-intervals is at most countable (they cannot accumulate since they have uniform length by Lemma \ref{lem:EstimationSizeCrossing}), we can order the starting parameter $s$ of all crossing $\gamma$-intervals $[s,t]$. Denote by $m$ the cardinality of $\text{Cross}(\gamma)$. Note that Lemma~\ref{lem:EstimationSizeCrossing} implies that there exists at most a finite number of disjoint crossing $\gamma$-intervals. Combining this with the fact that $\mathcal{R}$ has finite cardinality, we obtain that $m$ must be finite. We can then write:
\begin{equation*}
	\text{Cross}(\gamma) = \big\{[s_i,t_i], \ i \in \{1,\dots,m\}\big\},
\end{equation*}
where $s_1 \leq \dots \leq s_m$.

For any $i \in \{1,\dots,m\}$, define $j(i) \in \{1,\dots,m\}$ to be the integer such that:
\begin{equation*}
	s_{j(i)} = \inf\{s_j, \ j \in \{1,\dots,m\} \ \text{such that} \ s_j > t_i\}.
\end{equation*}
If such an integer does not exist, set $j(i)=m+1$ with $s_{m+1}=t_{m+1}=1$. The integer $s_{j(i)}$ is just the starting parameter of the next crossing $\gamma$-interval disjoint from $[s_i,t_i]$.

We are going to build the collection $\mathcal{C}_1$ by induction. The idea is to start with a crossing $\gamma$-interval $[s_i,t_i]$ that we add to $\mathcal{C}_1$. Then we can cover all the crossing $\gamma$-intervals that are not disjoint from $[s_i,t_i]$ by the union of $[s_i,t_i]$ and the last crossing $\gamma$-interval intersecting $[s_i,t_i]$. We also add this one to $\mathcal{C}_1$ and move to the next crossing interval disjoint from $[s_i,t_i]$ to repeat the argument. Lemma~\ref{lem:EstimationSizeCrossing} then gives a bound on $\mathcal{C}_1$. This is contained in the following claim. See Figure~\ref{fig:ProofCrossingIntervals}.

\begin{claim}
	Let $i \in \{1,\dots,m\}$. The following is true:
	\begin{equation*}
		[s_i,t_{j(i)-1}] \subset [s_i,t_i] \cup [s_{j(i)-1},t_{j(i)-1}].
	\end{equation*}
	\label{claim:CrossingIntervals}
\end{claim}

\begin{proof}
	Let $s \in [s_i,t_{j(i)-1}]$. For any $l \in \{i,\dots,j(i)-1\}$, the interval $[s_l,t_l]$ intersects the interval $[s_i,t_i]$ since, by definition of $j(i)$, we have $s_l \leq t_i$. Then, either we have $s_l \leq s_{j(i)-1}$ and $s \in [s_i,t_i]$, or $s_l > s_{j(i)-1}$ and $s \in [s_{j(i)-1},t_{j(i)-1}]$.
\end{proof}

\begin{figure}[h]
	\centering
	\includegraphics[scale=0.8]{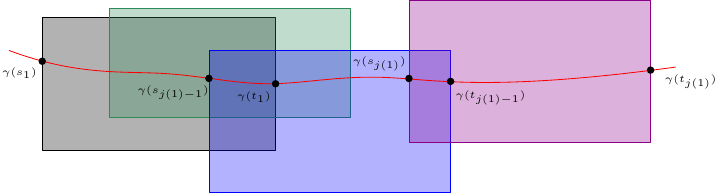}
	\caption{Proof of Claim~\ref{claim:CrossingIntervals}.}
	\label{fig:ProofCrossingIntervals}
\end{figure}

We now explain the induction. We start by letting $\mathcal{C}_{1,1} = \{[s_1,t_1],[s_{j(1)},t_{j(1)}]\}$. Note that by Claim~\ref{claim:CrossingIntervals} we have:
\begin{equation*}
	\bigcup_{i=1}^{j(1)-1} [s_i,t_i] \subset \bigcup_{[s,t] \in \mathcal{C}_{1,1}} [s,t].
\end{equation*}
Suppose that there exists a collection of crossing $\gamma$-intervals $\mathcal{C}_{1,n} = \{[\tilde{s}_1,\tilde{t}_1],\dots,[\tilde{s}_{2n},\tilde{t}_{2n}]\}$, where $\tilde{s}_{2n-1} = s_{j^{n-1}(1)}$ and $\tilde{s}_{2n} = s_{j^n(1)-1}$, and such that:
\begin{equation*}
	\bigcup_{i=1}^{j^n(1)-1}[s_i,t_i] \subset \bigcup_{i=1}^{2n} [\tilde{s}_i,\tilde{t}_i].
\end{equation*}
where we define $j^n(1) := j(j^{n-1}(1))$ by induction, with $j(m+1) = m+1$. Suppose also that every intervals with odd indices are pairwise disjoints. Let us explain how we build $\mathcal{C}_{1,n+1}$.

Let $k = j^n(1)$. Let $\mathcal{C}_{1,n+1} = \mathcal{C}_{1,n} \cup \{[s_k,t_k],[s_{j(k)-1},t_{j(k)-1}]\}$. Combining the induction hypothesis and Claim~\ref{claim:CrossingIntervals} we have:
\begin{equation*}
	\begin{aligned}
		\bigcup_{i=1}^{j^{n+1}(1)-1}[s_i,t_i] &\subset \bigcup_{i=1}^{2n} [\tilde{s}_i,\tilde{t}_i] \cup [s_k,t_k] \cup [s_{j(k)-1},t_{j(k)-1}].
	\end{aligned}
\end{equation*}
The disjointness property is also satisfied since $k = j(j^{n-1}(1))$ and $s_{2n-1} = s_{j^{n-1}(1)}$. We stop the induction when we reach $n_{\star}$ such that $j^{n_{\star}}(1) = m+1$. We let $\mathcal{C}_1 = \mathcal{C}_{1,n_{\star}}$. Note that we have by construction:
\begin{equation*}
	\bigcup_{i=1}^m [s_i,t_i] \subset \bigcup_{[s,t] \in \mathcal{C}_1} [s,t].
\end{equation*}
Note that the disjointness property combined with Lemma~\ref{lem:EstimationSizeCrossing} implies that $|\mathcal{C}_1| \leq 4C$ since there are at most $2C$ pairwise disjoint crossing $\gamma$-intervals. This proves the first two points of Proposition~\ref{prop:IntersectionCurveRectangle}.

\parbreak\textbf{Step 2: A bound on the number of rectangles needed to cover the folding $\gamma$-intervals.}
Contrary to Step 1, the number of folding $\gamma$-intervals can be infinite. Moreover, it may be possible to have a sequence of $(\gamma,R_i)$-folding intervals accumulating to a point, for some $i \in \{1,\dots,L\}$, and then the set $\text{Fold}(\gamma)$ may not be orderable as in Step 1. Nevertheless, since the family $\mathcal{R}$ is finite and their unstable boundaries are contained in different leaves of $\mathscr{F}^u$, there exists a constant $C'$ such that:
\begin{equation}
	\max_{i_1 \ne i_2 \in \{1,\dots,L\}} \max_{l_1,l_2 \in \{1,2\}} d\big(\partial^{u,l_1}R_{i_1},\partial^{u,l_2}R_{i_2}\big) \geq C'.
	\label{eq:SeparationRectangles}
\end{equation}
We construct inductively a family $\mathcal{R}^{\star}$ of rectangles which cover the curve $\gamma$ and such that the following holds. Let $t \in (0,1)$ be the ending parameter of a crossing $\gamma$-interval and let $s>t$ be the starting parameter of the first crossing $\gamma$-interval after $t$. Then the number of rectangle of $\mathcal{R}^{\star}$ needed to cover $\gamma([t,s])$ is uniformly bounded.
\begin{lemma}
	There exists a finite collection of $us$-rectangles $\mathcal{R}^{\star} := \{R_{i_1},\dots,R_{i_m}\} \subset \mathcal{R}$ and a finite collection of parameters $t_0 = 0 < t_1 < \dots < t_{m+1} = 1$ that have the following properties.
	\begin{itemize}[label={--}]
		\item For any $k = 1,\dots,m$, the parameter $t_k$ is the ending parameter of a crossing $\gamma$-interval $I_k$.
		\item The intervals $(I_k)_{k=1,\dots,m}$ are pairwise disjoint.
		\item $i_m \leq 4m$.
		\item We have the following covering property:
		\begin{equation*}
			\gamma\big(\text{Fold}(\gamma)\big) = \gamma\big(\text{Fold}(\gamma)\big) \cap \big(\bigcup_{i=1}^m R_{i_m}\big).
		\end{equation*}
	\end{itemize}
	\label{claim:OrderRectangleFolding}
\end{lemma}

\begin{proof}
	The proof goes by induction. Let $t_0 = 0$ and $\mathcal{R}^{\star}_0 = \emptyset$. Let $n \geq 1$. Suppose the following. There exists a collection of $us$-rectangles $\mathcal{R}^{\star}_{n-1} := \{R_{i_1},\dots,R_{i_{n-1}}\} \subset \mathcal{R}$ and a sequence of parameters $t_0 = 0 < \dots < t_{n-1}$ with the following properties.
	\begin{itemize}[label={--}]
		\item For any $k = 1,\dots,n-1$, the parameter $t_k$ is the ending parameter of a crossing $\gamma$-interval $I_k$.
		\item The intervals $(I_k)_{k=1,\dots,n-1}$ are pairwise disjoint.
		\item $i_{n-1} \leq 4(n-1)$.
		\item We have the following covering property. For any folding $\gamma$-interval $[s,t] \subset [0,t_{n-1}]$ we have:
		\begin{equation*}
			[s,t] \subset \bigcup_{j=1}^{i_{n-1}} R_j.
		\end{equation*}
	\end{itemize}
	If $t_{n-1} = 1$, the induction stops and the claim is proven. Suppose then that $t_{n-1} < 1$. We are now going to explain how to construct the family $\mathcal{R}^{\star}_{n}$ and the parameter $t_n$ such that all the previous listed properties are still satisfied at rank $n$.
	
	By Lemma~\ref{lem:GoodChartCurveRectangle}, we can suppose that the curve $\gamma$ and its intersection with any $us$-rectangle $R \in \mathcal{R}$ is contained in $(0,1)^2$. Moreover, for any $l = 1,2$, the curve $\partial^{u,l}R$ separates $(0,1)^2$ into two disjoint open connected components. Recall that this property gives a meaning to being on the left (or the right) of $\partial^{u,l}R$. Let $R,R' \in \mathcal{R}$ such that both $R$ and $R'$ intersect $\gamma$. We say that $R$ is on the left (or the right) of $R'$ if the left (or the right) unstable boundary of $R$ is on the left (or the right) of the left (or the right) unstable boundary of $R'$.
	
	We say that a folding $\gamma$-interval $[s,t]$ is a left (or right) folding $\gamma$-interval if $\gamma(s)$ belongs to the left (or the right) unstable boundary of its associated rectangle.
	
	Let us consider the next end of a crossing $\gamma$-interval starting after $t_{n-1}$. Define:
	\begin{equation*}
		t_c = \inf\{t \in (t_{n-1},1] \ \exists s>t_{n-1}, \ [s,t] \in \text{Cross}(\gamma)\}.
	\end{equation*}
	If such a crossing interval does not exist, we let $t_c = 1$. Let $s_c \in (s'_1,t_c)$ be the starting parameter of the crossing $\gamma$-interval ending at $t_c$ and let $R_c \in \mathcal{R}$ be the $us$-rectangle associated to $[s_c,t_c]$. Remark that since $s_c > t_{n-1}$, the interval $[s_c,t_c]$ is disjoint from all the $I_k$ for $k = 1,\dots,n-1$.
	
	From now on, and until the end of the proof of Lemma \ref{claim:OrderRectangleFolding}, we suppose that $t_{n-1}$ belongs to the right boundary of the rectangle associated to the crossing $\gamma$-interval $I_{n-1}$. The case where $t_{n-1}$ belongs to the left boundary is exactly symmetric.
	
	\parbreak\textbf{Step 2.a: Select the rectangles to cover the folding intervals starting before $t_{n-1}$ and ending before $t_c$.}
	Note that the induction allows us to cover the folding $\gamma$-intervals included in $[0,t_{n-1}]$. We begin by covering such intervals that intersect $[0,t_{n-1}]$ but are not contained in it.
	
	\begin{claim}
		There exist two $us$-rectangles $R_{i_n},R_{i_{n+1}} \in \mathcal{R}$ that display the following property. For any folding $\gamma$-interval $[s,t]$ satisfying:
		\begin{equation*}
			s<t_{n-1} \quad \text{and} \quad t_{n-1}<t<t_c,
		\end{equation*}
		then $\gamma([s,t])$ is contained in either $R_{i_n}$ or $R_{i_{n-1}}$.
		\label{claim:Folding:RankBefore}
	\end{claim}
	
	\begin{proof}
		We are first going to prove that there exists a rectangle $R_{i_n} \in \mathcal{R}$ that contains all the left folding $\gamma$-intervals intersecting $[0,t_{n-1}]$, which are not contained in it. Note that by~\eqref{eq:SeparationRectangles}, this set is finite. Denote it as $\{[s'_1,t'_1],\dots,[s'_l,t'_l]\}$ such that $s'_1 < \dots < s'_l$. We are going to prove that the following order holds:
		\begin{equation*}
			s'_1 < \dots < s'_l \leq t_{n-1} < t'_l < \dots < t'_1.
		\end{equation*}
		For any $k \in \{1,\dots,l\}$, denote the $us$-rectangle associated with $[s'_k,t'_k]$ by $R'_k$. Using the chart given by Lemma~\ref{lem:GoodChartCurveRectangle} and since $s'_1 < \dots < s'_l$, the following property holds. For any $k \in \{1,\dots,l-1\}$ the left unstable boundary of $R'_k$ is on the left of the left unstable boundary of $R'_{k+1}$. Indeed, suppose by contradiction that there exists $k \in \{1,\dots,l-1\}$ such that the left unstable boundary of $R'_k$ is on the right of the left unstable boundary of $R'_{k+1}$. Since $s'_k < s'_{k+1}$, the curve $\gamma$ crosses first the left unstable boundary of $R'_k$ and then the left unstable boundary of $R'_{k+1}$. But $\gamma$ must then cross again the left unstable boundary of $R'_k$ before $\gamma(s'_{k+1})$. It implies that $t'_k < s'_{k+1}$, which is a contradiction.
		
		Since any $s'_k \leq t_{n-1} < t'_1$ for $k = 1,\dots,l$, we know that $\gamma$ must cross first all the left unstable boundaries of $R'_k$ for $k = 2,\dots,l$ before crossing again the left unstable boundary of $R'_1$. This property combined with the order on the left unstable boundaries of the $R'_k$ implies the desired order $s'_1 < \dots < s'_l < t'_l < \dots < t'_1$.
		
		We then proved that for any $k = 2,\dots,l$ we have $[s'_k,t'_k] \subset [s'_1,t'_1]$. But since $\gamma([s'_1,t'_1]) \subset R'_1$, we obtain that any left folding $\gamma$-interval $[s,t]$ intersecting $[0,t_{n-1}]$, which is not contained in $[0,t_{n-1}]$, satisfies $\gamma([s,t]) \subset R'_1$. We then define $R_{i_n} := R'_1$.
		
		The argument is the same for the right folding $\gamma$-intervals intersecting $[0,t_{n-1}]$, but not contained in it. Denote the rectangle containing them by $R_{i_{n+1}}$. This concludes the proof of Claim~\ref{claim:Folding:RankBefore}.
	\end{proof}
	
	\parbreak\textbf{Step 2.b: Select the rectangles to cover the intervals starting after $t_{n-1}$ and find the next crossing $\gamma$-interval.}
	We now want to select rectangles to cover the folding $\gamma$-intervals included in $(t_{n-1},t_c)$. We are first going to define three time parameters that will be crucial to choose these rectangles.
	
	Define first $s'_1 \in [t_{n-1},1]$ to be the first starting parameter of a folding $\gamma$-interval after $t_{n-1}$.
	\begin{equation*}
		s'_1 = \inf\{s \in [t_{n-1},1], \ \exists \ t > s, \ [s,t] \in \text{Fold}(\gamma)\}.
	\end{equation*}
	If such a parameter does not exist, there is no folding $\gamma$-interval left after $t_{n-1}$ and we set $t_n = 1$ and $\mathcal{R}^{\star}_{n} = \{R_{i_1},\dots,R_{i_{n+1}}\}$, which stops the induction and proves Lemma~\ref{claim:OrderRectangleFolding}. Suppose then that such a parameter exists and let $t'_1 \in [0,1]$ and $R'_1 \in \mathcal{R}$ such that $[s'_1,t'_1] \in \text{Fold}(\gamma,R'_1)$. Without loss of generality, we can suppose that $[s'_1,t'_1]$ is a left folding $\gamma$-interval.
	
	Secondly, define $s'_2 \in [s'_1,1)$ to be the next starting parameter of a folding $\gamma$-interval associated to a rectangle disjoint from $R'_1$. Equation~\eqref{eq:SeparationRectangles} guarantees that $s'_2 > s'_1$.
	\begin{equation*}
		s'_2 = \inf\{s \in (s'_1,1), \ \exists \ R \in \mathcal{R}, \ R \cap R'_1 = \emptyset, \ \exists \ t > s, \ [s,t] \in \text{Fold}(\gamma,R)\}.
	\end{equation*}
	Let $t'_2 > s'_2$ and $R'_2 \in \mathcal{R}$ such that $[s'_2,t'_2] \in \text{Fold}(\gamma,R'_2)$. If such an interval does not exist, we set $s'_2 = t'_2 = 1$ and $R'_2 = \emptyset$.
	
	Finally, define $s'_3 > s'_2$ to be the starting parameter of the first folding $\gamma$-interval associated to a rectangle disjoint from $R'_1$ and $R'_2$. If $s'_2 = 1$, set $s'_3 = 1$. Otherwise, set:
	\begin{equation*}
		s'_3 = \inf\{s \in (s'_2,1), \ \exists \ R \in \mathcal{R}, \ R \cap R'_1 = \emptyset \ \text{and} \ R \cap R'_2 = \emptyset, \ \exists \ t > s, \ [s,t] \in \text{Fold}(\gamma,R)\}.
	\end{equation*}
	Let $t'_3 > s'_3$ and $R'_3 \in \mathcal{R}$ such that $[s'_3,t'_3] \in \text{Fold}(\gamma,R'_3)$. If such an interval does not exist, we set $s'_3 = t'_3 = 1$ and $R'_3 = \emptyset$. Remark the following important fact: $t_c<s_3^{\prime}$. Indeed, since the three rectangles $R_1^{\prime},R_2^{\prime}$ and $R_3^{\prime}$ are pairwise disjoints, we can suppose without loss of generality that $R_1^{\prime}$ is the leftmost, $R_2^{\prime}$ is in the middle and $R_{3}^{\prime}$ is the rightmost. Then $\gamma(s_1^{\prime})$ is on the left of the left unstable boundary of $R_2^{\prime}$ and $\gamma(s_3^{\prime})$ is on the right of the right unstable boundary of $R_2^{\prime}$. By Lemma \ref{lem:GoodChartCurveRectangle}, it implies the existence of a crossing $\gamma$-interval ending before $s_3^{\prime}$.
	
	\begin{claim}
		There exist two $us$-rectangles $R_{i_{n+2}},R_{i_{n+3}} \in \mathcal{R}$ that display the following property. For any folding $\gamma$-interval $[s,t] \subset [s'_1,t_c]$, we have $\gamma([s,t])\subset R_{i_{n+2}} \cup R_{i_{n+3}}$.
		\label{claim:Folding:ThreeDisjoint}
	\end{claim}
	
	\begin{proof}
		Let us first prove that for any left folding $\gamma$-interval $[s,t] \subset [s'_1,t_c]$, we have $\gamma([s,t]) \subset R'_1 \cup R_{i_n} \cup R_{i_{n+1}}$. Let $[s,t]$ be such an interval.
		
		Suppose first that $s'_1 < s < s'_2$. This implies that $[s,t]$ is associated to a $us$-rectangle $R \in \mathcal{R}$ that intersects $R'_1$. Define $R_1^{\star}$ to be the rectangle intersecting $R_1^{\prime}$ displaying the following properties.
		\begin{itemize}[label={--}]
			\item There exists a left folding $\gamma$-interval $[s_{\star},t_{\star}]$ associated to $R_1^{\star}$ such that $s_{\star}>t_{n-1}$.
			\item Let any rectangle $R$ intersecting $R_1^{\star}$ and such that there exists a left folding $\gamma$-interval, starting after $t_{n-1}$, associated to $R$. Then the left unstable boundary of $R$ is on the right of the left unstable boundary of $R_1^{\star}$.
		\end{itemize}  
		Let us prove that $\gamma([s,t])\subset R_1^{\star}$. Indeed, since the left unstable boundary of $R$ is on the right of the left unstable boundary of $R_1^{\star}$, then $\gamma([s,t])$ intersects $R_1^{\star}$. Suppose now that $\gamma([s,t])$ is not contained in $R_1^{\star}$. Then there must exists a parameter $u\in(s,t)$ such that $\gamma(u)$ is on the right of the right unstable boundary of $R_1^{\star}$. But since $\gamma(s_{\star})$ is on the left of the left unstable boundary of $R_1^{\star}$ and $s_{\star}<t_c$, there exists a crossing $\gamma$-interval associated to $R_1^{\star}$ ending before $t_c$, which is a contradiction. 
		
		Letting $R_{i_{n+2}}:=R_1^{\star}$, we then proved that the image by $\gamma$ of any folding $\gamma$-interval $[s,t] \subset [s'_1,t_c]$, such that $s < s'_2$, is contained in $R_{i_{n+2}}$.
		
		Suppose now that $s'_2 \leq s < s'_3$. This implies that $[s,t]$ is associated to a $us$-rectangle $R \in \mathcal{R}$ that intersects either $R'_1$ or $R'_2$. If $R$ intersects $R'_1$, then by the same argument as before we must have $\gamma([s,t]) \subset R_{i_{n+2}}$. Suppose then that $R$ intersects $R'_2$. Note that $R'_2$ must be on the left of $R'_1$, otherwise Lemma~\ref{lem:GoodChartCurveRectangle} would give the existence of a crossing $(\gamma,R'_1)$-interval that ends before $s'_2$, but we have $s'_2 \leq s \leq t_c$, which is a contradiction. A similar argument shows that $[s'_2,t'_2]$ must be a right folding $\gamma$-interval. So if $s > s'_2$ is the starting point of a left folding $\gamma$-interval, then, again, there must exist a crossing $(\gamma,R)$-interval ending before $s$, by Lemma~\ref{lem:GoodChartCurveRectangle}, which is a contradiction. The case $s'_2 \leq s < s'_3$ is then impossible.
		
		To sum up, since $t_c < s'_3$, any left folding $\gamma$-interval $[s,t]$ such that $[s,t] \subset [s'_1,t_c]$ satisfies $\gamma([s,t]) \subset R_{i_{n+2}}$. By symmetric arguments, there exists a rectangle $R_{i_{n+3}}\in \mathcal{R}$ displaying the following property. Any right folding $\gamma$-interval $[s,t]$ such that $[s,t] \subset [s'_1,t_c]$ satisfies $\gamma([s,t]) \subset R_{i_{n+3}}$. This proves Claim~\ref{claim:Folding:ThreeDisjoint}.
	\end{proof}
	
	We then define $t_n := t_c$ and $\mathcal{R}^{\star}_{n} = \{R_{i_1},\dots,R_{i_{n+3}}\}$. Claim~\ref{claim:Folding:RankBefore} and Claim~\ref{claim:Folding:ThreeDisjoint} allow us to complete the induction at rank $n$. We stop when we reach some $N$ such that $t_N = 1$, which has to happen since the number of crossing $\gamma$-intervals is finite.
	
	\parbreak\textbf{Step 2.d: The covering family $\mathcal{R}^{\star}$ and a bound on its cardinality.}
	By Lemma~\ref{claim:OrderRectangleFolding}, the family $\mathcal{R}^{\star} = \{R_{i_1},\dots,R_{i_m}\}$ contains the image under $\gamma$ of all the folding $\gamma$-intervals. Let us now give a bound on its cardinality. We know by Lemma~\ref{lem:EstimationSizeCrossing} that the number of disjoint crossing $\gamma$-intervals is smaller than $2C$. We then have that the number of intervals $(I_k)_{k=1,\dots,m}$ from Lemma~\ref{claim:OrderRectangleFolding} is smaller than $2C$ and then:
	\begin{equation*}
		|\mathcal{R}^{\star}| \leq 4m \leq 8C.
	\end{equation*}
	This proves the third and fourth points of Proposition~\ref{prop:IntersectionCurveRectangle}. This concludes the proof of Lemma~\ref{claim:OrderRectangleFolding}.
\end{proof}

\parbreak\textbf{Step 3: The remaining part of $\gamma$.}
In this step, we conclude the proof. The remaining part of the intersection between $\gamma$ and the union of the rectangles of $\mathcal{R}$ may be composed of the beginning and the end of the curve.

If $\gamma(0) \in M \setminus \cup_i R_i$, we define $s_1 = 0$. If $\gamma(0) \in R_i$ for some $i \in \{1,\dots,L\}$, we define $s_1 \in [0,1]$ to be the first $s \in [0,1]$ such that $\gamma(s) \in \partial^u R_i$. If such an $s$ does not exist, we let $s_1 = 1$. Note that if $s_1 = 1$, Proposition~\ref{prop:IntersectionCurveRectangle} holds trivially.

If $\gamma(1) \in M \setminus \cup_i R_i$, we define $s_2 = 1$. If $\gamma(1) \in R_j$ for some $j \in \{1,\dots,L\}$, we define $s_2$ to be the last $s \in [0,1]$ such that $\gamma(s) \in \partial^u R_j$. If such an $s$ does not exist, we let $s_2 = 0$. Again, note that in this case Proposition~\ref{prop:IntersectionCurveRectangle} holds trivially.

Finally, by~\eqref{eq:HypothesisSecCurveInterRectangles}, if $y \in \gamma([0,1]) \cap \cup_i R_i$ and $y \notin \gamma([0,s_1]) \cup \gamma([s_2,1])$, then there exists an interval $[s,t] \subset [0,1]$ with $y \in \gamma([s,t])$ such that $\gamma(s),\gamma(t) \in \partial^u R_i$ and $\gamma((s,t)) \in R_i$. Thus $y$ is contained either in a crossing or a folding $\gamma$-interval. This concludes the proof of Proposition~\ref{prop:IntersectionCurveRectangle}.

\section{Homoclinic classes for endomorphisms}\label{sec:HomClasses}
Throughout this section, $f:M\rightarrow M$ is a $\cC^r$ local diffeomorphism, for $r>1$.

\subsection{Homoclinic classes of hyperbolic periodic orbits and hyperbolic measures of saddle type}
Let $\hO \subset M_f$ be a periodic orbit of period $p$. We say that $\hO$ is hyperbolic if $d_{\hx}\hf^p$ is a hyperbolic matrix for some $\hx \in \hO$. By Theorem \ref{thm:Localu/sManifolds}, any $\hx \in \hO$ admits local stable and unstable manifolds $\Ws_{loc}(\hx),\Wu_{loc}(\hx) \subset M$ such that:
\begin{equation*}
	f(\Ws_{loc}(\hx)) \subset \Ws_{loc}(\hf(\hx)) \ \text{and} \ \Wu_{loc}(\hx) \subset f(\Wu_{loc}(\hf\inv(\hx))).
\end{equation*}
We can define the global stable and unstable manifolds:
\begin{equation*}
	\Ws(\hx) = \bigcup_{n \geq 0} f^{-n}\big(\Ws_{loc}(\hf^n(\hx))\big) \ \text{and} \ \Wu(\hx) = \bigcup_{n\geq 0} f^n\big(\Wu_{loc}(\hf^{-n}(\hx))\big).
\end{equation*}
We also write:
\begin{equation*}
	\Ws(\hO) = \bigcup_{i=1}^p \Ws(\hf^i(\hx)) \ \text{and} \ \Wu(\hO) = \bigcup_{i=1}^p \Wu(\hf^i(\hx))
\end{equation*}
for the stable and unstable manifolds of the whole orbit. Note that for any periodic orbit $\cO \subset M$, there exists a unique periodic orbit $\hO \subset M_f$ that lifts $\cO$. This justifies the notation $W^{s/u}(\cO)$.\\
As we saw in Section \ref{section:PesinTheory}, we can also define stable and unstable manifolds given an ergodic hyperbolic saddle $f$-invariant measure $\mu$. Let $\hmu := \pi\inv_{\star}\mu$. For $\hmu$-almost every $\hx$, there also exist local stable and unstable manifolds $\Ws_{loc}(\hx),\Wu_{loc}(\hx) \subset M$ that are invariant. Recall that $\hat{Y}^{\#}$ is contained in the set of $\hx \subset M_f$ that admits sequences of forward iterates $\hf^{n_k}(\hx)$ and backward iterates $\hf^{-m_k}(\hx)$ in a common Pesin block, both converging to $\hx$. Recall also that $\hmu(\hat{Y}^{\#}) = 1$. We can then define global stable and unstable manifolds, as for periodic orbits. Let $\hx \in \hat{Y}^{\#}$ and let $\hf^{n_k}(\hx), \hf^{-m_k}(\hx)$ be the sequences of forward and backward iterates in a common Pesin block converging to $\hx$. We then define:
\begin{equation*}
	\Ws(\hx) = \bigcup_{k\geq0}f^{-n_k}\big(\Ws_{loc}(\hf^{n_k}(\hx))\big) \ \text{and} \ \Wu(\hx) = \bigcup_{k\geq0}f^{m_k}\big(\Wu_{loc}(\hf^{-m_k}(\hx))\big).
\end{equation*}

We also want to define the analogous of these sets in the natural extension. First, recall that for any $\hx = (x_n)_{n \in \bbZ} \in M_f$, since $f$ is a local diffeomorphism, we can define its inverse on branches: $f\inv_{x_n}$ is the unique map such that $f\inv_{x_n}\circ f = \id$ in a neighborhood of $x_{n-1}$ for $n\leq0$.

\begin{defn}[Invariant sets in $M_f$.]
	The local stable set of $\hx \in \hat{Y}^{\#}$ is defined by:
	\begin{equation*}
		V^s_{loc}(\hx) = \{\hy, \ \pi(\hy) \in \Ws_{loc}(\hx)\}
	\end{equation*}
	and its local unstable set is defined by:
	\begin{equation*}
		V^u_{loc}(\hx) = \{\hy = (y_n)_{n\in \bbZ}, \ y_0 \in \Wu_{loc}(\hx), \ y_{n} = f\inv_{x_n}(y_{n+1}) \ \forall n<0\} = \{\hy, \ y_n \in \Wu_{loc}(\hf^{-n}(\hy)) \ \forall n<0\}.
	\end{equation*}
	As for the classical invariant manifolds, we can then define the global invariant manifolds:
	\begin{equation*}
		V^s(\hx) = \bigcup_{k\geq 0} \hf^{-n_k}\big(V^s_{loc}(\hf^{n_k}(\hx))\big) \ \text{and} \ V^u(\hx) = \bigcup_{k\geq 0} \hf^{m_k}\big(V^u_{loc}(\hf^{-m_k}(\hx))\big)
	\end{equation*}
	where $\hf^{n_k}(\hx),\hf^{-m_k}(\hx)$ are sequences of forward and backward iterates in the same Pesin block, both converging to $\hx$. We define it in the same way for hyperbolic periodic orbits.
	\label{def:LocalInvariantSet}
\end{defn}

Recall that if $U,V \subset M$ are two submanifolds, we denote their transverse intersection by $U\pitchfork V$. We are now able to define the homoclinic relation.

\begin{defn}[Homoclinic relation]
	Let $\hmu_1,\hmu_2$ be two ergodic hyperbolic saddle $\hf$-invariant measures. We write $\hmu_1\preceq\hmu_2$ if there exist $A_1,A_2 \subset M_f$ such that:
	\begin{itemize}[label={--}]
		\item $A_i$ is $\hmu_i$-measurable and $\hmu_i(A_i)>0$ for $i \in \{1,2\}$.
		\item For any $(\hx,\hy) \in A_1\times A_2$, we have $V^u(\hx)\pitchfork V^s(\hy) \ne \emptyset$.
	\end{itemize}
	We write $\hmu_1\simh\hmu_2$ if $\hmu_1\preceq \hmu_2$ and $\hmu_2 \preceq \hmu_1$. In this case, we say that $\hmu_1$ and $\hmu_2$ are homoclinically related.
	
	If $\hO\subset M_f$ is a hyperbolic periodic orbit, we write $\hmu_1 \simh \hO$ if $\hmu_1\simh\hmu_{\hat{\cO}}$, where $\hmu_{\hat{\cO}}$ is the Dirac measure on the orbit $\hat{\cO}$. This is equivalent to requiring that:
	\begin{equation*}
		V^u(\hx) \pitchfork V^s(\hO) \ne \emptyset \ \text{and} \ V^s(\hx)\pitchfork V^u(\hO)\ne \emptyset
	\end{equation*}
	for $\hmu_1$-almost every $\hx$. We still say that $\hmu_1$ is homoclinically related to $\hat{\cO}$.
	
	If $\hO_1,\hO_2 \subset M_f$ are two hyperbolic periodic orbits, we write $\hO_1\simh\hO_2$ if $\hmu_{\hat{\cO}_1}\simh\hmu_{\hat{\cO}_2}$. This is equivalent to requiring that:
	\begin{equation*}
		\Wu(\hO_1) \pitchfork \Ws(\hO_2) \ne \emptyset \ \text{and} \ \Ws(\hO_1) \pitchfork \Wu(\hO_2).
	\end{equation*}
	In this case, we also say that $\hO_1$ and $\hO_2$ are homoclinically related.
\end{defn}
Note that by graph transform arguments, which are local, the classical Katok Theorem holds also for hyperbolic measures of saddle type which are invariant by a local diffeomorphism.
\begin{prop}
	The homoclinic relation is an equivalence relation on the set of ergodic hyperbolic saddle $\hf$-invariant measures. In particular, if $\hmu_1,\hmu_2$ are two ergodic hyperbolic saddle $\hf$-invariant measures and $\hO$ is a hyperbolic saddle periodic orbit such that $\hmu_1 \simh \hO$ and $\hmu_2\simh \hO$, then $\hmu_1\simh \hmu_2$. The homoclinic class of a measure $\hmu$ is then defined as the set of all measures homoclinically related to $\hmu$. Moreover, for any hyperbolic $\hf$-invariant measure $\hmu$ of saddle type, there exists a hyperbolic saddle periodic orbit $\hat{\cO}$ such that $\hmu(\{\hx\in M_f, \ \hx\simh\hat{\cO}\})=1$.
	\label{prop:HomRelation}
\end{prop}

The proof of this proposition can be found in \cite[Proposition 4.3]{lima2024measures}. It is very similar to the case of diffeomorphisms, which can be found in \cite[Proposition 2.11]{buzzi2022measures}.
The following lemma is a version for endomorphisms of the well-known $\lambda$-lemma.

\begin{lemma}
	Let $\hy \in \hat{Y}^{\#}$ and let $\Delta \subset M$ be a disc of the same dimension as $\Ws(\hy)$. Let also $D\subset \Ws(\hy)$ be any disc. If $\Delta$ intersects transversely $\Wu(\hy)$, then there exists a sequence of discs $\Delta_k \subset \Delta$ and times $n_k \rightarrow +\infty$ such that $f^{-n_k}(\Delta_k)$ are discs that converge to $D$ for the $\cC^1$ topology.
	
	A similar statement holds for forward iterates $f^{n_k}(\Delta_k)$ of discs transverse to $\Ws(\hy)$.
	\label{lem:LambdaLemma}
\end{lemma}

Since this is a local statement, once a point in the natural extension is chosen, the proof proceeds by graph transform in the Pesin charts, and is thus the same as in the case of diffeomorphisms. Complete proofs can be found in \cite[Appendix]{lima2024measures} in the case of endomorphisms.

\subsection{Measures of maximal entropy in the same homoclinic class}
Using coding techniques that go back to Sarig \cite{sarig2013symbolic} and Buzzi-Crovisier-Sarig \cite{buzzi2022measures}, Lima, Obata, and Poletti proved in \cite[Theorem 4.1]{lima2024measures} that each homoclinic class carries at most one measure of maximal entropy. We will use this later to show the finiteness of ergodic measures of maximal entropy by showing the finiteness of homoclinic classes.

\begin{thm}[\cite{lima2024measures}]
	Let $\hmu_1,\hmu_2$ be two ergodic hyperbolic saddle $\hf$-invariant measures of maximal entropy. If $\hmu_1\simh\hmu_2$, then $\hmu_1=\hmu_2$.
	\label{thm:CodingMme}
\end{thm}

\section{Entropy, Yomdin theory and dynamical Sard lemma}\label{sec:EntropyYomdin}
Throughout this section, $f:M\rightarrow M$ is a $\cC^r$ local diffeomorphism, for $r>1$.

\subsection{Generalities on entropy}
We recall some well-known facts about measured and topological entropy. In this paragraph, $g$ will be either $f$ or $\hf$, or some iterate of these maps, and $X$ will be either $M$ or $M_f$. Given an integer $n$ and $\epsilon >0$, the $(g,n,\epsilon)$-Bowen ball at $x\in X$, or $(n,\epsilon)$-Bowen ball when the choice of $g$ is unambiguous, is the following set:
\begin{equation*}
	B_g(x,n,\epsilon) = \{y \in X, \ d\big(g^i(x),g^i(y)\big) < \epsilon, \ \forall \ 0\leq i\leq n-1\}.
\end{equation*}
Consider $Y\subset X$. A set $F\subset X$ is called $(g,n,\epsilon)$-spanning if:
\begin{equation*}
	Y\subset \bigcup_{x \in F} B_g(x,n,\epsilon).
\end{equation*}
Define $r_g(n,\epsilon,Y)$ to be the minimal cardinality of a $(g,n,\epsilon)$-spanning set of $Y$. When $Y=X$, we just write $r_g(n,\epsilon)$.

\parbreak\textbf{Topological entropy of a subset.} The topological entropy of $g$ on a set $Y\subset X$ is defined by:
\begin{equation}
	h_{top}(g,Y) = \lim_{\epsilon\rightarrow 0}h_{top}(g,Y,\epsilon) \ \text{where} \ h_{top}(g,Y,\epsilon) = \limsup_{n\rightarrow + \infty} \frac{1}{n}\log r_g(n,\epsilon,Y).
	\label{eq:TopologicalEntropy}
\end{equation}

\parbreak\textbf{Measured entropy and Katok formula.} Let $\nu$ be an ergodic $g$-invariant measure. For $0<\lambda<1$, an integer $n$ and $\epsilon>0$, let:
\begin{equation*}
	r_g(n,\epsilon,\nu,\lambda) = \inf\{r_g(n,\epsilon,Y), \ Y\subset X \ \text{measurable s.t.} \ \nu(Y)>\lambda\}.
\end{equation*}
For any $0<\lambda<1$, the measured entropy of $\nu$ satisfies:
\begin{equation}
	h(g,\nu) = \lim_{\epsilon \rightarrow 0} h(g,\nu,\epsilon) \ \text{where} \ h(g,\nu,\epsilon) = \limsup_{n\rightarrow + \infty} \frac{1}{n}\log r_g(n,\epsilon,\nu,\lambda).
	\label{eq:KatokFormula}
\end{equation}
In the case where $\nu$ is not ergodic, we proceed as follows. Let $\nu:=\int_M \nu_x \ d\nu(x)$ be an ergodic decomposition of $\nu$. We define $h(g,\nu)=\int_M h(g,\nu_x) \ d\nu(x)$.
 
\parbreak\textbf{Tail entropy.}
The tail entropy of $g$ is defined by:
\begin{equation}
	h^{\star}(f) = \lim_{\epsilon\rightarrow 0} h^{\star}(g,\epsilon) \ \text{where} \ h^{\star}(g,\epsilon) = \sup_{x\in X} h_{top}(g,\{y\in X, \ \forall n\geq 0 \ d\big(f^n(x),f^n(y)\big) < \epsilon\}).
	\label{eq:TailEntropy}
\end{equation}

The relevance of this concept lies in the following estimate. Let $\nu$ be an ergodic $g$-invariant measure. We have:
\begin{equation}
	h(g,\nu) \leq h(g,\nu,\epsilon) + h^{\star}(f,\epsilon).
	\label{eq:TailEntropyEstimate1}
\end{equation}
Secondly, if $(\nu_n)_n$ is a sequence of ergodic $g$-invariant measures weak-$\star$ converging to $\nu$, we have the following upper semicontinuity property:
\begin{equation}
	\limsup_{n\rightarrow + \infty} h(g,\nu_n) \leq h(g,\nu) + h^{\star}(g).
	\label{eq:TailEntropyEstimate2}
\end{equation}

Recall that $\lambda(g) := \lim_{n\rightarrow +\infty}\frac{1}{n}\log\norm{dg^n}$. The following control on the tail entropy holds:
\begin{equation}
	h^{\star}(g)\leq \frac{\lambda(g)}{r}.
	\label{eq:TailEntropyEstimate3}
\end{equation}
In the special case where $g\in\cC^{\infty}$, we then have $h^{\star}(g)=0$.

\subsection{Unstable entropy}\label{subsec:UnstableEntropy}
Let $\mu$ be an ergodic hyperbolic saddle $f$-invariant measure. Let $\hmu := \pi\inv_{\star}\mu$. We recall that this means that:
\begin{equation*}
	\lambda^s(f,\mu) < 0 < \lambda^u(f,\mu)
\end{equation*}
where $\lambda^{s/u}(f,\mu)$ are the two Lyapunov exponents of the measure $\mu$. The Pesin unstable manifold theorem (Theorem \ref{thm:Localu/sManifolds}) asserts that $\hmu$-almost every point $\hx$ belongs to an unstable set $V^u(\hx)$ (see Section \ref{sec:HomClasses} for a definition). This set is contained in the connected component of $M_f$ containing $\hx$ and is an injectively immersed $\cC^{r}$ curve characterized by:
\begin{equation*}
	V^u(\hx) = \{\hy \in M_f, \ \limsup_{n\rightarrow + \infty}\log d\big(\hf^{-n}(\hx),\hf^{-n}(\hy)\big)< 0\}.
\end{equation*}
A measurable partition $\xi$ is subordinated to the unstable lamination $V^u$ if for $\hmu$-almost every $\hx$, the atom $\xi(\hx)$ is a neighborhood of $\hx$ inside $V^u(\hx)$ and $\xi$ is increasing: every atom of $\hf(\xi)$ is a union of atoms of $\xi$. Qian and Zhu proved the existence of such partitions for endomorphisms in \cite[Proposition 3.2]{qian2002srb}. Their work was inspired by the famous Ledrappier-Young theory \cite{ledrappier1985metric1,ledrappier1985metric2}.

Since $\xi$ is measurable, Rokhlin's disintegration theorem applies and then, for $\hmu$-almost every $\hx$, there exists a probability measure $\hmu_{\hx}^u$ supported on $\xi(\hx)$ such that:
\begin{equation}
	\hmu = \int \hmu_{\hx}^u \ d\hmu(\hx).
	\label{eq:ConditionalMeasures}
\end{equation}
Such a family $\{\hmu_{\hx}^u\}_{\hx}$ is not unique but, given $\xi$, any two families displaying these properties are equal outside a set of $\hx$ of $\hmu$-measure zero. We then call $\hmu^u_{\hx}$ the conditional measure on $\xi(\hx)$.

One of the main consequences of the Ledrappier-Young theory is that the entropy along these unstable conditional measures is equal to the entropy of the measure. In the case of endomorphisms, this was proved by Qian, Xie, and Zhu in \cite[Proposition IX 2.14]{qian2009PesinTheoryEndo}.

\begin{thm}
	Let $\hmu$ be an ergodic hyperbolic saddle $\hf$-invariant measure. Let $\{\hmu^u_{\hx}\}_{\hx}$ be a system of conditional measures on local unstable sets. Then, for $\hmu$-almost every $\hx$, we have:
	\begin{equation*}
		h(\hf,\hmu) = \inf_{\lambda>0}\lim_{\epsilon \rightarrow 0}\liminf_{n\rightarrow +\infty}\frac{1}{n}\log r_{\hf}(n,\epsilon,\hmu^u_{\hx},\lambda).
	\end{equation*}
	\label{thm:UnstableEntropy}
\end{thm}
Note that the definition of $r_{\hf}$, mentioned in Theorem \ref{thm:UnstableEntropy}, extend to non-invariant measures.

The next lemma states that we can also compute the unstable entropy in the manifold, by looking at the projection of $\hmu_{\hx}^u$ which is supported on the local unstable manifold of $\hx$. It is an easy corollary once we know that $\pi:V^u_{loc}(\hx) \rightarrow \Wu_{loc}(\hx)$ is a true conjugacy between $\hf$ and $f$. This last statement is straightforward from the definition of the local unstable sets \ref{def:LocalInvariantSet}.

\begin{lemma}
	Let $\hmu$ be an ergodic hyperbolic saddle $\hf$-invariant measure. Let $\{\hmu^u_{\hx}\}_{\hx}$ be a system of conditional measures on local unstable sets. For any $\hx$, define $\mu_{\hx}^u := \pi_{\star}\hmu^u_{\hx}$. For $\hmu$-almost every $\hx$, we then have:
	\begin{equation*}
		h(\hf,\hmu) = \inf_{\lambda>0}\lim_{\epsilon\rightarrow 0}\liminf_{n\rightarrow +\infty} \frac{1}{n}\log r_{f}(n,\epsilon,\mu^u_{\hx},\lambda).
	\end{equation*}
	\label{lem:UnstableEntropy}
\end{lemma}

Recall also the following version of the Brin-Katok theorem, proved in \cite{ledrappier1985metric2} in the case of diffeomorphisms and adapted in \cite{qian2009PesinTheoryEndo} for endomorphisms.

\begin{prop}
	Let $\hmu$ be an ergodic hyperbolic saddle $\hf$-invariant measure. Let $\{\hmu^u_{\hx}\}_{\hx}$ be a system of conditional measures on local unstable sets. For any $\hx$, define $\mu_{\hx}^u := \pi_{\star}\hmu^u_{\hx}$. Then, for $\hmu$-almost every $\hx$, we have:
	\begin{equation*}
		h(\hf,\hmu) = \lim_{\epsilon \rightarrow 0} \liminf_{n\rightarrow +\infty} -\frac{1}{n}\log \mu^u_{\hx}\big(B_f(\pi(\hx),n,\epsilon)\big).
	\end{equation*}
	\label{prop:LocalBrinKatok}
\end{prop}

\subsection{Yomdin theory and reparametrizations}
Let $r>1$ and $\epsilon >0$. Recall the definition of an $(r,\epsilon)$-curve from Definition \ref{def:Curve}. Suppose that $\gamma:[0,1] \rightarrow M$ is a curve. If we cut $[0,1]$ into small intervals $[a_i,b_i]$ and reparametrize $\gamma$ by $\gamma\circ\psi_i$ where $\psi_i : [a_i,b_i] \rightarrow [0,1]$ is affine, then $\norm{\gamma\circ\psi_i}_{\cC^r} < \kappa\norm{\gamma}_{\cC^r}$ where $\kappa = |a_i-b_i|$. Thus, by cutting into sufficiently small pieces, we can obtain a covering by pieces with affine reparametrizations of $\cC^r$ size as small as we want. Yomdin measured the complexity of a curve by counting how many reparametrized pieces with $\cC^r$ size less than $\epsilon$ are needed to cover it.

\begin{defn}[Reparametrization]
	Let $\gamma : [0,1]\rightarrow M$ be a $\cC^r$ curve. A reparametrization of $\gamma$ is a non-constant affine map $\psi:[0,1] \rightarrow [0,1]$.
	
	A family of reparametrizations of $\gamma$ over a set $T\subset [0,1]$ is a collection $\mathscr{R}$ of reparametrizations such that:
	\begin{equation*}
		T \subset \bigcup_{\psi \in \mathscr{R}} \psi([0,1]).
	\end{equation*}
	\label{def:Reparametrization}
\end{defn}

\begin{defn}[Admissible reparametrizations]
	Let $\epsilon>0$, $N\geq 1$ and $T\subset[0,1]$. Let $\gamma:[0,1] \rightarrow M$ be a $\cC^r$ curve. A reparametrization $\psi$ of $\gamma$ is $(f,r,N,\epsilon)$-admissible up to time $n$ if there exists an increasing sequence of integers $(n_0,n_1,...,n_l)$ such that:
	\begin{itemize}[label={--}]
		\item $n_0=0$, $n_l=n$ and $n_j-n_{j-1} \leq N$ for $j\in \{1,...,l\}$,
		\item the curve $f^{n_j} \circ \gamma \circ \psi$ is an $(r,\epsilon)$-curve for each $j \in \{0,...,l\}$.
	\end{itemize}
	We call the integers $n_j$ the admissible times.
	
	A family of reparametrizations $\mathscr{R}$ of $\gamma$ over $T$ is $(f,r,N,\epsilon)$-admissible up to time $n$ if each $\psi \in \mathscr{R}$ is $(f,r,N,\epsilon)$-admissible up to time $n$.
	\label{def:AdmissibleReparametrization}
\end{defn}

Let us now cite Yomdin's theorem. It asserts that covers by Bowen balls generate reparametrizations with cardinality of the same order of magnitude.

\begin{thm}[Yomdin]
	Given a real number $2\leq r< \infty$, there exist $\Upsilon := \Upsilon(r)$ and $\epsilon_{\Upsilon} := \epsilon_{\Upsilon}(r)$ with the following properties. For any:
	\begin{itemize}[label={--}]
		\item regular $(r,\epsilon)$-curve $\gamma$ with $0<\epsilon<\epsilon_{\Upsilon}$,
		\item $x \in \gamma([0,1])$ and $T=\{t \in [0,1], \ f\circ\gamma(t) \in B(f(x),\epsilon)\}$,
	\end{itemize}
	there exists a family of reparametrizations $\mathscr{R}$ of $\gamma$ over $T$ such that:
	\begin{enumerate}
		\item for any $\psi \in \mathscr{R}$, $f\circ\gamma\circ\psi$ is a $(r,\epsilon)$-curve.
		\item $|\mathscr{R}|\leq \Upsilon\norm{df}^{1/r}$.
	\end{enumerate}
	\label{thm:Yomdin}
\end{thm}

\begin{cor}
	Let $2\leq r < \infty$ and let $\Upsilon := \Upsilon(r)$ and $\epsilon_{\Upsilon}:=\epsilon_{\Upsilon}(r)$ be the constants given by Theorem \ref{thm:Yomdin}. Suppose that:
	\begin{itemize}[label={--}]
		\item $\gamma : [0,1] \rightarrow M$ is a regular $(r,\epsilon)$-curve with $0<\epsilon<\epsilon_{\Upsilon}$,
		\item $N,n\geq 1$ and $T\subset [0,1]$.
	\end{itemize}
	Then there exists a family $\mathscr{R}_n$ of reparametrizations over the set $T$ which is $(f,r,N,\epsilon)$-admissible up to time $n$ such that:
	\begin{equation*}
		|\mathscr{R}_n|\leq \Upsilon^{\lceil n/N \rceil}\norm{df}^{n/r}\norm{df}^{N/r}r_f\big(n,\epsilon,f\circ\gamma(T)\big).
	\end{equation*}
	\label{cor:AdmissibleReparam}
\end{cor}

Yomdin theory combined with the characterization of entropy along unstable manifolds, Lemma \ref{lem:UnstableEntropy}, allows us to compute the entropy in terms of reparametrizations. Again, for a more precise proof in the case of diffeomorphisms, see \cite[Corollary 4.16]{buzzi2022continuity}.

\begin{prop}
	Let $\mu$ be an ergodic $f$-invariant hyperbolic saddle measure and let $\hmu := \pi_{\star}\inv\mu$. Let $\{\hmu^u_{\hx}\}$ be any system of conditional measures along unstable manifolds and let $\mu^u_{\hx} = \pi_{\star}\hmu^u_{\hx}$. The following holds for any $\hat{X}\subset M_f$ of positive $\hmu$-measure, for $\hmu$-almost every $\hx \in \hat{X}$ and for any choice of:
	\begin{itemize}[label={--}]
		\item $\gamma : [0,1] \rightarrow \Wu_{loc}(\hx)$ a $\cC^r$ curve with $r>2$,
		\item $T\subset [0,1]$ such that $\mu_{\hx}^u\big(\gamma(T)\big)>0$.
	\end{itemize}
	If $(\mathscr{R}_n)$ is a family of reparametrizations of $\gamma$ which are $(f,N,\epsilon)$-admissible up to time $n$ over the set $T$ for some $\epsilon>0$ and $N\geq1$ independent of $n$, then:
	\begin{equation*}
		h(f,\mu,\epsilon) \leq \lim_{n\rightarrow +\infty}\frac{1}{n}\log\#\mathscr{R}_n.
	\end{equation*}
	\label{prop:EntropyReparametrizations}
\end{prop}

\subsection{Dynamical laminations and Sard lemma}\label{subsec:Lamination}
\subsubsection{Generalities about laminations}
\textbf{Laminations.} Fix some $r>1$. Let us first define what we mean by a continuous lamination with $\cC^r$ leaves. Let $K\subset M$. A $\cC^r$ lamination $\mathscr{L}$ of $K$ is an atlas of coordinates such that for any $x \in K$, there exists $U\cap K$ a neighborhood of $x$ in $M$ homeomorphic to $\bbR \times T$, where $T$ is some given topological space. The image of any $\bbR\times\{t\}$ is an embedded $\cC^r$ curve which is called a plaque of the lamination. The neighborhood $U$ is called a lamination neighborhood of $x$. We also require that the composition of two intersecting change of coordinates is $\cC^r$ and that it preserves the plaques. The set of the maximal union of intersecting plaques forms a partition of $K$ such that the atoms are injectively immersed curves. For $x\in K$, we denote by $\mathscr{L}(x)$ the atom of the lamination containing $x$ and we call each of these atoms a leaf of the lamination.

\parbreak\textbf{Transversals.} A transversal to a continuous lamination with $\cC^r$ leaves $\mathscr{L}$ of a set $K$ is a $\cC^1$ embedded one-dimensional submanifold $\tau\subset M$ that intersects $K$ such that for every $x \in \tau \cap K$ and every plaque $L$ through $x$, $\tau$ is transverse to $L$. Denote by $\text{dim}_H$ the Hausdorff dimension.

\begin{defn}[Transverse dimension of a lamination]
	The transverse dimension of a continuous lamination $\mathscr{L}$ of a set $K$ is defined by:
	\begin{equation*}
		\text{\dj}(\mathscr{L}) = \sup \{\text{dim}_H(\tau\cap K), \ \tau \ \text{is a transversal to} \ \mathscr{L}\}.
	\end{equation*}
	\label{def:TransverseDimLamination}
\end{defn}

\parbreak\textbf{Holonomies.} Consider a continuous lamination $\mathscr{L}$ of a set $K$. Consider a point $x_0$, a lamination neighborhood $U_0$ of $x_0$ and a transversal $\tau_0$ containing $x_0$. Up to reducing $\tau_0$, the plaque $L_y$ at every $y$ sufficiently close to $x_0$ intersects $\tau_0$ in exactly one point. There exists then an open set $V$ containing $x_0$ such that for any transversal $\tau$ sufficiently uniformly $\cC^1$ close to $\tau_0$, the map $H_{\tau} : V\cap K \rightarrow \tau$ is well-defined as the unique intersection between a plaque through $y$ and $\tau$.

We say that $\mathscr{L}$ has Lipschitz holonomies if, given $x_0,U_0$ and $\tau_0$, there exists a constant $L>0$, a uniformly $\cC^1$ neighborhood $\mathcal{T}$ of $\tau_0$ and a neighborhood $V$ of $x_0$ with the following property. For any transversals $\tau_1,\tau_2 \in \mathcal{T}$, the holonomy projection map $H_{\tau_1\rightarrow \tau_2} : V\cap K\cap \tau_1 \rightarrow \tau_2$ has Lipschitz constant less than or equal to $L$. Here $H_{\tau_1\rightarrow \tau_2} : V\cap K\cap \tau_1 \rightarrow \tau_2$ denotes the restriction of $H_{\tau_2}$ to $\tau_1$.

\subsubsection{Dynamical laminations}
The following theorem gives a bound on the transverse dimension of the stable lamination of an ergodic hyperbolic saddle measure, which depends on the entropy and its positive exponent. For an ergodic hyperbolic $f$-invariant measure $\mu$, recall that $\lambda^u(\mu)$ denotes its largest Lyapunov exponent, which is positive.

\begin{thm}
	Let $\mu$ be an ergodic, hyperbolic saddle $f$-invariant measure. Let $\hmu := \pi\inv_{\star}\mu$. The following holds for $\hmu$-almost every $\hx$. Fix any $\delta>0$. There exists a measurable set $\hat{\Lambda}_{\hx} := \hat{\Lambda}_{\hx}(\delta)$ displaying following properties.
	\begin{enumerate}
		\item $\Ws_{loc}(\hat{\Lambda}_{\hx}):=\bigcup_{\hy\in\hat{\Lambda}_{\hx}} \Ws_{loc}(\hy)$ defines a continuous lamination with $\cC^{r}$ leaves which has Lipschitz holonomies.
		\item The following dimension estimate holds:
		\begin{equation*}
			\text{dim}_H\big(\Wu_{loc}(\hx)\cap\Lambda_{\hx}\big) \geq h(f,\mu)/\lambda^u(\mu) - \delta
		\end{equation*}
		where $\Lambda_{\hx} := \pi(\hat{\Lambda}_{\hx})$.
		\item In the case where $\text{dim}_H\big(\Wu_{loc}(\hx)\cap\Lambda_{\hx}\big)>0$, the point $\pi(\hx)$ is accumulated by points of $\Lambda_{\hx}$ on both sides of $\Wu_{loc}(\hx)$ belonging to the support of $\hmu$ restricted to some common Pesin block.
	\end{enumerate}
	\label{thm:LipschitzHolonomiesLamination}
\end{thm}

We emphasize that $\Ws_{loc}(\hx)$ and $\Wu_{loc}(\hx)$ are the $\cC^r$ embedded curves defined by Definition \ref{def:LocalInvManifolds}. This theorem is a combination of \cite{young1982dimension} and \cite[Theorem 4.1]{buzzi2022measures}. Let us give the main lines of the proof for the sake of completeness.

\begin{proof}
	Recall the definition of the local invariant sets from Definition \ref{def:LocalInvariantSet}. Recall also the definition of the conditional measures \eqref{eq:ConditionalMeasures} and let $\{\hmu_{\hx}^u\}_{\hx}$ be a system of conditional measures subordinate to the unstable lamination of $\hmu$. Let us choose $\hx$ such that Oseledets' theorem and Proposition \ref{prop:LocalBrinKatok} are satisfied. Note that this is a set of full $\hmu$-measure. Fix $\delta>0$. We are first going to construct a set $\hat{\Lambda}_{\hx}\subset M_f$ of positive $\hmu_{\hx}^u$-measure, contained in a Pesin block, such that for any $\delta,\epsilon>0$ small enough and any $\hy \in \hat{\Lambda}_{\hx}$ we have:
	\begin{equation*}
		\liminf_{\rho\rightarrow 0} \frac{\log\mu^u_{\hx}\big(B(y,\rho)\big)}{\log\rho} \geq h(f,\mu)/\lambda^u(\mu) - \delta
	\end{equation*}
	where $y := \pi(\hy)$ and $\mu_{\hx}:=\pi_{\star}\hmu^u_{\hx}$. This follows the same lines as \cite[Section 3]{young1982dimension}.
	
	Choose $\hat{\Lambda}_{\hx}$ such that for any $\hy \in \hat{\Lambda}_{\hx}$ it satisfies the following properties.
	\begin{itemize}[label={--}]
		\item There exists a positive constant $\rho_0$ such that $\lim_{n\rightarrow +\infty}-\frac{1}{n} \log \mu^u_{\hx}\Big(B_f(y,n,\rho_0)\Big) \geq h(f,\mu) - \delta_1$, for $\delta_1>0$ small and where $y := \pi(\hy)$.
		\item $\hy$ is in the support of $\hmu$ restricted to some fixed Pesin block of positive $\hmu$-measure.
		\item The map $\hy\mapsto\Ws_{loc}(\hy)$ is continuous for the $\cC^1$ topology.
	\end{itemize}
	The first item is possible using Proposition \ref{prop:LocalBrinKatok} and taking $\hy$ in a set of positive $\hmu_{\hx}^u$-measure. The second item is possible by restricting to a Pesin block of positive $\hmu^u_{\hx}$-measure. Now using the classical (and more precise, see \cite[Chapter S]{katok1995introduction}) version of Theorem \ref{thm:finPesinChart}, it is easy to see that for any $\delta_2>0$ small, there exists $\chi>\lambda^u(\mu)-\delta_2$ such that for all integer $n\geq 0$ and for all $\hy \in \hat{\Lambda}_{\hx}$ we have:
	\begin{equation*}
		B(y, C(\hy)\rho_0 e^{-n\chi}) \subset B_f(y,n,\rho_0)
	\end{equation*}
	for some constant $C(\hy)$. In other words, for any $\rho<<\rho_0$, there exists an integer $n(\rho)$ such that:
	\begin{itemize}[label={--}]
		\item $B(y,\rho)\subset B_f(y,n(\rho),\rho_0)$,
		\item $\lim_{\rho\rightarrow 0}\frac{\log\rho}{n(\rho)}\geq-\chi$.
	\end{itemize}
	We can now compute the following:
	\begin{equation*}
		\begin{aligned}
			\frac{\log\mu^u_{\hx}\big(B(y,\rho)\big)}{\log \rho} &\geq \frac{\log\mu^u_{\hx}\big(B_f(y,n(\rho),\rho_0)\big)}{\log\rho}\\
			&\geq -\frac{1}{n(\rho)}\log\mu^u_{\hx}\big(B_f(y,n(\rho),\rho_0)\big) \times -\frac{n(\rho)}{\log\rho}\\
			&\geq\big(h(f,\mu) - \delta_1\big)/\big(\lambda^u(\mu)-\delta_2\big).
		\end{aligned}
	\end{equation*}
	Taking the $\liminf$ as $\rho\rightarrow 0$ and taking $\delta_1,\delta_2$ small enough gives us the desired result. Now, using \cite[Proposition 2.1]{young1982dimension}, which is a purely topological result, we get that:
	\begin{equation*}
		\text{dim}_H\big(\Wu_{loc}(\hx)\cap \Lambda_{\hx}\big) \geq \frac{h(f,\mu)}{\lambda^u(\mu)}-\delta.
	\end{equation*}
	By construction, we can suppose that $\hat{\Lambda}_{\hx}$ is contained in a Pesin block. Moreover, since $f$ is $\cC^{r}$, the same proof as in \cite[Appendix A]{buzzi2022measures} shows that $\Ws_{loc}(\hat{\Lambda}_{\hx})$ defines a continuous lamination with $\cC^{r}$ leaves which have Lipschitz holonomies (their proof works for diffeomorphisms but is local, so the same exact arguments apply in our case). For the last statement, note that the points in $\Wu_{loc}(\hx)$ that are not accumulated on both sides of $\Wu_{loc}(\hx)$ constitute a set of $\hmu_{\hx}^u$-measure zero. This proves the last statement and thus concludes the proof.
\end{proof}

\subsubsection{Dynamical Sard lemma}
Fix a real number $r>1$ and let $\mathscr{L}$ be a continuous lamination with $\cC^r$ leaves on a set $K$. Let $\gamma:[0,1] \rightarrow M$ be a $\cC^r$ curve. Define the following lamination:
\begin{equation*}
	\mathscr{N}_{\gamma}(\mathscr{L}) := \big\{\mathscr{L}(\gamma(t)), \ t\in[0,1], \ \gamma(t) \cap L \ne \emptyset \ \text{for some leaf} \ L \ \text{and} \ \gamma^{\prime}(t) \in T_{\gamma(t)}L\big\}.
\end{equation*}
This is the lamination made with all leaves such that there exists at least one tangent intersection with $\gamma$. The following theorem asserts that the higher the smoothness of the leaves of $\mathcal{L}$, the lower the transverse dimension of the lamination $\mathscr{N}_{\gamma}(\mathscr{L})$. This is a topological statement which is proved in \cite[Theorem 4.2]{buzzi2022measures}.

\begin{thm}
	Let $r>1$ and let $\mathscr{L}$ be a continuous lamination with $\cC^r$ leaves. For any $\cC^r$ curve $\gamma$, we have:
	\begin{equation*}
		\text{\dj}\big(\mathscr{N}_{\gamma}(\mathscr{L})\big) \leq \frac{1}{r}.
	\end{equation*}
	\label{thm:Sard}
\end{thm}

\section{Unstable with stable intersections for saddle ergodic measures}\label{sec:UnstableStableIntersection}
The goal of this section is to prove the following propositions. Proposition \ref{prop:IntersectionsUnstableMaxEntropy} is a corollary of Proposition \ref{prop:UnstableStableIntersection}. Recall that $\lambda(f) = \lim_{n\rightarrow +\infty} \frac{1}{n}\log\norm{df^n}$. Recall the definition of the set $\hat{Y}^{\#}$ before Proposition \ref{prop:MarkovRectangles}. Recall that $\hat{Y}^{\#}$ is contained in the set of points satisfying the conclusion of Theorem \ref{thm:LipschitzHolonomiesLamination}. Recall that such an $\hx$ is accumulated by points on both sides of $\Wu(\hx)$ and the local stable manifolds of these points form a continuous lamination with $\cC^r$ leaves having Lipschitz holonomies.

\subsection{Statement of the main Propositions}
\begin{prop}[Unstable Intersections]
	Let $f:M\rightarrow M$ be a $\cC^r$, $r>1$, local diffeomorphism. Let $(\mu_k)_{k \in \bbN}$ be a sequence of ergodic hyperbolic $f$-invariant measures of saddle type such that:
	\begin{equation*}
		\mu_k \underset{k\rightarrow+\infty}{\longrightarrow} \nu
	\end{equation*}
	for the weak-$\star$ topology. Suppose that the following holds.
	\begin{itemize}[label={--}]
		\item There exists $\alpha \in (0,1]$ and two $f$-invariant measures $\nu_1,\nu_0$, possibly non-ergodic, such that $\nu = \alpha\nu_1 + (1-\alpha)\nu_0$.
		\item $\nu$-almost every ergodic component of $\nu_1$ is hyperbolic of saddle type and has positive entropy.
		\item If $\alpha<1$, we also suppose that $\nu$-almost every ergodic component $\nu_0^{\prime}$ of $\nu_0$ satisfies:
		\begin{equation*}
			h(f,\nu_0^{\prime}) < \frac{1}{1-\alpha}\big(\liminf_{k\rightarrow +\infty} h(f,\mu_k) - 2\lambda(f)/r\big).
		\end{equation*}
	\end{itemize}
	Then, there exists a finite number of points $\hy_1,\dots,\hy_L \in \hat{Y}^{\#}$ displaying the following properties. For $k$ sufficiently large, there exists a measurable set $\hat{X}_k\subset M_f$ and an integer $i(k)\in \{1,\dots,L\}$ such that the following holds.
	\begin{enumerate}[label=(UI\arabic*)]
		\item \label{property:UI1} $\hmu_k(\hat{X}_k) >0$, where $\hmu_k := \pi_{\star}^{-1} \mu_k$.
		\item For any $\hx \in \hat{X}_k$, the following intersection holds:
		\begin{equation*}
			\Wu(\hx) \cap \Ws_{loc}(\hy_{i(k)}) \ne \emptyset.
		\end{equation*}
		\label{property:IntersectionUnstable}
	\end{enumerate}
	\label{prop:UnstableStableIntersection}
\end{prop}
Note that in the case where $f$ is $\cC^{\infty}$, Proposition \ref{prop:UnstableStableIntersection} can be applied whenever $\nu$-almost every ergodic component $\nu_0^{\prime}$ of $\nu_0$ satisfies $h(f,\nu_0^{\prime}) < \frac{1}{1-\alpha}\liminf_{k\rightarrow +\infty} h(f,\mu_k)$. We emphasize that if $\alpha=1$, we have $\nu=\nu_1$ so the hypothesis on the entropy of the ergodic components of $\nu_0$ will always be satisfied. We also emphasize that the entropy hypothesis implies that $\liminf_{k\rightarrow +\infty} h(f,\mu_k) - 2\lambda(f)/r>0$. Finally, note that the integer $k$ and the points $\hy_1,\dots,\hy_L$ depend on the sequence $(\mu_k)_k$ and on the decomposition of the measure $\nu$.

The following proposition is a corollary of Proposition \ref{prop:UnstableStableIntersection}. Its principal application holds for sequences of measures such that their entropy converges to the topological entropy, and in particular for ergodic m.m.e.'s.
\begin{prop}[Homoclinic Unstable Intersections]
	Let $f:M\rightarrow M$ be a $\cC^r$, $r>1$, local diffeomorphism. Let $(\mu_k)_{k \in \bbN}$ be a sequence of ergodic hyperbolic saddle $f$-invariant measures such that:
	\begin{equation*}
		\mu_k \underset{k\rightarrow +\infty}{\longrightarrow} \nu
	\end{equation*}
	for the weak-$\star$ topology. Fix $\beta\in(0,1)$. Suppose that the following holds.
	\begin{itemize}[label={--}]
		\item There exists $\alpha \in (0,1]$ and two different $f$-invariant measures $\nu_1,\nu_0$, possibly non-ergodic, such that $\nu = \alpha\nu_1 + (1-\alpha)\nu_0$.
		\item $\nu$-almost every ergodic component of $\nu_1$ is hyperbolic of saddle type and has entropy larger than $\lambda(f)/r$. We denote by $\hnu_1:=\pi_{\star}^{-1}\nu_1$.
		\item For $\nu$-almost every ergodic component $\nu^{\prime}$ of $\nu$, we have:
		\begin{equation*}
			h(f,\nu^{\prime}) < \frac{1}{1-\alpha\beta}\big(\liminf_{k\rightarrow +\infty} h(f,\mu_k) - 2\frac{\lambda(f)}{r}\big).
		\end{equation*}
	\end{itemize}
	Then for any hyperbolic periodic saddle orbit $\hat{\cO}$ satisfying $\hnu_1(\{\hx\in M_f, \ \hx\simh\hat{\cO}\})>\beta$, the following property holds for $k$ sufficiently large. For $\hmu_k$-almost every $\hx$, the following intersection holds:
	\begin{equation*}
		\Wu(\hx) \pitchfork \Ws(\hat{\cO}) \ne \emptyset.
	\end{equation*}
	\label{prop:IntersectionsUnstableMaxEntropy}
\end{prop}
Note that in the case where $f\in \cC^{\infty}$, Proposition \ref{prop:IntersectionsUnstableMaxEntropy} can be applied whenever $\nu$-almost every ergodic component $\nu^{\prime}$ of $\nu$ satisfies $h(f,\nu^{\prime}) < \frac{1}{1-\alpha\beta}\liminf_{k\rightarrow +\infty} h(f,\mu_k)$. We emphasize that the hypothesis on the entropy of the ergodic components of $\nu$ implies that $\liminf_{k\rightarrow +\infty} h(f,\mu_k) - 2\lambda(f)/r>0$. Finally, note that the integer $k$ depends on the sequence of measures $(\mu_k)_k$, on the number $\beta$ and on the periodic orbit $\hat{\cO}$.

The rest of the section is devoted to the proof of Proposition \ref{prop:UnstableStableIntersection} and Proposition \ref{prop:IntersectionsUnstableMaxEntropy}.

\subsection{Proof of Proposition \ref{prop:UnstableStableIntersection}}
Before going into the details, let us explain first the ideas of the proof of Proposition \ref{prop:UnstableStableIntersection}. The main point is to analyze the entropic behavior of the typical orbits of the measures $\mu_k$. In Step \stepref{step:U1}, we settle the entropy and reparametrization parameters. In Step \stepref{step:U2}, we use Proposition \ref{prop:MarkovRectangles} to build the two families $\mathcal{R}$ and $\tilde{\mathcal{R}}$ of $us$-rectangles covering the measure $\nu_1$, and the neighborhood $\hat{U}_1$ of points having a good Pesin-behavior (see Section \ref{sec:Rectangles}), which has very large $\nu_1$-measure. This measure has to be seen as the saddle high-entropy part of $\nu$. Step~\stepref{step:U3} deals with $\nu_0$. It has to be seen as a small entropy part of the measure $\nu$, which can contain in its ergodic decomposition non-hyperbolic measures, zero entropy measures, expanding type measures etc. We construct $\hat{U}_0$, a neighborhood of its support, of large $\nu_0$-measure, which can be covered by a small number of Bowen-balls.

From Step \stepref{step:U4} to Step \stepref{step:U7}, we introduce the decomposition type of typical $\mu_k$-orbits. The idea is that for $k$ large enough, typical orbits are, provided they are long enough, made by three types of segments.
\begin{itemize}[label={--}]
	\item Long segments starting in $\hat{U}_1$, which do not intersect $\hat{U}_0$: the \textit{rectangle segments}. These segments have to be the ones where entropy is created.
	\item Long segments starting in $\hat{U}_0$, which do not intersect $\hat{U}_1$: the \textit{small entropy segments}. The orbit separation along these segments is small compared with the entropy.
	\item Segments of any size which do not intersect $\hat{U}_1$ nor $\hat{U}_0$: the \textit{wild segments}. These segments can see entropy creation, but their total length will be very small.
\end{itemize}
Step \stepref{step:U8} is the core of the proof. We compute the entropy of the $\mu_k$ by building families of admissible reparametrizations of the unstable local manifold of typical $\mu_k$-points. We suppose that the iterates of these curves do not cross the stable manifolds containing the stable boundaries of the $\nu_1$-rectangles. We build the reparametrizations by induction, depending on the class of the segment of the typical $\mu_k$-orbits we consider. Our goal will be to control the number of reparametrizations to show that the non-crossing hypothesis between $\mu_k$-unstable and $\nu_1$-stable manifolds prevents entropy creation. Let us explain our strategy in each class of segment.

\parbreak\textbf{Small entropy segments.} These segments are made of iterates staying close to the measure $\nu_0$ and have then a cover by a small number of Bowen-balls. We use Yomdin theorem to ensure that the number of reparametrizations is small inside each Bowen-ball.

\parbreak\textbf{Wild segments.} We control the number of reparametrizations by the norm of the derivative. It can be big, but since these segments are very short in total length, they do not allow entropy creation.

\parbreak\textbf{Rectangle segments.} Here, the control is more delicate. Let us develop the strategy, which is the core of the present work. At the beginning of a rectangle segment, any reparametrization of the unstable manifold we study intersects the $\nu_1$-rectangles. Pick one of these reparametrizations and let us denote it by $\gamma$. We apply Proposition \ref{prop:IntersectionCurveRectangle} to the curve $\gamma$, and treat separately the crossing and folding $\gamma$-intervals.

\parbreak\textbf{Rectangle segments: case of a crossing interval.} We treat each crossing $\gamma$-interval apart. Proposition \ref{prop:IntersectionCurveRectangle} bounds the number of such intervals. Since their image by $\gamma$ is connected, we use then Property \ref{property:MR5} and the fact that iterates of $\gamma$ do not cross the stable boundaries of the rectangles to trap the iterates of the unstable manifolds inside the rectangles. Indeed, during a rectangle segment, the iterates of any crossing interval is contained in a rectangle of the family $\tilde{R}$. See figure \ref{fig:ProofTopologicalIntersection}.

\begin{figure}[h]
	\centering
	\includegraphics[scale=1]{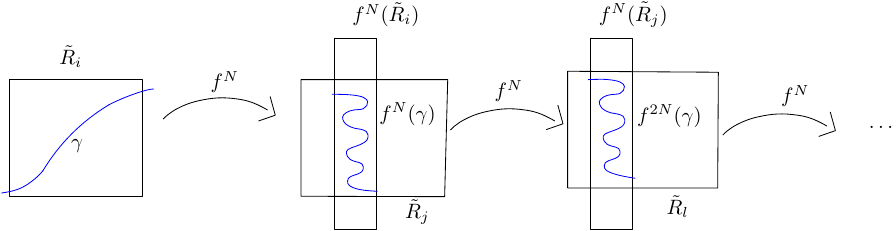}
	\caption{The proof in the case of a crossing $\gamma$-interval. The horizontal direction represents the stable boundaries of the rectangles.}
	\label{fig:ProofTopologicalIntersection}
\end{figure}

\parbreak\textbf{Rectangle segments: case of folding intervals.} Unlike in the case of crossing $\gamma$-intervals, we do not have a bound on the number of such intervals. Nevertheless, Proposition \ref{prop:IntersectionCurveRectangle} gives a uniform bound on the number of rectangles needed to contain all the images of the folding $\gamma$-intervals. We treat these rectangles separately. Pick such a rectangle $R_i$. The set of folding $\gamma$-intervals whose images are contained in $R_i$ is not connected, so we cannot directly apply the same argument as for crossing intervals. Nevertheless, the curve $\gamma$ is connected and is small, combining this with the fact that no iterate of $\gamma$ can cross the stable manifolds containing the stable boundaries of the rectangles of the family $\tilde{\mathcal{R}}$, Property \ref{property:MR5} still implies that the image by $f$ of all the folding $\gamma$-intervals contained in $R_i$ are still contained in a rectangle $\tilde{R}_j$. To apply this argument, we need to control the size of the iterates $f^n(\gamma)$. For this reason, we apply Yomdin Theorem every $N_{\star}$ iterates for some $N$ fixed and well-chosen. Then, we are able to repeat this argument starting with a new reparametrization, which will be contained in some rectangle $\tilde{R}_l$.

We conclude in Step \stepref{step:U9} and Step \stepref{step:U10} by giving a bound on the number of reparametrizations to conclude that the iterates of the unstable manifolds of typical $\mu_k$-points have to cross the stable manifolds containing the stable boundaries of the $\nu_1$-rectangles. See Figure \ref{fig:ProofTopologicalIntersectionFolding}.

\begin{figure}[h]
	\centering
	\includegraphics[scale=1]{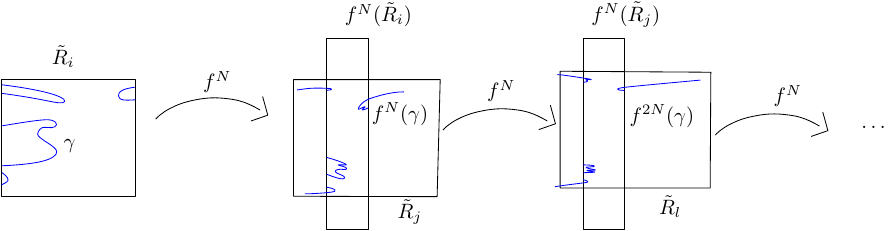}
	\caption{The proof in the case of a rectangle $R_i$ containing folding $\gamma$-intervals. The horizontal direction represents the stable boundaries of the rectangles. Note that the set of folding $\gamma$-intervals contained in $\tilde{R}_i$ is disconnected, whereas $\gamma$ is connected. The curve gets bigger as we iterate and we have to reparametrize to continue to apply the argument, using the fact that iterates of $\gamma$ do not cross the stable local manifolds containing the stable boundaries of the rectangles.}
	\label{fig:ProofTopologicalIntersectionFolding}
\end{figure}

\parbreak\textbf{Step U.I.0.} Let $(\mu_k)_{k\in\bbN}$ be a sequence of ergodic hyperbolic $f$-invariant measures of saddle type. Write $\hmu_k := \pi_{\star}^{-1}\mu_k$. We suppose that we have the following convergence:
\begin{equation*}
	\mu_k \underset{k \rightarrow +\infty}{\longrightarrow} \nu
\end{equation*}
in the weak-$\star$ topology for some $f$-invariant measure $\nu$, which may not be ergodic. Again, denote by $\hnu := \pi_{\star}^{-1}\nu$. By our hypothesis, there exist two $\hf$-invariant probability measures $\hnu_0,\hnu_1$ and a real number $\alpha \in (0,1]$ such that:
\begin{equation*}
	\hnu = \alpha \hnu_1 + (1-\alpha)\hnu_0.
\end{equation*}

Write the ergodic decomposition of $\hnu_0$ as $\hnu_0 = \int \hnu_{0,\hx}\ d\hnu_0(\hx)$. Take a positive constant $h>0$ such that:
\begin{equation}
	\frac{1}{1-\alpha}\big(\liminf_{k\rightarrow +\infty} h(f,\mu_k) - 2\lambda(f)/r\big)> h > \text{ess-sup}_{\hx} h(f,\hnu_{0,\hx}).
	\label{eq:Choiceofh}
\end{equation}

The two measures $\hnu_0$ and $\hnu_1$ satisfy the following. For $\hnu$-almost every ergodic component $\hnu_0^{\prime}$ of $\hnu_0$, we have $h(\hf,\hnu_0^{\prime})\leq h$. We also have that $\hnu$-almost every ergodic component of $\hnu_1$ is hyperbolic of saddle type with entropy larger than $\lambda(f)/r$. Note that we can then apply Proposition \ref{prop:MarkovRectangles} to the measure $\nu_1$.

\parbreak\UStep[Choice of the entropy and reparametrization parameters.]\label{step:U1}
Let us first choose the parameters of entropy. Pick first a constant $\kappa>0$ such that:
\begin{equation}
	\kappa > \frac{8(\log\norm{df}/r)-2}{1-\alpha}.
	\label{eq:ChoiceOfKappa}
\end{equation}
We then consider another arbitrary parameter $\delta>0$ which satisfies:
\begin{equation}
	\delta < \min\Big(\frac{1}{8h(1-\alpha)},\frac{\alpha }{10(1-\alpha)},1/2\Big).
	\label{eq:ChoiceOfDelta}
\end{equation}
Choose also $\delta$ small enough such that it satisfies also:
\begin{equation}
	h+\kappa\delta<\frac{1}{1-\alpha}\big(\liminf_{k\rightarrow +\infty} h(f,\mu_k) - 2\lambda(f)/r\big).
	\label{eq:ChoiceOfDelta2}
\end{equation}
Up to taking an iterate large enough, which will not change the rest of the proof since our analysis relies on entropy, we can suppose that:
\begin{equation}
	\lambda(f)(1-\delta)\leq \log\norm{df}\leq \lambda(f)(1+\delta).
	\label{eq:GreaterExponent}
\end{equation}
Then, we select the parameter $\epsilon>0$, which will be the scale at which we are going to compute the entropy. Recall that by \eqref{eq:TailEntropyEstimate1} and \eqref{eq:TailEntropyEstimate3}, we can take $\epsilon>0$ such that for any $f$-invariant measure $m$:
\begin{equation}
	h(f,m,\epsilon) \geq h(f,m)-\frac{\lambda(f)}{r}(1+\delta).
	\label{eq:MinorationEntropy}
\end{equation}
We also require that $\epsilon<\epsilon_{\Upsilon}(r)$, where $\epsilon_{\Upsilon}(r)$ is the scale given by Yomdin Theorem \ref{thm:Yomdin}.

We also fix the integer parameter $N_{\star}$, which will be the frequency at which we will reparametrize. We choose it such that:
\begin{equation}
	N_{\star}>\frac{10\log\Upsilon}{\alpha h}
	\label{eq:ChoiceOfNStar}
\end{equation}
where $\Upsilon := \Upsilon(r)$ is the constant from Yomdin Theorem \ref{thm:Yomdin}.

\parbreak\UStep[The $\nu_1$-rectangles.]\label{step:U2}
Take any $K>8$. Define the map $H(t) = -t\log t -(1-t)\log(1-t)$. Choose also a parameter $\eta \in (0,1)$ small such that:
\begin{equation}
	\eta < \min\Big(\frac{\delta r}{32N_{\star}\log\norm{df}},\frac{\delta}{40\log\big(24(K+1)+2\big)},\frac{\delta}{56\log\norm{2df}},\frac{\delta}{16h(1+\delta^2)}\Big) \ \text{and} \ H(5\eta)<\frac{\delta }{16}.
	\label{eq:ChoiceOfEta}
\end{equation}
We apply Proposition \ref{prop:MarkovRectangles} to the measure $\nu_1$ with these choices of $K$ and $\eta$. It gives us a Pesin block $\hat{\Lambda}$, an integer $N\geq 1$, a foliation $\mathscr{F}^u$ of an open set $V\supset \Lambda := \pi(\hat{\Lambda})$ and two families of $us$-rectangles $\mathcal{R} = (R_1,\dots,R_L)$ and $\tilde{\mathcal{R}} = (\tilde{R}_1,\dots,\tilde{R}_L)$ such that Properties \ref{property:MR1} to \ref{property:MR5} hold for some parameters $\alpha^{\prime},\sigma>0$ sufficiently small. Choose $\alpha^{\prime}$ and $\sigma$ such that all the results from Sections \ref{sec:Rectangles} and \ref{sec:CurvesIntersectingFamiliesRectangles} hold.

Up to reducing $\epsilon$, we also choose $\sigma$ such that:
\begin{equation}
	\sigma(1+K)<\epsilon<2\sigma(1+K).
	\label{eq:ChoiceOfSigma}
\end{equation}
Since $\epsilon$ and $\sigma$ are the last parameters we have chosen, we can reduce them as much as we want. We will do it in what follows, often to work with rectangles small enough compared to the size of the Pesin charts, which can be chosen uniformly since we work with charts of the Pesin block $\hat{\Lambda}$.

Let $\hat{U}_1$ be the neighborhood of $\hat{\Lambda}$ constructed in Lemma \ref{lem:ContractionExtendedUnstableConeField} and used in Property \ref{property:MR5} for the integer $N$. Recall that we have $\pi(\hat{U}_1) \subset \cup_iR_i$. Recall also that it satisfies:
\begin{equation*}
	\hnu_1(\hat{U}_1) > 1-\eta.
\end{equation*}
Since the most important property of Proposition \ref{prop:MarkovRectangles}, Property \ref{property:MR5}, holds for the iterate $f^N$, we will, from now on, fix $g=f^N$.

\parbreak\UStep[The small entropy part.]\label{step:U3}
Let us deal with the small entropy part of the measure $\hnu$, namely the measure $\hnu_0$. We are going to show that we can find open sets with large $\hnu_0$-measure which admit covers by Bowen balls of small cardinality.

\begin{claim}
	There exists an integer $N_0$ with the following property. For any $n\geq N_0$, there exists an open set $\hat{U}_0 := \hat{U}_0(n)$ and a finite set $\mathcal{C}_0 := \mathcal{C}_0(n)$ such that:
	\begin{itemize}[label={--}]
		\item $\hnu_0(\hat{U}_0) > 1-\eta$ and $\hnu_0(\partial\hat{U}_0)=0$.
		\item $\hat{U}_0 \subset \bigcup_{\hx \in \cC_0} B_{g}\big(\hx,n,\epsilon\big)$.
		\item $|\cC_0| \leq \exp\big(Nh(1+\delta^2)n+o(n)\big)$.
	\end{itemize}
	\label{claim:Step3}
\end{claim}
\begin{proof}
	Since the measure $\hnu_0$ may not be ergodic, let us consider an ergodic decomposition:
	\begin{equation*}
		\hnu_0 = \int_{M_f} \hnu_{0,\hx} \ d\hnu_0(\hx).
	\end{equation*}
	By our hypothesis, $\hnu_0$-almost every $\hnu_{0,\hx}$ satisfies $h(\hf,\hnu_{0,\hx}) \leq h$, which implies that $h(\hat{g},\hnu_{0,\hx}) \leq hN$. Since the $\hnu_{0,\hx}$ are ergodic, using \eqref{eq:KatokFormula}, for $\hnu_0$-almost every $\hx$ we have:
	\begin{equation*}
		\limsup_{n\rightarrow +\infty} \frac{1}{n} \inf_{\hnu_{0,\hx}(\hat{A})>(1-\eta)^{1/2}} \log r_{\hat{g}}\big(n,\epsilon,\hat{A}\big) \leq Nh.
	\end{equation*}
	This implies the existence of an integer $n(\hx)\geq 1$ such that for any $n\geq n(\hx)$, there exists a measurable set $\hat{A}_n(\hx)$, satisfying $\hnu_{0,\hx}(\hat{A}_n(\hx))>(1-\eta)^{1/2}$ and:
	\begin{equation*}
		r_{\hat{g}}\big(n,\epsilon,\hat{A}_n(\hx)\big)\leq \exp\big(Nh(1+\delta^2/10)n\big).
	\end{equation*}
	Let $\mathcal{C}_n(\hx) \subset M_f$ be a set of points such that $\#\mathcal{C}_n(\hx)\leq \exp\big(Nh(1+\delta^2/10)n\big)$ and:
	\begin{equation*}
		\hat{A}_n(\hx) \subset \bigcup_{\hy \in \mathcal{C}_n(\hx)} B_{\hat{g}}\big(\hy,n,\epsilon\big) := \hat{A}^{\prime}_n(\hx).
	\end{equation*}
	
	Take now a compact set $\hat{K}$ and a constant $r>0$ such that $\hnu_0(\hat{K})>(1-\eta)^{1/2}$ and the following properties hold.
	\begin{itemize}[label={--}]
		\item There exists an integer $N_0\geq 1$ such that for any $\hx \in \hat{K}$ and any $n\geq N_0$ we have:
		\begin{equation*}
			r_{\hat{g}}\big(n,\epsilon,\hat{A}^{\prime}_n(\hx)\big)\leq \exp\big(Nh(1+\delta^2/10)n\big).
		\end{equation*}
		\item For two points $\hx,\hy \in \hat{K}$ such that $d(\hx,\hy)<r$, we still have $\hnu_{0,\hy}(\hat{A}^{\prime}_n(\hx)) > (1-\eta)^{1/2}$.
	\end{itemize}
	
	Take a finite cover of $\hat{K}$ by balls of radius $r$ around $\hx_1,\dots,\hx_l$. For $n\geq N_0$, consider the set:
	\begin{equation*}
		\hat{U}_0 = \bigcup_{i=1}^l \hat{A}^{\prime}_n(\hx_i).
	\end{equation*}
	Note that:
	\begin{equation*}
		\hnu_0(\hat{U}_0) \geq \int_{\hat{K}} \hnu_{0,\hx}(\hat{U}_0) \ d\hnu_0(\hx) > 1-\eta.
	\end{equation*}
	Up to reducing $\hat{U}_0$, but keeping the property that $\hnu_0(\hat{U}_0) > 1-\eta$, we obtain that $\hnu_0(\partial \hat{U}_0)=0$.
	
	By construction, $\hat{U}_0$ can be covered by a number of $\big(n,\epsilon,g\big)$-Bowen balls smaller than $\exp\big(Nh(1+\delta^2)n+o(n)\big)$ for any $n\geq N_0$.
\end{proof}

\parbreak\UStep[The typical $\mu_k$-orbits do not jump often between $\nu_0$ and $\nu_1$.]\label{step:U4}
The aim of this step is to show that for $n$ large enough, the typical orbits of $\hmu_k$ do not jump often between the open sets $\hat{U}_0$ and $\hat{U}_1$. More precisely, we mean that these orbits have to spend long periods of time in some $\hat{U}_i$ before intersecting the other one.

Let us choose integers $n_1,n_0\geq 0$, which will be the time parameters of the orbits, and which satisfy the following:
\begin{equation}
	n_1 > 1/\eta \quad \text{and} \quad n_0 > \max(1/\eta, N_0).
	\label{eq:ChoiceOfTime}
\end{equation}

\begin{claim}
	There exist two open sets $\hat{U}_i^{\prime} \subset \hat{U}_i$, for $i=0,1$, satisfying the following.
	\begin{itemize}[label={--}]
		\item $\hnu_i(\hat{U}_i^{\prime}) > 1-\eta$ and $\hnu_i(\partial\hat{U}_i^{\prime}) = 0$.
		\item For any $0\leq j\leq n_1$ we have:
		\begin{equation*}
			\hat{g}^j\big(\text{Closure}(\hat{U}_1^{\prime})\big) \cap \text{Closure}(\hat{U}_0^{\prime}) = \emptyset.
		\end{equation*}
		\item The same property holds exchanging the roles of $\hat{U}_1^{\prime}$ and $n_1$ with $\hat{U}_0^{\prime}$ and $n_0$.
		\item $\hnu_0(\hat{U}_1^{\prime}) < \eta$ and $\hnu_1(\hat{U}_0^{\prime}) < \eta$.
	\end{itemize}
	\label{claim:Step4}
\end{claim}
To keep the notations as simple as possible, we still denote $\hat{U}_i^{\prime}$ by $\hat{U}_i$.

\begin{proof}
	Recall that we have the following ergodic decomposition for $i =1,2$:
	\begin{equation*}
		\hnu_i = \int_{M_f} \hnu_{i,\hx} \ d\hnu_i(\hx).
	\end{equation*}
	Since $\hnu=\alpha\hnu_1+(1-\alpha)\hnu_0$, the measures $\hnu_0$ and $\hnu_1$ are mutually singular in the following sense. For $\hnu_0\times\hnu_1$-almost every $(\hx,\hy)$, the measures $\hnu_{0,\hx}$ and $\hnu_{1,\hy}$ are singular. There exist then two disjoint $g$-invariant sets $\hat{X}_0$ and $\hat{X}_1$ such that $\hnu_{i_1}(\hat{X_{i_2}})$ equals one if and only if $i_1=i_2$. Choose now a compact subset $\hat{K}_i\subset\hat{U}_i$ of $\hat{X}_i$ such that $\hnu_i(\hat{K}_i)>1-\eta$. Necessarily, for all integer $j$ we have:
	\begin{equation*}
		\hat{g}^j(\hat{K}_{i_1}) \cap \hat{K}_{i_2} \subset \hat{X}_{i_1} \cap \hat{X}_{i_2} = \emptyset
	\end{equation*}
	for $i_1\ne i_2$. Replacing $\hat{K}_i$ with small enough neighborhoods $\hat{U}_{i}^{\prime} \subset \hat{U}_i$, we can then still guarantee that for $0\leq j\leq n_{i_1}$ we have:
	\begin{equation*}
		\hat{g}^j\big(\text{Closure}(\hat{U}_{i_1}^{\prime})\big) \cap \text{Closure}(\hat{U}_{i_2}^{\prime}) = \emptyset
	\end{equation*}
	for $i_1\ne i_2$. Up to reducing these open sets we can assume that $\hnu_i(\partial \hat{U}^{\prime}_i) = 0$. Note that if $i_1\ne i_2$, we have $\hat{U}_{i_1}^{\prime} \cap \hat{U}_{i_2}^{\prime} = \emptyset$ so $\hnu_{i_1}(\hat{U}_{i_2}^{\prime}) \leq \hnu_{i_1}(M_f\backslash\hat{U}_{i_1}^{\prime}) < \eta$.
\end{proof}

\parbreak\UStep[Decomposition of orbits into segments.]\label{step:U5}
We call a $\hat{g}$-orbit segment with length $t$ and initial point $\hx$ the string $(\hx,\dots,\hat{g}^{t-1}(\hx))$. Let us define the following three classes of $\hat{g}$-orbit segments $(\hat{g}^s(\hx),\dots,\hat{g}^{t+s-1}(\hx))$ of a $\hat{g}$-orbit $(\hx,\dots,\hat{g}^{n-1}(\hx))$.
\begin{itemize}[label={--}]
	\item \textbf{Rectangle segments:} $\hat{g}^s(\hx) \in \hat{U}_1$ and $t=n_1$.
	\item \textbf{Small entropy segments:} $\hat{g}^s(\hx) \in \hat{U}_0$ and $t=n_0$.
	\item \textbf{Wild segments:} one of the following situations happens.
	\begin{enumerate}[label=(\alph*)]
		\item $\hat{g}^{s+j}(\hx) \notin \hat{U}_0\cup\hat{U}_1$ for $j=0,\dots,t-1$ and $\hat{g}^{t+s}(\hx) \in \hat{U}_0\cup\hat{U}_1$.
		\item $\hat{g}^s(\hx) \in \hat{U}_0$ and $s+n_0>n-1$, in that case $t+s=n$.
		\item $\hat{g}^s(\hx) \in \hat{U}_1$ and $s+n_1>n-1$, in that case $t+s=n$.
	\end{enumerate}
\end{itemize}
Note that the rectangle segments and the small entropy segments are disjoint since $\hat{U}_0\cap \hat{U}_1=\emptyset$ and since by Claim \ref{claim:Step4} if $\hx \in \hat{U}_{i_1}$, then $\hat{g}^j(\hx) \notin \hat{U}_{i_2}$ for $j=1,\dots,n_{i_1}$ for $i_1\ne i_2$.

\begin{claim}
	There exists an integer $k_{\star}$ such that for any $k\geq k_{\star}$ and for $\hmu_k$-almost every $\hx$, there exists an integer $n_k(\hx)$ displaying the following property. For any $n\geq n_k(\hx)$, any orbit $(\hx,\dots,\hat{g}^{n-1}(\hx))$ can be decomposed as follows. There exists an integer $l$ and a string:
	\begin{equation*}
		\theta^{\prime}(\hx,n) := \big((t_0,b_0),\dots,(t_l,b_l)\big) \in (\bbZ\times\{1,2,3\})^{l+1}
	\end{equation*}
	such that $t_0=-1$, $t_l=n-1$, $b_l=1$ and for any $j\in\{0,\dots,l-1\}$ the orbit segment $(\hat{g}^{t_j+1}(\hx),\dots,\hat{g}^{t_{j+1}}(\hx))$ is: a rectangle segment if $b_j=1$, a small entropy segment if $b_j=2$ and a wild segment if $b_j=3$. Moreover, we have the following estimates:
	\begin{enumerate}[label=(\alph*)]
		\item\label{propMarkovRectangles:CaseA} rectangle segments have total length at most $\alpha n + 2\eta n$,
		\item\label{propMarkovRectangles:CaseB} small entropy segments have total length at most $(1-\alpha)n + 2\eta n$,
		\item\label{propMarkovRectangles:CaseC} wild segments have total length at most $2\eta n$.
	\end{enumerate}
	\label{claim:Step5}
\end{claim}
\begin{proof}
	By Step 1 of the proof and by Claim \ref{claim:Step4}, we have the following estimates:
	\begin{itemize}[label={--}]
		\item $\hnu(\hat{U}_0 \cup \hat{U}_1) > 1-\eta$,
		\item $\hnu(\hat{U}_0) = \alpha\hnu_1(\hat{U}_0) + (1-\alpha) \hnu_0(\hat{U}_0) < \eta\alpha + (1-\alpha) < \eta + (1-\alpha)$,
		\item $\hnu(\hat{U}_1) = \alpha\hnu_1(\hat{U}_1) + (1-\alpha) \hnu_0(\hat{U}_1) < \alpha + (1-\alpha)\eta < \alpha + \eta$.
	\end{itemize}
	Since the sequence $(\hmu_n)$ converges weak-$\star$ to $\hnu$, there exists an integer $k_{\star}\geq 1$ such that for any $k \geq k_{\star}$ we have:
	\begin{itemize}[label={--}]
		\item $\hmu_k(\hat{U}_0 \cup \hat{U}_1) > 1-\eta$,
		\item $\hmu_k(\hat{U}_0) < (1-\alpha) + \eta$,
		\item $\hmu_k(\hat{U}_1) < \alpha + \eta$.
	\end{itemize}
	Fix now some $k\geq k_{\star}$. For $\hmu_k$-almost every $\hx$, there exists an integer $n_k(\hx)$ such that for $n\geq n_k(\hx)$ the following is true by Birkhoff theorem:
	\begin{itemize}[label={--}]
		\item $\frac{1}{n}\sum_{i=0}^{n-1} \mathds{1}_{\hat{U}_0\cup\hat{U}_1}(\hat{g}^i(\hx)) > 1-\eta$,
		\item $\frac{1}{n}\sum_{i=0}^{n-1} \mathds{1}_{\hat{U}_0}(\hat{g}^i(\hx)) < (1-\alpha) + \eta$,
		\item $\frac{1}{n}\sum_{i=0}^{n-1} \mathds{1}_{\hat{U}_1}(\hat{g}^i(\hx)) < \alpha + \eta$.
	\end{itemize}
	Fix such an $\hx$. Up to increasing $n_k(\hx)$, we can suppose that $\max(n_0,n_1) < \eta n_k(\hx)$.
	
	For $k\geq k_{\star}$ and for $n\geq n_k(\hx)$, this allows us to define a decomposition of the $\hat{g}$-orbit $(\hx,\dots,\hat{g}^{n-1}(\hx))$. Indeed, we define the following sequence of times by induction, which depends on $\hx$. Let $t_0 = -1$. Assume that $t_{i}$ is defined. We are going to define $t_{i+1}$ in the following way.
	\begin{itemize}[label={--}]
		\item If $\hat{g}^{t_i+1} \in \hat{U}_0$ and $t_i + n_0\leq n-1$, we define $t_{i+1} = t_i + n_0$. The orbit segment $(\hat{g}^{t_i+1}(\hx),\dots,\hat{g}^{t_{i+1}}(\hx))$ is then a small entropy segment.
		\item If $\hat{g}^{t_i+1} \in \hat{U}_1$ and $t_i + n_1\leq n-1$, we define $t_{i+1}=t_i+n_1$. The orbit segment $(\hat{g}^{t_i+1}(\hx),\dots,\hat{g}^{t_{i+1}}(\hx))$ is then a rectangle segment.
		\item If $\hat{g}^{t_i+1}\in \hat{U}_0$ but $t_i+n_0>n-1$, or if $\hat{g}^{t_i+1}\in \hat{U}_1$ but $t_i+n_1>n-1$, we define $t_{i+1}=n-1$. The orbit segment $(\hat{g}^{t_i+1}(\hx),\dots,\hat{g}^{n-1}(\hx))$ is then a wild segment.
		\item If $\hat{g}^{t_i+1} \in \big(\hat{U}_0\cup\hat{U}_1\big)^c$, then we define:
		\begin{equation*}
			t_{i+1} = \inf\{j>t_i, \ \hat{g}^{j+1}(\hx) \in \hat{U}_0\cup\hat{U}_1\},
		\end{equation*}
		or $t_{i+1} = n-1$ if the $\inf$ does not exist. The orbit segment $(\hat{g}^{t_i+1}(\hx),\dots,\hat{g}^{t_{i+1}}(\hx))$ is then a wild segment.
	\end{itemize}
	We repeat the process until we reach some $l$ such that $t_l = n-1$.
	
	For $k\geq k_{\star}$, for $\hmu_k$-almost every $\hx$ and $n\geq n_k(\hx)$ we can then define the following desired data:
	\begin{equation*}
		\theta^{\prime}(\hx,n):=\big((t_0,b_0),\dots,(t_l,b_l)\big),
	\end{equation*}
	where $b_i \in \{1,2,3\}$, defining the class of the segment $(\hat{g}^{t_i+1}(\hx),\dots,\hat{g}^{t_{i+1}}(\hx))$. We introduce the convention that $b_i=1$ means that $\vartheta_i$ is a rectangle segment, $b_i=2$ means that $\vartheta_i$ is a small entropy segment and $b_i=3$ means that $\vartheta_i$ is a wild segment. Since $t_l=n-1$, we pose arbitrarily $b_l=1$. This gives the desired decomposition. We are now going to estimate the size of each class of segment inside this decomposition.
	
	Fix such an $\hx$ and consider the $\hat{g}$-orbit segment $(\hx,\dots,\hat{g}^{n-1}(\hx))$. Consider the orbit segments which do not intersect $\hat{U}_0$ nor $\hat{U}_1$. Their total length is at most $\eta n$. The wild segments are segments which have this property, or last segment of an orbit, with starting point in $\hat{U}_0$ or $\hat{U}_1$, which have length smaller than $\max(n_0,n_1)<\eta n$. This then proves the last point \ref{propMarkovRectangles:CaseC}.
	
	It remains to prove cases \ref{propMarkovRectangles:CaseA} and \ref{propMarkovRectangles:CaseB}. Set $\hat{W} := \big(\hat{U}_0 \cup \hat{U}_1\big)^c$. Denote by $R(\hx)$ the union of all points of the $\hx$-orbit belonging to a rectangle segment of the decomposition defined earlier. Recall that Claim \ref{claim:Step4} implies that a rectangle segment does not intersect $\hat{U}_0$. We then have the following:
	\begin{equation*}
		|R(\hx)| = |R(\hx)\cap \hat{U}_1| + |R(\hx)\cap\hat{W}| \leq |\hat{U}_1| + |\hat{W}| < \alpha n +\eta n + \eta n < \alpha n + 2\eta n.
	\end{equation*}
	The argument to prove the estimate of case \ref{propMarkovRectangles:CaseB} is the same.
\end{proof}
For $k\geq k_{\star}$, for $\hmu_k$-almost every $\hx$ and for $n\geq n_k(\hx)$, Claim \ref{claim:Step5} allows us to construct the decomposition string $\theta^{\prime}(\hx,n) \in (\bbZ\times\{1,2,3\})^{l+1}$ for some integer $l$. For a given orbit $(\hx,\dots,\hat{g}^{n-1}(\hx))$, we will also need to keep track of which iterates are inside $\hat{U}_0\cup\hat{U}_1$ and which are not. For that, we introduce the word $v(\hx,n) \in \{0,1\}^n$ which satisfies $v(\hx,n)_i = 1$ if $\hat{g}^i \in \hat{U}_0\cup\hat{U}_1$ and $v(\hx,n)_i = 0$ otherwise.

We define the \textit{decomposition type} of an orbit $(\hx,\dots,\hat{g}^{n-1}(\hx))$ as the following data:
\begin{equation*}
	\theta(\hx,n) := \Big[v(\hx,n), \theta^{\prime}(\hx,n)\Big] \in \{0,1\}^{n} \times \big(\bbN \times \{1,2,3\}\big)^{l+1}.
\end{equation*}

\parbreak\UStep[A bound on the number of decomposition types.]\label{step:U6}
Recall that we denote by $H$ the map $H(t) = -t\log t -(1-t)\log(1-t)$. The following Claim asserts that the number of decomposition types is asymptotically small compared to the entropy of the $\mu_k$'s.

\begin{claim}
	There exists an integer $n_H := n_H(\eta)$ such that the number of decomposition types of any $\hat{g}$-orbit as in Claim \ref{claim:Step5} with length $n>n_H$ is at most $\exp\big(n(5\eta + 2H(5\eta))\big)$.
	\label{claim:Step6}
\end{claim}
\begin{proof}
	By Claim \ref{claim:Step5}, a decomposition of an orbit segment of length $n$ large enough has:
	\begin{itemize}[label={--}]
		\item at most $\eta n$ rectangle segments (because they have length $n_1 > 1/\eta$ by \eqref{eq:ChoiceOfTime}),
		\item at most $\eta n$ small entropy segments by the same reason,
		\item at most $2\eta n$ wild segments (because their total length is at most $2\eta n$).
	\end{itemize}
	This gives a total of at most $\lfloor 4\eta n\rfloor$ segments. Any type $\theta = \Big[\big(v_0,\dots,v_{n-1}\big),\big((t_0,b_0),\dots,(t_l,b_l)\big)\Big]$ satisfies $l< \lfloor 4\eta n\rfloor+1$. A decomposition type consists of times $t_i$, classes $b_i$ and a word $v$ of $n$ letters with at most $\eta n$ letter $0$. There can be at most $\sum_{i=1}^{\lfloor 4\eta n\rfloor}\binom{n}{i-1}3^i$ different strings of couples of times and classes $\big((t_0,b_0),\dots,(t_l,b_l)\big)$. There also can be at most $\sum_{i=1}^{\lfloor \eta n\rfloor}\binom{n}{i-1}$ such words $v$. So the total number of decomposition types possible is bounded by $\sum_{i=1}^{\lfloor 4\eta n\rfloor}\binom{n}{i-1}3^i\times\sum_{i=1}^{\lfloor \eta n\rfloor}\binom{n}{i-1}$. Since $4\eta<1/2$, the first sum is bounded by $4\eta n\binom{n}{\lfloor 4\eta n\rfloor}\times 3^{\lfloor 4\eta n\rfloor}$ and the second by $\eta n\binom{n}{\lfloor \eta n\rfloor}$. By the De Moivre-Laplace approximation, this is less than $\exp\big(4\eta n\log 3 + nH(4\eta) +nH(\eta) + o(n)\big)$ as $n\rightarrow +\infty$. The claim follows because $5\eta < 1/2$ and $4\log3<5$.
\end{proof}

\parbreak\UStep[Conditional measures.]\label{step:U7}
Let $k\geq k_{\star}$ from Claim \ref{claim:Step5}. As explained in Section \ref{subsec:UnstableEntropy}, let $\xi_k$ be a measurable partition subordinated to the unstable lamination of $\hmu_k$. Denote by $\{\hmu_k^{u,\hx}\}_{\hx}$ the family of conditional measures associated to $\xi_k$. Recall that $\hmu_k^{u,\hx}$ is supported on $V^u_{loc}(\hx)$. Denote by $\mu_k^{u,\hx} := \pi_{\star}\hmu_k^{u,\hx}$. This measure is then supported on $\Wu_{loc}(\hx)$. We fix an integer $N_k$ and a measurable set $\hat{F}_k$ with $\hmu_k(\hat{F}_k)>1/2$ such that every point $\hx \in F_k$ satisfies the following:
\begin{itemize}[label={--}]
	\item $\hx$ has a well defined unstable manifold $\Wu(\hx) \subset M$,
	\item $\mu_k^{u,\hx}$ is well defined and $x$ belongs to the support of the measure $\mu_k^{u,\hx}$ restricted to $F_k$, where $x := \pi(\hx)$ and $F_k := \pi(\hat{F}_k)$,
	\item $\hx$ satisfies Claim \ref{claim:Step5} with $n_k(\hx)\leq N_k$. In particular, for each $n\geq N_k$ the orbit $(\hx,\dots,\hat{g}^{n-1}(\hx))$ has decomposition as in Claim \ref{claim:Step5}. Write $\theta(\hx,n)$ for its decomposition type.
\end{itemize}

\parbreak\UStep[Construction of the reparametrizations.]\label{step:U8}
Pick any $\hx \in \hat{F}_k$ which satisfies Proposition \ref{prop:EntropyReparametrizations} for the map $g$. Recall Definition \ref{def:LocalInvariantSet} of the local unstable set at $\hx$, denoted by $V^u_{loc}(\hx) \subset M_f$. Let $\hat{\gamma} : [0,1] \rightarrow V^u_{loc}(\hx)$ be a $\cC^r$ curve of $\cC^r$ size $\epsilon$, with the choice of $r>2$ defined in Step \stepref{step:U1}, and such that $\hx \in \hat{\gamma}((0,1))$. Let $\gamma :[0,1] \rightarrow \Wu_{loc}(\hx)$ be its projection. Note that $\gamma$ is an $(r,\epsilon)$-curve. The properties of the set $\hat{F}_k$ guarantee that $\mu_k^{u,\hx}\big(F_k\cap\gamma([0,1])\big) >0$. Define $\hat{T}:=\hat{\gamma}^{-1}(\hat{F}_k)$. Define also $T:=\gamma^{-1}(F_k)$. Fix $n\geq \max(N_k,n_H)$ and fix also $N_{\star}, \epsilon$ and $\sigma$ as in Step \stepref{step:U1} and Step \stepref{step:U2}.

Recall the definition of the set $\hat{Y}^{\#}$ before Proposition \ref{prop:MarkovRectangles}. Recall the definition of the family of rectangles $\mathcal{R}$ from Step \stepref{step:U1}. Recall also that for any $i \in \{1,\dots,L\}$, we have $\partial^{s,l} R_i \subset \Ws_{loc}(\hy_i^l)$ where $\hy_i^l \in \hat{\Lambda} \cap \hat{Y}^{\#}$ for $l=1,2$. Our main goal is to show that some positive iterate of $\Wu_{loc}(\hx)$ intersects some $\Ws_{loc}(\hy_i^l)$, which will complete the proof of Proposition \ref{prop:UnstableStableIntersection}. We are going to proceed by contradiction. From now on, we assume the following non-intersection hypothesis.

\begin{hyp}
	For any $m\geq 0$, for any $i\in \{1,\dots,L\}$ and any $l\in\{1,2\}$ we have:
	\begin{equation*}
		f^m(\Wu_{loc}(\hx))\cap \Ws_{loc}(\hy^l_i)=\emptyset.
	\end{equation*}
	\label{hypothesis:Unstable}
\end{hyp}
Assuming Hypothesis \ref{hypothesis:Unstable}, our goal from now on is to construct a family of reparametrization $\mathscr{R}_n$ over $T$ which is $(g,N_{\star},\epsilon)$-admissible up to time $n$. To simplify the notation and since we will always work with the map $g$, we just write $(N_{\star},\epsilon)$-admissible. After constructing the family of reparametrization, we will use Proposition \ref{prop:EntropyReparametrizations} to get an upper bound on the entropy of $\mu_k$.

We begin by fixing a decomposition type $\theta = \big[\big(v_0,\dots,v_{n-1}\big),\big((t_0,b_0),\dots,(t_l,b_l)\big)\big]$ with $t_0=-1$ and $t_l=n-1$. Recall that for any $y \in \gamma(T)$, there exists a unique $\hy \in \hat{\gamma}(\hat{T})$ which lifts $y$. For what follows, we will always denote this special lift by $\hy$. We will first construct a family of reparametrization $\mathscr{R}_n^{\theta}$ which is $(N_{\star},\epsilon)$-admissible up to time $n$ over the set:
\begin{equation*}
	T_{\theta} := \gamma^{-1}\big(\{y \in \gamma(T) \ \text{such that for the unique lift} \ \hy \in \hat{\gamma}(\hat{T}), \ (\hy,\dots,\hat{g}^{n-1}(y)) \ \text{has type} \ \theta\}\big).
\end{equation*}
We will then take the union over all possible decomposition types and obtain the desired family $\mathscr{R}_n$. Let us sum up what we have just explained in the following Claim. The proof will be contained in the present Step and the next one.

\begin{claim}
	For any $k\geq k_{\star}$, for any $\hx\in \hat{F}_k$ and any $n$ large enough, there exists a family of reparametrization $\mathscr{R}_n$ of the curve $\gamma$ over the set $\gamma^{-1}(F_k)$ which is $(N_{\star},\epsilon)$-admissible up to time $n$ and such that:
	\begin{equation*}
		\mathscr{R}_n = \bigcup_{\theta} \mathscr{R}_n^{\theta}.
	\end{equation*}
	The union is taken over all decomposition types $\theta$ and $\mathscr{R}^{\theta}_n$ are families of reparametrization of $\gamma$ over the set of parameters of $\gamma^{-1}(F_k)$ such that their image by $\gamma$ has decomposition type $\theta$. The family $\mathscr{R}^{\theta}_n$ is $(N_{\star},\epsilon)$-admissible up to time $n$ and is obtained by induction, constructing families of reparametrization for all the segments composing $\theta$. Moreover, the families $\mathscr{R}_n$ and $\mathscr{R}^{\theta}_n$ have bounded cardinality, where a bound is given in \eqref{eq:BoundReparamTheta} and \eqref{eq:BoundMainReparam}.
\end{claim}

The family $\mathscr{R}_n^{\theta}$ will be obtained inductively by defining families of reparametrization $\mathscr{R}^{\theta}_{t_i}$ which will be $(N_{\star},\epsilon)$-admissible over the set $T_{\theta}$ up to time $t_i$. The first family of the induction is $\mathscr{R}^{\theta}_0 := \{id\}$. It is $(N_{\star},\epsilon)$-admissible over $T_{\theta}$ up to time $0$ because $\gamma$ is an $(r,\epsilon)$-curve. Assume now that $\mathscr{R}^{\theta}_{t_{i-1}}$ is well defined. We proceed by concatenation. We will set:
\begin{equation*}
	\mathscr{R}^{\theta}_{t_i} = \{\psi\circ\phi, \ \psi \in \mathscr{R}_{t_{i-1}}, \ \phi \in \mathscr{R}(\psi,t_i)\}
\end{equation*}
where $\mathscr{R}(\psi,t_i)$ will be a well-chosen family of reparametrization of the curve $g^{t_{i-1}} \circ \gamma \circ \psi$ and which will be $(N_{\star},\epsilon)$-admissible up to time $t_i-t_{i-1}$ over the set $\psi^{-1}(T_{\theta})$.

We fix then $\psi \in \mathscr{R}^{\theta}_{t_{i-1}}$. We let $\gamma^{\prime} := g^{t_{i-1}}\circ\gamma\circ\psi$. Recall that $\gamma^{\prime}$ is an $(r,\epsilon)$-curve. We let also $T^{\prime}_{\theta} := \psi^{-1}(T_{\theta})$. By definition of $T_{\theta}$, for any $y \in \gamma(T_{\theta})$ there exists a unique lift $\hat{y} \in \hat{\gamma}(\hat{T})$ such that the orbit $(\hy,\dots,\hat{g}^{n-1}(\hy))$ has decomposition type $\theta$. We will then construct $\mathscr{R}(\psi,t_i)$ distinguishing three cases for the different values of $b_{i-1}$. We will start with the two easy cases where $b_{i-1}$ takes the value $2$ or $3$, which means that the orbit segment $(\hat{g}^{t_{i-1}+1}(\hy),\dots,\hat{g}^{t_i}(\hy))$ is respectively a small entropy segment or a wild segment. The case of rectangle segment is the main difficulty of the proof.

\parbreak\textbf{Case A: small entropy segment.} We first treat the case $b_{i-1}=2$. For $y \in \gamma^{\prime}(T^{\prime}_{\theta})$, recall that $\hy$ is the unique lift of $y$ such that $\hy \in \hat{g}\circ\hat{\gamma}\circ\psi(\hat{T})$. Since $b_{i-1}=2$, we know that every $y \in \gamma^{\prime}(T^{\prime}_{\theta})$ satisfies $\hy \in \hat{U}_0$ and $t_i-t_{i-1} = n_0$. By Claim \ref{claim:Step3} and since $n_0\geq N_0$, there exists a set $\cC_0$ such that $|\cC_0|\leq \exp\big(Nh(1+\delta^2)n_0\big)$ and:
\begin{equation*}
	\gamma^{\prime}(T^{\prime}_{\theta}) \subset \bigcup_{\hx \in \cC_0} B(\pi(\hx),n_0,\epsilon)
\end{equation*}
We are now able to apply Corollary \ref{cor:AdmissibleReparam} with the regularity $r$, at the scale $\epsilon$, with times $N_{\star}$ and $n_0$ to the curve $\gamma^{\prime}$ and over the set $T^{\prime}_{\theta}$. It gives us the family $\mathcal{R}(\psi,t_i)$ of reparametrization of $\gamma^{\prime}$ over $T^{\prime}_{\theta}$, which is $(N_{\star},\epsilon)$-admissible up to time $t_i-t_{i-1}$ and with the following cardinality:
\begin{equation*}
	\begin{aligned}
		|\mathscr{R}(\psi,t_i)|&\leq \Upsilon^{\lceil n_0/N_{\star} \rceil}\norm{dg}^{n_0/r}\norm{dg}^{N_{\star}/r}r_g\big(n,\epsilon,\gamma^{\prime}(T_{\theta})\big)\\
		&\leq \Upsilon^{\lceil n_0/N_{\star} \rceil}\norm{dg}^{n_0/r}\norm{dg}^{N_{\star}/r}|\cC_0|\\
		&\leq \Upsilon^{\lceil n_0/N_{\star} \rceil}\norm{dg}^{n_0/r}\norm{dg}^{N_{\star}/r}\exp\big(Nh(1+\delta^2)n_0\big).
	\end{aligned}
\end{equation*}
Doing this for every $\psi \in \mathscr{R}^{\theta}_{t_{i-1}}$, we obtain the desired family $\mathscr{R}_{t_i}^{\theta}$ which then satisfies:
\begin{equation}
	|\mathscr{R}_{t_i}^{\theta}| \leq \Upsilon^{\lceil (t_i-t_{i-1})/N_{\star} \rceil}\norm{dg}^{(t_i-t_{i-1})/r}\norm{dg}^{N_{\star}/r}\exp\big(Nh(1+\delta^2)(t_i-t_{i-1})\big)|\mathscr{R}^{\theta}_{t_{i-1}}|.
	\label{eq:CardinalReparamCaseB}
\end{equation}

\parbreak\textbf{Case B: wild segment.} We continue with the case $b_{i-1}=3$. Since $\gamma^{\prime}$ is a $(r,\epsilon)$-curve, for any $j=1,\dots,t_i-t_{i-1}$ we have the following estimate:
\begin{equation*}
	\norm{d(g^{j}\circ\gamma^{\prime})} \leq \epsilon\norm{dg^j} \leq \epsilon \norm{dg^{t_i-t_{i-1}}}.
\end{equation*}
Let $t_1,t_2 \in T^{\prime}_{\theta}$ such that $|t_1-t_2|\leq 1/\norm{dg^{(t_i-t_{i-1})}}.$ Write $x_1 = \gamma^{\prime}(t_1)$ and $x_2 = \gamma^{\prime}(t_1)$. For all $j=0,\dots,t_i-t_{i-1}$ we then have:
\begin{equation*}
	d\big(g^j(x_1),g^j(x_2)\big) \leq \epsilon.
\end{equation*}
Construct $\mathcal{D}\subset [0,1]$ a $1/\norm{dg^{t_i-t_{i-1}}}$-dense set such that $\#\mathcal{D}\leq2\norm{dg^{t_i-t_{i-1}}}.$ We then obtain that:
\begin{equation*}
	\gamma^{\prime}(T^{\prime}_{\theta}) \subset \bigcup_{t\in \mathcal{D}} B_{g}\big(\gamma^{\prime}(t),t_i-t_{i-1},\epsilon\big).
\end{equation*}
As in Case A, we apply Corollary \ref{cor:AdmissibleReparam} to obtain the family $\mathscr{R}(\psi,t_i)$. It satisfies the following estimate:
\begin{equation*}
	|\mathscr{R}(\psi,t_i)|\leq \Upsilon^{\lceil (t_i-t_{i-1})/N_{\star} \rceil}\norm{dg}^{(t_i-t_{i-1})/r}\norm{dg}^{N_{\star}/r}2\norm{dg}^{t_i-t_{i-1}}.
\end{equation*}
Again, doing this for any $\psi \in \mathscr{R}^{\theta}_{t_{i-1}}$, we obtain the desired family $\mathscr{R}^{\theta}_{t_i}$ which satisfies:
\begin{equation}
	|\mathscr{R}^{\theta}_{t_i}|\leq \Upsilon^{\lceil (t_i-t_{i-1})/N_{\star} \rceil}\norm{dg}^{(t_i-t_{i-1})/r}\norm{dg}^{N_{\star}/r}2\norm{dg}^{t_i-t_{i-1}}|\mathscr{R}^{\theta}_{t_{i-1}}|.
	\label{eq:CardinalReparamCaseC}
\end{equation}

\parbreak\textbf{Case C: rectangle segment.} We now move on to the more technical and difficult case. We set $b_{i-1}=1$. Since, for any $y \in \gamma(T_{\theta})$, the segment $\vartheta^i(\hy) := (\hat{g}^{(t_{i-1}+1)}(\hy),\dots,\hat{g}^{t_i}(\hy))$ is a rectangle segment, we have $v_{t_{i-1}+1}=1$ (here again $\hy$ is the unique lift of $y$ belonging to $\hat{\gamma}(\hat{T})$). Moreover, since $t_i-t_{i-1}=n_1$, we know that $\hat{g}^j(\hy) \notin \hat{U}_0$ for any $j=t_{i-1}+1,\dots,t_i$. From this, we conclude that if, for such a $j$, the letter $v_j$ is equal to $1$, it means that $\hat{g}^j(\hy) \in \hat{U}_1$.

Set $\tau_0=0$, $\tau_0^{\prime}=0$ and define the following sequence of times $\tau_0=\tau^{\prime}_0\leq\tau_1\leq\tau_1^{\prime}\leq\dots\leq\tau_M\leq\tau_M^{\prime}$ for some integer $M\geq1$. Assume that $\tau_j$ and $\tau_j^{\prime}$ are well-defined. Let:
\begin{equation*}
	\tau_{j+1} := \inf\{d>\tau_k^{\prime}, \ v_{t_{i-1}+d+1}=0\} \quad \text{and} \ \quad \tau_{j+1}^{\prime} := \inf\{d>\tau_{k+1}, \  v_{t_{i-1}+d+1}=1\}
\end{equation*}
with the convention that $\inf \emptyset =n-1$. In the case where there exists $j$ such that $\tau_j\geq n_1-1$ but $\tau^{\prime}_{j-1}<n_1-1$, we let $\tau^{\prime}_M=\tau_M=n-1$. In the case where there exists $j$ such that $\tau_j^{\prime}\geq n_1-1$ but $\tau_{j}<n_1-1$, we pose $\tau_M=\tau_{j}<\tau_M^{\prime}=n-1$. Note that $\tau_1$ is well defined since $v_{t_{i-1}+1}=1$.

To sum up, the sub-segments $\big(\hat{g}^{t_{i-1}+1+\tau_j^{\prime}+1}(\hy),\dots,\hat{g}^{t_{i-1}+1+\tau_{j+1}}(\hy)\big)$ satisfy that each point is contained in $\hat{U}_1$, while the sub-segments $\big(\hat{g}^{t_{i-1}+1+\tau_j+1}(\hy),\dots,\hat{g}^{t_{i-1}+1+\tau_j^{\prime}}(\hy)\big)$ are the ones such that each point is outside $\hat{U}_1\cup\hat{U}_0$. Note that this last type of sub-segments are wild segments inside the rectangle segment $\big(\hat{g}^{t_{i-1}+1}(\hy),\dots,\hat{g}^{t_i}(\hy)\big)$. Let us introduce the following useful quantity:
\begin{equation*}
	c^0_i = \#\{1\leq j\leq n_1-1, \ v_{t_{i-1}+j} = 0\}.
\end{equation*}
This quantity represents the number of iterates in the segment $\big(\hat{g}^{t_{i-1}+1}(\hy),\dots,\hat{g}^{t_i}(\hy)\big)$ outside of $\hat{U}_1$.

We will construct the family $\mathscr{R}(\psi,t_i)$ by induction, again by concatenations of family of reparametrizations, as explained before. We start by letting $\mathscr{R}_{\tau_0} = \mathscr{R}_{\tau^{\prime}_0} = \mathscr{R}^{\theta}_{t_{i-1}}$.

Suppose first that we dispose of a family of reparametrizations $\mathscr{R}_{\tau_k}$ of the curve $\gamma^{\prime}$, which is $(N_{\star},\epsilon)$-admissible up to time $\tau_j$ over the set $\psi^{-1}(T_{\theta})$. We will set:
\begin{equation*}
	\mathscr{R}_{\tau^{\prime}_j} := \{\varphi\circ\tilde{\varphi}, \ \varphi \in \mathscr{R}_{\tau_j}, \ \tilde{\varphi} \in \mathscr{R}(\varphi,\tau_j^{\prime})\}
\end{equation*}
where $\mathscr{R}(\varphi,\tau^{\prime}_j)$ is a family of reparametrization of the curve $g^{\tau_j}\circ\gamma^{\prime}\circ\varphi$ which is $(N_{\star},\epsilon)$-admissible up to time $\tau_j^{\prime}-\tau_j+1$ over $\varphi^{-1}(\psi^{-1}(T_{\theta}))$. The resulting family $\mathscr{R}_{\tau^{\prime}_j}$ of reparametrizations of the curve $\gamma^{\prime}$ over the set $\psi^{-1}(T_{\theta})$ will then be $(N_{\star},\epsilon)$-admissible up to time $\tau_j^{\prime}+1$. To construct this family, we do the same thing as in the Case B of wild segments. Indeed, the sub-segment $\big(\hat{g}^{t_{i-1}+1+\tau_j+1}(\hy),\dots,\hat{g}^{t_{i-1}+1+\tau_j^{\prime}}(\hy)\big)$ is a wild segment inside the rectangle segment. We then obtain the following estimate:
\begin{equation}
	|\mathscr{R}_{\tau^{\prime}_j}| \leq \Upsilon^{\lceil (\tau_j^{\prime}-\tau_j+1)/N_{\star}\rceil} \norm{dg}^{(\tau_j^{\prime}-\tau_j+1)/r}\norm{dg}^{N_{\star}/r}2\norm{dg}^{\tau_j^{\prime}-\tau_j+1}|\mathscr{R}_{\tau_j}|.
	\label{eq:CardinalReparamCaseA1}
\end{equation}

Let us now move onto the last case. Suppose that we dispose of a family of reparametrizations $\mathscr{R}_{\tau^{\prime}_j}$ of the curve $\gamma^{\prime}$, which is $(N_{\star},\epsilon)$-admissible up to time $\tau_j^{\prime}+1$ over the set $\psi^{-1}(T_{\theta})$. We will set:
\begin{equation*}
	\mathscr{R}_{\tau_{j+1}} := \{\varphi\circ\tilde{\varphi}, \ \varphi \in \mathscr{R}_{\tau^{\prime}_j}, \ \tilde{\varphi} \in \mathscr{R}(\varphi,\tau_{j+1})\}
\end{equation*}
where $\mathscr{R}(\varphi,\tau_{j+1})$ is a family of reparametrization of the curve $g^{\tau^{\prime}_j+1}\circ\gamma^{\prime}\circ\varphi$ which is $(N_{\star},\epsilon)$-admissible up to time $\tau_{j+1}-\tau^{\prime}_j-1$ over $\varphi^{-1}(\psi^{-1}(T_{\theta}))$. The resulting family $\mathscr{R}_{\tau_{j+1}}$ of reparametrizations of the curve $\gamma^{\prime}$ over the set $\psi^{-1}(T_{\theta})$ will then be $(N_{\star},\epsilon)$-admissible up to time $\tau_{j+1}$.

Fix then $\varphi \in \mathscr{R}_{\tau^{\prime}_j}$. Let $\tilde{\gamma} := g^{\tau_j^{\prime}+1} \circ \gamma^{\prime}\circ \varphi$ and let $\tilde{T}_{\theta} :=\varphi^{-1}(\psi^{-1}(T_{\theta}))$. Recall that $\pi(\hat{U}_1)\subset \cup R_i$, the family of rectangles obtained by applying Proposition \ref{prop:MarkovRectangles} in Step \stepref{step:U1}. Recall also that for any $y \in \gamma(T_{\theta})$, the sub-segment $(\hat{g}^{t_{i-1}+1+\tau_j^{\prime}+1}(\hy),\dots,\hat{g}^{t_{i-1}+1+\tau_{j+1}}(\hy))$ is contained in $\hat{U}_1$ (here $\hy$ is again the unique lift of $y$ belonging to $\hat{\gamma}(\hat{T})$). We then obtain that $\tilde{\gamma}$ is an $(r,\epsilon)$-curve and that $\tilde{\gamma}(\tilde{T}_{\theta})\subset \cup R_i$. Since $\epsilon<2\sigma(1+K)$ by \eqref{eq:ChoiceOfSigma}, we can apply Proposition \ref{prop:IntersectionCurveRectangle} to the curve $\tilde{\gamma}$ with the constant $C:=2(1+K)$. Recall the definition of $\text{Cross}(\tilde{\gamma})$ and $\text{Fold}(\tilde{\gamma})$ from Definition \ref{def:Crossing/TurningIntervals}, the set of crossing, respectively folding $\tilde{\gamma}$-intervals. Let $\mathcal{C}_1 \subset \text{Cross}(\tilde{\gamma})$ and $\mathcal{C}_2 \subset \{1,\dots,L\}$ be the two collections given by Proposition \ref{prop:IntersectionCurveRectangle}.

We start by dealing with the case of crossing $\tilde{\gamma}$-intervals. Let $I = [s,t] \in \mathcal{C}_1$ such that $I\cap\tilde{T}_{\theta} \ne \emptyset$. Denote by $\tilde{\gamma}_I$ the curve $\tilde{\gamma}$ restricted to the interval $I$. It defines then a map $\tilde{\gamma}_{I} : I=[s,t] \subset [0,1] \rightarrow M$. Let $i_0 \in \{1,\dots,L\}$ be the integer such that $\tilde{\gamma}_I \subset R_{i_0}$. Pick any $y\in \tilde{\gamma}_{I}(\tilde{T}_{\theta})$ and let $\hy$ be its lift belonging to $\hat{U}_1$. Note that for any $m=0,\dots,\tau_{j+1}-\tau^{\prime}_j-1$, we have $\hat{g}^{m}(\hy) \in \hat{U}_1$, and then $g^m(y) \in \cup_i R_i$. Let then $i_m \in \{1,\dots,L\}$ be an integer such that $g^m(y) \in R_{i_m}$ (which may not be unique, since the rectangles may intersect). See Figure \ref{fig:ProofTopologicalIntersection} for a visual proof of the following Claim.

\begin{claim}
	For any $m=0,\dots,\tau_{j+1}-\tau^{\prime}_j-1$, the following holds:
	\begin{equation*}
		g^m\big(\text{Int}(\tilde{\gamma}_{I})\big) \subset \tilde{R}_{i_m}.
	\end{equation*}
	\label{claim:Step8CaseCrossing}
\end{claim}
\begin{proof}
	The proof goes by induction. We already have that $\tilde{\gamma}_{I}\subset R_{i_0}$. Suppose that $g^{m-1}(\tilde{\gamma}_{I})\subset R_{i_{m-1}}$ for some $0\leq j<\tau_{j+1}-\tau^{\prime}_j-1$. Since $\hat{g}^{m-1}(\hy)\in \hat{U}_1$ and since $g(R_{i_{m-1}})\cap R_{i_m}\ne \emptyset$, the $us$-rectangles $\tilde{R}_{i_{m-1}}$ and $\tilde{R}_{i_m}$ satisfy Property \ref{property:MR5}:
	\begin{equation*}
		g(\tilde{R}_{i_{m-1}}) \cap \partial^u\tilde{R}_{i_m} = \emptyset.
	\end{equation*}
	It implies that the curve $g^m(\tilde{\gamma}_I)$, which is included in $g(\tilde{R}_{i_{m-1}})$ cannot cross $\partial^u\tilde{R}_{i_m}$. Hypothesis \ref{hypothesis:Unstable} also implies that $g^m(\tilde{\gamma}_I)$ cannot cross $\partial^{s,l}\tilde{R}_{i_m} \subset \Ws_{loc}(\hy_{i_m}^l)$ for $l=1,2$ neither. Since $g^m(\tilde{\gamma}_{I})$ meets $\tilde{R}_{i_m}$ but does not cross any of its boundaries, $g^m(\tilde{\gamma}_{I})$ has to be contained in $\tilde{R}_{i_m}$. This concludes the induction. We emphasize that the important property that displays the curve $\tilde{\gamma}_{I}$ is its connectedness.
\end{proof}

Combining Property \ref{property:MR4} and Lemma \ref{lem:DiameterRectangle}, for any $i \in \{1,\dots,L\}$ we have $\text{diam}(\tilde{R}_i)\leq \epsilon$, by \eqref{eq:ChoiceOfSigma}. This implies in particular that for any $I \in \mathcal{C}_1$, the set $\tilde{\gamma}_I(\tilde{T}_{\theta})$ is covered by a single one $(g,\tau_{j+1}-\tau^{\prime}_j-1,\epsilon)$-Bowen ball. We can now apply Corollary \ref{cor:AdmissibleReparam} to conclude that there exists a family of reparametrization $\mathscr{R}(\varphi,\tau_{j+1})^{\text{Cross}}$ of the curve $\tilde{\gamma}$ over the set $\tilde{T}_{\theta}\cap\text{Cross}(\tilde{\gamma})$, which is $(N_{\star},\epsilon)$-admissible up to time $\tau_{j+1}-\tau_j^{\prime}-1$ such that:
\begin{equation*}
	|\mathscr{R}(\varphi,\tau_{j+1})^{\text{Cross}}|\leq |\cC_1|\Upsilon^{\lceil (\tau_{j+1}-\tau_j^{\prime}-1)/N_{\star}\rceil}\norm{dg}^{(\tau_{j+1}-\tau_j^{\prime}-1)/r}\norm{dg}^{N_{\star}/r}.
\end{equation*}
For the intervals $[0,s_1]$ and $[s_2,1]$ given by Proposition \ref{prop:IntersectionCurveRectangle}, since they are contained in some $us$-rectangle and are connected, we can deal with them as for $\tilde{\gamma}$-crossing intervals. This gives us $\mathscr{R}(\varphi,\tau_{j+1})^{\text{Extrem}}$, a family of reparametrizations of $\tilde{\gamma}$ over the set $\tilde{T}_{\theta}\cap\big([0,s_1]\cup[s_2,1]\big)$, which is $(N_{\star},\epsilon)$-admissible up to time $\tau_{j+1}-\tau_{j}^{\prime}-1$ and which satisfies the estimate:
\begin{equation*}
	|\mathscr{R}(\varphi,\tau_{j+1})^{\text{Extrem}}|\leq 2\Upsilon^{\lceil (\tau_{j+1}-\tau_j^{\prime}-1)/N_{\star}\rceil}\norm{dg}^{(\tau_{j+1}-\tau_j^{\prime}-1)/r}\norm{dg}^{N_{\star}/r}.
\end{equation*}

We conclude Case C by dealing with the case of folding $\tilde{\gamma}$-intervals. This part of the argument contains a crucial idea, which is a bit subtle. Let us then explain it before going into the details. The important property we use in the case of crossing $\tilde{\gamma}$-intervals was their connectedness. Here, since we do not have a bound on the number of folding $\tilde{\gamma}$-intervals, we are forced to work with several intervals, but such that their images by $\tilde{\gamma}$ are contained in some fixed rectangle. The set of such intervals becomes then disconnected. However the iterates by $g^m$ of the curve $\tilde{\gamma}$ for $m=0,\dots,N_{\star}$ are curves which are still small compared to the size of the local stable manifolds $\Ws_{loc}(\hy_i^l)$ for $i\in \{1,\dots,L\}$ and $l=1,2$. We are then able to use again Hypothesis \ref{hypothesis:Unstable} combined with Property \ref{property:MR5}, indeed, if $g^m(\tilde{\gamma})$ intersects two different rectangles of the family $\tilde{\mathcal{R}}$, since it is small, $g^m(\tilde{\gamma})$ has to intersect some $\Ws_{loc}(\hy_i^l)$. For $m\geq N_{\star}$ iterates, we have a bad control on the size of the curve $g^m(\tilde{\gamma})$ and the previous argument cannot be applied. To bypass that problem, we use Yomdin Theorem after $N_{\star}$ iterates to be able to work again with small curves intersecting some rectangle of $\mathcal{R}$, and repeat the previous argument.

Let $i \in \cC_2$ such that the set $\tilde{\gamma}(\tilde{T}_{\theta})\cap R_i$ is non-empty. Let $q_1,q_2\geq 0$ be two integers such that $\tau_{k+1}-\tau_k^{\prime}-1 = q_1N_{\star} + q_2$, with $q_2<N_{\star}$. For $q=0,\dots,q_1 +1$, we are going to build by induction families of reparametrizations $\mathscr{I}_q$ (which depends on $i$) of $\tilde{\gamma}$ over the set $\tilde{T}_{\theta}\cap \tilde{\gamma}^{-1}(R_i)$, which will be $(N_{\star},\epsilon)$-admissible up to time $qN_{\star}$, with the following property. For each $\varphi^{\prime} \in \mathscr{I}_q$, the set $g^m\circ\tilde{\gamma}\circ\varphi^{\prime}\big((\varphi^{\prime})^{-1}(\tilde{\gamma}^{-1}(R_i))\cap\tilde{T}_{\theta}\big)$ will be contained in a single one rectangle of the family $\tilde{\mathcal{R}}$ for any $m=0,\dots,qN_{\star}$ (the rectangle depends on the iterate $m$). Taking the union of the families $\mathscr{I}_{q_1+1}$ for each $i \in \cC_2$ will give us $\mathscr{R}(\varphi,\tau_{j+1})^{\text{Fold}}$, a family of reparametrization of the curve $\tilde{\gamma}$ over the set $\tilde{T}_{\theta}\cap\text{Fold}(\tilde{\gamma})$ which is $(N_{\star},\epsilon)$-admissible up to time $\tau_{j+1}-\tau_j-1$.

Set $\mathscr{I}_0 = \{\text{Id}\}$. Suppose that such a family $\mathscr{I}_{q-1}$ exists for some $q \in \{1,\dots,q_1\}$. We are going to build now the family $\mathscr{I}_q$. As explained previously, we proceed by concatenation. Let $\varphi^{\prime} \in \mathscr{I}_{q-1}$. Denote $\kappa := g^{(q-1)N_{\star}}\circ\tilde{\gamma}\circ\varphi^{\prime}$. By the hypothesis of induction, $\kappa$ is an $(r,\epsilon)$-curve. Let $G_i := (\varphi^{\prime})^{-1}\big(\tilde{T}_{\theta}\cap \tilde{\gamma}^{-1}(R_i)\big)$, the set over which we are going to reparametrize. By the induction hypothesis, there exists $i_0 \in \{1,\dots,L\}$ such that $\kappa(G_i) \subset \tilde{R}_{i_0}$. Let $y \in \kappa(G_i)$ and let $\hy$ be its lift to $\hat{U}_1$. Recall that $\hat{g}^m(\hy) \in \hat{U}_1$ for $m=0,\dots,N_{\star}$, so there exists an integer $i_m \in \{1,\dots,L\}$ such that $g^m(y) \in R_{i_m}$ for those times $m$ (again, $i_m$ may not be unique).

\begin{claim}
	For any $m=0,\dots,N_{\star}$, the following holds:
	\begin{equation*}
		g^m\circ\kappa(G_i) \subset \tilde{R}_{i_m}.
	\end{equation*}
	\label{claim:Step8CaseFolding}
\end{claim}
\begin{proof}
	Let $\big(\Psi_{\hz_m},q_{\tilde{\epsilon}}(\hz_m)\big)$ be an adapted chart to the rectangle $R_{i_m}$. First, remark that the curve $g^m\circ\kappa$ is contained in the domain $\Psi_{\hz_m}\big(R(q_{\tilde{\epsilon}}(\hz_m))\big)$. Indeed, $g^m\circ\kappa$ intersects $R_{i_m}$ and for any $t_1,t_2 \in [0,1]$ we have:
	\begin{equation*}
		d\big(g^m\circ\kappa(t_1), g^m\circ\kappa(t_1)\big)\leq \norm{dg}^{N_{\star}}d(\kappa(t_1),\kappa(t_2)) <\norm{dg}^{N_{\star}}\epsilon.
	\end{equation*}
	This ensures that $g^m\circ\kappa([0,1])\subset\Psi_{\hz_j}\big(R(q_{\tilde{\epsilon}}(\hz_j))\big)$, for $m=0,\dots,N_{\star}$, by the choices of $\epsilon$, and then $\sigma$, in Step \stepref{step:U1} and Step \stepref{step:U2}.
	
	We proceed now by induction again. We already have that $\kappa(G_i)\subset R_{i_0}\subset \tilde{R}_{i_0}$. Suppose that $g^{m-1}\circ\kappa(G_i)\subset \tilde{R}_{i_{m-1}}$ for some $0\leq j<N_{\star}$. Note that we cannot apply directly the same argument as in Claim \ref{claim:Step8CaseCrossing} since the set $g^m\circ\kappa(G_i)$ is no longer connected.
	
	The proof goes by contradiction. We know that $g^m\circ\kappa(G_i)\subset \cup_i R_i$. Suppose then that there exists a point $z \in g^m\circ\kappa(G_i)$ such that $z \notin \tilde{R}_{i_m}$. Recall that the stable boundaries $\partial^{s,l}R_{i_j}$ are contained inside local stable manifolds $\Ws_{loc}(\hy_{i_m}^l)$. These curves are Lipschitz graphs in the chart $\big(\Psi_{\hz_{i_m}},q_{\tilde{\epsilon}}(\hz_{i_m})\big)$ of size $q_{\tilde{\epsilon}}(\hz_{i_m})$. Denote them by $V^{s,l}$. Consider the open set $Q\subset \bbR^2$ delimited by the Jordan curve composed by the union of the two graphs $V^{s,1}$ and $V^{s,2}$ and the vertical boundaries of the chart $\{\pm q_{\tilde{\epsilon}}(\hz_{i_m})\} \times [-q_{\tilde{\epsilon}}(\hz_{i_m}),q_{\tilde{\epsilon}}(\hz_{i_m})]$. Since $g^m\circ\kappa(G_i)\cap R_{i_m}\ne \emptyset$, the image in the chart $\big(\Psi_{\hz_{i_m}},q_{\tilde{\epsilon}}(\hz_{i_m})\big)$ of the curve $g^m\circ\kappa$ has to intersect $Q$. Since $\hat{g}^{m-1}(\hy)\in \hat{U}_1$ and since $g(R_{i_{m-1}})\cap R_{i_m}\ne \emptyset$, the $us$-rectangles $\tilde{R}_{i_{m-1}}$ and $\tilde{R}_{i_m}$ satisfy Property \ref{property:MR5}:
	\begin{equation*}
		g(\tilde{R}_{i_{m-1}}) \cap \partial^u\tilde{R}_{i_m} = \emptyset.
	\end{equation*}
	Combining this and the fact that $z \notin \tilde{R}_{i_m}$, we then obtain that $g^m\circ\kappa$ intersects also $R\big(q_{\tilde{\epsilon}}(\hz_{i_m})\big)\backslash\overline{Q}$ in the chart $\big(\Psi_{\hz_{i_m}},q_{\tilde{\epsilon}}(\hz_{i_m})\big)$. It implies that $g^m\circ\kappa$ has to cross one of the $V^{s,l}$, which contradicts Hypothesis \ref{hypothesis:Unstable}. Indeed, we use the connectedness of $g^m\circ\kappa$ and the fact that it has small size to bypass the fact that $g^m\circ\kappa(G_i)$ is not connected.
\end{proof}

Combining Claim \ref{claim:Step8CaseFolding} and the fact that $\diam(\tilde{R}_i)<\epsilon$ for any $i \in \{1,\dots,L\}$, we can apply Yomdin Theorem \ref{thm:Yomdin} with the map $g^{N_{\star}}$ and the curve $\kappa$. It gives us a family of reparametrization of $(r,\epsilon)$-curves over the set of $t \in [0,1]$ such that $g^{N_{\star}}\circ\kappa(t) \in \tilde{R}_{i_{N_{\star}}}$ with cardinality smaller than $\Upsilon\norm{dg}^{N_{\star}/r}$. Repeating this for any $\varphi^{\prime} \in \mathscr{I}_{q-1}$ gives us the family $\mathscr{I}_q$ by concatenation. As explained, taking the union of the families $\mathscr{I}_{q_1+1}$ constructed for each $i \in \cC_2$ gives the family of reparametrizations $\mathscr{R}(\varphi,\tau_{k+1})^{\text{Fold}}$ of the curve $\tilde{\gamma}$ over the set $\tilde{T}_{\theta}\cap\text{Fold}(\tilde{\gamma})$, which is $(N_{\star},\epsilon)$-admissible up to time $\tau_{k+1}-\tau_k^{\prime}-1$. We then have the following bound on its cardinality:
\begin{equation*}
	|\mathscr{R}(\varphi,\tau_{k+1})^{\text{Fold}}|\leq |\cC_2|\Upsilon^{\lceil (\tau_{k+1}-\tau_k^{\prime}-1)/N_{\star}\rceil}\norm{dg}^{(\tau_{k+1}-\tau_k^{\prime}-1)/r}\norm{dg}^{N_{\star}/r}.
\end{equation*}

To conclude with Case C, our goal was to construct $\mathscr{R}(\varphi,\tau_{k+1})$, where $\varphi \in \mathscr{R}_{\tau_k^{\prime}}$, a family of reparametrization of the curve $\gamma^{\prime} := g^{t_{i-1}}\circ\gamma\circ\psi$, which is $(N_{\star},\epsilon)$-admissible up to time $\tau_j^{\prime}+1$ over the set $\psi^{-1}(T_{\theta})$. We have built three families of reparametrizations of $\tilde{\gamma}$, all $(N_{\star},\epsilon)$-admissible up to time $\tau_{j+1}-\tau_j^{\prime}-1$. Let us describe which set they reparameterize and what are their cardinalities.
\begin{itemize}[label={--}]
	\item $\mathscr{R}(\varphi,\tau_{j+1})^{\text{Cross}}$ reparameterizes over the set $\tilde{T}_{\theta}\cap\text{Cross}(\tilde{\gamma})$ and has cardinality:
	\begin{equation*}
		|\mathscr{R}(\varphi,\tau_{j+1})^{\text{Cross}}|\leq |\cC_1|\Upsilon^{\lceil (\tau_{j+1}-\tau_j^{\prime}-1)/N_{\star}\rceil}\norm{dg}^{(\tau_{j+1}-\tau_j^{\prime}-1)/r}\norm{dg}^{N_{\star}/r}.
	\end{equation*}
	\item $\mathscr{R}(\varphi,\tau_{j+1})^{\text{Extrem}}$ reparameterizes over the set $\tilde{T}_{\theta}\cap\big([0,s_1]\cup[s_2,1]\big)$ and has cardinality:
	\begin{equation*}
		|\mathscr{R}(\varphi,\tau_{j+1})^{\text{Extrem}}|\leq 2\Upsilon^{\lceil (\tau_{j+1}-\tau_j^{\prime}-1)/N_{\star}\rceil}\norm{dg}^{(\tau_{j+1}-\tau_j^{\prime}-1)/r}\norm{dg}^{N_{\star}/r}.
	\end{equation*}
	\item $\mathscr{R}(\varphi,\tau_{k+1})^{\text{Fold}}$ reparameterizes over the set $\tilde{T}_{\theta}\cap\text{Fold}(\tilde{\gamma})$ and has cardinality:
	\begin{equation*}
		|\mathscr{R}(\varphi,\tau_{j+1})^{\text{Fold}}|\leq |\cC_2|\Upsilon^{\lceil (\tau_{j+1}-\tau_j^{\prime}-1)/N_{\star}\rceil}\norm{dg}^{(\tau_{j+1}-\tau_j^{\prime}-1)/r}\norm{dg}^{N_{\star}/r}.
	\end{equation*}
\end{itemize}
By Proposition \ref{prop:IntersectionCurveRectangle}, we know that:
\begin{equation*}
	\tilde{\gamma}^{-1}\Big(\tilde{\gamma}\bigcap\big(\cup_i R_i\big)\Big) \subset \text{Cross}(\tilde{\gamma})\cup \text{Fold}(\tilde{\gamma}) \cup [0,s_1] \cup [s_2,1].
\end{equation*}
This implies that taking the union of these three families of reparametrization gives us the desired family $\mathscr{R}(\varphi,\tau_{j+1})$. This concludes then the construction of $\mathscr{R}_{\tau_{j+1}}$, which has the following bound on its cardinality:
\begin{equation}
	|\mathscr{R}_{\tau_{j+1}}| \leq \big(|\cC_1| + 2 + |\cC_2|\big) \Upsilon^{\lceil (\tau_{j+1}-\tau_j^{\prime}-1)/N_{\star}\rceil}\norm{dg}^{(\tau_{j+1}-\tau_j^{\prime}-1)/r}\norm{dg}^{N_{\star}/r}|\mathscr{R}_{\tau_j^{\prime}}|.
	\label{eq:CardinalReparamCaseA2}
\end{equation}

By the induction process, we then have defined $\mathscr{R}(\psi,t_i)$, for $\psi \in \mathscr{R}^{\theta}_{t_{i-1}}$. We recall that it is a family of reparametrization of the curve $g^{t_i-1}\circ\gamma\circ\psi$ over the set $\psi^{-1}(T_{\theta})$, which is $(N_{\star},\epsilon)$-admissible up to time $t_i-t_{i-1}$. By \eqref{eq:CardinalReparamCaseA1} and \eqref{eq:CardinalReparamCaseA2}, it satisfies:
\begin{equation*}
	|\mathscr{R}(\psi,t_i)| \leq 2^{c_i^0}\norm{dg}^{2c_i^0}\big(|\cC_1| + 2 + |\cC_2|\big)^{c_i^0} \Upsilon^{\lceil (t_i-t_{i-1})/N_{\star}\rceil}\norm{dg}^{(t_i-t_{i-1})/r}\norm{dg}^{N_{\star}/r}.
\end{equation*}
Doing this for every $\psi \in \mathscr{R}^{\theta}_{t_{i-1}}$, we obtain the desired family $\mathscr{R}^{\theta}_{t_i}$ which satisfies:
\begin{equation}
	|\mathscr{R}^{\theta}_{t_i}| \leq 2^{c_i^0}\norm{dg}^{2c_i^0}\big(|\cC_1| + 2 + |\cC_2|\big)^{c_i^0} \Upsilon^{\lceil (t_i-t_{i-1})/N_{\star}\rceil}\norm{dg}^{(t_i-t_{i-1})/r}\norm{dg}^{2N_{\star}/r} |\mathscr{R}^{\theta}_{t_{i-1}}|.
	\label{eq:CardinalReparamCaseA}
\end{equation}
This then concludes Case C and Step \stepref{step:U8}.

\parbreak\UStep[A bound on the cardinality of the reparametrizations.]\label{step:U9}
For each $n\geq N_k$ and each decomposition type $\theta := \big[\big(v_0,\dots,v_{n-1}\big),\big((t_0,b_0),\dots,(t_m,b_m)\big)\big]$, Step \stepref{step:U8} gave us families of reparametrizations $\mathscr{R}^{\theta}_{t_i}$ of the curve $\gamma$ over $T_{\theta}$, for each time $t_i$. These families are $(N_{\star},\epsilon)$-admissible up to time $t_i$. Their cardinality depends on the class of the orbit segment $t_i-t_{i-1}$ and the length of each class of segment is given by Claim \ref{claim:Step5}. Combining this with \eqref{eq:CardinalReparamCaseA}, \eqref{eq:CardinalReparamCaseB}, \eqref{eq:CardinalReparamCaseC}, the estimates from Proposition \ref{prop:IntersectionCurveRectangle} ($\gamma$ is an $(r,\epsilon)$-curve with $\epsilon<2\sigma(K+1)$) and the fact that $m$ is bigger than the number of rectangle segments, we obtain the following bound on the cardinality of the family $\mathscr{R}_n^{\theta}$:
\begin{equation}
	\begin{aligned}
		|\mathscr{R}_n^{\theta}|\leq& \exp\Big(Nh(1+\delta^2)\big((1-\alpha)n + 2\eta n\big)\Big)\\
		&\times (2\norm{dg})^{2\eta n}\\
		&\times \big(2\norm{dg}(24(K+1)+2)\big)^{m+\sum_ic_i^0}\\
		&\times \Upsilon^{n/N_{\star}}\norm{dg}^{n/r}\norm{dg}^{mN_{\star}/r}.
	\end{aligned}
	\label{eq:BoundReparamTheta}
\end{equation}
The first term comes from the small entropy segments, the second from the wild segments, the third from the rectangle segments and the last one from Yomdin Theorem to bound the number of reparametrizations. We obtain the family $\mathscr{R}_n$ by taking the union of the $\mathscr{R}_n^{\theta}$ over all the decomposition types $\theta$. Note now that $\sum_i c_i^0$ is the number of letter $0$ in the word $v$, so it is smaller than $\eta n$. Note also that $m$ is the total number of different segments in the decomposition type and is smaller than $4\eta n$ (see the proof of Claim \ref{claim:Step6}). Combining this with Claim \ref{claim:Step6}, we obtain the following bound:
\begin{equation}
	\begin{aligned}
		|\mathscr{R}_n| \leq& \exp\big(n(5\eta + 2H(5\eta))\big)\\
		&\times \exp\Big(Nh(1+\delta^2)\big((1-\alpha)n + 2\eta n\big)\Big)\\
		&\times (2\norm{dg})^{2\eta n}\\
		&\times \big(2\norm{dg}(24(K+1)+2)\big)^{5\eta n}\\
		&\times \Upsilon^{n/N_{\star}}\norm{dg}^{n/r}\norm{dg}^{4\eta nN_{\star}/r}.
	\end{aligned}
	\label{eq:BoundMainReparam}
\end{equation}
Using now Proposition \ref{prop:EntropyReparametrizations}, we get the following bound on the entropy of the $\mu_k$:
\begin{equation*}
	\begin{aligned}
		h(f,\mu_k,\epsilon) = \frac{1}{N}h(g,\mu_k,\epsilon) &\leq \lim_{n\rightarrow +\infty}\frac{1}{Nn} \log |\mathscr{R}_n|\\
		&\leq 5\eta + 2H(5\eta)\\
		&+ h(1+\delta^2)\big((1-\alpha) +2\eta\big)\\
		&+2\eta \log(2\norm{df})\\
		&+5\eta \log(2\norm{df}) + 5\eta\log\big(24(K+1)+2\big)\\
		&+\frac{\log \Upsilon}{N_{\star}} + \frac{\log\norm{df}}{r} + 4\eta N_{\star}/r \log\norm{df}.
	\end{aligned}
\end{equation*}

\parbreak\UStep[End of the proof.]\label{step:U10}
We now analyze the bound we obtained in Step 9. By \eqref{eq:ChoiceOfNStar}, we have:
\begin{equation*}
	\frac{\log\Upsilon}{N_{\star}} < \delta/8.
\end{equation*}
Since we obviously have $\log\norm{2df}>1$ we have that $5\eta<7\eta\log\norm{2df}$ and then by \eqref{eq:ChoiceOfEta} we have:
\begin{equation*}
	\begin{array}{ccc}
		5\eta<\delta/8, & 2H(5\eta)<\delta/8, & (2+5)\eta\log\norm{2df}<\delta/8, \\
		4\eta N_{\star}/r \log\norm{df}< \delta/8, & \text{and}  & 5\eta\log\big(24(K+1)+2\big)<\delta/8.
	\end{array}
\end{equation*}
Finally, by \eqref{eq:ChoiceOfDelta} and \eqref{eq:ChoiceOfEta}, we have:
\begin{equation*}
	h(1+\delta^2)\big((1-\alpha) + 2\eta\big) = h(1-\alpha) + h\delta^2(1-\alpha) + h(1+\delta^2)2\eta < h(1-\alpha) + \delta/8 + \delta/8.
\end{equation*}
Gathering everything together, we obtain the following bound on the entropy of the $\mu_k$ for $k\geq k_{\star}$:
\begin{equation*}
	h(f,\mu_k,\epsilon) < \frac{\log\norm{df}}{r} + h(1-\alpha) + \delta.
\end{equation*}
On the other hand, by combining \eqref{eq:MinorationEntropy}, \eqref{eq:ChoiceOfDelta2} and \eqref{eq:GreaterExponent}, we also have for $k\geq0$ large enough:
\begin{equation*}
	\begin{aligned}
		h(f,\mu_k,\epsilon) &> h(f,\mu_k) - \frac{\lambda(f)}{r} (1+\delta)\\
		&>\liminf_{k\rightarrow+\infty} h(f,\mu_k)-\delta-\frac{\log\norm{df}}{r}\frac{(1+\delta)}{(1-\delta)}\\
		&>(h+\kappa\delta)(1-\alpha)+2\lambda(f)/r -\delta-\frac{\log\norm{df}}{r}\frac{(1+\delta)}{(1-\delta)}\\
		&>(h+\kappa\delta)(1-\alpha) + \frac{\log\norm{df}}{r}\Big(\frac{2}{1+\delta}-\frac{(1+\delta)}{(1-\delta)}\Big) - \delta\\
		&>(h+\kappa\delta)(1-\alpha)+\frac{\log\norm{df}}{r}\big(1-8\delta\big)-\delta
	\end{aligned}
\end{equation*}
where the last inequality is a simple computation true whenever $0<\delta<1/2$. Finally, by \eqref{eq:ChoiceOfKappa}, we also know that:
\begin{equation*}
	\kappa>\frac{(8\log\norm{df}/r)-2}{1-\alpha} \iff \kappa(1-\alpha) - 8\frac{\log\norm{df}}{r}-2>0.
\end{equation*}
This then implies that:
\begin{equation*}
	(h+\kappa\delta)(1-\alpha)+\frac{\log\norm{df}}{r}\big(1-8\delta\big)-\delta > \frac{\log\norm{df}}{r} + h(1-\alpha) +\delta
\end{equation*}
which is impossible.

We then conclude by contradiction that Hypothesis \ref{hypothesis:Unstable} does not hold true. For any $k$ large enough and for any $\hx \in \hat{F}_k$, there must then exist an $m\geq0$, $i \in \{1,\dots,L\}$ and $l \in \{1,2\}$ such that:
\begin{equation*}
	f^m(\Wu_{loc}(\hx)) \cap \Ws_{loc}(\hy_i^l) \ne \emptyset.
\end{equation*}
Since $\hmu_k(\hat{F}_k)>0$ and the number of rectangles in the family $\mathcal{R}$ is finite, there exists $i_{\star} \in \{1,\dots,L\}$ and $l_{\star}\in\{1,2\}$ such that $\hmu_k(\hat{Y}_k)>0$ where:
\begin{equation*}
	\hat{Y}_k := \{\hx \in \hat{F}_k, \ \exists m\geq 0, \ f^m(\Wu_{loc}(\hx))\cap\Ws_{loc}(\hy_{i_{\star}}^{l_{\star}}) \ne \emptyset\}.
\end{equation*}
Recall that $V^u(\hx)\subset M_f$ is the unstable set of $\hx$ from Definition \ref{def:LocalInvariantSet}. Recall also that $\{\hmu_k^{u,\hx}\}_{\hx}$ is the family of conditional measures defined in Step \stepref{step:U7}. Up to passing to a set of the same $\hmu_k$ measure as $\hat{Y}_k$, for any $\hx \in \hat{Y}_k$ we have that:
\begin{equation*}
	\hmu_k^{u,\hx}\Big(\{\hy \in V^u(\hf^m(\hx)), \ \hy \ \text{has a well-defined unstable manifold}\}\Big)>0.
\end{equation*}
We then obtain that the set:
\begin{equation*}
	\hat{X}_k := \{\hx\in M_f, \ \Wu(\hx)\cap\Ws_{loc}(\hy_{i_{\star}}^{l_{\star}})\ne \emptyset\}
\end{equation*}
has positive $\hmu_k$-measure.

The family of $us$-rectangles $\mathcal{R}$ and the set $\hat{X}_k$, for $k\geq k_{\star}$ are the desired objects from Proposition \ref{prop:UnstableStableIntersection}. This concludes then the proof of Proposition \ref{prop:UnstableStableIntersection}.

\subsection{Proof of Proposition \ref{prop:IntersectionsUnstableMaxEntropy}}
We move on to the proof of Proposition \ref{prop:IntersectionsUnstableMaxEntropy}. It is a Corollary of Proposition \ref{prop:UnstableStableIntersection}. Before moving into the details, let us explain quickly the main points. We are going first to further decompose the measure $\hnu_1$ into two parts, the first one, $\hnu_1^{\star}$, satisfies that almost every measure in its ergodic decomposition is homoclinically related to $\hat{\cO}$, and the second part is the rest. This gives the decomposition of the measure $\hnu$ into $\hnu_1^{\star}$ and $\hnu_0^{\star}:=1-\hnu_1^{\star}$. The idea is then to apply Proposition \ref{prop:UnstableStableIntersection} to this decomposition. We conclude that unstable manifolds of typical $\mu_k$-points cross some stable manifold containing the stable boundaries of the rectangles associated to $\hnu_1^{\star}$.

Then, since $\hnu_1^{\star}$ is hyperbolic of saddle type with entropy large enough, we can use Theorem \ref{thm:LipschitzHolonomiesLamination} and Theorem \ref{thm:Sard} to find transverse intersections between  $\mu_k$-unstable manifold and the stable manifolds of a set of positive measure of some ergodic measure which is part of the ergodic decomposition of $\hnu_1^{\star}$. We conclude since the measure $\hnu_1^{\star}$ has almost all its ergodic components homoclinically related to $\hat{\cO}$.

\parbreak\EStep[Obtaining topological intersections between $\mu_k$-unstable manifolds and the $\hnu_1^{\star}$-rectangles.]\label{step:HUI1}
Recall that we have the following convergence:
\begin{equation*}
	\mu_k \rightarrow \nu
\end{equation*}
and that there exists $\alpha \in (0,1]$ such that $\hnu = \alpha\hnu_1 + (1-\alpha)\hnu_0$. Recall that $\hnu$-almost every ergodic component of $\hnu_1$ is hyperbolic of saddle type with entropy larger than $\lambda(f)/r$. We denote the ergodic decomposition of $\hnu_1$ by $\hnu_1 = \int \hnu_{1,\hx} \ d\hnu_1(\hx)$. Define the following set:
\begin{equation*}
	\hat{X} := \{\hx \in M_f, \ \hnu_{1,\hx}\simh \hat{\cO}\} = \{\hx \in \supp \hnu_1, \ \hx\simh\hat{\cO}\}.
\end{equation*}
Let $\alpha^{\prime} := \hnu(\hat{X})$. By the hypothesis on $\hnu_1$ we have $\hnu_1(\hx \in \supp\hnu_1, \ \hx\simh\hat{\cO})>\beta>0$. It implies that $\hnu(\hx \in \supp\hnu_1, \ \hx\simh\hat{\cO}) >\alpha\beta > 0$ and then $\alpha^{\prime}>\alpha\beta$. Define the following probability measure:
\begin{equation*}
	\hnu_1^{\star} := \frac{1}{\alpha^{\prime}}\int_{\hat{X}} \hnu_{1,\hx} \ d\hnu(\hx).
\end{equation*}
Note that this measure satisfies that $\hnu$-almost every of its ergodic components are hyperbolic of saddle type, with entropy larger than $\lambda(f)/r$, and are homoclinically related to $\hat{\cO}$. Denote the ergodic decomposition of $\hnu$ by $\hnu = \int \hnu_{\hx} \ d\hnu(\hx)$. Define also the probability measure:
\begin{equation*}
	\hnu_0^{\star} = \frac{1}{1-\alpha^{\prime}}\int_{\hat{X}^c} \hnu_{\hx} \ d\hnu(\hx).
\end{equation*}
These measures satisfy that:
\begin{equation*}
	\hnu = \alpha^{\prime}\hnu_1^{\star}+(1-\alpha^{\prime})\hnu^{\star}_0.
\end{equation*}
Moreover, combining the fact that $\alpha^{\prime}>\alpha\beta$ and the hypothesis on the entropy of ergodic components of $\hnu$, we get that $\hnu$-almost every ergodic component $\hnu^{\star}_{0,\hx}$ of $\hnu_0^{\star}$ satisfies the following bound:
\begin{equation*}
	h(f,\hnu^{\star}_{0,\hx}) < \frac{1}{1-\alpha\beta}\big(\liminf_{k\rightarrow +\infty} h(f,\mu_k) - 2\frac{\lambda(f)}{r}\big) < \frac{1}{1-\alpha^{\prime}}\big(\liminf_{k\rightarrow +\infty} h(f,\mu_k) - 2\frac{\lambda(f)}{r}\big).
\end{equation*}

We can then apply Proposition \ref{prop:UnstableStableIntersection} to this decomposition. We recall the definition of $Y^{\#}$ before Proposition \ref{prop:MarkovRectangles}. We obtain that there exists a finite family of points $\hy_1,\dots,\hy_L\in\supp\hnu_1^{\star}\cap\hat{Y}^{\#}$ and an integer $k_{\star}$ with the following properties. For $k\geq k_{\star}$, there exists a measurable set $\hat{X}_k$ of positive $\hmu_k$-measure and an integer $i(k) \in \{1,\dots,L\}$ such that for any $\hx \in \hat{X}_k$ we have:
\begin{equation*}
	\Wu(\hx) \cap \Ws_{loc}(\hy_{i(k)}) \ne \emptyset.
\end{equation*}

\parbreak\EStep[Going from topological to transverse intersections.]\label{step:HUI2}
The previous step gave us topological intersections between the $\hx \in \hat{X}_k$ and some stable manifolds of typical $\hnu_1^{\star}$-points. We are going to show that there exist actually transverse intersections.

\begin{claim}
	Pick any $\hx \in \hat{X}_k$. For any $i \in \{1,\dots,L\}$, there exists $\hat{m}_i$, an ergodic hyperbolic $\hf$-invariant measure of saddle type with positive entropy which satisfies the following. Suppose that $\Wu(\hx) \cap \Ws_{loc}(\hy_i)\ne \emptyset$. Then for any $\epsilon^{\prime}$ small enough, there exists a measurable set $\hat{A} \subset B(\hy_i,\epsilon^{\prime})$ with $\hat{m}_i(\hat{A})>0$ and such that for any $\hz \in \hat{A}$ we have:
	\begin{equation*}
		f^m\big(\Wu(\hx)\big) \pitchfork \Ws_{loc}(\hz)\ne \emptyset.
	\end{equation*}
	for some $m\geq 0$.
	\label{claim:TopologicalToTransverse}
\end{claim}
\begin{proof}
	Let $i \in \{1,\dots,L\}$ such that $\Wu(\hx) \cap \Ws_{loc}(\hy_i) \ne \emptyset$. Let $\gamma :[0,1] \rightarrow \Wu(\hx)$ be a $\cC^{r}$ curve which contains the intersection with $\Ws_{loc}(\hy_i)$ in its interior. Recall that $\hnu_1^{\star}$ can be decomposed as the following:
	\begin{equation*}
		\hnu_1^{\star} = \int \hnu^{\star}_{1,\hx} \ d\hnu_1^{\star}(\hx)
	\end{equation*}
	where each $\hnu^{\star}_{1,\hx}$ is an ergodic hyperbolic of saddle type $\hf$-invariant measure with entropy larger than $\lambda(f)/r$. Take $\hat{m}_i$ some measure in the ergodic decomposition of $\hnu_1^{\star}$ containing $\hy_i$ in its support. Denote by $\lambda^u(\hat{m}_i)$ the positive Lyapunov exponent of the ergodic measure $\hat{m}_i$. We have the following:
	\begin{equation*}
		h(\hf,\hat{m}_i) > \lambda(f)/r \implies h(\hf,\hat{m}_{i,l})/\lambda^u(\hat{m}_i) > 1/r
	\end{equation*}
	since we trivially have $\lambda(f)\geq \lambda^u(\hat{m}_i)$.
	
	Apply Theorem \ref{thm:LipschitzHolonomiesLamination} with the measure $\hat{m}_i$ and $\delta>0$ small enough so that $h(\hf,\hat{m}_i)/\lambda(f)-\delta>1/r$. There exists then some $\Lambda_{\hy_i}$ such that:
	\begin{equation*}
		\dim_H\big(\Wu_{loc}(\hy_i) \cap \Lambda_{\hy_i}\big) >1/r.
	\end{equation*}
	Since $\hy_i\in\hat{Y}^{\#}$ and is in particular recurrent in the same Pesin block, and since the length of $f^m(\Ws_{loc}(\hy_i))$ goes to zero, up to taking a big enough iterate $m\geq0$ we can suppose that:
	\begin{equation*}
		f^m(\Wu(\hx)) \cap \Ws_{loc}(\hy_i) \in B\big(\pi(\hy_i),\epsilon_{\star}\big).
	\end{equation*}
	for $\epsilon_{\star}>0$ as small as we want. Since $\hy_i\in \hat{Y}^{\#}$, it is accumulated on both sides of $\Wu_{loc}(\hy_i)$ by points belonging to $\Lambda_{\hy_i}$. Up to passing to a subset of $\Lambda_{\hy_i}$, we obtain by continuity that there exists some sub-curve of $f^m(\gamma)$ which intersects any leaf of the lamination $\mathscr{L}^s:=\Ws_{loc}\big(\Lambda_{\hy_i}\cap B(\pi(\hy_i),\epsilon_{\star})\cap\Wu_{loc}(\hy_i)\big)$. We still denote this sub-curve of $f^m(\gamma)$ by $\gamma$ to simplify the notations. By Theorem \ref{thm:LipschitzHolonomiesLamination}, $\mathscr{L}^s$ is a continuous lamination with $\cC^r$ leaves and with Lipschitz holonomies. We have the following bound on its transverse dimension:
	\begin{equation*}
		\text{\dj}(\mathscr{L}^s) \geq \dim_H\big(\Wu_{loc}(\hy^i_l)\cap \Lambda_{\hy^l_i}\big)>1/r
	\end{equation*}
	where the first inequality holds since $\Wu_{loc}(\hy_i^l)$ intersects transversely every leaf of $\mathscr{L}^s$. Recall now that $\mathscr{N}_{\gamma}(\mathscr{L}^s)$ is the lamination made of leaves of $\mathscr{L}^s$ which intersect $\gamma$ in a non-transverse way. By Theorem \ref{thm:Sard}, we have:
	\begin{equation*}
		\text{\dj}\big(\mathscr{N}_{\gamma}(\mathscr{L}^s)\big) \leq 1/r.
	\end{equation*}
	This implies that $\mathscr{L}^s$ cannot be contained in $\mathscr{N}_{\gamma}(\mathscr{L}^s)$.
	
	There exists then some $\hz \in \hat{\Lambda}_{\hy_i}\cap B(\hy_i,\epsilon_{\star})$ such that:
	\begin{equation*}
		f^m\big(\Wu(\hx)\big) \pitchfork \Ws_{loc}(\hz) \ne \emptyset.
	\end{equation*}
	By continuity of the local stable manifolds on $\hat{\Lambda}_{\hy_i}$, the same is true in a small neighborhood of $\hy^i$. Since $\hz$ is in the support of $\hat{m}_i$ restricted to some Pesin block, we obtain that there exists some set $\hat{A}$ of positive $\hat{m}_i$-measure and $\epsilon^{\prime}>0$ small enough satisfying the following. For any $\hz \in \hat{A}$, we have:
	\begin{equation*}
		f^m\big(\Wu(\hx)\big) \intert \Ws_{loc}(\hz)\ne \emptyset
	\end{equation*}
	and $\hat{A}\subset B(\hy_i,\epsilon^{\prime})$. This concludes the proof of Claim \ref{claim:TopologicalToTransverse}.
\end{proof}

\parbreak\EStep[End of the proof.]\label{step:HUI3}
We are now able to conclude the proof of Proposition \ref{prop:IntersectionsUnstableMaxEntropy}. Since by construction, $\hat{m}_i$ is an ergodic component of $\hnu_1^{\star}$, it can be chosen homoclinically related to $\hat{\cO}$. Since $\hy_i \in \text{supp}\ \hat{m}_i\cap\hat{Y}^{\#}$, we can find arbitrarily close to $\hy_i$, points belonging to the same Pesin block as $\hy_i$ and which are homoclinically related to $\hat{\cO}$. Continuity of stable and unstable local manifolds on Pesin blocks, in the $\cC^1$ topology, implies that $\hy_i$ is homoclinically related to $\hat{\cO}_i$. Using the $\lambda$-Lemma \ref{lem:LambdaLemma}, we conclude that there are discs of $\Ws(\hat{\cO})$ which accumulate on $\Ws_{loc}(\hy_i)$. Now by Claim \ref{claim:TopologicalToTransverse}, we obtain that for any $\hx \in \hat{X}_k$ we have:
\begin{equation*}
	f^m(\Wu(\hx)) \pitchfork \Ws(\hat{\cO}) \ne \emptyset
\end{equation*}
for some $m\geq0$. But the invariance of $\Ws(\hat{\cO})$ implies that this intersection also holds for $\Wu(\hx)$, and then for a set of full $\hmu_k$-measure by ergodicity. This then concludes the proof of Proposition \ref{prop:IntersectionsUnstableMaxEntropy}.

\section{Stable with unstable intersections for saddle ergodic measures}\label{sec:StableUnstableIntersections}
Throughout this section, $f:M\to M$ is a $\cC^r$, $r>1$ local diffeomorphism.

\subsection{Topological family of $us$-rectangles associated with a hyperbolic saddle measure}
Let us recall some well-known facts about the disintegration of measures. Let $\hnu$ be an $\hf$-invariant, possibly non-ergodic, measure such that for $\hnu$-almost every $\hx$ we have $\lambda^u(\hx)>0>\lambda^s(\hx)$, where $\lambda^{u/s}(\hx)$ are the two Lyapunov exponents at $\hx$. The Pesin stable manifold theorem (Theorem \ref{thm:Localu/sManifolds}) asserts that $\hnu$-almost every $\hx$ belongs to a stable set $V^s(\hx)$ (see Definition \ref{def:LocalInvariantSet}). As in Subsection \ref{subsec:UnstableEntropy}, for a given measurable partition subordinate to the stable lamination, we have the following disintegration:
\begin{equation*}
	\hnu = \int \hnu_{\hx}^s \ d\hnu(\hx).
\end{equation*}
We call $\hnu_{\hx}^s$ the conditional measure on the stable set $V^s(\hx)$. Locally, the stable set $V^s_{loc}(\hx)$ is homeomorphic to the product of a curve and a Cantor set. We can then further disintegrate every $\hnu^s_{\hx}$ on connected components of $V^s_{loc}(\hx)$ as:
\begin{equation*}
	\hnu^s_{\hx} = \int \hnu^s_{\hx,W} \ d\hnu_{V^s_{loc}(\hx)}(W)
\end{equation*}
where $\{\hnu^s_{\hx,W}\}_W$ is a family of conditionals such that each $\hnu^s_{\hx,W}$ is supported on the connected component $W$ of $V^s_{loc}(\hx)$ and $\hnu_{V^s_{loc}(\hx)}$ is a probability measure on the set of connected components of $V^s_{loc}(\hx)$.

The following proposition should be compared to Proposition \ref{prop:MarkovRectangles}. Given a (possibly non-ergodic) $f$-invariant measure $\nu$, we also construct a family of $us$-rectangles with large $\nu$ measure. The main difference lies in the fact that we require the unstable boundaries to be unstable local manifolds, but we do not require the dynamical Markov property \ref{property:MR5}.

\begin{prop}
	Let $\nu$ be a $f$-invariant probability measure, possibly non-ergodic. Suppose it satisfies the following for $\hnu := \pi_{\star}^{-1} \nu$-almost every $\hx$:
	\begin{itemize}[label={--}]
		\item $\lambda^u(\hx)>0>\lambda^s(\hx)$.
		\item $\hnu$-almost every ergodic component of $\hnu$ has positive entropy.
		\item There exists a connected component $W$ of $V^s_{loc}(\hx)$ such that $\hnu^s_{\hx,W}$ is non-atomic on $W$, where $\{\hnu^s_{\hx,W}\}_{W}$ is a system of conditional measures on connected components of $V^s_{loc}(\hx)$.
	\end{itemize}
	Let $\eta \in (0,1)$. There exists a Pesin block $\hNUH := \hNUH(\eta)$ such that the following holds for $\sigma:=\sigma(\hNUH)$ sufficiently small.
	
	There exists a family of $us$-rectangles $\mathcal{T}:=\{T_1,\dots,T_L\}$ such that:
	\begin{enumerate}[label=(R\arabic*)]
		\item\label{property:TR1} For any $i \in \{1,\dots,L\}$, we have that $T_i$ is a $us$-rectangle satisfying:
		\begin{itemize}[label={--}]
			\item there exists $\hx_i \in \hNUH$ such that $\big(\Psi_{\hx_i},q_{\tilde{\epsilon}}(\hx_i)\big)$ is an adapted chart for $T_i$.
			\item $\partial^{s,j}T_i \subset \Ws_{loc}(\hy_i^j)$ and $\partial^{u,j} T_i \subset \Wu_{loc}(\hz^j_i)$ where $\hy^j_i,\hz^j_i \in \hNUH \cap \hat{Y}^{\#}$ for $j\in\{1,2\}$.
		\end{itemize}
		\item\label{property:TR2} $\hnu(\hNUH)>1-\eta$ and $\pi\big(\hNUH\big)\subset \bigcup_{i=1}^L T_i$.
		\item\label{property:TR3} For any $i \in \{1,\dots,L\}$ we have the following:
		\begin{equation*}
			\diam(T_i)<\sigma.
		\end{equation*}
	\end{enumerate}
	\label{prop:TopologicalRectangles}
\end{prop}

Before proving Proposition \ref{prop:TopologicalRectangles}, let us briefly explain the main ideas. As previously explained, this proposition is close to Proposition \ref{prop:MarkovRectangles}. The fact that we do not require a Markov property such as Property \ref{property:MR5} greatly simplifies the proof. In fact, we can obtain Proposition \ref{prop:TopologicalRectangles} as a corollary of Proposition \ref{prop:MarkovRectangles}. Since the argument is very close to Step 2 of the proof of Proposition \ref{prop:MarkovRectangles}, we will not provide all the details.

\begin{proof}[Proof of Proposition \ref{prop:TopologicalRectangles}.]
	Let $\eta \in (0,1)$. We begin by applying Proposition \ref{prop:MarkovRectangles} to the measure $\nu$ and with the constant $\eta/2$, with any $K>8$ and with $\sigma$ small enough. This gives us, in particular, a Pesin block $\hNUH$ and a family of $us$-rectangles $\mathcal{R}:=(R_1,\dots,R_L)$ such that Properties \ref{property:MR1}, \ref{property:MR2}, and \ref{property:MR4} hold.
	
	Let $\{\hnu^s_{\hx}\}_{\hx}$ be a system of conditionals on stable local sets. Consider the set of points $\hat{X}\subset M_f$ such that $\hnu_{\hx}^s(\hNUH)>0$. Note that $\hnu(\hat{X}\cap\hNUH)>1-\eta$ by Property \ref{property:MR2}. Take $\hat{K}$ a compact subset of $\hat{X}\cap\hNUH$ such that $\hnu(\hat{K})>1-\eta$. For $\hnu$-almost every $\hx\in\hat{K}$, there exists a connected component $W\subset V^s_{loc}(\hx)$ such that $\hnu_{\hx,W}^s$ is non-atomic on $W$. In particular, there exist two distinct points $\hx^+,\hx^- \in W\cap \hNUH$, such that we can choose them as close as we want from $\hx$. Recall that by Property \ref{property:MR1}, $\partial^{s,l}R_i \subset \Ws_{loc}(\hy_i^l)$ where $\hy_i^l \in\hNUH$, for $i\in\{1,\dots,L\}$ and $l\in\{1,2\}$. Then, Proposition \ref{prop:IntersectionAdmissiblemanifolds} implies that if $\pi(\hx)\in R_i$, the curve $\Wu_{loc}(\hx^{\pm})$ intersects in a unique point $\Ws_{loc}(\hy^i_l)$ for both $l\in \{1,2\}$. The union of these four curves bounds a $us$-rectangle $T_{\hx,i}$ such that its stable, respectively unstable, boundaries are contained in local stable, respectively unstable, manifolds of points of $\hNUH$. Using Lemma \ref{lem:DiameterRectangle}, taking $\hx^+$ and $\hx^-$ close enough guarantees that $\diam(T_{\hx,i})<\sigma$. Now, since the union of all $T_{\hx,i}$ covers $\hNUH$, we can extract a finite cover. This is the desired family $\mathcal{T}$, which concludes the proof.
\end{proof}

\subsection{Stable entropy}
The results contained in this subsection are proved in \cite{liu2008invariant}. Recall that a partition $\xi_1$ is said to be finer than another partition $\xi_2$ if each atom of $\xi_1$ is contained in an atom of $\xi_2$. We denote this by $\xi_2\leq\xi_1$. Recall also that a partition $\xi$ is measurable if there exists a sequence of countable partitions $(\alpha_n)_{n\geq 1}$ such that for any $n\geq 1$ we have $\alpha_n\leq\alpha_{n+1}$ and such that $\xi =\vee_{n\geq 1}\alpha_n$.

Let $\mu$ be an $f$-invariant measure and let $\xi_1,\xi_2$ be two measurable partitions. Recall that $H_{\mu}(\xi_1|\xi_2)$ denotes the classical conditional entropy of $\xi_1$ relative to $\xi_2$:
\begin{equation*}
	H_{\mu}(\xi_1|\xi_2) := \sum_{B\in\xi_2}\sum_{A\in\xi_1} \mu(A\cap B) \log\big(\frac{\mu(A\cap B)}{\mu(B)}\big).
\end{equation*}

\begin{prop}[\cite{liu2008invariant}]
	Let $\mu$ be an ergodic, hyperbolic of saddle type $f$-invariant measure. Then, there exists $\xi^s$, a measurable partition subordinate to $W^s$, which satisfies:
	\begin{equation*}
		h(f,\mu) = h(f,\mu,\xi^s) := H_{\mu}(\xi|f^{-1}\xi) = H_{\hmu}\big(\pi^{-1}\xi|\hf^{-1}(\pi^{-1}\xi)\big).
	\end{equation*}
\end{prop}

Remark that in this case, $\pi^{-1}\xi$ is a measurable partition subordinate to the stable lamination $V^s$. A version of Theorem \ref{thm:UnstableEntropy} applied to $\hf^{-1}$ gives the following proposition.

\begin{prop}
	Let $\hmu$ be an ergodic hyperbolic of saddle type $\hf$-invariant measure. Let $\{\hmu^s_{\hx}\}_{\hx}$ be a system of conditional measures on stable sets. Then for $\hmu$-almost every $\hx$ we have:
	\begin{equation*}
		h(\hf,\hmu) = \inf_{\lambda>0}\lim_{\epsilon \to 0}\liminf_{n\to +\infty}\frac{1}{n}\log r_{\hf}(n,\epsilon,\hmu^s_{\hx},\lambda).
	\end{equation*}
	\label{prop:StableEntropy}
\end{prop}

\subsection{Stable with unstable intersections for saddle ergodic measures}
Let us now state the main proposition of the present section. It should be seen as the "stable version" of Proposition \ref{prop:UnstableStableIntersection}. Recall that for a local diffeomorphism $f:M\to M$, we denote by $\deg(f)$ the number of pre-images of a point $x \in M$, which is constant.

\begin{prop}
	Let $f : M\to M$ be a $\cC^{\infty}$ local diffeomorphism. Let $(\mu_k)_{k\in \bbN}$ be a sequence of ergodic $f$-invariant measures such that:
	\begin{equation*}
		\mu_k \underset{k\to+\infty}{\longrightarrow} \nu
	\end{equation*}
	for the weak-$\star$ topology. Suppose the following:
	\begin{equation*}
		\liminf_{k\to+\infty} h(f,\mu_k) > \log \deg(f).
	\end{equation*}
	Then, there exists a finite number of hyperbolic saddle periodic orbits $\hat{\cO_1},\dots,\hat{\cO}_L \subset M_f$ satisfying the following properties.
	\begin{enumerate}
		\item For any $k$ sufficiently large, there exists $j(k)\in\{1,\dots,L\}$ such that for $\hmu_k$ almost every $\hx$, the following intersection holds:
		\begin{equation*}
			\Ws(\hx) \pitchfork \Wu(\hat{\cO}_{j(k)}) \ne \emptyset,
		\end{equation*}
		where $\hmu_k :=\pi^{-1}_{\star}\mu_k$.
		\item For any $j \in \{1,\dots,L\}$ the set $\{\hx\in M_f, \ \hx\simh\cO_j\}$ has positive $\hnu$-measure, where $\hnu:=\pi^{-1}_{\star}\nu$.
	\end{enumerate}
	
	\label{prop:StableUnstableIntersection}
\end{prop}

The rest of this section is devoted to the proof of Proposition \ref{prop:StableUnstableIntersection}. Since the proof is close to the proof of Proposition \ref{prop:UnstableStableIntersection}, but in a simpler setting, we will not provide all the details and will refer to the corresponding parts of the proof of Proposition \ref{prop:UnstableStableIntersection}.

\subsection{Proof of Proposition \ref{prop:StableUnstableIntersection}}
Recall that $f:M\to M$ is a $\cC^{\infty}$ local diffeomorphism and that $(\mu_k)_{k\in \bbN}$ is a sequence of ergodic $f$-invariant measures. Let $\nu$ be the (possibly non-ergodic) $f$-invariant measure such that:
\begin{equation*}
	\mu_k \underset{k\to+\infty}{\longrightarrow} \nu.
\end{equation*}
Recall also that we made the following assumption:
\begin{equation*}
	\liminf_{k\to +\infty} h(f,\mu_k) > \log d(f).
\end{equation*}
Fix then a constant $h>0$ such that:
\begin{equation}
	\liminf_{k\to +\infty} h(f,\mu_k)-\log d(f) >h.
	\label{eq:HypothesisEntropyLogd}
\end{equation}

\parbreak\SStep[Decomposition of the measure $\nu$.]\label{step:SI1}
Let us write the ergodic decomposition of $\hnu$, the lift of $\nu$ to $M_f$ using $\pi$, as follows:
\begin{equation*}
	\hnu = \int_M \hnu_x \ d\nu(x).
\end{equation*}
Let us define the following set:
\begin{equation*}
	\hat{X} = \{\hx \in M_f, \ h(\hf,\hnu_x)>\log d(f)\}.
\end{equation*}
Since $f$ is $\cC^{\infty}$, the entropy is upper-semi continuous and we obtain that $h(f,\nu) > \log d(f)$, which implies that $\hnu(\hat{X})>0$. Recall the following Ruelle type inequality from Proposition \ref{prop:Ruelle}, which holds for any ergodic $f$-invariant measure $\mu$:
\begin{equation*}
	h(f,\mu)\leq \log d(f) + \sum_{\lambda_i(\mu)<0} \lambda_i(\mu)
\end{equation*}
where $\lambda_i(\mu)$ denote the $i$-th Lyapunov exponent of $\mu$. It implies in particular that for $\hnu$-almost any $\hx \in M$, the measure $\hnu_x$ has one positive and one negative Lyapunov exponent.

We can now decompose the measure $\nu$ in the following way. Let us define first the following $f$-invariant measures:
\begin{equation*}
	\nu_1 := \frac{1}{\nu(X)}\int_X \nu_x \ d\nu(x) \quad \text{and} \quad \nu_0 := \frac{1}{\nu(X^c)} \int_{X^c} \nu_x \ d\nu(x).
\end{equation*}
Letting $\alpha = \nu(X)$, we obtain the decomposition:
\begin{equation*}
	\nu=\alpha\nu_1 + (1-\alpha)\nu_0.
\end{equation*}
Let $\hx \in \hat{X}$ and let $\{\hnu_{\hx,\hy}^s\}_y$ be a system of conditional measures of the measure $\hnu_x$ on local stable sets. Let $\{\hnu^s_{\hx,\hy,W}\}_W$ be a system of conditional measures of $\hnu_{\hx,\hy}^s$ on the connected components $W$ of $V^s_{loc}(\hy)$. Note that by Proposition \ref{prop:StableEntropy}, if there exists a set of positive $\hnu_x$-measure of points $\hy \in M_f$ such that $\hnu_{\hx,\hy,W}^s$ is atomic for any connected component $W$ of $V^s_{loc}(\hy)$, then we would have $h(\hf,\hnu_{\hx}) \leq \log d(f)$. In particular, $\hnu_1$ satisfies the hypothesis of Proposition \ref{prop:TopologicalRectangles}.

Remark also that the measure $\nu_0$ satisfies that $\nu_0$-almost every ergodic component has entropy less or equal to $\log d(f)$.

\parbreak\SStep[Choice of the entropy parameter.]\label{step:SI2}
Let $\delta>0$ be an arbitrary parameter which has to be taken small enough, and which depends only on $h$. By \eqref{eq:TailEntropyEstimate1}, we can choose $\epsilon>0$ such that for any $f$-invariant measure $m$ we have:
\begin{equation}
	h(f,m,\epsilon)>h(f,m)-\delta.
	\label{eq:StableScaleEntropy}
\end{equation}
Since the way to choose $\delta$ is very similar to Step \stepref{step:U1} of the proof of Proposition \ref{prop:UnstableStableIntersection}, we do not specify the value of $\delta$, but it will be clear in what follows that it would have been possible to make it explicit.

\parbreak\SStep[The $\nu_1$-rectangles.]\label{step:SI3}
This step can be compared with Step \stepref{step:U2} in the proof of Proposition \ref{prop:UnstableStableIntersection}. Let $\eta \in (0,1)$ be small enough, depending only on $h$ and $\delta$. Again, we choose not to specify the value of $\eta$. Apply now Proposition \ref{prop:TopologicalRectangles} to the measure $\nu_1$ with the previous choice of $\eta$ and with $\sigma$ small. It gives us a Pesin block $\hat{\Lambda}$ and a collection of $us$-rectangles $\mathcal{T} := \{T_1,\dots,T_L\}$ such that properties \ref{property:TR1} to \ref{property:TR3} hold.

Up to reducing $\sigma$ and $\epsilon$, we can suppose that the following items are true.
\begin{itemize}[label={--}]
	\item For any $i \in \{1,\dots,L\}$, there exists a homeomorphism $\varphi_i : (0,1)^2\times K\to \pi^{-1}(T_i)$, where $K$ is the Cantor set. Moreover, for any $t \in K$ the map $\varphi_i$ is a diffeomorphism on $(0,1)^2\times\{t\}$.
	\item For any $i \in \{1,\dots,L\}$, each connected component of $\pi^{-1}(T_i)$ has diameter less than $\epsilon/3$.
\end{itemize}

\parbreak\SStep[The decomposition of typical $\mu_k$-orbits.]\label{step:SI4}
This step can be compared to Steps \stepref{step:U3}, \stepref{step:U4}, \stepref{step:U5}, and \stepref{step:U6} from Proposition \ref{prop:UnstableStableIntersection}. We omit the details since the proofs are the same.

Recall that $\nu=\alpha\nu_1+(1-\alpha)\nu_0$ for $\alpha>0$. Let $\hnu_i := \pi^{-1}_{\star}\nu_i$ for $i=0,1$. The following claim is obtained by combining Claim \ref{claim:Step3} and Claim \ref{claim:Step4}.

\begin{claim}
	There exist two integers $n_0,n_1\geq 1$ and two open sets $\hat{U}_0,\hat{U}_1\subset M_f$ such that the following holds.
	\begin{itemize}[label={--}]
		\item $\hat{U}_0 \subset \bigcup_{\hx \in \cC_0} B_{\hf^{-1}}(\hx,n_0,\epsilon)$ where $\cC_0\subset M_f$ such that $|\cC_0|\leq \exp\big(n_0\log \deg(f)(1+\delta)\big)$.
		\item $\hat{U}_1\subset \pi^{-1}(\cup_i T_i)$,
		\item $\hnu_i(\hat{U}_i)>1-\eta$, $\hnu_i(\partial\hat{U}_i) =0$ for $i=0,1$ and $\hnu_i(\hat{U}_j)<\eta$ for $i\ne j$.
		\item For any $0\leq j\leq n_1$ we have:
		\begin{equation*}
			\hf^{-j}\big(\text{Closure}(\hat{U}_1)\big) \cap \text{Closure}(\hat{U}_0) = \emptyset.
		\end{equation*}
		The symmetric property holds exchanging the roles of $n_1$ and $\hat{U}_1$ with $n_0$ and $\hat{U}_0$.
	\end{itemize}
	\label{claim:StableStep4-1}
\end{claim}

As in the proof of Proposition \ref{prop:UnstableStableIntersection}, we define the following three classes of a $\hf^{-1}$-orbit segment $(\hf^{-s}(\hx),\dots,\hf^{-(t+s-1)}(\hx))$ of a $\hf^{-1}$-orbit $(\hx,\dots,\hf^{-(n-1)}(\hx))$.

\begin{itemize}[label={--}]
	\item \textbf{Rectangle segment:} $\hf^{-s} \in \hat{U}_1$ and $t=n_1$.
	\item \textbf{Small entropy segment:} $\hf^{-s} \in \hat{U}_0$ and $t=n_0$.
	\item \textbf{Wild segment:} one of the following situations occurs.
	\begin{enumerate}[label=(\alph*)]
		\item $\hat{f}^{-(s+j)}(\hx) \notin \hat{U}_0\cup\hat{U}_1$ for $j=0,\dots,t-1$ and $\hat{f}^{-(t+s)}(\hx) \in \hat{U}_0\cup\hat{U}_1$.
		\item $\hat{f}^{-s}(\hx) \in \hat{U}_0$ and $s+n_0>n-1$, in which case $t+s=n$.
		\item $\hat{f}^{-s}(\hx) \in \hat{U}_1$ and $s+n_1>n-1$, in which case $t+s=n$.
	\end{enumerate}
\end{itemize}

Claim \ref{claim:Step5} gives the following claim.

\begin{claim}
	There exists an integer $k_{\star}$ such that for any $k\geq k_{\star}$ and for $\hmu_k$-almost every $\hx$, there exists an integer $n_k(\hx)$ displaying the following property. For any $n\geq n_k(\hx)$, any orbit $(\hx,\dots,\hat{f}^{-(n-1)}(\hx))$ can be decomposed as follows. There exists an integer $l$ and a string:
	\begin{equation*}
		\theta^{\prime}(\hx,n) := \big((t_0,b_0),\dots,(t_l,b_l)\big) \in (\bbZ\times\{1,2,3\})^{l+1}
	\end{equation*}
	such that $t_0=-1$, $t_l=n-1$, $b_l=1$ and for any $j\in\{0,\dots,l-1\}$ the orbit segment $(\hat{f}^{-(t_j+1)}(\hx),\dots,\hat{f}^{-t_{j+1}})$ is: a rectangle segment if $b_j=1$, a small entropy segment if $b_j=2$, and a wild segment if $b_j=3$. Moreover, we have the following estimates:
	\begin{enumerate}[label=(\alph*)]
		\item rectangle segments have total length at most $\alpha n + 2\eta n$,
		\item small entropy segments have total length at most $(1-\alpha)n + 2\eta n$,
		\item wild segments have total length at most $2\eta n$.
	\end{enumerate}
	\label{claim:StableStep4-2}
\end{claim}

For $k\geq k_{\star}$, for $\hmu_k$-almost every $\hx$, and for $n\geq n_k(\hx)$, Claim \ref{claim:Step5} allows us to construct the decomposition string $\theta^{\prime}(\hx,n) \in (\bbZ\times\{1,2,3\})^{l+1}$ for some integer $l$. For a given orbit $(\hx,\dots,\hat{f}^{-(n-1)}(\hx))$, we will also need to keep track of which iterates are inside $\hat{U}_0\cup\hat{U}_1$ and which are not. For that, we introduce the word $v(\hx,n) \in \{0,1\}^n$ which satisfies $v(\hx,n)_i = 1$ if $\hat{g}^i \in \hat{U}_0\cup\hat{U}_1$ and $v(\hx,n)_i = 0$ otherwise.

We define the \textit{decomposition type} of an orbit $(\hx,\dots,\hat{f}^{-(n-1)}(\hx))$ as the following data:
\begin{equation*}
	\theta(\hx,n) := \Big[v(\hx,n), \theta^{\prime}(\hx,n)\Big] \in \{0,1\}^{n} \times \big(\bbN \times \{1,2,3\}\big)^{l+1}.
\end{equation*}

Recall that we denote by $H$ the map $H(t) = -t\log t -(1-t)\log(1-t)$. Claim \ref{claim:Step6} gives us the following bound on the number of decomposition types.

\begin{claim}
	There exists an integer $n_H := n_H(\eta)$ such that the number of decomposition types of any $\hat{f}^{-1}$-orbit as in Claim \ref{claim:StableStep4-2} with length $n>n_H$ is at most $\exp\big(n(5\eta + 2H(5\eta))\big)$.
	\label{claim:StableStep4-3}
\end{claim}

\parbreak\SStep[Conditional measures.]\label{step:SI5}
This step can be compared with Step \stepref{step:U7} in the proof of Proposition \ref{prop:UnstableStableIntersection}. Let $k\geq k_{\star}$ from Claim \ref{claim:StableStep4-2}. Denote by $\{\hmu_k^{s,\hx}\}_{\hx}$ a family of conditional measures of $\hmu_k$ on local stable sets. Recall that the measure $\hmu_k^{s,\hx}$ is supported on $V^s_{loc}(\hx)$. We fix an integer $N_k\geq 0$ and a measurable set $\hat{F}_k$ with $\hmu_k(\hat{F}_k)>1/2$ such that every point $\hx \in \hat{F}_k$ satisfies:
\begin{itemize}[label={--}]
	\item $\hx$ has a well-defined stable manifold $V^s(\hx) \subset M_f$,
	\item $\hmu^{s,\hx}_k$ is well-defined and $\hx$ is in the support of $\hmu_k^{s,\hx}$ restricted to $\hat{F}_k$,
	\item $\hx$ satisfies Claim \ref{claim:StableStep4-2} with $n_k(\hx)\leq N_k$. In particular, for each $n\geq N_k$, the orbit $(\hx,\dots,\hf^{-n}(\hx))$ has a decomposition as in Claim \ref{claim:StableStep4-2}. Write $\theta(\hx,n)$ for its decomposition.
\end{itemize}

\parbreak\SStep[Construction of a cover by Bowen-balls.]\label{step:SI6}
This step can be compared to Step \stepref{step:U8} in the proof of Proposition \ref{prop:UnstableStableIntersection}. The cases of small entropy segments and wild segments are treated in the same way.

Recall that for any $x \in M$, the fiber $\pi^{-1}(x)$ is homeomorphic to the Cantor set $K$. Then, up to reducing the size of $\Ws_{loc}$, the stable set $V^s_{loc}(x)$ is homeomorphic to $[0,1]\times K$. Denote this homeomorphism by $\varphi_x$. On each connected component $[0,1]\times \{t\}$, for $t\in K$, the map $\varphi_x$ is a diffeomorphism and its image is homeomorphic to the curve $\Ws_{loc}(x)$.

Fix any $\hx \in \hat{F}_k$. Let $x=\pi(\hx)$. Fix $n\geq \max(N_k,n_H)$ and fix also $\epsilon,\sigma>0$ as in Step \sstepref{step:SI2}. Let $\mathcal{B}_0 = V^s_{loc}(x) \cap B(\hx,\epsilon) \cap \hat{F}_k$. Note that by our hypothesis on the set $\hat{F}_k$ we have $\hmu^{s,\hx}_k(\mathcal{B}_0)>0$. Recall the definition of the family $\mathcal{T} = \{T_1,\dots,T_L\}$. We are going to show that some $\hf^{-1}$-iterate of $V^s_{loc}(\hx)$ intersects some $\pi\inv(\partial^{u,l}T_i)$ for some $i\in \{1,\dots,L\}$ and $l\in \{1,2\}$. We will proceed by contradiction. From now on, assume the following non-intersection hypothesis.

\begin{hyp}
	For any $m\geq 0$ and for any $i \in \{1,\dots,L\}$ and for any connected component $W\subset \hf^{-m}(V^s_{loc}(\hx))$ such that $W \cap \pi^{-1}(T_i) \ne \emptyset$, we have $W \subset \pi^{-1}(T_i)$.
	\label{eq:HypothesisStableIntersection}
\end{hyp}

Our goal is to construct a cover $\mathcal{B}_n$ of $\mathcal{B}_0$ by $(n,\epsilon)$-Bowen balls of small cardinality. The strategy is similar to the one explained in Step 8 of the proof of Proposition \ref{prop:UnstableStableIntersection}.

We begin by fixing a decomposition type $\theta = \big[(v_0,\dots,v_{n-1}),\big((t_0,b_0),\dots,(t_m,b_m)\big)\big]$ with $t_0=-1$ and $t_m=n-1$. We will construct a cover $\mathcal{B}_n^{\theta}$ of $\mathcal{B}^{\prime} := \mathcal{B}_0 \cap \{\hx, \ \theta(\hx,n)=\theta\}$ by $(n,\epsilon)$-Bowen balls. We will then take the union over all possible decomposition types and obtain the desired cover $\mathcal{B}_n$.

The cover $\mathcal{B}^{\theta}_n$ will be obtained inductively by defining covers $\mathcal{B}_{t_i}^{\theta}$ of $\mathcal{B}_0$ with $(t_i,\epsilon)$-Bowen balls. The first cover is $\mathcal{B}^{\theta}_0 := \mathcal{B}^{\prime}$. Assume now that the cover $\mathcal{B}^{\theta}_{t_{i-1}}$ is well-defined. We proceed by concatenation. We will set:
\begin{equation*}
	\mathcal{B}_{t_i}^{\theta} = \{B\cap \hf^{t_{i-1}}(B^{\prime}), \ B\in \mathcal{B}^{\theta}_{t_{i-1}}, \ B^{\prime}\in \mathcal{B}(B,t_i)\}
\end{equation*}
where $\mathcal{B}(B,t_i)$ is a well-chosen cover of $\hf^{-t_{i-1}}\big(B\cap \mathcal{B}^{\prime}\big)$ by $(t_i-t_{i-1},\epsilon)$-Bowen balls. Note that $B$ is a $(t_{i-1},\epsilon)$-Bowen ball. Let us now construct these covers $\mathcal{B}(B,t_i)$.

We fix then $B \in \mathcal{B}^{\theta}_{t_{i-1}}$. As in the proof of Proposition \ref{prop:UnstableStableIntersection}, the cases $b_{i-1}=2,3$ are easy and are treated in the same way, while the case $b_{i-1}=1$ needs more attention.

\parbreak\textbf{Case A: small entropy segment.} Here we treat the case $b_{i-1}=2$. Recall that we then have $t_i-t_{i-1}=n_0$. By Claim \ref{claim:StableStep4-1}, there exists a set $\cC_0\subset M_f$ such that $|\cC_0|\leq \exp\big(n_0\log \deg(f)(1+\delta)\big)$ and such that:
\begin{equation*}
	\hf^{-t_{i-1}}\big(B\cap \mathcal{B}^{\prime}\big) \subset \bigcup_{\hy \in \cC_0} B_{\hf^{-1}}(\hy,n_0,\epsilon).
\end{equation*}
We then set $\mathcal{B}(B,t_i) = \{B_{\hf^{-1}}(\hy,n_0,\epsilon), \ \hy \in \cC_0\}$. Note that it satisfies the following bound:
\begin{equation*}
	|\mathcal{B}(B,t_i)|\leq \exp\big(n_0\log \deg(f)(1+\delta)\big).
\end{equation*}
Repeating this for any $B\in \mathcal{B}^{\theta}_{t_{i-1}}$ allows us to construct the cover $\mathcal{B}^{\theta}_{t_i}$, which satisfies:
\begin{equation}
	|\mathcal{B}^{\theta}_{t_i}|\leq \exp\big(n_0\log \deg(f)(1+\delta)\big)|\mathcal{B}^{\theta}_{t_{i-1}}|.
	\label{eq:StableBoundSmallEntropy}
\end{equation}

\parbreak\textbf{Case B: wild segment.} We suppose now $b_{i-1}=3$. The idea is the same as in Case B in the proof of Proposition \ref{prop:UnstableStableIntersection}. We bound the number of Bowen-balls by the derivative. Indeed, if two points $\hy,\hz \in \hf^{-t_{i-1}}\big(B\cap \mathcal{B}^{\prime}\big)$ are such that $d(\hy,\hz) < \epsilon/(\norm{d\hf^{-1}}^{t_i-t_{i-1}})$, then they must stay $\epsilon$ close until time $t_i-t_{i-1}$. We construct the cover $\mathcal{B}(B,t_i)$ by taking the union of such balls of radius $\epsilon/(\norm{d\hf^{-1}}^{t_i-t_{i-1}})$. There exists then a constant $C\geq1$, depending on $M$ and $f$, such that:
\begin{equation*}
	|\mathcal{B}(B,t_i)|\leq C\norm{d\hf^{-1}}^{t_i-t_{i-1}}.
\end{equation*}
Repeating this argument for all $B \in \mathcal{B}^{\theta}_{t_{i-1}}$ gives us the desired family $\mathcal{B}^{\theta}_{t_i}$, which satisfies:
\begin{equation}
	|\mathcal{B}^{\theta}_{t_i}| \leq C\norm{d\hf^{-1}}^{t_i-t_{i-1}}|\mathcal{B}^{\theta}_{t_{i-1}}|.
	\label{eq:StableBoundWild}
\end{equation}

\parbreak\textbf{Case C: rectangle segment.} Let us finally treat the last case $b_{i-1}=1$. Recall that we then have that $t_i-t_{i-1}=n_1$ and that for any $\hy \in \hf^{-t_{i-1}+1}(\mathcal{B}^{\prime})$, we have $\hy \in \cup_i T_i$, where the $us$-rectangles $T_i$ are defined in Step \sstepref{step:SI3}. Define the two sequences of times $(\tau_k)_k$ and $(\tau^{\prime}_k)_k$ as in Step 8 of the proof of Proposition \ref{prop:UnstableStableIntersection}. Set $\tau_0=\tau^{\prime}_0 = 0$. Assume that $\tau_k$ and $\tau_k^{\prime}$ are well-defined. Let:
\begin{equation*}
	\tau_{k+1} := \inf\{j>\tau_k^{\prime}, \ v_{t_{i-1}+j+1}=0\} \quad \text{and} \ \quad \tau_{k+1}^{\prime} := \inf\{j>\tau_{k+1}, \  v_{t_{i-1}+j+1}=1\}.
\end{equation*}
The induction stops when some $\tau_k$ or $\tau_k^{\prime}$ is equal to $n_1-1$. In the case where some $\tau_k = n_1-1$, we set $\tau^{\prime}_k = n_1-1$ too. Note that $\tau_1$ is well defined since $v_{t_{i-1}+1}=1$. Recall that by Claim \ref{claim:StableStep4-1} and by the definition of a rectangle segment, if $v_{t_{i-1}+j} = 1$ then $\hf^{-(t_{i-1}+j)}(\hy) \in \hat{U}_1$ for any $\hy \in \mathcal{B}^{\prime}$. Let us introduce the following quantity:
\begin{equation*}
	c^0_i = \#\{0\leq j\leq n_1-1, \ v_{t_{i-1}+j} = 0\}.
\end{equation*}
We are going to proceed again by induction, constructing sequences $\mathcal{B}_{\tau_k}$ and $\mathcal{B}_{\tau_k^{\prime}}$ of covers of the set $\hf^{-t_{i-1}}(B\cap\mathcal{B}^{\prime})$ by Bowen balls. For the integer $l\geq0$ such that $\tau_l=n-1$, we will then let $\mathcal{B}(B,t_i) = \mathcal{B}_{\tau_l}$. Let first $\mathcal{B}_{\tau_0} = \mathcal{B}_{\tau^{\prime}_0} = \hf^{t_{i-1}}(B\cap\mathcal{B}^{\prime})$. Assume that the cover $\mathcal{B}_{\tau_k^{\prime}}$ is well-defined. We are going to construct $\mathcal{B}_{\tau_{k+1}}$. To do this, as before, we fix some $D \in \mathcal{B}_{\tau^{\prime}_k}$ and we want to build $\mathcal{B}(D,\tau_{k+1})$, a cover of $\hf^{-(t_{i-1}+\tau^{\prime}_k+1)}(D\cap\mathcal{B}^\prime)$ by $(\tau_{k+1}-\tau_{k}^{\prime}-1,\epsilon)$-Bowen balls. Proceeding by concatenation as before will give us $\mathcal{B}_{\tau_{k+1}}$. See Case C of Step \stepref{step:U8} in the proof of Proposition \ref{prop:UnstableStableIntersection} for more details on the strategy. Let then $D \in \mathcal{B}_{\tau_k^{\prime}}$.

Recall now the following crucial fact. For any $j \in \llbracket\tau_{k}^{\prime},\tau_{k+1}\rrbracket$, Hypothesis \ref{eq:HypothesisStableIntersection} gives that for any $i \in \{1,\dots,L\}$ and any connected component $V\subset \hf^{-(t_{i-1}+j)}(V^s_{loc}(\hx))$ such that $V \cap \pi^{-1}(T_i) \ne \emptyset$, we have:
\begin{equation*}
	V\subset \pi^{-1}(T_i).
\end{equation*}
For any connected component $V$ of $\hf^{-(t_{i-1}+\tau_k^{\prime})}(D\cap V^s_{loc}(\hx))$, denote by $V^{\prime}$ the set $V \cap \hf^{-(t_{i-1}+\tau_k^{\prime})}(\mathcal{B}^{\prime})$ when the intersection is non-empty.

\begin{claim}
	There exists $i_1\in\{1,\dots,L\}$ such that $\hf^{-(t_{i-1}+\tau_k^{\prime}+1)}(D\cap V^s_{loc}(\hx))\subset \pi^{-1}(T_{i_1})$. Moreover, for any connected component $V\subset \hf^{-(t_{i-1}+\tau_k^{\prime})}(D\cap V^s_{loc}(\hx))$ and any $j \in \llbracket2,\tau_{k+1}-\tau_k^{\prime}\rrbracket$, there exists $i_j \in \{1,\dots,L\}$ depending on $V$ such that:
	\begin{equation*}
		\hf^{-j}(V^{\prime}) \subset \pi^{-1}(T_{i_j}).
	\end{equation*}
	As a consequence, $V^{\prime}$ is contained in a single $(\tau_{k+1}-\tau^{\prime}_k,\epsilon/3)$-Bowen ball.
	\label{claim:StableStep6CaseA}
\end{claim}

\begin{proof}
	Note first that $\hf^{-(t_{i-1}+\tau^{\prime}_k+1)}(D\cap V^s_{loc}(\hx)\cap\mathcal{B}^{\prime})$ has to project by $\pi$ onto a connected component of $f^{-(t_{i-1}+\tau^{\prime}_k+1)}(\Ws_{loc}(\hx))$. Indeed, suppose that $\hf^{-(t_{i-1}+\tau^{\prime}_k)}(D\cap\mathcal{B}^{\prime})$ projects onto two different connected components $W_1$ and $W_2$. There exists some $j\in\{1,\dots,t_{i-1}+\tau_k^{\prime}\}$ such that $f^j(W_1)=f^j(W_2)$ but $f^{j-1}(W_1)\cap f^{j-1}(W_2)=\emptyset$. But since $f$ is a local diffeomorphism, this would imply that $f^{j-1}(W_1)$ and $f^{j-1}(W_2)$ are separated by a uniform distance depending only on $f$, which contradicts the fact that the image by $f^j$ of points of $\pi(\hf^{-(t_{i-1}+\tau^{\prime}_k+1)}(D\cap\mathcal{B}^{\prime}))$ have to be $\epsilon$-close for $j=0,\dots,t_{i-1}+\tau_k^{\prime}+1$.
	
	In particular, $\pi(\hf^{-(t_{i-1}+\tau^{\prime}_k+1)}(D\cap\mathcal{B}^{\prime}))$ is a curve of diameter less than $\epsilon$ which intersects $\cup_i T_i$. Hypothesis \ref{eq:HypothesisStableIntersection} implies that there exists $i_0\in \{1,\dots,L\}$ such that:
	\begin{equation*}
		V^{\prime} \subset V\subset \pi^{-1}(T_{i_0}).
	\end{equation*}
	Now Claim \ref{claim:StableStep6CaseA} is a direct consequence of Hypothesis \ref{eq:HypothesisStableIntersection} and the fact that for any $j \in \llbracket1,\tau_{k+1}-\tau_k^{\prime}\rrbracket$ and any $\hy \in V^{\prime} \subset \hf^{-(t_{i-1}+\tau_k^{\prime})}(\mathcal{B}^{\prime})$ we have $\hf^{-j}(\hy) \in \cup_i\pi^{-1}(T_i)$. Indeed, since $\hf^{-j}(V)$ is connected, it has to be contained in a connected component of $\pi^{-1}(T_{i_j})$, which has diameter less than $\epsilon/3$.
\end{proof}

Choose an arbitrary $x_1 \in T_{i_1}\cap\pi(\hf^{-(t_{i-1}+\tau_k^{\prime}+1)}(V^s_{loc}(\hx)\cap D)$. We are going to build a cover by Bowen balls of the set $\pi^{-1}(x_1)$ with controlled cardinality.

\begin{claim}
	For any $n\geq0$, there exists a set $\mathcal{F}^n_{x_1}$ such that $|\mathcal{F}^n_{x_1}|\leq K \deg(f)^{n}$, for some constant $K>0$ independent of $n$, which satisfies the following:
	\begin{equation*}
		\pi^{-1}(x_1) \subset \bigcup_{\hy \in \mathcal{F}^n_{x_1}} B_{\hf^{-1}}(\hy,n,\epsilon/3).
	\end{equation*}
	\label{claim:StableBoundFiber}
\end{claim}

\begin{proof}
	Recall that the distance in $M_f$ is given by:
	\begin{equation*}
		d(\hy,\hz) = \sum_{i=0}^{+\infty} \frac{d(y_i,z_i)}{2^i}.
	\end{equation*}
	Since $x_1$ is fixed and the number of pre-images is constant equal to $\deg(f)$, the dynamics of $\hf^{-1}$ on $\pi^{-1}(x_1)$ is conjugate to a Markov shift on a finite alphabet with $\deg(f)$ letters. It is now classical that the number of $(n,\epsilon/3)$-Bowen balls to cover such a system is proportional to $\deg(f)^n$, the proportionality constant depending only on $\epsilon$.
\end{proof}

Apply Claim \ref{claim:StableBoundFiber} with $\tilde{n}=\tau_{k+1}-\tau_{k}^{\prime}-1$. For any $\hy \in \mathcal{F}^{\tilde{n}}_{x_1}$, denote by $\hat{T}_{i_1,\hy}$ the union of all the connected components of $\pi^{-1}(T_{i_1})$ intersecting $B_{\hf^{-1}}(\hy,\tau_{k+1}-\tau_k^{\prime}-1,\epsilon/3)$.

\begin{claim}
	For any $\hy \in \mathcal{F}_{x_1}^{\tilde{n}}$ the set $\hat{T}_{i_1,\hy}\cap \hf^{-(t_{i-1}+\tau_k^{\prime}+1)}(D\cap V^s_{loc}(\hx))$ is contained in one $(\tau_{k+1}-\tau_{k}^{\prime}-1,\epsilon)$-Bowen ball.
	\label{claim:StableStep6Cover}
\end{claim}

\begin{proof}
	Let $j \in \llbracket0,\tau_{k+1}-\tau_{k}^{\prime}-1\rrbracket$ and let $\hz_1,\hz_2 \in \hat{T}_{i_1\hy}\cap \hf^{-(t_{i-1}+\tau_k^{\prime}+1)}(D\cap V^s_{loc}(\hx))$ for some $\hy \in \mathcal{F}^{\tilde{n}}_{x_1}$. Denote by $\hx_1^l$ the unique point in $\pi^{-1}(x_1)$ which belongs to the same connected component of $\hf^{-(t_{i-1}+\tau_k^{\prime}+1)}(D\cap V^s_{loc}(\hx))$ as $\hz_l$, for $l\in\{1,2\}$. Combining Claim \ref{claim:StableStep6CaseA} and the fact that $\hx_1^1$ and $\hx_1^2$ belong to the same $(\tau_{k+1}-\tau_{k}^{\prime},\epsilon/3)$-Bowen ball we have:
	\begin{equation*}
		d\big(\hf^{-j}(\hz_1),\hf^{-j}(\hz_2)\big) \leq d\big(\hf^{-j}(\hz_1),\hf^{-j}(\hx_1^1)\big) + d\big(\hf^{-j}(\hx_1^1),\hf^{-j}(\hx_1^2)\big) + d\big(\hf^{-j}(\hx_1^2),\hf^{-j}(\hz_2)\big) < \epsilon
	\end{equation*}
\end{proof}

Define now $\mathcal{B}(D,\tau_{k+1})$ to be the minimal collection of $(\tau_{k+1}-\tau_{k}^{\prime}-1,\epsilon)$-Bowen balls containing every set $\hat{T}_{i_1,\hy}\cap \hf^{-(t_{i-1}+\tau_k^{\prime}+1)}(D\cap V^s_{loc}(\hx))$ for $\hy \in \mathcal{F}^{\tilde{n}}_{x_1}$. Note that combining Claim \ref{claim:StableStep6CaseA} and Claim \ref{claim:StableBoundFiber} gives the following:
\begin{equation*}
	\hf^{-(t_{i-1}+\tau_k^{\prime}+1)}(D\cap\mathcal{B}^{\prime}) \subset \bigcup_{\hy \in \mathcal{F}^{\tilde{n}}_{x_1}}\hat{T}_{\hy}\cap \hf^{-(t_{i-1}+\tau_k^{\prime}+1)}(D\cap V^s_{loc}(\hx)).
\end{equation*}
In particular, $\mathcal{B}(D,\tau_{k+1})$ is a cover by $(\tau_{k+1}-\tau_k^{\prime}-1,\epsilon)$-Bowen balls of the set $\hf^{-(t_{i-1}+\tau_k^{\prime}+1)}(D\cap\mathcal{B}^{\prime})$. Note also that by Claim \ref{claim:StableStep6Cover}, we have the following bound on the cardinality of this cover:
\begin{equation*}
	|\mathcal{B}(D,\tau_{k+1})|\leq |\mathcal{F}^{\tilde{n}}_{x_1}|\leq K\deg(f)^{(\tau_{k+1}-\tau_{k}^{\prime})}.
\end{equation*}
This concludes the construction of the cover $\mathcal{B}_{\tau_{k+1}}$ and gives us the following:
\begin{equation*}
	|\mathcal{B}_{\tau_{k+1}}|\leq K\deg(f)^{(\tau_{k+1}-\tau_{k}^{\prime}-1)}|\mathcal{B}_{\tau_k^{\prime}}|.
\end{equation*}
To conclude with Case C, assume now that $\mathcal{B}_{\tau_k}$ is well-defined and that you want to construct $\mathcal{B}_{\tau_k^{\prime}}$. Here the argument is the same as in Case A and we construct our cover by taking balls of radius $\epsilon/\norm{d\hf^{-1}}^{\tau_k^{\prime}-\tau_k}$. This gives us the following bound on the cardinality of $\mathcal{B}_{\tau_k^{\prime}}$:
\begin{equation*}
	|\mathcal{B}_{\tau_k^{\prime}}|\leq C\norm{d\hf^{-1}}^{\tau_k^{\prime}-\tau_k}
\end{equation*}
for some constant $C>0$.

This concludes the induction and allows us to define $\mathcal{B}(B,t_i)$ and, by induction, it gives us the following bound on its cardinality:
\begin{equation*}
	|\mathcal{B}(B,t_i)|\le(KC)^{c_i^0}\times d(f)^{n_1} \times \norm{d\hf^{-1}}^{c_i^0}.
\end{equation*}
Repeating this construction for all $B\in \mathcal{B}_{t_{i-1}}^{\theta}$ allows us to define the cover $\mathcal{B}_{t_i}^{\theta}$, which has cardinality:
\begin{equation}
	|\mathcal{B}_{t_i}^{\theta}|\leq (KC)^{c_i^0}\times d(f)^{n_1} \times \norm{d\hf^{-1}}^{c_i^0}|\mathcal{B}_{t_{i-1}}^{\theta}|.
	\label{eq:StableBoundRectangle}
\end{equation}
This finishes the induction over the $t_i$ and allows us to define $\mathcal{B}_n^{\theta}$. Repeating all this construction for every decomposition type $\theta$ defines the cover $\mathcal{B}_n$.

\parbreak\SStep[A bound on the cardinality of the covers by Bowen balls.]\label{step:SI7}
This step can be compared to Step \sstepref{step:U9} in the proof of Proposition \ref{prop:UnstableStableIntersection}. By Step \sstepref{step:SI6}, gathering all the estimates \eqref{eq:StableBoundSmallEntropy}, \eqref{eq:StableBoundWild} and \eqref{eq:StableBoundRectangle} together, we obtain the following bound for $n>N_k$:
\begin{equation*}
	\begin{aligned}
		|\mathcal{B}_n|\leq &\exp\big(n(5\eta+2H(\eta)\big)\\
		&\times \exp\Big(\big((1-\alpha)n +2\eta n\big)\log d(f)(1+\delta)\big)\\
		&\times \big(2KC\norm{d\hf^{-1}}\big)^{2\eta n}\\
		&\times d(f)^{\alpha n + 2\eta n}.
	\end{aligned}
\end{equation*}
Using Proposition \ref{prop:StableEntropy}, this gives us the following bound on the entropy of $\hmu_k$ for $k\geq k_{\star}$:
\begin{equation*}
	\begin{aligned}
		h(f,\hmu_k,\epsilon) \leq \lim_{n\to +\infty} \frac{1}{n} \log |\mathcal{B}_n|&\leq 5\eta + 2H(\eta)\\
		&+(1-\alpha)\log d(f) + \big(\delta(1-\alpha)+2\eta\big)\log d(f)\\
		&+2\eta \log(2KC\norm{d\hf^{-1}})\\
		&+\alpha\log d(f) + 2\eta \log d(f).
	\end{aligned}
\end{equation*}
Choosing $\eta$ and $\delta$ small enough in Steps \sstepref{step:SI2} and \sstepref{step:SI3} leads to: $h(f,\hmu_k,\epsilon)\leq \log d(f) + h$.
On the other hand, we have, again choosing $\delta$ correctly:
\begin{equation*}
	h(f,\hmu_k,\epsilon) > h(f,\hmu_k) - \delta > \log d(f) + h
\end{equation*}
which leads to a contradiction. Hypothesis \ref{eq:HypothesisStableIntersection} can then not be true. This implies that there exists some $m\geq 0$, some $i \in \{1,\dots,L\}$ and a connected component $V\subset \hf^{-m}(V^s_{loc}(\hx))$ such that $V$ intersects $\pi^{-1}(T_i)$ and also $M_f\backslash\overline{\pi^{-1}(T_i)}$. Recall that the stable boundaries of any $us$-rectangle of $\mathcal{T}$ is contained in a stable manifold. Recall also that stable manifolds are either disjoint or equal. We then conclude that for any $k\geq k_{\star}$, there exists $i(k)\in \{1,\dots,L\}$ such that for any $\hx \in \hat{F}_k$, there exists some positive integer $m$ such that:
\begin{equation}
	\hf^{-m}(V^s_{loc}(\hx)) \cap V^u_{loc}(\hy_{i(k)}) \ne \emptyset
	\label{eq:StableUnstableIntersection}
\end{equation}
where $\Wu_{loc}(\hy_{i(k)})$ contains one of the unstable boundaries of $T_{i(k)}$.

\parbreak\SStep[End of the proof.]\label{step:SI8}
We are now able to conclude. This step uses Sard's theorem \ref{thm:Sard} and is very similar to Steps \estepref{step:HUI2} and \estepref{step:HUI3} in the proof of Proposition \ref{prop:IntersectionsUnstableMaxEntropy}.

Fix then $k\geq k_{\star}$, an integer $i(k) \in \{1,\dots,L\}$, a point $\hx \in \hat{F}_k$ such that $\hx$ satisfies Theorem \ref{thm:LipschitzHolonomiesLamination}, a positive integer $m\geq 0$ and a connected component $V\subset \hf^{-m}(V^s_{loc}(\hx))$ such that $V$ intersects both $T_{i(k)}$ and $M\backslash\overline{T_{i(k)}}$. Projecting by $\pi$, this implies that there exists a curve $\gamma:[0,1]\to M$ such that:
\begin{itemize}[label={--}]
	\item $f^m(\gamma([0,1])) \subset \Ws_{loc}(\hx)$,
	\item $\gamma$ intersects both $T_{i(k)}$ and $M\backslash\overline{T_{i(k)}}$,
	\item $\gamma([0,1])\cap \partial^{u,l} T_{i(k)} \ne \emptyset$ for some $l \in \{1,2\}$.
\end{itemize}
In particular, $\Ws_{loc}(\hx)$ intersects both $f^m(T_{i(k)})$ and $M\backslash\overline{f^m(T_{i(k)})}$. This implies that for any chart small enough sending $\Ws_{loc}(\hx)$ to the horizontal, the curve $f^m(\partial^{u,l}T_{i(k)})$ has a component above $\Ws_{loc}(\hx)$ and a component below $\Ws_{loc}(\hx)$. Let $\kappa:[0,1]\to f^m(\partial^{u,l}T_{i(k)})$ be an embedding which intersects $\Ws_{loc}(\hx)$ and has one component above and one below $\Ws_{loc}(\hx)$ in the previous sense. Since $\hx$ satisfies Theorem \ref{thm:LipschitzHolonomiesLamination}, there exists a compact set $\Lambda_{\hx}$ such that:
\begin{itemize}[label={--}]
	\item$\text{dim}_H\big(\Wu_{loc}(\hx)\cap\Lambda_{\hx}\big)>0$,
	\item $\pi(\hx)$ is accumulated on both sides of $\Wu(\hx)$ by points of $\Lambda_{\hx}$.
\end{itemize}
Up to taking $\epsilon>0$ small enough, the fact that $\pi(\hx)$ is accumulated on both sides of $\Wu(\hx)$ by points of $\Lambda_{\hx}$ implies that the curve $\kappa$ intersects any leaf of the stable lamination $\Ws_{loc}(\Lambda_x\cap B(\hx,\epsilon))$. Now, since $\text{dim}_H\big(\Wu_{\epsilon}(\hx)\cap\Lambda_{\hx}\big)>0$, Theorem \ref{thm:Sard} and the continuity of the stable lamination of $\Lambda_{\hx}$ for the $\cC^1$ topology implies that there exists a set $\Lambda_k$ of positive $\mu_k$-measure such that for any $x \in \Lambda_k$ we have:
\begin{equation*}
	\Ws(x)\pitchfork \Wu_{loc}(\hy_{i(k)}) \ne \emptyset.
\end{equation*}
Using Katok Theorem \ref{prop:HomRelation} and the $\lambda$-lemma \ref{lem:LambdaLemma}, we can show that $\Wu_{loc}(\hy_{i(k)})$, respectively $\Ws_{loc}(\hy_{i(k)})$, is $\cC^1$-accumulated by a sequence of discs belonging to the unstable manifold, respectively the stable manifold, of some hyperbolic saddle periodic orbit $\hat{\cO}_{i(k)}$. We then conclude that for any $x \in \Lambda_k$, we have:
\begin{equation*}
	\Ws(x)\pitchfork\Wu(\hat{\cO}_{i(k)})\ne\emptyset.
\end{equation*}
Note that since the set of $x \in M$ such that $\Ws(x)\pitchfork\Wu(\cO_{i(k)})\ne\emptyset$ is invariant and $\mu_k$ is ergodic, it holds on a set of full $\mu_k$-measure.

We consider the family of hyperbolic saddle periodic orbits $\{\cO_1,\dots,\cO_L\}$. Fix $i \in \{1,\dots,L\}$. We still have to show that $\hnu(\{\hx\in M_f, \hx \simh\hat{\cO}_i\})>0$. Since $\hy_i \in \hat{Y}^{\#}$, it is in the support of $\hnu$ restricted to some Pesin block $\hat{\Lambda}$ and $\hnu(B(\hy_i,\epsilon)\cap \hat{\Lambda})>0$ for any $\epsilon>0$. By continuity of the stable and unstable manifolds in Pesin blocks, for the $\cC^1$ topology, we can find some small $\epsilon>0$ such that any $\hz\in B(\hy_i,\epsilon)\cap \hat{\Lambda}$ is homoclinically related to $\hat{\cO}_i$, which concludes the proof of Proposition \ref{prop:StableUnstableIntersection}.

\section{Proof of the main theorems}\label{sec:ProofMainThm}
This last section is devoted to the proofs of the main theorems presented in the introduction. We begin by proving Theorem \ref{mainthm:A} assuming Theorem \ref{mainthm:B}, as a direct corollary.

\begin{proof}[Proof of Theorem \ref{mainthm:A} assuming Theorem \ref{mainthm:B}]
	By Theorem \ref{thm:CodingMme}, any measured homoclinic class contains at most one ergodic measure of maximal entropy. By Theorem \ref{mainthm:B}, the set of measured homoclinic classes containing a measure of maximal entropy is finite, which concludes the proof.
\end{proof}

We continue with the proof of Theorem \ref{mainthm:B}. We proceed by contradiction, which allows us to find a sequence of ergodic hyperbolic $f$-invariant measures $(\mu_k)_k$ of saddle type, which are not homoclinically related and such that their entropy converges to the topological entropy. The idea is then to apply first Proposition \ref{prop:StableUnstableIntersection} to the sequence $(\mu_k)_k$, which converges to $\nu$. We then obtain the existence of a finite number of hyperbolic periodic orbits such that the stable manifold of $\mu_k$-almost every point intersects transversely the unstable manifold of some of these orbits. Then, Proposition~\ref{prop:IntersectionsUnstableMaxEntropy} will imply that the unstable manifold of $\mu_k$-almost every point intersects transversely the stable manifold of all these hyperbolic periodic orbits. This will imply that for $k$ large enough, the $\mu_k$ are homoclinically related to a finite number of periodic hyperbolic orbits, which is a contradiction.

\begin{proof}[Proof of Theorem \ref{mainthm:B}]
	We proceed by contradiction. Suppose that for any $\epsilon>0$, the set of measured homoclinic classes $\mathcal{M}(\mu)$ containing a measure $\mu^{\prime}\in\mathcal{M}(\mu)$ such that $h(f,\mu^{\prime})>h_{\text{top}}(f)-\epsilon$ is not finite. This allows us to find a sequence of measures $(\mu_k)_k$ such that $\liminf_{k\to+\infty}h(f,\mu_k)=h_{\text{top}}(f)$ and such that for $i\ne j$, the measures $\mu_i$ and $\mu_j$ are not homoclinically related. Denote its lift to $M_f$ by $(\hmu_k)_k$. Up to taking a subsequence, we can assume that $(\hmu_k)_k$ weak-$\star$ converges to a $\hf$-invariant measure $\hnu$. Again, denote its projection to $M$ by $\nu$. Since $f$ is $\cC^{\infty}$, the entropy is upper semi-continuous and $\hnu$ is a measure of maximal entropy. In particular, $\hnu$-almost every ergodic component is a measure of maximal entropy. Since $h_{\text{top}}(f)>\log\deg(f)$, Ruelle's inequality (Proposition \ref{prop:Ruelle}) implies that any measure of maximal entropy is hyperbolic of saddle type. The sequence $(\mu_k)_k$ and its limit $\nu$ satisfy the hypotheses of Propositions \ref{prop:IntersectionsUnstableMaxEntropy} and \ref{prop:StableUnstableIntersection} with $\alpha=1$.
	
	We begin by applying Proposition \ref{prop:StableUnstableIntersection}. We obtain that there exists a finite number of hyperbolic saddle $\hf$-periodic orbits $\hat{\cO}_1,\dots,\hat{\cO}_{L}$ such that the following properties hold.
	\begin{itemize}[label={--}]
		\item There exists an integer $k_0\geq 0$ such that for any $k\geq k_0$ and for $\hmu_k$-almost every $\hx$, the following intersection holds:
		\begin{equation}
			\Ws(\hx)\pitchfork\Wu(\hat{\cO}_j)\ne\emptyset
			\label{eq:ProofMainThmAStable}
		\end{equation}
		for some $j\in\{1,\dots,L\}$.
		\item For any $i \in \{1,\dots,L\}$, the set $\{\hx \in M_f, \ \hx\simh\hat{\cO}_i\}$ has positive $\hnu$-measure.
	\end{itemize}
	We now apply Proposition \ref{prop:IntersectionsUnstableMaxEntropy} for all the orbits $\hat{\cO}_1,\dots,\hat{\cO}_L$, in the trivial case where $\alpha=1$. We obtain that there exist integers $k_1,\dots,k_L\geq 1$ satisfying the following property for any $i\in\{1,\dots,L\}$. For any $k\geq k_i$, for $\hmu_k$-almost every $\hx$, the following intersection holds:
	\begin{equation}
		\Wu(\hx)\pitchfork W^s(\hat{\cO}_i)\ne \emptyset.
		\label{eq:ProofMainThmAUnstable}
	\end{equation}
	
	Take now $k_{\star}=\max(k_0,\dots,k_L)$. Combining \eqref{eq:ProofMainThmAStable} and \eqref{eq:ProofMainThmAUnstable}, we obtain that for any $k\geq k_{\star}$, there exists an integer $i(k)\in \{1,\dots,L\}$ such that $\hmu_k$ is homoclinically related to $\hat{\cO}_{i(k)}$. This concludes the proof of Theorem \ref{mainthm:B}, since there exist $k_1,k_2\geq k_{\star}$ such that $i(k_1)=i(k_2)$ and then $\hmu_{k_1}\simh\hmu_{k_2}$, which is a contradiction.
\end{proof}

Let us now finish with the proof of Theorem \ref{mainthm:C}. It is a corollary of Proposition \ref{prop:UnstableStableIntersection}. The idea is to use Proposition \ref{prop:UnstableStableIntersection} to get topological intersections and conclude with Sard's theorem \ref{thm:Sard}.

\begin{proof}[Proof of Theorem \ref{mainthm:C}]
	Let $(\mu_k)_k$ be a sequence of ergodic hyperbolic measures of saddle type with maximal entropy. Let $(\hmu_k)_k$ be the lifted sequence to $M_f$. Up to taking a subsequence, we can assume that $\hmu_k$ weak-$\star$ converges to a $\hf$-invariant measure $\hnu$. Since $f$ is $\cC^{\infty}$, the entropy is upper semi-continuous and $\hnu$ is then a measure of maximal entropy. In particular, $\hnu$-almost every ergodic component is also a measure of maximal entropy.
	
	\begin{claim}
		There exists $\alpha\in (0,1]$ and two $\hf$-invariant probabilities $\hnu_1$ and $\hnu_0$ such that:
		\begin{itemize}[label={--}]
			\item $\hnu$-almost every ergodic component of $\hnu_1$ is hyperbolic of saddle type,
			\item $\hnu=\alpha\hnu_1+(1-\alpha)\hnu_0$.
		\end{itemize}
		\label{claim:ProofMainThmB}
	\end{claim}
	
	\begin{proof}
		Write $\hnu=\int_{M_f}\hnu_{\hx} \ d\hnu(\hx)$ as the ergodic decomposition of the measure $\hnu$. The Lyapunov exponents of $\hnu$ are defined as:
		\begin{equation*}
			\lambda^{\pm}(\hnu)=\int_{M_f}\lambda^{\pm}(\hnu_{\hx}) \ d\hnu(\hx).
		\end{equation*}
		Using the subadditive ergodic theorem, it is well known that the upper Lyapunov exponent is upper semi-continuous:
		\begin{equation*}
			\limsup_{k\to+\infty}\lambda^+(\hmu_k) \leq \lambda^+(\hnu).
		\end{equation*}
		See \cite[Theorem 3.5]{buzzi2022continuity} for a proof. The following sequence of inequalities then holds:
		\begin{equation*}
			\begin{aligned}
				\lambda^+(\hnu) &\geq \limsup_{k\to +\infty} \lambda^+(\hmu_k)\\
				&\geq \limsup_{k\to+\infty} \lambda^+(\hmu_k)+\lambda^{-}(\hmu_k) = \limsup_{k\to +\infty}\int_{M_f} \log|\det d_{\hx}\hf| \ d\hmu_k(\hx)\\
				&= \int_{M_f} \log|\det d_{\hx}\hf| \ d\hnu(\hx) = \lambda^+(\hnu)+\lambda^-(\hnu)
			\end{aligned}
		\end{equation*}
		where the convergence of the integral holds since $(\hmu_k)_k$ converges weak-$\star$ to $\hnu$ and $\det d_{\hx}\hf$ is continuous. We then get that $\lambda^-(\hnu)\leq0$. But since almost-every ergodic component of $\hnu$ is hyperbolic, we have $\lambda^-(\hnu)<0$. Let $\hat{X}=\{\hx, \ \lambda^-(\hnu_{\hx})<0\}$. The fact that $\lambda^-(\hnu)$ is negative implies that $\hat{X}$ has positive $\hnu$-measure. Denote this measure by $\alpha$. The following two invariant probabilities give us the desired decomposition:
		\begin{equation*}
			\hnu_1=\frac{1}{\alpha}\int_{\hat{X}}\hnu_{\hx} \ d\hnu(\hx) \quad \text{and} \quad \hnu_0=\frac{1}{1-\alpha} \int_{\hat{X}^c} \hnu_{\hx} \ d\hnu(\hx).
		\end{equation*}
	\end{proof}
	
	Let us check that the decomposition given by Claim \ref{claim:ProofMainThmB} satisfies the hypothesis of Proposition \ref{prop:UnstableStableIntersection}. Denote by $\nu,\nu_1$ and $\nu_0$ the projections to $M$ by $\pi$ of the measures $\hnu,\hnu_1$ and $\hnu_0$. Recall that we supposed that $h_{\text{top}}(f)>0$. Recall also that $\hnu$-almost every ergodic component of $\hnu$ is a measure of maximal entropy. We then have that $\hnu$-almost every ergodic component of $\hnu_1$ has positive entropy and is hyperbolic of saddle type. Since $\liminf_{k\to +\infty} h(f,\mu_k)=h_{\text{top}}(f)$ and $\alpha>0$, we also have:
	\begin{equation*}
		h(f,\nu^{\prime}) < \frac{1}{1-\alpha}\liminf_{k\to+\infty} h(f,\mu_k)
	\end{equation*}
	for $\nu$-almost every ergodic component of $\nu_0$. We can then apply Proposition \ref{prop:UnstableStableIntersection} to the sequence $(\mu_k)_k$ and the measure $\nu$. There exists then a finite number of recurrent points $\hy_1,\dots,\hy_L$, which are accumulated on both sides of their unstable manifold by points such that their stable manifolds form a continuous lamination with Lipschitz holonomies (recall the definition of the set $\hat{Y}^{\#}$), and displaying the following property. For $k$ sufficiently large, there exists an integer $i(k)\in \{1,\dots,L\}$ and a set $\hat{X}_k$ of positive $\hmu_k$-measure such that for any $\hx\in \hat{X}_k$, the following intersection holds:
	\begin{equation*}
		\Wu(\hx)\cap \Ws_{loc}(\hy_{i(k)})\ne \emptyset.
	\end{equation*}
	We are going to apply now the same argument as in Step \estepref{step:HUI2} of the proof of Proposition \ref{prop:IntersectionsUnstableMaxEntropy}. Let us explain it again. Since $\hy_{i(k)}$ is accumulated by points of the same Pesin block on both sides of its unstable manifold, any $\Wu(\hx)$, for $\hx\in \hat{X}_k$, intersects a continuous lamination by stable local manifolds with Lipschitz holonomies. Since $f$ is $\cC^{\infty}$, we can apply Sard Theorem \ref{thm:Sard}, which implies that $\Wu(\hx)$ intersects transversely local stable manifolds of points belonging to a set of positive $\hnu_1$-measure. This implies that $\mu_k$ is homoclinically related to some ergodic, hyperbolic measure of saddle type $\hnu_{i(k)}$ which has positive entropy, since $\hnu_{i(k)}$ can be chosen as an ergodic component of $\hnu_1$. This concludes the proof of Theorem \ref{mainthm:C}.
\end{proof}

We conclude this work by explaining why our results extend to the setting of the Corollary in Subsection \ref{maincor}. Let $f$ be a map of a two-dimensional manifold and $K$ be a non-singular compact attractor. Recall that $K$ is a non-singular compact attractor if $K$ is compact, invariant, and there exists an open connected neighborhood $U$ of $K$ such that $f(\overline{U})\subset U$ and $U\cap (\cup_{n\geq 0}f^n(\mathscr{C}))=\emptyset$, where $\mathscr{C}$ is the set of critical points of $f$, and the orbit of any point in $U$ converges to $K$. First, since $U$ is connected, the number of pre-images is constant in $K$ and we denote it by $\deg(f_{|K})$. We can also suppose that for any hyperbolic measure of saddle type $f$-invariant measure supported in $K$, its local stable and unstable manifolds are contained in $U$. The result of Section \ref{sec:UnstableStableIntersection} holds trivially since we only iterate in the future. The results of Section \ref{sec:StableUnstableIntersections} also hold since the local stable manifolds are contained in $U$, which implies that the $n$-th pre-images of any point in such a curve cannot be a critical point, for any $n$. The setting is then the same as in Section \ref{sec:StableUnstableIntersections}. We conclude as in Section \ref{sec:ProofMainThm}.

\printbibliography

\parbreak\begin{center}
	\emph{Mat\'eo Ghezal} (mateoghezalmath@gmail.com)\\
	Laboratoire de Math\'ematiques d'Orsay\\
	CNRS - UMR 8628\\
	Universit\'e Paris-Saclay\\
	Orsay 91405, France.
\end{center}
\end{document}